\documentclass[reqno,oneside]{amsart}
\usepackage[textwidth=16.1cm]{geometry}
\usepackage{mathrsfs}  
\usepackage{cancel} 
\usepackage[dvipsnames]{xcolor}
\usepackage{enumitem}
\usepackage{mathtools}
\usepackage[final]{hyperref}
\usepackage{bm}
\usepackage{subcaption}
\usepackage{comment}
\usepackage{xargs}
\usepackage{soul}
\usepackage{subcaption}
\usepackage{amssymb}
\usepackage{tikz, pgf}
\usetikzlibrary{decorations.pathmorphing}

\usepackage[textsize=tiny,textwidth=2.05cm]{todonotes}
\newcommandx{\task}[2][1=]{\todo[linecolor=blue,backgroundcolor=blue!30,#1]{#2}}
\newcommandx{\note}[2][1=]{\todo[linecolor=green,backgroundcolor=green!30,#1]{#2}}
\newcommandx{\error}[2][1=]{\todo[linecolor=red,backgroundcolor=red!30,#1]{#2}}
\usepackage{marginnote}
\usepackage[dvipsnames]{xcolor}
\usepackage{mathrsfs}
\usepackage{mathtools}
\usepackage{esint}
\usepackage{stackengine}
\usepackage{hyperref}
\usepackage{bbm}

\newtheorem{theorem}{Theorem}[section]
\newtheorem{lemma}[theorem]{Lemma}
\newtheorem{proposition}[theorem]{Proposition}
\newtheorem{corollary}[theorem]{Corollary}

\theoremstyle{definition}
  \newtheorem{definition}[theorem]{Definition}

\theoremstyle{remark}
  \newtheorem{remark}[theorem]{Remark}
  \def\myellipse{(0,0) ellipse (4cm and 1cm)}

\newcommand{\compo}{H}
\newcommand{\dx}{\di x}

\newcommand{\eps}{\varepsilon}
\newcommand{\N}{\mathbb{N}}
\newcommand{\Z}{\mathbb{Z}}
\newcommand{\R}{\mathbb{R}}
\newcommand{\C}{\mathbb{C}}

\newcommand{\defas}{\coloneqq}

\usepackage{stackengine}

\newcommand*{\di}{\mathop{}\!\mathrm{d}}

\DeclareMathOperator{\dist}{dist}
\DeclareMathOperator*{\argmin}{argmin}

\newcommand{\BBB}{\color{black}}

\newcommand{\EEE}{\color{black}}

\newcommand{\MMM}{\color{black}}

\newcommand{\GGG}{\color{black}}

\newcommand{\JJJ}{\color{black}}
\newcommand{\VVV}{\color{black}}
\newcommand{\VIT}{\color{black}}
\newcommand{\JOS}{\color{black}}

\newcommand{\T}{\mathcal{T}}

\newcommand{\mres}{\mathbin{\vrule height 1.6ex depth 0pt width 0.13ex\vrule height 0.13ex depth 0pt width 1.3ex}}

\newcommand{\ol}{\overline}
\newcommand{\sm}{\setminus}

\newcommand{\tsubset}{\, \tilde{\subset} \, }

 \newcommand{\dhn}{\mathrm{d}\mathcal{H}^{d-1}}
 \newcommand{\hn}{\mathcal{H}^{d-1}}
 \newcommand{\dhu}{\mathrm{d}\mathcal{H}^{1}}
 \newcommand{\hu}{\mathcal{H}^{1}}

\DeclarePairedDelimiterX\setof[1]\{\}{#1}
\DeclarePairedDelimiterX\abs[1]\lvert\rvert{#1}
\DeclarePairedDelimiterX\norm[1]\lVert\rVert{#1}
\DeclarePairedDelimiterX\sprod[2]\langle\rangle{#1, #2}

\usepackage{enumitem}
\numberwithin{equation}{section}

\setlist[enumerate,1]{font=\normalfont}
\setlist[itemize,1]{font=\normalfont}
\newlist{thmlist}{enumerate}{1}
\setlist[thmlist]{label=(\roman{thmlisti}),
	ref=(\roman{thmlisti}),font=\normalfont,
	noitemsep}

\title[Adaptive finite element approximation for quasi-static crack growth]{Adaptive finite element approximation  \\ for quasi-static crack growth}


\author[V.~Crismale]{Vito Crismale}
\address[Vito Crismale]{
  Dipartimento di Matematica “G. Castelnuovo”, Sapienza Universita` di Roma, Piazzale A. Moro 2, I-00185, Rome, Italy}
\email{crismale@mat.uniroma1.it}

\author[M.~Friedrich]{Manuel Friedrich} 
\address[Manuel Friedrich]{
  Department of Mathematics \\
  Friedrich-Alexander Universit\"at Erlangen-N\"urnberg \\
   Cauerstr.~11, D-91058 Erlangen, Germany 
}
\email{manuel.friedrich@fau.de}

\author[J.~Seutter]{Joscha Seutter}
\address[Joscha Seutter]{
 Department of Mathematics \\
 Friedrich-Alexander Universit\"at Erlangen-N\"urnberg \\
 Cauerstr.~11, D-91058 Erlangen, Germany
}
\email{joscha.seutter@fau.de}

\begin{document}
\begin{abstract}
  We provide an  adaptive finite element approximation  for a model of quasi-static crack growth  in dimension two. The discrete setting  consists of  integral functionals that are defined on continuous, piecewise affine functions, where the triangulation is a part of the unknown of the problem and adaptive in each minimization step. The limit passage is conducted simultaneously in the vanishing mesh size and discretized time step, and results in an evolution for the  continuum Griffith model of brittle fracture with isotropic surface energy \cite{FriedrichSolombrino} which  is characterized by  an irreversibility condition, a global stability, and an energy balance. Our result corresponds to an evolutionary counterpart of the static $\Gamma$-convergence result in \cite{BonBab} for which, as a byproduct, we provide an alternative proof.

    \end{abstract}
    \maketitle
    
    \section{Introduction}
    
    The fracture behavior of brittle materials has been a central focus of research in mechanical engineering, beginning with {\sc Griffith}'s pioneering work in the 1920s \cite{griffith}. Griffith's theory revolutionized the understanding of crack formation and propagation by framing it as a competition between the elastic bulk energy  of a material and the energy needed to increase the area of the cracked surface. Building on this foundation, {\sc Francfort} and {\sc Marigo}  \cite{frma98} introduced a variational approach to fracture, where displacement fields and crack paths are determined by minimizing  so-called \emph{Griffith energies}. They proposed an evolutionary model in the framework of rate independent systems, called an \EEE \emph{irreversible quasi-static crack evolution}, which is governed by three fundamental principles: irreversibility of the crack, static equilibrium at every time, and an energy balance that ensures that ensures that the process is non-dissipative. \EEE Unlike traditional fracture theories, this framework does not rely on prescribed crack paths and provides a more effective description of crack initiation. The present paper is devoted to an approximation result of such fracture evolutions based on adaptive finite elements. We start by giving a nonexhaustive account on the relevant literature.

    \textbf{Existence of quasi-static crack evolutions:} The mathematical well-posedness of the model from \cite{frma98} was initiated in \cite{DM-Toa} for a  $2d$ antiplane model   with  restrictive assumptions on the crack topology. The topological restrictions have then been removed in the  breakthrough result  \cite{Francfort-Larsen:2003} by passing to the so-called \emph{weak formulation}, i.e., expressing the problem in the functional setting of $SBV$-functions  \cite{Ambrosio-Fusco-Pallara:2000} and replacing the crack by the jump set of the displacement $u$. This study was subsequently generalized to nonlinear elasticity \cite{dMasoFranToad, DalGiac},  including the setting  of non-interpenetration \cite{Lazzaroni}. We also refer to \cite{DM-Toa2, Larsen}  for some results based on local minimization.  All such existence results are derived from solving certain time-incremental problems where one  fundamental  challenge \EEE consists in proving that the static equilibrium property at all times    is conserved in the passage to   time-continuous solutions. The extension of this strategy to the Griffith energy in linearized elasticity  gives rise to several additional difficulties inherent to the presence of symmetrized gradients. Indeed, due to  the lack of Korn's inequality, there is only control on  the symmetric part of the gradient $e(u) = \frac{1}{2}(\nabla u^T + \nabla u)$ but not on the full gradient $\nabla u$, \VVV leading \EEE to an analytically more intricate formulation  in  the larger space of  special   functions of bounded deformation. Only recently, departing from {\sc Dal Maso}'s seminal paper \cite{DM13} on the generalized space $GSBD$,  there have been significant developments  for the analysis of linear Griffith models. We refer the reader to static existence results, both in the weak \cite{AlmiTassssso, CC-JEMS, FriedrichPerugini2} and the strong \cite{Iu3, Crismale-Cham-ex, FLS} \EEE setting.  In \cite{FriedrichSolombrino}, the existence of quasi-static  crack evolutions has been established in dimension two, generalizing the seminal work    \cite{Francfort-Larsen:2003} to the vectorial, geometrically linear setting.     \EEE

  \textbf{Approximation of the Griffith functional:} Due to the presence of unknown surfaces,  minimization problems for the Griffith functional   are    notoriously difficult to be
  solved numerically in an efficient and  robust   way.  Consequently, the  {regularization} of free-discontinuity problems by  {computationally more  viable models} is of fundamental importance. With the aim of rigorously proving convergence of such approximation schemes, one usually resorts to the variational notion of $\Gamma$-convergence  \cite{dalmasoIntroductionGConvergence1993}   that ensures convergence of minimizers. Over the last years, a variety of different ways to approximate the Griffith energy has been proposed, including so-called phase-field approximations where the sharp discontinuity is smoothened into a diffuse crack in terms of an  auxiliary phase-field variable. This well-known approach was not only extensively studied from a mathematical point of view, see \cite{Chambolle:2004, Chambolle-Conti-Francfort-phase,  ChaCri19, Iurlano:13} for recent results in the  linearly elastic setting,     but constitutes also one of the  most popular computational methods for simulating brittle fracture. Practically, this approach is combined with an additional spatial discretization of both variables in terms of  finite-difference or finite-elements  (for rigorous results by means of $\Gamma$-convergence see  \cite{Bach, Belle, Crismale-Scilla}). \EEE One major drawback lies in the fact that the two-parameter approximation in terms of the  diffuse crack and the mesh-size parameter gives rise to a multiscale numerical problem requiring the use of a very thin mesh. Similar issues may appear for approximations by non-local functionals, where the energy density depends on some  `averaged behavior' of $u$ in a neighborhood of vanishing size, see \cite{Marziani, non-local1, Scilla} for results in the  setting of linear fracture models. To avoid dealing with a further approximation parameter, one therefore searches for discrete approximations of single-scale type,  see \cite{Frater07, PodGon21, RanChi2015a}.

    In this paper, we want to focus on one prominent example of those, namely a discrete finite-element approximation that makes use of adaptive mesh-refinements and was first introduced for the Mumford-Shah functional by {\sc Dal Maso and Chambolle} \cite{ChamboDalMaso}.   The main feature of this approach is that it involves an implicit mesh-optimization and hence gives enough flexibility to approximate isotropic crack energies. As described in \cite{chambolleBourdain}, this  is a real advantage as this method does not require to use very fine meshes.    
     In \cite{negri}, this approximation scheme was then generalized to the Griffith functional in linearized elasticity. Yet, in this result, the approximating sequence of functionals still depends on the full gradient $\nabla u$ in an unnatural way. Only recently, {\sc Babadjian \EEE and Bonhomme} \cite{BonBab} succeeded in extending the results from \cite{ChamboDalMaso} to the linear case which is exclusively written in terms of the symmetric gradient $e(u)$ of the displacement  $u$. We refer to \cite[Introduction]{BonBab} for more details and for a review of related discrete-to-continuous approximations.  
     
  \textbf{Evolutionary approximation results:} \EEE  All the aforementioned results are purely static and do not take into account the irreversible nature of crack growth. In fact, the literature on approximation of evolutionary fracture models  is  comparably scarce. In \cite{Giacomini:2005}, a phase-field approximation of quasi-static fracture evolution was provided for the antiplane case, and in the setting of finite  elasticity a result on discontinuous-finite-element approximation is available   \cite{GP1, GP2}. Additionally, crack evolutions have been identified as the effective variational limit of sequences of problems in various settings of applicative relevance such as atomistic models \cite{FriedrichSeutter}, homogenization \cite{GiacPonsi}, linearization \cite{steinke}, or a cohesive-to-brittle passage in the limit of infinite specimen size \cite{Giacomini:2005b}. However, rigorous approximation schemes for quasi-static crack growth in the framework of linearized elasticity are still pending. In the present paper,  we move a first step in this direction by providing  a discrete-to-continuous passage for the adaptive finite element model analyzed in  \cite{BonBab, ChamboDalMaso}.  Besides providing a first  approximation result for an evolutionary Griffith model, we believe that the techniques developed in this paper may also contribute to tackle other relevant approximation schemes in the future,  such as (spatially discretized) phase-field models. \EEE

  \textbf{The present paper:}  Following  \EEE the setting  in \cite{BonBab, ChamboDalMaso}, we call a triangulation of  the reference domain \EEE $\Omega \subset \R^2$ admissible if the angles of all contained triangles exceed a given value $\theta_{0}$ and the size of their edges lies between $\eps$ and a given function $\omega(\eps)>\eps$, with $\omega(\eps) \to 0$ as $\eps \to 0$. We  consider functionals defined on the family $\mathcal{A}_{\eps}(\Omega)$ of continuous functions $u\colon \Omega \to \R^2$ which are piecewise affine on admissible triangulations of $\Omega$, and take the form 
    \begin{equation*}
      E_{\eps}(u)\defas 
      \frac{1}{\eps} \int_{\Omega}  f(\eps \GGG \, \C e(u) \colon e(u)\EEE)  \, {\rm d}x \,.
    \end{equation*} \EEE
  Here, $f$ is a nondecreasing function behaving as the identity near $0$ and as a constant near $\infty$, see  \eqref{f-assump},  and by $\C$ we denote a \EEE fourth order elasticity tensor, bounded from above and below, see \eqref{C-assump}. \EEE  For simplicity, we will mainly focus on the special case $f(t)= t \wedge \kappa $  and   $\C=\mathrm{Id}_{2{\times}2 \times 2 \times 2}$. \EEE For any function $u\in \mathcal{A}_{\eps}(\Omega)$, we regard a triangle $T$ as \VVV `cracked' \EEE if the absolute value of the constant symmetric gradient $|e(u)_{T}|$ on the triangle exceeds the threshold $\sqrt{\kappa/\eps}$. This criterion enables us to set up a time-discrete evolution by incrementally minimizing the energy at discrete time sets $(t^k_\delta)_k \subset [0,T]$ depending on a time-discretization parameter $\delta$. In each step $t^k_\delta$, the energy should depend on  the displacements $(u^j_{\eps,\delta})_{j<k}$ at all previous time steps. To this end,  \EEE we consider
  the union of all \VVV `cracked \EEE triangles' at previous time steps as the discrete crack set  $\Omega_{\eps,\delta, k}^{\rm crack}$, and define the energy
    \begin{equation}\label{hist-depend}
      E_{\eps}(u, (u^j_{\eps,\delta})_{j<k})= \int_{\Omega\sm \Omega_{\eps,\delta, k}^{\rm crack}} |e(u)|^2 + \kappa \frac{|\Omega_{\eps,\delta,k}^{\rm crack}|}{\eps}\,.
    \end{equation}
    This corresponds to the energy  $E_{\eps}$, accounting  additionally for `cracked triangles' at all the previous time steps.
    In order to implement an irreversibility condition, we however have to require additional properties for the admissible triangulations, depending on the  previous time steps. More precisely,  for each $t^{k}_\delta$, we assume that all triangles that have been  `cracked'  at previous time steps $t^{j}_\delta<t^k_\delta$ are contained in the triangulation. This requirement restricts the flexibility in the choice of the mesh and  gives rise to a notion of an increasing crack set on the time-discrete level.  Secondly, we need an additional technical condition related to a background mesh, which is inspired by   the construction of recovery sequences in \cite{ChamboDalMaso}. \EEE


    Starting from an incremental minimization scheme of the history-dependent energy in \eqref{hist-depend}, we obtain a time-discrete evolution that is piecewise constant in time and piecewise affine in space. In our main result (Theorem \ref{maintheorem}), we show that in the simultaneous limit $\delta\to 0, \eps \to 0$ this evolution converges to an irreversible quasi-static crack evolution in the sense of \cite{frma98}. This means that we find a pair $t\mapsto (u(t),K(t))$ for  $t \in [0,T]$, where $u(t)$ lies in the space $GSBD^2(\Omega)$, see \cite{DM13}, and $K(t)$ is a rectifiable set \MMM containing the jump set of $u(t)$, such that  $t\mapsto (u(t),K(t))$ satisfies: \EEE  (a) an irreversibility condition, (b) a global stability at all times (sometimes referred to as unilateral minimality), and (c) an energy balance law, see Definition \ref{main def} for details. As a byproduct,  we get that the energies converge at all times. Moreover,  we show the convergence of crack sets in the sense of $\sigma$-convergence introduced in \cite{GiacPonsi}  (see Definition \ref{def:sigmaconv}) and we obtain strong convergence of the linear strains. Besides providing an approximation result for fracture evolution, our main theorem also provides an alternative proof of the existence result in \cite{FriedrichSolombrino}.

     \textbf{Proof strategy:}  In the proof, we   essentially face three challenges: (1)  we need a suitable compactness argument to identify the continuum crack $K(t)$. (2) This notion needs to be compatible for deriving a  $\liminf$-inequality for the history-dependent energy $E_{\eps}$ \JJJ in \eqref{hist-depend}\EEE. (3) Eventually, we need to show the stability of the unilateral minimality property, which corresponds to constructing a mutual recovery sequence for displacements and crack sets.  We now briefly sketch our strategies to tackle these issues.

  (1) A suitable notion for convergence of crack sets, the so-called \emph{$\sigma$-convergence}, is already available and has been employed successfully for proving the  stability of unilateral minimality   \cite{GiacPonsi}. We adapt this notion to the $GSBD$-setting by resorting to tools developed in \cite{FriPerSol20} which allow to separate effects of bulk and surface energies for linear Griffith functionals,  see  Section~\ref{subsec:comp}.

  (2)  The proof of the lower bound in the $\Gamma$-convergence result  \cite{BonBab} relies on a careful blow-up analysis. As it appears to be incompatible to combine this strategy with $\sigma$-convergence, we approach the problem from a different perspective: we consider the energy as split into a bulk and a surface part, called \emph{displacement-void representation},  where the surface part is essentially of the form $c\mathcal{H}^1(\partial \Omega_{\eps,\delta,k}^{\rm crack})$ for a suitable constant $c>0$, cf.\ \eqref{hist-depend}.  Since  \EEE then the energy can be represented as a pair of function and set, we can resort to recent results on such functionals \cite{Friedrich-Crismale}, obtained in connection with models for material voids in elastically stressed solids. Note that a similar proof strategy was adopted already in \cite{chambolleBourdain} for the Mumford-Shah functional. 
    Roughly speaking, the proof therein relies on the fact that the \VVV `averaged \EEE volume' of the set of `cracked \EEE triangles' can be bounded from below by the boundary of a slightly smaller set of triangles, where the deformation gradient can still be controlled on the removed triangles.  We are able \JJJ to \EEE derive an   analogous statement for the vector-valued case and the symmetric gradient $e(u)$ which provides a sharp lower bound for the boundary of void sets after a suitable modification, see Theorem \ref{generaltheorem}. The original argument in \cite{chambolleBourdain} strongly relies on a scalar-valued displacement field and consists in locally controlling the gradient in  `cracked \EEE triangles',  see  \cite[Remark 3.5]{ChamboDalMaso} or \cite[Introduction]{BonBab}.  Our proof instead is considerably more involved as it is inherently nonlocal and uses some techniques from planar graph theory.  This  theorem on `void modifications' lies at the core of our analysis and, in our opinion, represents the crucial novelty of our work. As a byproduct, this result also allows us to give an alternative short proof of the $\Gamma$-convergence result in \cite{BonBab}, see Theorem \ref{thm:liminf}. \EEE

    (3) The third major issue in the proof of the main result lies in showing the stability of unilateral minimizers.  Here, relying on some arguments from \cite{FriedrichSeutter}, we adapt the \emph{jump-transfer construction} from \cite{Francfort-Larsen:2003} to our discrete setting. One challenge is to find  an admissible triangulation for the construction of a piecewise affine interpolation of the function with  `transferred \EEE jump'.  As this triangulation has to contain the `cracked triangles' of previous time steps, we set up a triangulation by combining the triangulation of the former time step with the construction of \EEE \cite{ChamboDalMaso}. \EEE Since our crack set consists of boundaries of voids, we also have to make sure to choose the correct interpolation of $u$ on these triangles.

    \textbf{Organization of the paper:}  The paper is organized as follows. In Section \ref{sec: gamma} we introduce the finite-element approximation for the Griffith functional and we state the result on modification of voids, which is fundamental for our analysis. With this  at hand, we then give a short proof  of  the \MMM (static)  $\Gamma$-liminf inequality \EEE from \cite{BonBab}. Section \ref{sec: mod} is devoted to the proof of the void modification   by combining results from planar graph theory with extension arguments relying on suitable Korn-type inequalities. Then, in Section \ref{sec: evo}, we introduce the quasi-static adaptive finite-element model.  After setting up  a time-discrete evolution by inductively minimizing the history-dependent energy, we state our main result on the approximation of quasi-static crack growth.  Subsequently, in Section \ref{sec: preps},  we establish some preparatory compactness and semicontinuity results in our functional setting and  recall the notion of $\sigma$-convergence,  together with proofs of the corresponding irreversibility and compactness properties.  Section \ref{sec: main} is devoted to the proof of our main result, where we first  derive properties of the time-discrete evolutions and then pass to the continuum limit. To confirm that the latter   is indeed   a quasi-static crack evolution, we depend critically on the aforementioned stability of unilateral minimality, whose proof  is deferred to Section \ref{sec: stability}.
    
   \textbf{Notation:}   We close the introduction by introducing some notation: The \MMM 2-dimensional \EEE Lebesgue and the $1$-dimensional Hausdorff measures in $\R^2$ are denoted by $\mathcal{L}^2$ or $|\cdot|$ and $\mathcal{H}^1$ respectively. Depending on the context $\, \tilde{\subset} \, $ stands for inclusions up to sets that are negligible with respect to  either $\mathcal{H}^1$ or $\mathcal{L}^2$. The interior \MMM and closure \EEE of a set $A\subset \R^2$ are denoted by ${\rm int}(A)$ \MMM and $\overline{A}$, \EEE   by $\partial^* A$ we denote its reduced boundary, \MMM and by $\chi_A$ the corresponding characteristic function. \EEE  By $A \triangle B$ we indicate the symmetric difference of two sets $A,B \subset \R^2$. \EEE   
    The set of $2\times 2$ matrices is denoted by $\R^{2\times2}$ and the subset of symmetric and skew-symmetric matrices by $\R^{2\times2}_{\rm sym}$, respectively $\R^{2\times2}_{\rm skew}$.
    For an open and bounded set $U\subset\R^2$, 
    we denote by $ L^0(  U; \R^2)$ the $\mathcal{L}^2$-measurable functions from $U$ to $\R^2$. \MMM By $a \wedge b$ we denote the minimum of $a,b \in \R$. \EEE Finally, in the following $C>0$ denotes a universal constant that may change from line to line. 

\section{Finite element approximation of Griffith: The displacement-void approach}\label{sec: gamma}

In this section we revisit the recent $\Gamma$-convergence result \cite{BonBab} on the approximation of the Griffith functional \EEE by adaptive finite elements. We follow a different approach,   based on a \emph{displacement-void representation }of the energy, which will form the basis for  our study for quasi-static fracture evolution starting in Section~\ref{sec: evo}. 

 Let $\Omega\subset \R^2$ be a bounded domain with Lipschitz boundary.  We call a \emph{triangulation} of $\Omega$ a collection of closed triangles that only intersect on common edges or vertices and whose union contains $\Omega$.  The vertices of the triangles are called \emph{nodes} of the triangulation.  For a given angle $\theta_0>0$ and a function $\omega\colon\R^{+}\to \R^{+}$ with $\omega(\eps)\geq 6\eps$ for all $\eps>0$ and $\lim_{\eps\to 0}\omega(\eps)=0$, we denote by $\mathcal{T}_{\eps}(\Omega) := \mathcal{T}_{\eps}(\Omega,\theta_0,\omega)$ all triangulations of $\Omega$ whose edges have length between $\eps$ and $\omega(\eps)$ and whose angles are greater than or equal to $\theta_0$. We say that $u$ is \emph{piecewise affine on a triangulation $\mathbf{T}\MMM \in \mathcal{T}_{\eps}(\Omega)$\EEE} if $u$ is affine on each \MMM triangle \EEE $T\in \mathbf{T}$. The corresponding  constant symmetrized gradient of $u$ on each $T$ is denoted by $e(u)_T$. We then define 
\begin{align}\label{Veps}
\mathcal{A}_{\eps}(\Omega)\defas \{u\colon\Omega \to \R^2\, \GGG \text{continuous}\EEE \colon \, \text{there exists a $\mathbf{T}\in \mathcal{T}_{\eps}(\Omega)$ such that $u$ is  piecewise affine on $\mathbf{T}$}\} \,.
\end{align}
We \EEE will associate to each function $u\in \mathcal{A}_{\eps}(\Omega)$ a possibly non-unique triangulation $\mathbf{T}(u)\in \mathcal{T}_{\eps}(\Omega)$ as the ambiguity in the choice of $\mathbf{T}(u)$ does not pose any issues in the following.  
 

Let $f\colon [0,\infty)\to [0,\infty) $ be a nondecreasing continuous function,   which is differentiable at $0$ and satisfies  
 \begin{equation}\label{f-assump}
 f(0)=0\,,\quad \lim_{t\to 0^{+}}\frac{f(t)}{t}=1\,, \quad \lim_{t\to \infty}f(t)=\kappa 
\end{equation}
 for some $\kappa >0$.  Moreover, let $\C\in\R^{2 \times 2 \times 2 \times 2}$ be \EEE a symmetric fourth order tensor such that  
 \begin{equation}\label{C-assump}
c_1|\xi|^2\leq \C \xi\colon \xi \leq c_2 |\xi|^2 \quad \text{\MMM for all $\xi \in \R^{2 \times 2}_{\rm sym}$\EEE}
 \end{equation}
 for some constants   $0 < c_1 \le c_2$. \EEE
 \EEE We define an energy   $E_{\varepsilon} \colon L^0(\Omega;\R^2)\to \R$ by 
  \begin{equation}\label{energy-static}
    E_{\varepsilon}(u)\defas \begin{cases}
    \frac{1}{\eps} \int_{\Omega}f(\eps  \,\C e(u)\colon e(u)\EEE)\,  {\rm d}x \EEE & \text{if $u\in \mathcal{A}_{\varepsilon}(\Omega)$}\,,\\
    +\infty & \text{if $u\in L^0(\Omega;\R^2)\setminus \mathcal{A}_{\varepsilon}(\Omega)$}\,.
  \end{cases}\end{equation}
  In \cite{BonBab}, the authors showed that $E_{\eps}$ $\Gamma$-converges \EEE to the \emph{Griffith functional} $E\colon L^0(\Omega\VIT; \EEE\R^2)\to \R$ defined by  
\begin{equation}
  E(u)\defas \begin{cases}\int_{\Omega} \C e(u)\colon e(u)\EEE\,   {\rm d}x \EEE + \kappa\sin(\theta_{0})\,\mathcal{H}^1(J_{u}) & \text{if $u \in GSBD^2(\Omega)$}\,,\\ 
  +\infty & \text{otherwise},
  \end{cases}
\end{equation}  
  in terms of the topology induced by measure convergence.  For basic notation and properties of $GSBD$ functions, we refer the reader to \cite{DM13}.   Our first goal is to rederive this result  based on a different technique.

\subsection{Displacement-void representation and modification of voids}  \EEE  

To explain the idea, we consider the special case
\begin{align}\label{eq: special case}
f(t) =  t \wedge \kappa \quad \text{for } t \ge 0  \quad \text{ and } \quad    \C=\mathrm{Id}_{2{\times}2 \times 2 \times 2}. \EEE
\end{align}
For a function $u\in \mathcal{A}_{\eps}(\Omega)$ and a corresponding triangulation $\mathbf{T}(u) \in \mathcal{T}_{\eps}(u)$, we define     $ \mathbf{T}_{\eps}^{\rm big} \EEE (u)\defas \{T \in   {\mathbf{T}} \EEE (u)\colon \eps |e(u)_{T}|^2\geq  \kappa \}$,  corresponding to the set of triangles where $|e(u)_{T}|$ is big. We \EEE define their union in $\Omega$ as
\begin{align}\label{omegacrack}
     \Omega_{\eps}^{\rm big} (u)\EEE\defas \mathrm{int}\Big( \bigcup_{T\in   \mathbf{T}_{\eps}^{\rm big}(u)  } T \Big) \cap \Omega.
\end{align}
Note that $\kappa$ corresponds to the point where $f$ is not differentiable. With this choice, the energy can be split into
\begin{equation}\label{first-lb} 
E_{\eps}(u) =   \sum_{T\in \mathbf{T}(u)\setminus   \mathbf{T}_{\eps}^{\rm big} \EEE (u)} |T\cap \Omega| |e(u)_{T}|^2  \, +\, \kappa  \frac{|  \Omega_{\eps}^{\rm big}(u)\EEE  |}{\eps} \,.\end{equation}
Our idea relies on rewriting this as an energy featuring bulk and surface terms of the form
\begin{align}\label{apprxi}
 E_{\eps}(u) \sim  \int_{\Omega \setminus   \Omega_{\eps}^{\rm big}(u)\EEE} |e(u)|^2 \, {\rm d}x +  \frac{ \kappa \sin \theta_0}{2} \EEE \mathcal{H}^1(\partial  \Omega_{\eps}^{\rm big}(u)\EEE). 
 \end{align}
This will  allow us to directly apply $\Gamma$-convergence results for a class of energies defined on pairs of function-set \cite{Friedrich-Crismale}. Since in \cite{Friedrich-Crismale} the main motivation was a model for material voids inside elastically stressed solids, we call this a \emph{displacement-void representation} of the energy.

However, in general \eqref{apprxi} is not an identity as one can only guarantee $\frac{|  \Omega_{\eps}^{\rm big}\EEE |}{\eps} \ge  \frac{\sin \theta_0}{2} c \EEE  \mathcal{H}^1(\partial  \Omega_{\eps}^{\rm big}\EEE)$ for some $0 < c < 1$. Our first main result states that a sharp lower bound up to an arbitrarily small error can be achieved by a suitable modification of $  \Omega_{\eps}^{\rm big}\EEE$.  In the following,  we say that a set $E$  is induced by  (a subset
of triangles) \EEE $\mathbf{T}_E  \subset \mathbf{T}$ if it is given by the  interior of the union of $\mathbf{T}_E$  intersected with $\Omega$, \EEE i.e.,   
\begin{align}\label{EEE} 
 E=\mathrm{int}\Big( \bigcup_{T\in\mathbf{T}_E} T \Big)  \cap \Omega. \EEE
\end{align}
    
\begin{theorem}[Void modification]\label{generaltheorem}
Let $u \in H^{1}(\Omega;\R^2)$ be a function which is piecewise affine \GGG on a \EEE
 triangulation  
$\mathbf{T} \GGG \in \mathcal{T}_{\eps}(\Omega), \EEE $
and suppose that,  for a given subset  $\mathbf{T}_A\subset \mathbf{T}$ and $A$ induced by $\mathbf{T}_A$, it holds that  
\begin{align}\label{energybound1}
\int_{\Omega \setminus A}  |e(u)|^2 \, {\rm d}x   +   \frac{2}{\eps  \sin\theta_0}| A \EEE |  \le C_0.
  \end{align}
 Then, given $\eta >0$ \EEE  there exist   $A_{\rm mod}$   induced by some subset of triangles \EEE in \GGG $\mathbf{T}$ \EEE and $u_{\rm mod} \in H^1( \Omega_{\eps,\eta}; \EEE \R^2)$ such that
 \begin{align}\label{themainthing0}
    |A_{\rm mod}| \le C_\eta\, \eps \EEE , \quad \quad \quad   | \lbrace u \neq u_{\rm mod} \rbrace  \cap \Omega_{\eps,\eta} \EEE | \le C_\eta\, \eps \EEE    
    \end{align}
and 
\begin{align}\label{themainthing}
 \Vert   e(u_{\rm mod}) \Vert_{L^2(  \Omega_{\eps,\eta}  \EEE  \setminus A_{\rm mod })} \le C_\eta \EEE, \quad \quad \quad  \mathcal{H}^1(\partial A_{\rm mod }     ) \le  \frac{2}{\eps  \sin\theta_0}| A \EEE   |  +  C\eta,   
 \end{align}
 where $C>0$ depends only on $C_0$ in \eqref{energybound1}, $C_\eta$  is a constant times a negative power of $\eta$,
  and $\Omega_{\eps,\eta} := \lbrace x \in \Omega \colon {\rm dist}(x,\partial\Omega) >  2 \EEE \omega(\eps) +  \frac{\eps}{\eta^3} \EEE \rbrace$. \EEE Furthermore, if $\mathcal{C}(A_{\rm mod})$ denotes the set of all connected components of $A_{\rm mod}$, \EEE we have \begin{equation}\label{extra-statement}
  \# \mathcal{C}(A_{\rm mod})\leq C \frac{\eta}{\eps}.
  \end{equation}
\end{theorem}

The result will be proven in Section \ref{sec: mod}. We now show that this allows to obtain a short proof of the $\Gamma$-liminf inequality in \cite{BonBab}.

\subsection{$\Gamma$-convergence}\label{subsec: Gamma-conv}

The upper bound follows from   a density result   (see \cite{ChaCri19,Cortesanietal}) and an explicit construction, relying on the construction of the recovery sequence in \cite{ChamboDalMaso}. The lower bound in \cite{BonBab} is based on a careful blow-up analysis. We provide an alternative proof based on  Theorem~\ref{generaltheorem}.

\begin{theorem}\label{thm:liminf}
  Let $(u_\eps)_{\eps}  \subset  L^{0}(\Omega;\R^2)$ with  $u_{\eps}\to u$ in measure on $\Omega$. Then it holds 
  \begin{align}\label{gammaliminf}
  \liminf_{\eps \to 0} E_{\eps}(u_{\eps})\geq E(u)\,.
  \end{align}
\end{theorem}

\begin{proof} 
It is not restrictive to assume that $E_{\eps}(u_{\eps}) \le C_0$ for some $C_0 >0$.  We start by expressing the energy similarly as in \eqref{first-lb}  for general densities $f$. To this end, given $  R >0$, we define $  \mathbf{T}_\eps^{\rm big} \EEE(u)\subset \mathbf{T}(u)$ by 
\[ \mathbf{T}_\eps^{\rm big} \EEE(u)\defas \{T \in   {\mathbf{T}} \EEE (u)\colon \eps |e(u)_T|^2_{\C} \EEE   \geq \GGG R \EEE\}\,,\]
where for shorthand we set $|e(u)_T|_{\C} := (\C e(u)_T\colon e(u)_T)^{1/2}$.   Correspondingly, we define  the set $  \Omega_{\eps}^{\rm big}(u)\EEE$ as in \eqref{omegacrack}. Since $f$ is non-decreasing, we obtain the estimate 
\begin{equation}\label{first-lb.new} E_{\eps}(u)\geq \sum_{T\in \mathbf{T}(u)\setminus  \mathbf{T}_\eps^{\rm big} \EEE(u)} \frac{|T\cap \Omega|}{\eps}f(\eps |e(u)_T|^2_{\C}) \, +\, f( R \EEE )\, \frac{|   \Omega_{\eps}^{\rm big}(u)\EEE  |}{\eps} \,.\end{equation}
Note that in the general case we will need to  consider the limit $R \to \infty$ whereas in the special case  \eqref{eq: special case} one can fix $R =\kappa$.

We first deal with  the second term in \eqref{first-lb.new}. We let $\eta>0$,  $\eps_0 >0$, \EEE and fix $\Omega_* \subset \subset \Omega$  such that \EEE $\Omega_{\eps,\eta} \supset \Omega_*$ for all $0 < \eps \le \eps_0$,  with $\Omega_{\eps,\eta} = \lbrace x \in \Omega \colon {\rm dist}(x,\partial\Omega) >  2 \EEE \omega(\eps) +  \frac{\eps}{\eta^3} \EEE \rbrace$. \EEE We  apply Theorem~\ref{generaltheorem} for  $u_\eps$ and  $A =  \Omega_{\eps}^{\rm big}(u_{\eps})$ 
to obtain $(u_{\eps})_{\rm mod}$ and $ \Omega^{\rm mod}_\eps (u_\eps)\EEE \defas A_{\rm mod}$. Then, we define $v_{\eps} \colon \Omega_* \to \R^2$ by

 \[v_{\eps}(x)\defas \begin{cases}
    (u_{\eps})_{\rm mod}  (x) \EEE  & \text{if $x\in \Omega_*\setminus  \Omega^{\rm mod}_\eps (u_\eps) $}\,, \\
      0 & \text{if $x\in   \Omega^{\rm mod}_\eps (u_\eps) \cap \Omega_{*}$}\,.
    \end{cases}\] 
Then, in view of the \GGG energy bound and \eqref{themainthing}, \EEE we have 
    \[\sup_{0 < \eps \le \eps_0} \Big( \|e(v_{\eps})\|_{L^2(\Omega_*)}+ \mathcal{H}^1(J_{v_{\eps}}) \Big) <\infty.\]
      Since \GGG $u_\varepsilon\to u$ in measure, and \EEE  $\{ v_{\eps} \neq u_{\eps} \}\leq |  \Omega^{\rm mod}_\eps (u_\eps)\EEE|+ |\{u_{\eps} \neq (u_{\eps})_{\rm mod} \}|\leq  C_\eta \EEE \eps$ by \eqref{themainthing0}, we obtain  $v_{\eps}\to u$ in measure on $\Omega_*$ and then by compactness (see e.g.\  \cite[Theorem 3.5]{Friedrich-Crismale} or \cite{CC-JEMS})  that  $u \in GSBD^2(\Omega_*)$.     Moreover, from $ |   \Omega^{\rm mod}_\eps (u_\eps)\EEE| \leq  C_\eta \EEE \eps$ we have $\chi_{   \Omega^{\rm mod}_\eps (u_\eps)\EEE}\to 0$ in $L^1(\Omega)$. Therefore, we can apply the   \GGG lower semicontinuity result for  surface measures of voids stated in  \cite[Theorem 5.1]{Friedrich-Crismale}. This     together with the second estimate in \eqref{themainthing} gives 
    \begin{equation}\label{liminf-jump}
    \liminf_{\eps \to 0}\frac{|  \Omega_{\eps}^{\rm big}(u_{\eps})\EEE|}{\eps} + C\eta \geq  \liminf_{\eps \to 0} \frac{\sin(\theta_{0})}{2}\mathcal{H}^1(\partial   \Omega^{\rm mod}_\eps (u_\eps)\EEE \cap \Omega_*)\geq \sin(\theta_{0})\mathcal{H}^1(J_{u} \cap \Omega_*)\,.
    \end{equation}
        Now, we prove the lower semicontinuity of the elastic part of the energy. By \eqref{f-assump} and  a   Taylor expansion we obtain,  for $s \ge 0$, \EEE  \[f(s)=f(0)+f'(0)s+\gamma(s)= s+\gamma(s)\,,\] where $\gamma\colon \R^{+}\to \R^{+}$ with $\frac{\gamma(s)}{s}\to 0$ for $s\to 0$. For any subset $\mathbf{T}' \subset \mathbf{T}(u_{\eps}) $ this leads to
           \begin{equation}\label{elastic-split}
      \begin{aligned}
      &\sum_{T\in \mathbf{T}'\setminus  \mathbf{T}_\eps^{\rm big} \EEE(u_{\eps})} \frac{|T\cap \Omega|}{\eps}f(\eps|e(u_\eps)_T|_{\C}^2) \\  &  \quad =  \sum_{T\in \mathbf{T}'\setminus  \mathbf{T}_\eps^{\rm big} \EEE(u_{\eps})} |T\cap \Omega| |e(u_\eps)_T|_{\C}^2 + \sum_{T\in \mathbf{T}'\setminus  \mathbf{T}_\eps^{\rm big} \EEE(u_{\eps})} \frac{|T\cap \Omega|}{\eps} \gamma(\eps|e(u_\eps)_T|_{\C}^2).
      \end{aligned}
    \end{equation}     
    We   define the function $\chi_{\eps}\defas \chi_{[0,\eps^{-1/4})}(  |e(u_\eps)|_{\C} \EEE )$ and observe that,  for $\eps$ small enough, we have $\eps^{-1/4}\leq \sqrt{R} \EEE \eps^{-1/2}$, \EEE i.e., in particular $\chi_{\eps}\leq \chi_{\Omega\setminus   \Omega^{\rm big}_{\eps}(u_{\eps})}\EEE$. Altogether, \MMM from \eqref{elastic-split} \EEE we thus obtain
    \begin{equation}\label{elastic-split2}
      \sum_{T\in \mathbf{T}(u_{\eps})\setminus  \mathbf{T}_\eps^{\rm big} \EEE(u_{\eps})} \frac{|T\cap \Omega|}{\eps}f(\eps|e(u_\eps)_T|_{\C}^2)\geq \int_{\Omega_*} \chi_{\eps} \MMM |e(u_\eps)|_{\C}^2 \EEE\, {\rm d}x \EEE\;+\frac{1}{\eps}\int_{\Omega_*}\, \chi_{\eps}\gamma(\eps \MMM |e(u_\eps)|_{\C}^2)\EEE\,  {\rm d}x \,.
      \end{equation}
  We first deal with the second term in \eqref{elastic-split2}. Note that we can  write \EEE
  \[\frac{1}{\eps}\int_{\Omega_*} \chi_{\eps}\gamma(\eps  |e(u_\eps)|_{\C}^2\EEE)\,  {\rm d}x \EEE  = \EEE \int_{\Omega_*} \chi_{\eps} | e(u_\eps)|_{\C}^2 \EEE \frac{\gamma(\eps|e(u_\eps)|_{\C}^2)}{\eps|e(u_\eps)|_{\C}^2} \,  {\rm d}x \EEE\,.\] 
  By the definition of $\chi_{\eps}$ and $\gamma$ we get that $\chi_{\eps} \frac{\gamma(\eps|e(u_\eps)|_{\C}^2)}{\eps|e(u_\eps)|_{\C}^2}  \EEE $ converges uniformly to zero as $\eps \to 0$. Since by the energy bound $\chi_{\eps}  |e(u_\eps)|_{\C} \EEE $ is bounded in   $L^2(\Omega)$, \EEE we conclude
  \begin{equation}\label{rest-to-zero}
  \lim_{\eps \to 0}\frac{1}{\eps}\int_{\Omega_*}\, \chi_{\eps}\gamma(\eps  |e(u_\eps)|_{\C}^2 \EEE)\,  {\rm d}x \EEE=0\,.\end{equation}
  By the energy bound on $u_{\eps}$ we know that $| \Omega^{\rm big}_{\eps} \EEE(u_{\eps})|\leq \MMM  C \EEE \eps$ and that $  \chi_{\Omega\setminus \Omega^{\rm big}_{\eps}(u_{\eps})} \EEE e(u_{\eps})$ is bounded in  $L^2(\Omega;   \R^{2 \times 2}_{\rm sym} \EEE )$. \EEE We hence conclude that $\chi_{\eps}\to 1$ in measure and   $\chi_\eps e(u_\varepsilon) \rightharpoonup e(u)$ weakly in $L^2(\Omega;  \R^{2 \times 2}_{\rm sym} \EEE)$ by \cite[Theorem 3.5, (3.7)(ii)]{Friedrich-Crismale}.   In view of \eqref{elastic-split2} and \eqref{rest-to-zero}, we obtain
  \begin{equation}\label{liminf-bulk}
  \begin{split}
  \liminf_{\eps \to 0} \sum_{T\in \mathbf{T}(u_{\eps})\setminus  \mathbf{T}_\eps^{\rm big} \EEE(u_{\eps})} \frac{|T\cap \Omega|}{\eps }f(\eps|e(u_\eps)_T|_{\C}^2)&\geq  
 \liminf_{\eps \to 0}\int_{\Omega_*} \chi_{\eps} |e(u_\eps)|_{\C}^2 \EEE \,  {\rm d}x \EEE 
  \geq \int_{\Omega_*} \JJJ |e(u)|_{\C}^2 \EEE \,  {\rm d}x \EEE \,.
\end{split}  
  \end{equation}
  Finally,  combining \EEE  \eqref{first-lb.new}, \eqref{liminf-jump}, and \eqref{liminf-bulk}, we see that, for any   $R>0$, \EEE it holds 
  \[\liminf_{\eps \to 0} E_{\eps}(u_{\eps})\geq \int_{\Omega_*} \JJJ |e(u)|_{\C}^2 \EEE \,  {\rm d}x \EEE + \sin(\theta_{0})f(R)\, \mathcal{H}^1(J_{u} \cap \Omega_*)-C\eta\,.\] Sending $R\to \infty$ and $\eta\to 0$, by  using \eqref{f-assump} and \EEE the arbitrariness of $\Omega_* \subset \subset \Omega$ we conclude $u \in GSBD^2(\Omega)$ and that \eqref{gammaliminf} holds. 
  \end{proof} 


We mention that with this technique we could   prove \EEE also $\Gamma$-convergence under Dirichlet boundary conditions and the convergence of minimizers, as done in \cite[Section 4]{BonBab}.  We omit this here but refer to  Proposition \ref{liminf-ineq'} and Corollary \ref{cor: stability} \EEE later where this is performed in the evolutionary framework.

\section{Void modification: Proof of Theorem \ref{generaltheorem}}\label{sec: mod}

This section is devoted to the proof of Theorem \ref{generaltheorem}. \MMM In this section, we will also show the following corollary on the inclusion of modified void sets. \EEE

\begin{corollary}\label{monotonicity-corollary}
    Let $A^1,A^2 $ be induced by $\mathbf{T}_{A^1}$ and $\mathbf{T}_{A^2}$.  Suppose that $\mathbf{T}_{A^1} \subset \mathbf{T}_{A^2}$. \EEE Then,  the sets $A_{\rm mod}^1$, $ A_{\rm mod}^2$ in Theorem~\ref{generaltheorem} can be chosen such that \EEE  $A_{\rm mod}^1\subset A_{\rm mod}^2$  and $\mathbf{T}^{\rm mod}_{A^1} \subset \mathbf{T}^{\rm mod}_{A^2}$, where $\mathbf{T}^{\rm mod}_{A^j} := \lbrace T \in \mathbf{T}_{A^j} \colon T \subset A_{\rm mod}^j \rbrace$ for $j=1,2$. \EEE  
\end{corollary}

  We start by presenting the main idea of the proof of Theorem \ref{generaltheorem}. \GGG For $T\in \mathbf{T}$, we denote by $\mathcal{N}(T)$ the three nodes (corners) of $T$ \EEE and by   $\partial^j T$, $j=1,2,3$,  the three edges of $T$. Using that $\eps$ is the minimal length on an edge and $\theta_0$ is the minimal interior angle, \EEE an elementary computation shows 
\begin{align}\label{ti}
|T| \ge  \frac{1}{2}\eps \sin\theta_0 \max_{j=1,2,3} \mathcal{H}^1(\partial^j T).  
\end{align}
Fixing ideas, for the moment we assume  for simplicity that the sidelength of all triangles is $\sim \eps$, \MMM and that $A \subset \subset \Omega$.  \EEE Then, by using \eqref{ti} one can estimate
\begin{align}\label{mot1}
 \frac{2}{\eps  \sin\theta_0}|A |  \ge  \mathcal{H}^1(\partial A) -C\eps \# \mathbf{T}_A^{\rm ex},
 \end{align}
where  $\mathbf{T}_A^{\rm ex} \subset  \mathbf{T}_A$ denotes the triangles for which more than one side is `exposed' to $\R^2 \setminus A$, see the triangles highlighted in   Figure~\ref{fig:touching-components}.  Now, \JJJ to validate the second inequality in \eqref{themainthing} \EEE it would be enough to show that $\#\mathbf{T}_A^{\rm ex} \le C\eta/\eps$, which however in general does not hold. Another option  is to `heal' the triangles $\mathbf{T}_A^{\rm ex}$, i.e., we define $ A_{\rm mod} = A \setminus \bigcup_{T \in \mathbf{T}^{\rm ex}_A} T$ and observe that indeed it holds
$$ \frac{2}{\eps  \sin\theta_0}|A |  \ge  \mathcal{H}^1(\partial A_{\rm mod }).$$
Yet, in order to do so, we need to ensure that  
$$ \Vert   e(u) \Vert_{L^2(\Omega  \setminus A_{\rm mod })} \le C  \Vert   e(u) \Vert_{L^2(\Omega  \setminus A)}.   $$  
Such a strategy has been implemented in \cite{chambolleBourdain} for a scalar-valued problem. In the present vectorial setting, however, this procedure only works partially. For \VVV `good' \EEE exposed triangles $T \in \mathbf{T}_A^{\rm ex, good}$ (highlighted dark blue in Figure \ref{fig:touching-components}), one can show that $\Vert e(u) \Vert_{L^2(T)} \le  C \EEE \Vert e(u) \Vert_{L^2(N_T)}$ for a suitable neighborhood $N_T \subset \R^2 \setminus A$, see Lemma \ref{healing1}  below.  For \VVV `bad' \EEE exposed triangles $T \in \mathbf{T}_A^{\rm ex, bad}$ (highlighted in light blue in Figure \ref{fig:touching-components}), in contrast to \cite{chambolleBourdain}, the \MMM fact that only the symmetric gradient  $e(u)$ is controlled prohibits  \EEE  to obtain a similar estimate. 

Therefore, roughly speaking, our argument will feature both: (1) healing of triangles in $\mathbf{T}_A^{\rm ex, good}$ and (2) estimating the number $ \# \EEE \mathbf{T}_A^{\rm ex, bad}$ in terms of $C\eta/\eps$ which allows to obtain a small error in \eqref{mot1}. The latter counting argument will borrow some arguments from planar graph theory and will be based on healing small components, see Lemma \ref{heal entire com}.



\subsection{Preparations} \label{se: subset prep}
We start by introducing some notions. First, we define the \emph{saturation} of a connected set $Z \subset \R^2$, denoted by  ${\rm sat}(Z)$, as ${\rm int}(  \overline{Z} \EEE \cup h_Z)$, where $h_Z$ denotes  the union of the bounded connected components of $  \R^2\sm \overline{Z}$. Note that ${\rm sat}(Z)$ arises from $Z$ by `filling its holes'. We call a connected set $Z$ \emph{saturated} if it holds ${\rm sat}(Z)=Z$. \EEE In the following, we consider generic sets $\compo$ of the form  \eqref{EEE}. \VIT We extend the notation for specific triangles introduced for $A$ to a generic set $H$ of the form  \eqref{EEE}. Whenever $H$, $K$ are of the form  \eqref{EEE}, with $H \subset K$, we set $\mathbf{T}_H^{\textbf{•}}:=\mathbf{T}_K^{\textbf{•}} \cap \mathbf{T}_H$, with $\textbf{•}$ a standpoint for $\rm ex$, $\rm ex,bad$, $\rm ex,good$.  \EEE \\
\noindent \textbf{Graph related to a set $\compo$:} 
We denote the (open) connected components of $\compo$  
by $\mathcal{C}(\compo) = \lbrace \compo_1, \ldots, \compo_n\rbrace$. (Here, note that the connected components of $\compo$ are in general different from the ones of $\overline{\compo}$, see e.g. Figure~\ref{fig:touching-components}.) We introduce a graph related to $\compo$.  We define the \emph{vertices} $\mathcal{V}(\compo)$ and the \emph{edges} $\mathcal{E}(\compo)$ of the graph as
\begin{equation*}
\mathcal{V}(\compo):=\{ v \in \mathcal{N}(T)\colon T \in \mathbf{T}_\compo,\,  \ v \in \bigcup\nolimits_{j=1}^n \partial \compo_j\}, \quad \mathcal{E}(\compo):=\big\{\partial^i T \colon T \in \mathbf{T}_\compo,\, \partial^i T \subset \bigcup\nolimits_{j=1}^n \partial \compo_j\big\}.
\end{equation*} 
Note that any  saturated \EEE connected component $\compo_j$ with $\partial \compo_j \cap \partial \compo_k=\emptyset$ for all $k \neq j$ is represented by a closed cycle where each vertex has exactly two edges. Whenever for two components $\compo_j$ and $\compo_k$ the boundaries $\partial \compo_j$ and $ \partial \compo_k$ have a nonempty intersection, this is related to a vertex with four edges, see Figure \ref{fig:touching-components}. More generally, if $l$ different connected components meet at a vertex, this vertex has $2l$  edges. Let us denote
 \begin{equation}\label{1408240018}
 \mathcal{V}_{2l}(\compo):=\{ v \in \mathcal{V}(\compo) \colon n(v) = 2l\}, \quad \text{for } n(v):=\# \{S \in \mathcal{E}(\compo)\colon v \in S\}.
 \end{equation} 

\begin{figure}[h]
  \centering

\begin{tikzpicture}[scale=0.3]


  \filldraw[fill=blue!40!white, draw=white]{} (8,0)++(60:4) -- ++(60:4) -- ++(-60:4)-- cycle;
  \draw[color=white] (8,0)++(60:4) -- ++(60:4) -- ++(-60:4)-- cycle;

  \filldraw[fill=gray!60!white, draw=white]{} (8,0)++(60:4) ++(60:4) -- ++(4,0) -- ++ (-120:4) -- cycle;
  \draw[color=white] (8,0)++(60:4) ++(60:4) -- ++(4,0) -- ++ (-120:4) -- cycle; 

  \filldraw[fill=gray!60!white, draw=white]{} (12,0)++(60:4) -- ++(60:4) -- ++(-60:4)-- cycle;
  \draw[color=white] (12,0)++(60:4) -- ++(60:4) -- ++(-60:4)-- cycle;

  \filldraw[fill=gray!60!white, draw=white]{} (12,0)++(60:4) ++(60:4) -- ++(4,0) -- ++ (-120:4) -- cycle;
  \draw[color=white] (12,0)++(60:4) ++(60:4) -- ++(4,0) -- ++ (-120:4) -- cycle; 
  
  \filldraw[fill=blue!40!white, draw=white]{} (16,0)++(60:4) -- ++(60:4) -- ++(-60:4)-- cycle;
  \draw[color=white] (16,0)++(60:4) -- ++(60:4) -- ++(-60:4)-- cycle;

  \filldraw[fill=blue!70!white, draw=white]{} (0,0) -- ++(60:4) -- ++(-60:4)-- cycle;
  \draw[color=white] (0,0) -- ++(60:4) -- ++(-60:4)-- cycle;

   \filldraw[fill=gray!60!white, draw=white]{} (0,0) ++(60:4) -- ++(4,0) -- ++ (-120:4) -- cycle;
  \draw[color=white] (0,0) ++(60:4) -- ++(4,0) -- ++ (-120:4) -- cycle; 

  \filldraw[fill=gray!60!white, draw=white]{} (4,0) -- ++(60:4) -- ++(-60:4)-- cycle;
  \draw[color=white] (4,0) -- ++(60:4) -- ++(-60:4)-- cycle;
  
   \filldraw[fill=blue!40!white, draw=white]{} (4,0) ++(60:4) -- ++(4,0) -- ++ (-120:4) -- cycle;
  \draw[color=white] (4,0) ++(60:4) -- ++(4,0) -- ++ (-120:4) -- cycle; 

\filldraw[fill=blue!40!white, draw=white]{} (20,0) ++(60:4) -- ++(4,0) -- ++ (-120:4) -- cycle;
  \draw[color=white] (20,0) ++(60:4) -- ++(4,0) -- ++ (-120:4) -- cycle; 

  \filldraw[fill=blue!70!white, draw=white]{} (24,0) -- ++(60:4) -- ++(-60:4)-- cycle;
  \draw[color=white] (24,0) -- ++(60:4) -- ++(-60:4)-- cycle;
  \node[scale=1.2, color=blue!55!white] at (3,5) {$H$};
  
  \node[scale=1.2, color=blue!40!white] at (24,5.5) {$\mathbf{T}_H^{\rm ex, bad}  $};  
  \node[scale=1.2, color=blue!70!white] at (30,2.5) {$\mathbf{T}_H^{\rm ex, good}  $};

\end{tikzpicture}
\begin{tikzpicture}[scale=0.3]
  \hspace{-2.3em}

  \filldraw[red] (8,0)+(60:4) circle (4pt);
  \filldraw[black] (8,0)++(60:4) ++ (60:4) circle (4pt);
  \filldraw[black] (8,0)++(60:4)+(4,0) circle (4pt);
  \draw[color=black] (8,0)++(60:4) -- ++(60:4) -- ++(4,0) -- ++ (4,0) -- ++ (-60:4)-- cycle;

   \filldraw[black] (12,0)+(60:4) circle (4pt);
  \filldraw[black] (12,0)++(60:4) ++ (60:4) circle (4pt);
  \filldraw[black] (12,0)++(60:4)+(4,0) circle (4pt);

     \filldraw[black] (16,0)+(60:4) circle (4pt);
  \filldraw[black] (16,0)++(60:4) ++ (60:4) circle (4pt);
  \filldraw[red] (16,0)++(60:4)+(4,0) circle (4pt);




   \filldraw[black] (0,0) circle (4pt);
  \filldraw[black] (0,0) ++ (60:4) circle (4pt);
  \filldraw[black] (0,0)+(4,0) circle (4pt);
 
  \draw[color=black] (0,0) -- ++(60:4)-- ++(4,0)--++(4,0) -- ++ (-120:4) -- cycle;
  

    \filldraw[black] (4,0) circle (4pt);
  \filldraw[black] (4,0) ++ (60:4) circle (4pt);
  \filldraw[black] (4,0)+(4,0) circle (4pt);

  \draw[color=black] (20,0) ++(60:4) -- ++(4,0) ; 
  \draw[color=black] (20,0) ++(60:4) -- ++(-60:4) ;
  \filldraw[black] (24,0) circle (4pt);
  \filldraw[black] (24,0) ++ (60:4) circle (4pt);
  \filldraw[black] (24,0)+(4,0) circle (4pt);

  \draw[color=black] (24,0)  ++(60:4) -- ++(-60:4)-- (24,0);
  \node[scale=1, color=black] at (3,5) {$(\mathcal{V}(\compo),\mathcal{E}(\compo))$};


\end{tikzpicture}
\caption{An example for \EEE $\compo$ consisting of three components and the corresponding graph $(\mathcal{V}(\compo),\mathcal{E}(\compo))$. Note that $\overline{\compo}$, instead,  is connected. The triangles highlighted in blue are part of $\mathbf{T}^{\rm ex}_{H} = \mathbf{T}_H^{\rm ex, good} \cup \mathbf{T}_H^{\rm ex, bad} \EEE$. The vertices depicted in red are part of $\mathcal{V}_{4}(\compo)$ whereas the vertices depicted in black are all part of $\mathcal{V}_{2}(\compo)$. \GGG $\mathcal{E}(\compo)$ is the set of edges of triangles contained in $\partial \compo$. Here and in the following figures, we always use subsets of a regular triangular lattice for illustration purposes. 
} \label{fig:touching-components}
\end{figure}
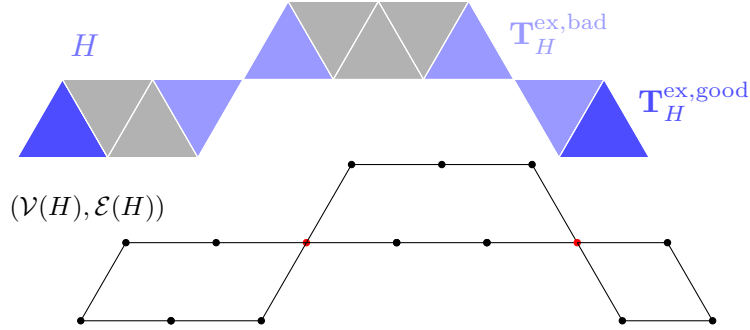
 We can understand $(\mathcal{V}(\compo), \mathcal{E}(\compo))$ as a planar graph. Denoting by $\mathcal{F}(\compo)$ the bounded faces of this graph \GGG (i.e., the bounded planar regions delimited by edges in $\mathcal{E}(\compo)$), \EEE we observe that each component in $\mathcal{C}(\compo)$ corresponds to such a face, and additional faces come from the bounded connected components of $\R^2 \setminus \overline{\compo}$, see Figure~\ref{fig:self-intersection}. \EEE

We recall the \emph{Euler formula} for planar graphs:
\begin{align}\label{Euler}
\#\mathcal{V}(\compo) - \#\mathcal{E}(\compo) + \#\mathcal{F}(\compo) = \nu(\mathcal{V}(\compo), \mathcal{E}(\compo)), 
\end{align}
where $\nu(\mathcal{V}(\compo), \mathcal{E}(\compo))$ denotes the number of connected components of the planar graph $(\mathcal{V}(\compo), \mathcal{E}(\compo))$.

We cover $\partial \compo$ by cycles in the graph. To this end,  fix $\compo_i\JJJ \in \mathcal{C}(H)\EEE$. 
 For two vertices $v_1,v_2\in \mathcal{V}(\compo) \cap \partial \compo_i$, we denote by $\overline{v_1\,v_2}$ the segment with endpoints $v_1$ and $v_2$. We now consider vertices $v_i^j \in \mathcal{V}(\compo) \cap \partial \compo_i$ that fulfill $\{\overline{v_i^j \, v_i^{j+1}}\}_{j=1}^{J-1}  \subset \mathcal{E}(\compo)\,
$ and $ \overline{ v_i^J \, v_i^{1}} \subset \mathcal{E}(\compo)$. If $\partial \compo_i$ is connected, we can choose a tuple in such a way that  the points in the tuple $(v_i^1, \ldots, v_i^J)$ coincide with  $ \mathcal{V}(\compo) \cap \partial \compo_i$. Note that $(v_i^1, \ldots, v_i^J)$  can contain the same vertex multiple times, which corresponds to additional cycles in the graph,  see the first example in Figure \ref{fig:self-intersection}. If $\partial \compo_i$ consists of several connected components (see e.g.\ the second example in Figure \ref{fig:self-intersection}), we repeat the argument for each component, and for simplicity collect all tuples in a single tuple, still denoted by $(v_i^1, \ldots, v_i^J)$. \EEE  For $l\geq 0$, we define 
\begin{equation}\label{Dll}
  \mathcal{D}_l(\compo):=\Big\{\compo_i \in \mathcal{C}(\compo)\colon  \#\lbrace j =1,\ldots, J \colon n(v_i^j) \ge 4 \rbrace  \EEE = l\Big\} \,.
  \end{equation}\EEE 


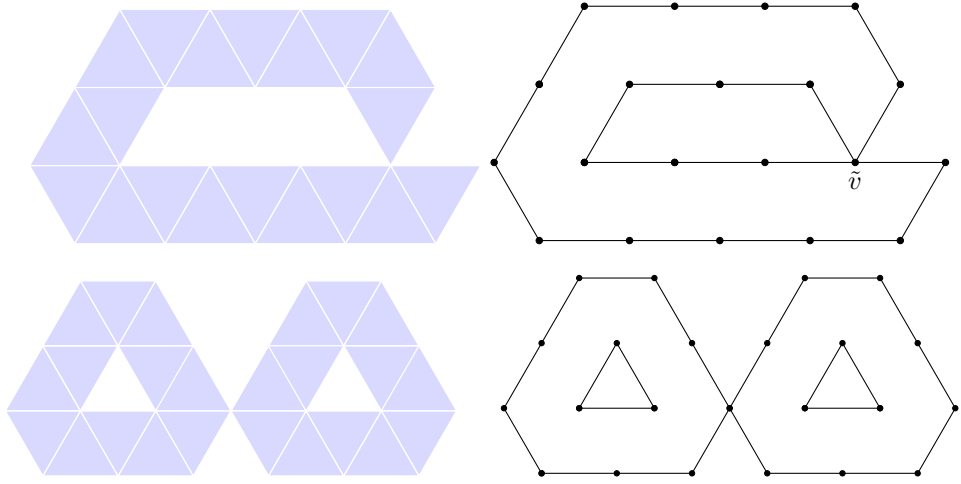
\begin{figure}[h]
\begin{center}
  \begin{tikzpicture}[scale=0.3]
      \filldraw[fill=blue!15!white, draw=white]{} (4,0)++(60:4) -- ++(60:4) -- ++(-60:4)-- cycle;
      \draw[color=white] (4,0)++(60:4) -- ++(60:4) -- ++(-60:4)-- cycle;
  
      \filldraw[fill=blue!15!white, draw=white]{} (4,0)++(60:4) ++(60:4) -- ++(4,0) -- ++ (-120:4) -- cycle;
      \draw[color=white] (4,0)++(60:4) ++(60:4) -- ++(4,0) -- ++ (-120:4) -- cycle; 
  
      \filldraw[fill=blue!15!white, draw=white]{} (8,0)++(60:4) -- ++(60:4) -- ++(-60:4)-- cycle;
      \draw[color=white] (8,0)++(60:4) -- ++(60:4) -- ++(-60:4)-- cycle;
  
      \filldraw[fill=blue!15!white, draw=white]{} (8,0)++(60:4) ++(60:4) -- ++(4,0) -- ++ (-120:4) -- cycle;
      \draw[color=white] (8,0)++(60:4) ++(60:4) -- ++(4,0) -- ++ (-120:4) -- cycle; 
  
      \filldraw[fill=blue!15!white, draw=white]{} (12,0)++(60:4) -- ++(60:4) -- ++(-60:4)-- cycle;
      \draw[color=white] (12,0)++(60:4) -- ++(60:4) -- ++(-60:4)-- cycle;
  
      \filldraw[fill=blue!15!white, draw=white]{} (12,0)++(60:4) ++(60:4) -- ++(4,0) -- ++ (-120:4) -- cycle;
      \draw[color=white] (12,0)++(60:4) ++(60:4) -- ++(4,0) -- ++ (-120:4) -- cycle; 
      
      \filldraw[fill=blue!15!white, draw=white]{} (16,0)++(60:4) -- ++(60:4) -- ++(-60:4)-- cycle;
      \draw[color=white] (16,0)++(60:4) -- ++(60:4) -- ++(-60:4)-- cycle;

  
  
      \filldraw[fill=blue!15!white, draw=white]{} (4,0) -- ++(60:4) -- ++(-60:4)-- cycle;
      \draw[color=white] (4,0) -- ++(60:4) -- ++(-60:4)-- cycle;
      
       \filldraw[fill=blue!15!white, draw=white]{} (4,0) ++(60:4) -- ++(4,0) -- ++ (-120:4) -- cycle;
      \draw[color=white] (4,0) ++(60:4) -- ++(4,0) -- ++ (-120:4) -- cycle; 
  
  \filldraw[fill=blue!15!white, draw=white]{} (16,0) ++(60:4) -- ++(4,0) -- ++ (-120:4) -- cycle;
      \draw[color=white] (16,0) ++(60:4) -- ++(4,0) -- ++ (-120:4) -- cycle; 
  
  
      \filldraw[fill=blue!15!white, draw=white]{} (0,0)++(-60:4) ++(60:4) -- ++(4,0) -- ++ (-120:4) -- cycle;
      \draw[color=white] (0,0)++(-60:4) ++(60:4) -- ++(4,0) -- ++ (-120:4) -- cycle;
  
   \filldraw[fill=blue!15!white, draw=white]{} (4,0)++(-60:4) -- ++(60:4) -- ++(-60:4)-- cycle;
      \draw[color=white] (4,0)++(-60:4) -- ++(60:4) -- ++(-60:4)-- cycle;
  
      \filldraw[fill=blue!15!white, draw=white]{} (4,0)++(-60:4) ++(60:4) -- ++(4,0) -- ++ (-120:4) -- cycle;
      \draw[color=white] (4,0)++(-60:4) ++(60:4) -- ++(4,0) -- ++ (-120:4) -- cycle;

        \filldraw[fill=blue!15!white, draw=white]{} (8,0)++(-60:4) ++(60:4) -- ++(4,0) -- ++ (-120:4) -- cycle;
      \draw[color=white] (8,0)++(-60:4) ++(60:4) -- ++(4,0) -- ++ (-120:4) -- cycle;
  
   \filldraw[fill=blue!15!white, draw=white]{} (8,0)++(-60:4) -- ++(60:4) -- ++(-60:4)-- cycle;
      \draw[color=white] (8,0)++(-60:4) -- ++(60:4) -- ++(-60:4)-- cycle;
  
      \filldraw[fill=blue!15!white, draw=white]{} (12,0)++(-60:4) ++(60:4) -- ++(4,0) -- ++ (-120:4) -- cycle;
      \draw[color=white] (12,0)++(-60:4) ++(60:4) -- ++(4,0) -- ++ (-120:4) -- cycle; 
  
   \filldraw[fill=blue!15!white, draw=white]{} (12,0)++(-60:4) -- ++(60:4) -- ++(-60:4)-- cycle;
      \draw[color=white] (12,0)++(-60:4) -- ++(60:4) -- ++(-60:4)-- cycle;

      \filldraw[fill=blue!15!white, draw=white]{} (16,0)++(-60:4) ++(60:4) -- ++(4,0) -- ++ (-120:4) -- cycle;
      \draw[color=white] (16,0)++(-60:4) ++(60:4) -- ++(4,0) -- ++ (-120:4) -- cycle; 

      \filldraw[fill=blue!15!white, draw=white]{} (16,0)++(-60:4) -- ++(60:4) -- ++(-60:4)-- cycle;
      \draw[color=white] (16,0)++(-60:4) -- ++(60:4) -- ++(-60:4)-- cycle;
  \end{tikzpicture}
  \begin{tikzpicture}[scale=0.3]
      
      \filldraw[black] (4,0)+(60:4) circle (4pt);
      \filldraw[black] (4,0)++(60:4) ++ (60:4) circle (4pt);

      \filldraw[black] (8,0)+(60:4) circle (4pt) ;
      \filldraw[black] (8,0)++(60:4) ++ (60:4) circle (4pt) ;
      \filldraw[black] (8,0)++(60:4)+(4,0) circle (4pt) ;
  
       \filldraw[black] (12,0)+(60:4) circle (4pt);
      \filldraw[black] (12,0)++(60:4) ++ (60:4) circle (4pt);
      \filldraw[black] (12,0)++(60:4)+(4,0) circle (4pt);
  
         \filldraw[black] (16,0)+(60:4) circle (4pt);
      \filldraw[black] (16,0)++(60:4) ++ (60:4) circle (4pt);
      \filldraw[black] (16,0)++(60:4)+(4,0) circle (4pt);
  
  


  
     

  
      \filldraw[black] (4,0) circle (4pt);
      \filldraw[black] (4,0) ++ (60:4) circle (4pt);
      \filldraw[black] (4,0)+(4,0) circle (4pt);
      
  


      \filldraw[black] (4,0)++(-60:4)  circle (4pt);
      \filldraw[black] (4,0)++(-60:4)  ++ (60:4) circle (4pt);
      \filldraw[black] (4,0)++(-60:4) ++(60:4)+(4,0) circle (4pt);


     \filldraw[black] (8,0)++(-60:4)  circle (4pt);
      \filldraw[black] (8,0)++(-60:4)  ++ (60:4) circle (4pt);
      \filldraw[black] (8,0)++(-60:4) ++(60:4)+(4,0) circle (4pt);


     \filldraw[black] (12,0)++(-60:4)  circle (4pt);
  
  
     \filldraw[black] (16,0)++(-60:4)  circle (4pt);
      \filldraw[black] (16,0)++(-60:4)  ++ (60:4) circle (4pt)node[anchor=north]{$\tilde{v}$ };


      \filldraw[black] (20,0)++(-60:4)  circle (4pt) ; 
      \filldraw[black] (20,0)++(-60:4)  ++ (60:4)  circle (4pt) ;


      \draw[color=black] (4,0)++(60:4) -- ++(60:4) -- ++(4,0) -- ++(4,0) -- ++(4,0) -- ++(-60:4) -- ++(-120 :4) -- ++ (120:4) -- ++(-4,0) -- ++(-4,0) -- ++(-120:4) -- ++ (4,0) -- ++(4,0)-- ++(4,0)-- ++(4,0) -- ++ (-120:4) -- ++ (-4,0) -- ++ (-4,0) -- ++ (-4,0) -- ++ (-4,0) -- ++ (120:4)-- ++ (60:4);
      
  \end{tikzpicture}
\end{center}
\vspace*{0.3cm}
  \begin{tikzpicture}[scale=0.25]
    \filldraw[fill=blue!15!white, draw=white]{} (4,0)++(60:4) -- ++(60:4) -- ++(-60:4)-- cycle;
      \draw[color=white] (4,0)++(60:4) -- ++(60:4) -- ++(-60:4)-- cycle;
  
      \filldraw[fill=blue!15!white, draw=white]{} (4,0)++(60:4) ++(60:4) -- ++(4,0) -- ++ (-120:4) -- cycle;
      \draw[color=white] (4,0)++(60:4) ++(60:4) -- ++(4,0) -- ++ (-120:4) -- cycle; 
  
      \filldraw[fill=blue!15!white, draw=white]{} (8,0)++(60:4) -- ++(60:4) -- ++(-60:4)-- cycle;
      \draw[color=white] (8,0)++(60:4) -- ++(60:4) -- ++(-60:4)-- cycle;


\filldraw[fill=blue!15!white, draw=white]{} (16,0)++(60:4) -- ++(60:4) -- ++(-60:4)-- cycle;
  \draw[color=white] (16,0)++(60:4) -- ++(60:4) -- ++(-60:4)-- cycle;
  
\filldraw[fill=blue!15!white, draw=white]{} (16,0)++(60:4) ++(60:4) -- ++ (4,0) -- ++ (-120:4) -- cycle;
\draw[color=white] (16,0)++(60:4) ++(60:4) -- ++(4,0) -- ++ (-120:4) -- cycle; 
  
\filldraw[fill=blue!15!white, draw=white]{} (20,0)++(60:4) -- ++(60:4) -- ++(-60:4)-- cycle;
\draw[color=white] (20,0)++(60:4) -- ++(60:4) -- ++(-60:4)-- cycle;

  
  
      \filldraw[fill=blue!15!white, draw=white]{} (4,0) -- ++(60:4) -- ++(-60:4)-- cycle;
      \draw[color=white] (4,0) -- ++(60:4) -- ++(-60:4)-- cycle;
      
       \filldraw[fill=blue!15!white, draw=white]{} (4,0) ++(60:4) -- ++(4,0) -- ++ (-120:4) -- cycle;
      \draw[color=white] (4,0) ++(60:4) -- ++(4,0) -- ++ (-120:4) -- cycle; 
  
  \filldraw[fill=blue!15!white, draw=white]{} (8,0) ++(60:4) -- ++(4,0) -- ++ (-120:4) -- cycle;
      \draw[color=white] (8,0) ++(60:4) -- ++(4,0) -- ++ (-120:4) -- cycle; 
  
     \filldraw[fill=blue!15!white, draw=white]{} (12,0) -- ++(60:4) -- ++(-60:4)-- cycle;


\filldraw[fill=blue!15!white, draw=white]{} (16,0) -- ++(60:4) -- ++(-60:4)-- cycle;
\draw[color=white] (16,0) -- ++(60:4) -- ++(-60:4)-- cycle;

 \filldraw[fill=blue!15!white, draw=white]{} (16,0) ++(60:4) -- ++(4,0) -- ++ (-120:4) -- cycle;
\draw[color=white] (4,0) ++(60:4) -- ++(4,0) -- ++ (-120:4) -- cycle; 

\filldraw[fill=blue!15!white, draw=white]{} (20,0) ++(60:4) -- ++(4,0) -- ++ (-120:4) -- cycle;
\draw[color=white] (20,0) ++(60:4) -- ++(4,0) -- ++ (-120:4) -- cycle; 

\filldraw[fill=blue!15!white, draw=white]{} (24,0) -- ++(60:4) -- ++(-60:4)-- cycle;

      \filldraw[fill=blue!15!white, draw=white]{} (0,0)++(-60:4) ++(60:4) -- ++(4,0) -- ++ (-120:4) -- cycle;
      \draw[color=white] (0,0)++(-60:4) ++(60:4) -- ++(4,0) -- ++ (-120:4) -- cycle;
  
   \filldraw[fill=blue!15!white, draw=white]{} (4,0)++(-60:4) -- ++(60:4) -- ++(-60:4)-- cycle;
      \draw[color=white] (4,0)++(-60:4) -- ++(60:4) -- ++(-60:4)-- cycle;
  
      \filldraw[fill=blue!15!white, draw=white]{} (4,0)++(-60:4) ++(60:4) -- ++(4,0) -- ++ (-120:4) -- cycle;
      \draw[color=white] (4,0)++(-60:4) ++(60:4) -- ++(4,0) -- ++ (-120:4) -- cycle;

        \filldraw[fill=blue!15!white, draw=white]{} (8,0)++(-60:4) ++(60:4) -- ++(4,0) -- ++ (-120:4) -- cycle;
      \draw[color=white] (8,0)++(-60:4) ++(60:4) -- ++(4,0) -- ++ (-120:4) -- cycle;
  
   \filldraw[fill=blue!15!white, draw=white]{} (8,0)++(-60:4) -- ++(60:4) -- ++(-60:4)-- cycle;
      \draw[color=white] (8,0)++(-60:4) -- ++(60:4) -- ++(-60:4)-- cycle;
  
\filldraw[fill=blue!15!white, draw=white]{} (12,0)++(-60:4) ++(60:4) -- ++(4,0) -- ++ (-120:4) -- cycle;
\draw[color=white] (12,0)++(-60:4) ++(60:4) -- ++(4,0) -- ++ (-120:4) -- cycle;

\filldraw[fill=blue!15!white, draw=white]{} (16,0)++(-60:4) -- ++(60:4) -- ++(-60:4)-- cycle;
\draw[color=white] (16,0)++(-60:4) -- ++(60:4) -- ++(-60:4)-- cycle;

\filldraw[fill=blue!15!white, draw=white]{} (16,0)++(-60:4) ++(60:4) -- ++(4,0) -- ++ (-120:4) -- cycle;
\draw[color=white] (16,0)++(-60:4) ++(60:4) -- ++(4,0) -- ++ (-120:4) -- cycle;

  \filldraw[fill=blue!15!white, draw=white]{} (20,0)++(-60:4) ++(60:4) -- ++(4,0) -- ++ (-120:4) -- cycle;
\draw[color=white] (20,0)++(-60:4) ++(60:4) -- ++(4,0) -- ++ (-120:4) -- cycle;

\filldraw[fill=blue!15!white, draw=white]{} (20,0)++(-60:4) -- ++(60:4) -- ++(-60:4)-- cycle;
\draw[color=white] (20,0)++(-60:4) -- ++(60:4) -- ++(-60:4)-- cycle;

  \end{tikzpicture}
  \begin{tikzpicture}[scale=0.25]
    \filldraw[black] (4,0)+(60:4) circle (4pt);
      \filldraw[black] (4,0)++(60:4) ++ (60:4) circle (4pt);
      \filldraw[black] (4,0)++(60:4) ++ (60:4) ++ (60:4)circle (4pt);
      \filldraw[black] (4,0)++(60:4) ++ (60:4) ++ (60:4)+(4,0) circle (4pt);
      \filldraw[black] (4,0)++(60:4)+(4,0) circle (4pt) ;
      \filldraw[black] (8,0)+(60:4) circle (4pt) ;
      \filldraw[black] (8,0)++(60:4) ++ (60:4) circle (4pt) ;

      \filldraw[black] (8,0)++(60:4)+(4,0) circle (4pt) ;
       \filldraw[black] (12,0)+(60:4) circle (4pt);
      \filldraw[black] (12,0)++(60:4) ++ (60:4) circle (4pt);
      \filldraw[black] (12,0)++(60:4)+(4,0) circle (4pt);
      \filldraw[black] (12,0)++(4,0) circle (4pt);
      \filldraw[black] (12,0)  circle (4pt);
      \filldraw[black] (8,0)  circle (4pt);

      \filldraw[black] (16,0)+(60:4) circle (4pt);
      \filldraw[black] (16,0)++(60:4) ++ (60:4) circle (4pt);
      \filldraw[black] (16,0)++(60:4) ++ (60:4) ++ (60:4)circle (4pt);
      \filldraw[black] (16,0)++(60:4) ++ (60:4) ++ (60:4)+(4,0) circle (4pt);
      \filldraw[black] (16,0)++(60:4)+(4,0) circle (4pt) ;
      \filldraw[black] (20,0)+(60:4) circle (4pt) ;
      \filldraw[black] (20,0)++(60:4) ++ (60:4) circle (4pt) ;

      \filldraw[black] (20,0)++(60:4)+(4,0) circle (4pt) ;
       \filldraw[black] (24,0)+(60:4) circle (4pt);
      \filldraw[black] (24,0)++(60:4) ++ (60:4) circle (4pt);
      \filldraw[black] (24,0)++(60:4)+(4,0) circle (4pt);
      \filldraw[black] (24,0)++(4,0) circle (4pt);
      \filldraw[black] (24,0)  circle (4pt);
      \filldraw[black] (8,0)  circle (4pt);
      \filldraw[black] (20,0)  circle (4pt);

      \draw[color=black] (4,0)++(60:4) -- ++(60:4) -- ++(60:4)-- ++(4,0) -- ++(-60:4) -- ++(-60:4) -- ++(60:4) -- ++(60 :4) -- ++ (4,0) -- ++(-60:4) -- ++(-60:4) -- ++(-120:4) -- ++ (-4,0) -- ++(-4,0)-- ++(120:4)-- ++(-120:4)  -- ++ (-4,0) -- ++ (-4,0) -- ++ (120:4);
      \draw[color=black] (8,0)++(60:4) -- ++(60:4)-- ++(-60:4)-- ++(-4,0);
      \draw[color=black] (20,0)++(60:4) -- ++(60:4)-- ++(-60:4)-- ++(-4,0);
  \end{tikzpicture}
  \caption{The first graphic depicts an example for a self-intersecting component $\compo_i\in \mathcal{D}_2(\compo)$ and the corresponding graph. Note that the vertex $\tilde{v}$ is contained twice in the corresponding tuple $(v_i^{1}, ..., v_i^{J})$ and that $\R^2 \setminus \overline{\compo_i}$ consists of two components.
  In the second example, the two components $\compo_1,\compo_2 \in \mathcal{D}_1(H)$ have boundaries $\partial \compo_1, \partial \compo_2$ that consist of \JJJ two \EEE connected components \MMM each. \EEE  
  } \label{fig:self-intersection}
  \end{figure}
  Roughly speaking,  $\mathcal{D}_l(\compo)$ collects the set of components which touch  other components at $l$ different vertices. (Also \EEE self-intersections of one component are possible and taken into account  depending on how often the vertex appears in $(v_i^1, \ldots, v_i^J)$ , \EEE see Figure~\ref{fig:self-intersection}.) 
By an elementary computation we have that
\begin{align}\label{iddd}
\sum_{l \ge 1} l \# \mathcal{D}_l(\compo) = \sum_{k \ge 2} k   \# \mathcal{V}_{2k}(\compo). 
\end{align}
 Indeed, the left-hand side can be written as $\sum_{i=1}^n \#\lbrace j =1\ldots J \colon n(v_i^j) \ge 4 \rbrace $, and each $v \in \mathcal{V}_{2k}(\compo)$, $k \ge 2$, appears $k$-times in these cycles. \EEE 
 
\GGG In the proof of Theorem~\ref{generaltheorem}, \EEE our strategy will be to bound $\# \mathcal{V}_{2k}(\compo)$ for $k \ge 2$ since these vertices are related to triangles in $\mathbf{T}_H^{\rm ex, bad}$ (see Figure \ref{fig:touching-components}). In view of \eqref{iddd}, this can be achieved by finding a  suitable bound on \EEE    
  $\# \mathcal{D}_l(\compo)$, $l \ge 1$. The following lemma shows that, under suitable assumptions on the components  of $\R^2 \setminus \overline{\compo}$, it actually suffices to control $\mathcal{D}_1(\compo)$ and $\mathcal{D}_2(\compo)$. \EEE 
\begin{lemma}\label{lemma:afirststep}
Let $\compo$ satisfy $\mathcal{H}^1(\partial \compo)\leq  M \EEE$ and, denoting by $(F_i)_i$ the connected components of $\R^2\setminus \ol \compo$, assume that $ |F_i|  >  \varepsilon^2/\eta^2$ for every $i$.
Then,
\begin{align}\label{afirststep}
\sum_{l \ge 3} l \# \mathcal{D}_l(\compo) \le  C     \# \mathcal{D}_1(\compo) +  C \frac{\eta}{\eps},  
\end{align}
 where $C>0$ only depends on $M$. \EEE
\end{lemma}

\begin{proof}
By the assumption on    $ F_i \EEE $, using that $(F_i)_i$ are pairwise disjoint and the isoperimetric inequality, we get 
$ \frac{\eps}{\eta}\# (F_i)_i \MMM \le \sum_i|F_i|^{1/2} \EEE \le C\sum_i \mathcal{H}^1(\partial F_i) \le C \mathcal{H}^1(\partial \compo)$. Noticing that $\mathcal{F}(\compo)=\mathcal{C}(\compo) \cup (F_i)_i$,  we thus have 
\begin{align}\label{comp bound}
\# \mathcal{F}(\compo) \le \# \mathcal{C}(\compo)    +   C \EEE \frac{\eta}{\eps}  \mathcal{H}^1(\partial \compo) \le  \# \mathcal{C}(\compo)    +  C \frac{\eta}{\eps}. 
\end{align}
Further, we note that 
\begin{align}\label{iddd2}
\# \mathcal{E}(\compo)   =   \sum_{k \ge 1}   k  \# \mathcal{V}_{2k}(\compo)    
\end{align}
as each edge is associated to exactly two vertices. Then, from \eqref{Euler} we obtain  
$$ \# \mathcal{V}(\compo)   +  \big( \#\mathcal{F}(\compo) -  \# \mathcal{C}(\compo) \big) +   \# \mathcal{C}(\compo)  =  \nu(\mathcal{V}(\compo), \mathcal{E}(\compo)) +  \# \mathcal{E}(\compo).  $$
Using the definition of $(\mathcal{V}_{2k}(\compo))_k$  and $(\mathcal{D}_{l}(\compo))_l$ in this formula, as well as  \eqref{iddd2}, we get
$$ \sum_{k \ge 1} \#  \mathcal{V}_{2k}(\compo)   +  \big(\#\mathcal{F}(\compo) -  \# \mathcal{C}(\compo) \big) + \sum_{l \ge 0}  \# \mathcal{D}_l(\compo)  =  \nu(\mathcal{V}(\compo), \mathcal{E}(\compo)) +  \sum_{k \ge 1}   k  \# \mathcal{V}_{2k}(\compo).  $$
As  $k-1 \ge k/2$ for $k \ge 2$, this yields
$$      \big(\#\mathcal{F}(\compo) -  \# \mathcal{C}(\compo) \big) + \sum_{l \ge 0}  \# \mathcal{D}_l(\compo)  \ge  \nu(\mathcal{V}(\compo), \mathcal{E}(\compo)) +  \sum_{k \ge 2}   \frac{k}{2}  \# \mathcal{V}_{2k}(\compo).  $$
This along with \eqref{iddd} shows
$$      \big(\#\mathcal{F}(\compo) -  \# \mathcal{C}(\compo) \big) + \sum_{l \ge 0}  \# \mathcal{D}_l(\compo)  \ge  \nu(\mathcal{V}(\compo), \mathcal{E}(\compo)) +  \sum_{l \ge 1}  \frac{l}{2} \# \mathcal{D}_l(\compo),  $$
and thus, as $l/2 - 1 \ge l/6$ for $\l \ge 3$, we deduce   
$$      \big(\#\mathcal{F}(\compo) -  \# \mathcal{C}(\compo) \big) +   \# \mathcal{D}_0(\compo) + \frac{1}{2}  \# \mathcal{D}_1(\compo) \ge  \nu(\mathcal{V}(\compo), \mathcal{E}(\compo)) +    \sum_{l \ge 3}  \frac{l}{6}  \# \mathcal{D}_l(\compo).  $$
Eventually, we observe that clearly $  \nu(\mathcal{V}(\compo), \mathcal{E}(\compo)) \ge \# \mathcal{D}_0(\compo)$ as each component  in   $ \mathcal{D}_0(\compo)$ induces a connected component of the planar graph. 
This along with \eqref{comp bound} shows  \eqref{afirststep}.
\end{proof}

\noindent \textbf{Healing of suitable sets.}
Let $\mathbf{N}(T;\compo)$ denote the triangles in $\overline{\compo \setminus T}$ sharing an edge with $T    \in \mathbf{T}$.  
We let 
\begin{equation}\label{defMj}
\mathbf{M}_j(\compo):=\{T \in \mathbf{T}_\compo \colon \# \mathbf{N}(T;\compo) = j\} \quad \text{for }j=0,1,2,3. 
\end{equation}
In other words, $\mathbf{M}_j(\compo)$ consists of triangles in $\ol \compo$ that expose $3-j$ edges to $\Omega\setminus \compo$ (see Figure \ref{fig:differ-exposing}).
\begin{figure}[h]
  \centering
  \begin{tikzpicture}[scale=0.42]
    \filldraw[fill=blue!15!white, draw=white]{} (0,0) -- ++(60:4) -- ++(-60:4)-- cycle;
    \draw[color=black] (0,0) -- ++(60:4) -- ++(-60:4)-- cycle;

    \filldraw[fill=blue!15!white, draw=white]{} (0,0) ++(60:4) -- ++(4,0) -- ++ (-120:4) -- cycle;
    \draw[color=black] (0,0) ++(60:4) -- ++(4,0) -- ++ (-120:4) -- cycle; 

    \filldraw[fill=blue!15!white, draw=white]{} (4,0) -- ++(60:4) -- ++(-60:4)-- cycle;
    \draw[color=black] (4,0) -- ++(60:4) -- ++(-60:4)-- cycle;

    \filldraw[fill=blue!15!white, draw=white]{} (0,0)++(60:4) -- ++(60:4) -- ++(-60:4)-- cycle;
    \draw[color=black] (0,0)++(60:4) -- ++(60:4) -- ++(-60:4)-- cycle;

    \node[scale=1 ] at (4,2) {$T$};
    \node[scale=1, color=blue!40!white] at (0.5,3) {$\compo$};

    \filldraw[fill=blue!15!white, draw=white]{} (10,0) -- ++(60:4) -- ++(-60:4)-- cycle;
    \draw[color=black] (10,0) -- ++(60:4) -- ++(-60:4)-- cycle;

    \filldraw[fill=blue!15!white, draw=white]{} (10,0) ++(60:4) -- ++(4,0) -- ++ (-120:4) -- cycle;
    \draw[color=black] (10,0) ++(60:4) -- ++(4,0) -- ++ (-120:4) -- cycle; 

    \filldraw[fill=blue!15!white, draw=white]{} (14,0) -- ++(60:4) -- ++(-60:4)-- cycle;
    \draw[color=black] (14,0) -- ++(60:4) -- ++(-60:4)-- cycle;

    \filldraw[fill=red!15!white, draw=white]{} (10,0)++(60:4) -- ++(60:4) -- ++(-60:4)-- cycle;
    \draw[color=black] (10,0)++(60:4) -- ++(60:4) -- ++(-60:4)-- cycle;

    \node[scale=1 ] at (14,2) {$T$};
    \node[scale=1, color=red!60!white] at (16.5,5.5) {$\Omega\setminus \compo$};
  \node[scale=1, color=blue!40!white] at (10.5,2.5) {$\compo$};
  \node[scale=1] at (4,-1.5) {$  T \EEE\in \mathbf{M}_{3}(\compo)$};

    \filldraw[fill=blue!15!white, draw=white]{} (20,0) -- ++(60:4) -- ++(-60:4)-- cycle;
    \draw[color=black] (20,0) -- ++(60:4) -- ++(-60:4)-- cycle;

    \filldraw[fill=blue!15!white, draw=white]{} (20,0) ++(60:4) -- ++(4,0) -- ++ (-120:4) -- cycle;
    \draw[color=black] (20,0) ++(60:4) -- ++(4,0) -- ++ (-120:4) -- cycle; 

    \filldraw[fill=red!15!white, draw=white]{} (24,0) -- ++(60:4) -- ++(-60:4)-- cycle;
    \draw[color=black] (24,0) -- ++(60:4) -- ++(-60:4)-- cycle;

    \filldraw[fill=red!15!white, draw=white]{} (20,0)++(60:4) -- ++(60:4) -- ++(-60:4)-- cycle;
    \draw[color=black] (20,0)++(60:4) -- ++(60:4) -- ++(-60:4)-- cycle;
    \filldraw[green] (24,0)++(60:4) circle (4pt); 

    \node[scale=1 ] at (24,2) {$T$};
    \node[scale=1, color=blue!40!white] at (20.5,2.5) {$\compo$};
    \node[scale=1, color=red!60!white] at (26.5,5.5) {$\Omega\setminus \compo$};
    \node[scale=1] at (14,-1.5) {$  T \EEE\in \mathbf{M}_{2}(\compo)$};

    \filldraw[fill=red!15!white, draw=white]{} (30,0) -- ++(60:4) -- ++(-60:4)-- cycle;
    \draw[color=black] (30,0) -- ++(60:4) -- ++(-60:4)-- cycle;

    \filldraw[fill=blue!15!white, draw=white]{} (30,0) ++(60:4) -- ++(4,0) -- ++ (-120:4) -- cycle;
    \draw[color=black] (30,0) ++(60:4) -- ++(4,0) -- ++ (-120:4) -- cycle; 

    \filldraw[fill=red!15!white, draw=white]{} (34,0) -- ++(60:4) -- ++(-60:4)-- cycle;
    \draw[color=black] (34,0) -- ++(60:4) -- ++(-60:4)-- cycle;

    \filldraw[fill=red!15!white, draw=white]{} (30,0)++(60:4) -- ++(60:4) -- ++(-60:4)-- cycle;
    \draw[color=black] (30,0)++(60:4) -- ++(60:4) -- ++(-60:4)-- cycle;

    \node[scale=1 ] at (34,2) {$T$};
    \node[scale=1, color=blue!40!white] at (30.5,2.5) {$\compo$};
    \node[scale=1, color=red!60!white] at (36.5,5.5) {$\Omega\setminus \compo$};
    \node[scale=1] at (24,-1.5) {$ T \EEE \in \mathbf{M}_{1}(\compo)$};
    \node[scale=1] at (34,-1.5) {$ T \EEE\in \mathbf{M}_{0}(\compo)$};
    \filldraw[green] (30,0)++(60:4) circle (4pt);
    \filldraw[green] (30,0)++(60:4)++(4,0) circle (4pt);\filldraw[green] (30,0)++(4,0) circle (4pt);

\end{tikzpicture}

\caption{Triangles in the sets $\mathbf{M}_{j}(\compo)$. If we suppose that in all four pictures the other triangles which are not depicted are contained in $\Omega \setminus H$, the green vertices correspond to the vertices  given in \eqref{defMheal}.}\label{fig:differ-exposing}
\end{figure}
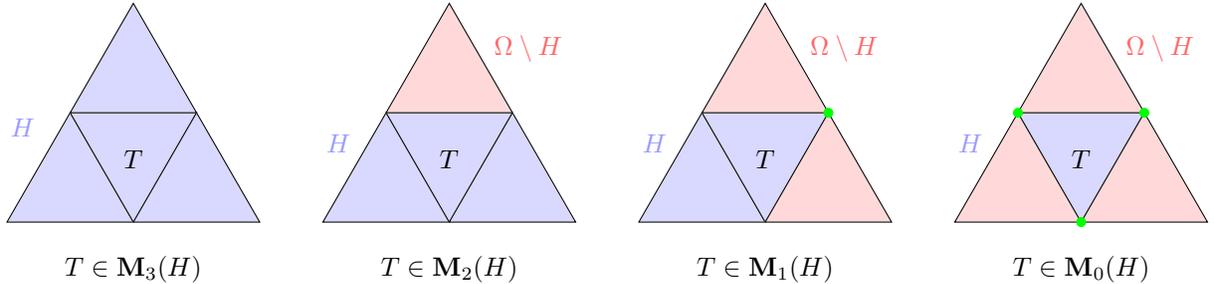
\GGG We let
\begin{equation}\label{defMheal}
\mathbf{M}^{\rm heal}_j(\compo):=\{T\in \mathbf{M}_j(\compo) \colon \text{there exists }v \in \mathcal{N}(T) \text{ s.t. }v \notin \mathcal{N}(T') \EEE \text{ for all }T'\in \mathbf{T}_\compo \setminus \{T\}\} \quad\text{for }j=0,1.
\end{equation}
 We notice that for $T\in \mathbf{M}^{\rm heal}_1(\compo)$ the vertex $v$ as in \eqref{defMheal} is necessarily the one which is not contained in the edge shared with any other (different) triangle in $\mathbf{T}_H$, while for $\mathbf{M}^{\rm heal}_0(\compo)$ there could be more vertices satisfying the condition $v \notin \mathcal{N}(T') \text{ for all }T'\in \mathbf{T}_\compo \setminus \{T\}$, see Figure~\ref{fig:differ-exposing}. \EEE In Figure \ref{fig:touching-components}, there is a set $\compo$ with six triangles in $\mathbf{M}_1(\compo)$, of which two (the external ones  in dark blue) are in $\mathbf{M}^{\rm heal}_1(\compo)$. 

 We also define the triangles which \MMM  `cannot be healed' (in the sense of Lemma \ref{healing1}  below) \EEE by  
\begin{equation}\label{defMnotheal}
\mathbf{M}^{\rm nh}_j(\compo) := \mathbf{M}_j(\compo) \setminus \mathbf{M}^{\rm heal}_j(\compo) \quad\text{for }j=0,1,
\end{equation}
and 
\begin{align}\label{healoftri} 
\compo_{\rm heal} := \compo \setminus \bigcup_{j=0,1}  \bigcup_{T \in  \mathbf{M}^{\rm heal}_j(\compo)} T.
\end{align}

\noindent
Each \EEE $T \in \mathbf{M}^{\rm heal}_0(\compo) $ \EEE satisfies $\mathcal{H}^1(\GGG \overline{\compo_{\rm heal}} \cap T \EEE ) = 0$ which yields
\begin{align*}\label{1208241839}
\mathcal{H}^1(\partial \GGG \compo \EEE) & =  \sum_{T \in \mathbf{T}_H}  \mathcal{H}^1(\partial \compo \cap T) \EEE \\ & \ge \mathcal{H}^1(\partial \compo_{\rm heal}) + \sum_{\GGG T \EEE \in \mathbf{M}^{\rm heal}_1(\compo)}  \big(   \mathcal{H}^1(\partial \compo \cap T) -    \mathcal{H}^1(\partial \compo_{\rm heal} \cap T)  \big)  +  \sum_{T \in \mathbf{M}^{\rm heal}_0(\compo)} \mathcal{H}^1(\partial \compo \cap T). 
\end{align*}
For $T \in \mathbf{M}^{\rm heal}_1(\GGG \compo \EEE) $ we have that $T \cap   \overline{\compo_{\rm heal}}$  coincides with one edge of $T$ \MMM or is empty, \EEE and thus by \eqref{ti} we get $\mathcal{H}^1(\partial \compo_{\rm heal}\cap T) \le \frac{2}{\eps \sin\theta_0 }|T|$, so that
\begin{equation}\label{1308241138}
\mathcal{H}^1(\partial \compo ) \ge \mathcal{H}^1(\partial \compo_{\rm heal}) - \sum_{j=0,1} \sum_{T  \in \mathbf{M}^{\rm heal}_j(\compo)  }   \Big( \frac{2}{\eps  \sin\theta_0}|T| -   \mathcal{H}^1( \partial \compo \cap T) \Big) .   
\end{equation}
 Moreover, \EEE we observe that 
\begin{equation}\label{1408240756}
\# \mathbf{M}^{\rm nh}_0(\compo) + \# \mathbf{M}^{\rm nh}_1(\compo) \leq  \sum_{k \ge 2} k \# \mathcal{V}_{2k}(\compo). \EEE
\end{equation}
In fact, setting $\mathcal{V}^{\rm nh}(\compo) := \bigcup_{l \ge 2} \mathcal{V}_{2l}(\compo)$, for each $T \in \mathbf{M}^{\rm nh}_1(\compo)$, the  set $\mathcal{V}^{\rm nh}(\compo)\cap \mathcal{N}(T)$   contains the unique vertex $v$ \EEE of $T$ which is contained in both edges exposed to $\R^2 \setminus \compo$, otherwise $T$ would lie in \EEE  $\mathbf{M}^{\rm heal}_1(\compo)$ by \eqref{defMheal}. Similarly, any $T \in \mathbf{M}^{\rm nh}_0(\compo)$ necessarily contains at least one vertex in $\mathcal{V}^{\rm nh}(\compo)$, by definition. \EEE

 Recall that in Theorem \ref{generaltheorem} one considers a function $u\in H^1(\Omega;\R^2)$ that is piecewise affine on $\mathbf{T}$ with $\|e(u)\|_{L^2(\Omega\sm \compo)} \le \EEE C$. The symmetric gradient of such a function can still be controlled on $\Omega \setminus \compo_{\rm heal} $, as the following lemma shows.  

\begin{lemma}[Healing of triangles]\label{healing1} 
  Let $u \in H^{1}(\Omega;\R^2)$ be piecewise affine on
  $\mathbf{T} $
  and   let   $H$   be   induced by $\mathbf{T}_H$, for a given subset $\mathbf{T}_H\subset \mathbf{T}$. Suppose that ${\rm dist}(H,\partial\Omega) \ge \omega(\eps)$. Then, there exists a uniform constant $C>0$ such that  
$$  \Vert e(u) \Vert_{L^2(\Omega \setminus \GGG \compo_{\rm heal}\EEE)} \le C  \Vert e(u) \Vert_{L^2(\Omega \setminus \GGG \compo \EEE)}. $$
\end{lemma}

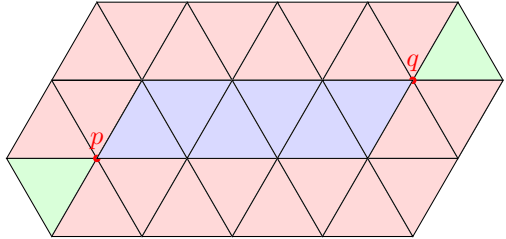
\begin{figure}[h]
  \begin{tikzpicture}[scale=0.3]
        \filldraw[fill=red!15!white, draw=white]{} (0,0)++(60:4) -- ++(60:4) -- ++(-60:4)-- cycle;
        \draw[color=black] (0,0)++(60:4) -- ++(60:4) -- ++(-60:4)-- cycle;
    
        \filldraw[fill=red!15!white, draw=white]{} (0,0)++(60:4) ++(60:4) -- ++(4,0) -- ++ (-120:4) -- cycle;
        \draw[color=black] (0,0)++(60:4) ++(60:4) -- ++(4,0) -- ++ (-120:4) -- cycle;
    
        \filldraw[fill=red!15!white, draw=white]{} (4,0)++(60:4) -- ++(60:4) -- ++(-60:4)-- cycle;
        \draw[color=black] (4,0)++(60:4) -- ++(60:4) -- ++(-60:4)-- cycle;
    
        \filldraw[fill=red!15!white, draw=white]{} (4,0)++(60:4) ++(60:4) -- ++(4,0) -- ++ (-120:4) -- cycle;
        \draw[color=black] (4,0)++(60:4) ++(60:4) -- ++(4,0) -- ++ (-120:4) -- cycle;
    
        \filldraw[fill=red!15!white, draw=white]{} (8,0)++(60:4) -- ++(60:4) -- ++(-60:4)-- cycle;
        \draw[color=black] (8,0)++(60:4) -- ++(60:4) -- ++(-60:4)-- cycle;
    
        \filldraw[fill=red!15!white, draw=white]{} (8,0)++(60:4) ++(60:4) -- ++(4,0) -- ++ (-120:4) -- cycle;
        \draw[color=black] (8,0)++(60:4) ++(60:4) -- ++(4,0) -- ++ (-120:4) -- cycle; 
    
        \filldraw[fill=red!15!white, draw=white]{} (12,0)++(60:4) -- ++(60:4) -- ++(-60:4)-- cycle;
        \draw[color=black] (12,0)++(60:4) -- ++(60:4) -- ++(-60:4)-- cycle;
    
        \filldraw[fill=red!15!white, draw=white]{} (12,0)++(60:4) ++(60:4) -- ++(4,0) -- ++ (-120:4) -- cycle;
        \draw[color=black] (12,0)++(60:4) ++(60:4) -- ++(4,0) -- ++ (-120:4) -- cycle; 
        
        \filldraw[fill=green!15!white, draw=white]{} (16,0)++(60:4) -- ++(60:4) -- ++(-60:4)-- cycle;
        \draw[color=black] (16,0)++(60:4) -- ++(60:4) -- ++(-60:4)-- cycle;

        \filldraw[fill=red!15!white, draw=white]{} (0,0) -- ++(60:4) -- ++(-60:4)-- cycle;
        \draw[color=black] (0,0) -- ++(60:4) -- ++(-60:4)-- cycle;
    
         \filldraw[fill=red!15!white, draw=white]{} (0,0) ++(60:4) -- ++(4,0) -- ++ (-120:4) -- cycle;
        \draw[color=black] (0,0) ++(60:4) -- ++(4,0) -- ++ (-120:4) -- cycle; 
        
        \draw (0,0)++(60:4) -- ++ (4,0);
        
        \draw (4,0)  -- +(60:4);
        \draw  (4,0) -- +(0:4);
        \draw (4,0)++(60:4) -- ++(-60:4);
    
         \draw (4,0)++(60:4) -- ++ (4,0);
       
        \draw (8,0)  -- +(60:4);
        \draw  (8,0) -- +(0:4);
        \draw (8,0)++(60:4) -- ++(-60:4); 
        
        \draw (8,0)++(60:4) -- ++ (4,0);
        
        \draw (12,0)  -- +(60:4);
        \draw  (12,0) -- +(0:4);
        \draw (12,0)++(60:4) -- ++(-60:4);
    
        \draw[color=black] (12,0) ++(60:4) -- ++(4,0) -- ++ (-120:4) -- cycle;
    
        \draw[color=black] (16,0) -- ++(60:4) -- ++(-60:4)-- cycle;
        \draw[color=black] (16,0) ++(60:4) -- ++(4,0) -- ++ (-120:4) -- cycle;

        \filldraw[fill=red!15!white, draw=white]{} (16,0) -- ++(60:4) -- ++(-60:4)-- cycle;
        \draw[color=black] (16,0) -- ++(60:4) -- ++(-60:4)-- cycle;
    
        \filldraw[fill=red!15!white, draw=white]{} (16,0) ++(60:4) -- ++(4,0) -- ++ (-120:4) -- cycle;
        \draw[color=black] (16,0) ++(60:4) -- ++(4,0) -- ++ (-120:4) -- cycle;
    
    
        \filldraw[fill=green!15!white, draw=white]{} (0,0) -- ++(-60:4) -- ++(60:4)-- cycle;
        \draw[color=black] (0,0) -- ++(-60:4) -- ++(60:4)-- cycle;
    
        \filldraw[fill=red!15!white, draw=white]{} (0,0) ++(-60:4) -- ++(60:4) -- ++ (-60:4)-- ++ (-4,0)--  cycle;
        \draw[color=black] (0,0) ++(-60:4) -- ++(60:4) -- ++ (-60:4)-- ++ (-4,0)--  cycle;

        \filldraw[fill=red!15!white, draw=white]{} (4,0) -- ++(-60:4) -- ++(60:4)-- cycle;
        \draw[color=black] (4,0) -- ++(-60:4) -- ++(60:4)-- cycle;
    
        \filldraw[fill=red!15!white, draw=white]{} (4,0) ++(-60:4) -- ++(60:4) -- ++ (-60:4)-- ++ (-4,0)--  cycle;
        \draw[color=black] (4,0) ++(-60:4) -- ++(60:4) -- ++ (-60:4)-- ++ (-4,0)--  cycle;
    
        \filldraw[fill=red!15!white, draw=white]{} (8,0) -- ++(-60:4) -- ++(60:4)-- cycle;
        \draw[color=black] (8,0) -- ++(-60:4) -- ++(60:4)-- cycle;
    
         \filldraw[fill=red!15!white, draw=white]{} (8,0) ++(-60:4) -- ++(60:4) -- ++ (-60:4)-- ++ (-4,0)--  cycle;
        \draw[color=black] (8,0) ++(-60:4) -- ++(60:4) -- ++ (-60:4)-- ++ (-4,0)--  cycle;
    
        \filldraw[fill=red!15!white, draw=white]{} (12,0) -- ++(-60:4) -- ++(60:4)-- cycle;
        \draw[color=black] (12,0) -- ++(-60:4) -- ++(60:4)-- cycle;
    
         \filldraw[fill=red!15!white, draw=white]{} (12,0) ++(-60:4) -- ++(60:4) -- ++ (-60:4)-- ++ (-4,0)--  cycle;
        \draw[color=black] (12,0) ++(-60:4) -- ++(60:4) -- ++ (-60:4)-- ++ (-4,0)--  cycle;
    
        \filldraw[fill=red!15!white, draw=white]{} (16,0) -- ++(-60:4) -- ++(60:4)-- cycle;
        \draw[color=black] (16,0) -- ++(-60:4) -- ++(60:4)-- cycle;
    
    
    \filldraw[red] (0,0)+(0:4) circle (4pt) node[anchor=south]{$p$};
    
    \filldraw[red] (16,0)+(60:4) circle (4pt) node[anchor=south]{$q$};
    
    
    \filldraw[fill=blue!15!white, draw=white]{} (4,0) -- ++(60:4) -- ++(-60:4)-- cycle;
        \draw[color=black] (4,0) -- ++(60:4) -- ++(-60:4)-- cycle;
        
         \filldraw[fill=blue!15!white, draw=white]{} (8,0) -- ++(60:4) -- ++(180:4)-- cycle;
        \draw[color=black] (8,0) -- ++(60:4) -- ++(180:4)-- cycle;
        
        \filldraw[fill=blue!15!white, draw=white]{} (8,0) -- ++(60:4) -- ++(-60:4)-- cycle;
        \draw[color=black] (8,0) -- ++(60:4) -- ++(-60:4)-- cycle;
        
         \filldraw[fill=blue!15!white, draw=white]{} (12,0) -- ++(60:4) -- ++(180:4)-- cycle;
        \draw[color=black] (12,0) -- ++(60:4) -- ++(180:4)-- cycle;
        
        \filldraw[fill=blue!15!white, draw=white]{} (12,0) -- ++(60:4) -- ++(-60:4)-- cycle;
        \draw[color=black] (12,0) -- ++(60:4) -- ++(-60:4)-- cycle;
        
         \filldraw[fill=blue!15!white, draw=white]{} (16,0) -- ++(60:4) -- ++(180:4)-- cycle;
        \draw[color=black] (16,0) -- ++(60:4) -- ++(180:4)-- cycle;
    
    \end{tikzpicture}
    \caption{The sets $Z$ (violet), $Y$ (green), $N_Z\setminus Y$ (red), and $\{p, \, q\}=\overline{Y}\cap \overline{Z}$.} \label{fig: two-domains-pic}
\end{figure}
  This result will be crucial in the proof of Theorem \ref{generaltheorem} as it enables us to `heal away' the triangles in $H\sm H_{\rm heal}$ without loosing control on the symmetric gradient. The condition on ${\rm dist}(H,\partial\Omega)$ ensures that a suitable neighborhood of  $\compo$ \EEE lies in $\Omega$.  We will also employ the following result, allowing to even \EEE heal suitable unions of triangles. 

\begin{lemma}[Healing of entire components]\label{heal entire com}
Let $\mathbf{T}_Z \subset \mathbf{T}$ be such that the set $Z$ induced by $\mathbf{T}_Z$   is connected with $| Z| \le \eps^2/\eta^2$ and $Z = {\rm sat}(Z)$,  i.e., $Z$ has no holes.  Suppose that ${\rm dist}(Z,\partial\Omega) \ge \omega(\eps)$. \EEE Let $\mathbf{T}_{N_Z} \subset \mathbf{T}$ denote the set of all $T\in \mathbf{T}\setminus \mathbf{T}_Z$ with $T \cap \overline{Z} \neq \emptyset$, and let $N_Z$ be the set induced by $\mathbf{T}_{N_Z}$. Moreover, let \EEE $\mathbf{T}_Y\subset \mathbf{T}_{N_Z}$ be such that  $\overline{Y} \cap \overline{Z}$ consists of at most two points, where $Y$ is induced by $\mathbf{T}_Y$, \EEE see Figure \ref{fig: two-domains-pic}. Then, given $u \in H^1( N_Z \setminus  \overline{Y} \EEE ;\R^2)$   being piecewise affine on $\mathbf{T}_{N_Z} \MMM \setminus \mathbf{T}_{Y} \EEE$, there exists  $u_{\rm heal} \in H^1({\rm int}( \overline{Z \cup (N_Z \setminus {Y})}); \EEE \R^2)$ with $u_{\rm heal}= u$ on $N_{Z}\sm  \ol{Y}$ \EEE such that
$$ \Vert   e(u_{\rm heal})  \Vert_{L^2(Z \cup   (N_Z \setminus Y)  )} \le  \frac{C}{\eta^\alpha} \EEE  \Vert  e(u)  \Vert_{L^2( N_Z \setminus Y  )}$$
 for a universal constant $C>0$ and some $\alpha \in \N$. \EEE    
\end{lemma}


 Note that the case $Y = \emptyset$ corresponds to the situation that $e(u)$ is controlled in an entire neighborhood $N_Z$ of $Z$ which allows to extend $u$ inside $Z$ according to classical extension theorems. \EEE The main point of this lemma is that such an extension of $u$ from $N_Z$ to $N_Z \cup Z$ is still possible in the presence of $Y$, as long as  $\overline{Y} \cap \overline{Z}$ consists of \emph{at most two points} (see Figure \ref{fig: two-domains-pic}). Already for $\#(\overline{Y} \cap \overline{Z}) \ge 3$, the situation is less rigid and a statement as in Lemma~\ref{heal entire com} cannot be expected. 
Note that we will apply this result for sets $Z$ which \MMM are  \EEE  connected components $\compo_j \in \mathcal{C}(\compo)$ for suitable $\compo$. 
 In particular, we observe that the  lemma can be applied for all components in $\compo_j \in \mathcal{D}_l(\compo)$, $l=0,1,2$, which satisfy $| \compo_j| \le \eps^2/\eta^2$ and have no holes. \EEE For convenience, we postpone the proofs of Lemma \ref{healing1} and Lemma \ref{heal entire com} to \MMM Section \ref{sec: proflem} below. \EEE

 \subsection{Proof of  Theorem \ref{generaltheorem} and  Corollary \ref{monotonicity-corollary}\EEE} We proceed with the proof of Theorem \ref{generaltheorem}.

\begin{proof}[Proof of Theorem~\ref{generaltheorem}]
For convenience, we first prove the result under the additional assumption \MMM that \EEE  ${\rm dist}(A,\partial\Omega) \ge \omega(\eps)$. We indicate the necessary adaptations for the general case at the end of the proof. Moreover, it is not restrictive to prove the result only for $\eta \le \eta_0$ for some universal $\eta_0$ chosen in \eqref{manus-blackboard} below. 

\noindent \textbf{First relation between area and boundary of $A$:}
From \eqref{ti} and \eqref{defMj}  we 
find 
\begin{align}\label{energy1}
\frac{2}{\eps  \sin\theta_0}|A | &  =  \sum_{T  \in \mathbf{T}_A  }  \frac{2}{\eps  \sin\theta_0}|T|     \ge  \sum_{T  \in \mathbf{M}_2(A)  }     \mathcal{H}^1( \partial A \cap T)   +  \sum_{j=0,1} \sum_{T  \in \mathbf{M}_j(A)  }    \frac{2}{\eps  \sin\theta_0}|T|     \notag   \\ &
 =  \mathcal{H}^1(\partial A) +  \sum_{j=0,1} \sum_{T  \in \mathbf{M}_j(A)  }   \Big( \frac{2}{\eps  \sin\theta_0}|T| -   \mathcal{H}^1( \partial A \cap T)     \Big).       
\end{align}
At this point, the energy bound \eqref{energybound1} and \eqref{ti} already give us the (unsharp) bounds
\begin{align}\label{eq: control on A}
\# \GGG \mathbf{T}_A \EEE \le C/\eps, \quad \quad \quad      \mathcal{H}^1(\partial A) \le C
\end{align}
for some $C  >0 \EEE$ depending on $C_0$, where we used that $|T| \ge  c\eps^2$ for some  $c>0$ \MMM depending on \EEE $\theta_0$.\\

\noindent\textbf{Modification 1: Connected components and filling holes:} We recall that $A$ is an open set by definition. In the sequel, we will consider connected components of $A$ and its complement. We will also consider connected components of $\overline{A}$ and its complement which may lead to different objects: while  in  connected components of $A$ adjacent triangles share an edge, in connected components of $\overline{A}$ triangles may be linked solely by a vertex, see Figure \ref{fig:touching-components}.

For technical reasons related to Lemma \ref{lemma:afirststep}  and  Lemma \ref{heal entire com}, we need to avoid that $\overline{A}$ contains small holes. Therefore, we fill small holes of $\overline{A}$ as follows. We denote by $\mathcal{A}_{\rm small}$ the collection of  connected components $A_c$ of $\R^2 \setminus \overline{A}$ satisfying $  | A_c| \le\eps^2/\eta^2 $. We define
\begin{align}\label{def of bbb}
 B \VIT := \EEE {\rm int}\Big( \overline{A}  \cup  \bigcup_{A_c \in \mathcal{A}_{\rm small}}  A_c\Big),   
 \end{align}
and denote by \GGG $\mathbf{T}_B$ \EEE the triangles contained in $B$.  (For the illustration of a hole, we refer to Figure~\ref{fig:self-intersection}.) \EEE Then, we clearly have  that $B$ is induced by $\mathbf{T}_B$, i.e.,\ $B = {\rm int} (\bigcup_{T \in \mathbf{T}_B} T)$. \EEE  Moreover,   each connected component of $A$ is contained in a connected component of $B$.  
 Note that by the definition of $B$ we have $ \mathbf{M}_j(B) \subset \bigcup_{i=0}^j \mathbf{M}_i(A) $ for $j=0,1,2$.  From \eqref{ti}, $\partial B \subset \partial A$, and \JJJ arguing as in \EEE the first line of \eqref{energy1} we then obtain   
\begin{align}\label{energy3}
\frac{2}{\eps  \sin\theta_0}|A | & \ge  \sum_{T  \in \mathbf{M}_2(B)  }     \mathcal{H}^1( \partial B \cap T )   +  \sum_{j=0,1} \sum_{T  \in \mathbf{M}_j(B)  }    \frac{2}{\eps  \sin\theta_0}|T |     \notag   \\ & =  \mathcal{H}^1(\partial B) +  \sum_{j=0,1} \sum_{T  \in \mathbf{M}_j(B)  }   \Big( \frac{2}{\eps  \sin\theta_0}|T| -   \mathcal{H}^1( \partial B \cap T)     \Big).
\end{align}
\noindent
The isoperimetric inequality along with  \eqref{eq: control on A} and $\sqrt{|A_c|} \le \eps/\eta$  \EEE for all $A_c \in \mathcal{A}_{\rm small}$ yield
\begin{align*}
\sum_{A_c \in \mathcal{A}_{\rm small}}   |A_c|  \le \frac{\eps}{\eta} \sum_{A_c \in \mathcal{A}_{\rm small}}   \sqrt{|A_c|} \le    C\frac{\eps}{\eta} \sum_{A_c \in \mathcal{A}_{\rm small}}   \mathcal{H}^1(\partial A_c)     \le    C\frac{\eps}{\eta}.
\end{align*}
Again using \eqref{eq: control on A} \MMM and the fact that $|T| \ge  c\eps^2$, \EEE this shows
\begin{align}\label{global bound on b}
 \# \mathbf{T}_B \le  \# \mathbf{T}_A  +  C \frac{1}{\eps\, \eta} \le C \frac{1}{\eps\, \eta}, \quad \quad    | B  | \le C  \frac{\eps}{\eta},  \quad \quad    \mathcal{H}^1 (\partial B)   \le C \EEE 
\end{align}
  for some $C>0$ depending on $C_0$ \MMM and $\theta_0$.  \EEE  We also note that, by the assumption on $A$, it clearly holds $B \subset \Omega$ with  ${\rm dist}(B,\partial\Omega) \ge \omega(\eps)$. \EEE

\noindent  \textbf{Motivation of next steps:}  For motivating the next steps of the proof, let us also introduce $B_{\rm heal}$ related to  $B$   as defined in \eqref{healoftri}. \EEE \GGG Taking $\compo=B$ in \eqref{1308241138}, \EEE along with   \eqref{defMnotheal} and \eqref{energy3}
we deduce  
\begin{align*}
  \frac{2}{\eps  \sin\theta_0}|A | & \ge  \mathcal{H}^1(\partial B_{\rm heal}) +  \sum_{j=0,1} \sum_{T\in \mathbf{M}^{\rm nh}_j(B)} \Bigl(\frac{2}{\eps  \sin\theta_0} |T|- \mathcal{H}^1(\partial T \cap B)\Bigr)\,.
\end{align*} 
We observe that  $\frac{2}{\eps  \sin\theta_0}|T| -  \mathcal{H}^1( \partial  T) \ge -C\eps$ for each $T \in  \mathbf{M}^{\rm nh}_0(B) \cup \mathbf{M}^{\rm nh}_1(B)$. 
In fact, it is elementary to show that $|T| \ge c(\mathcal{H}^1( \partial  T))^2 $ for $c$ only depending on $\theta_0$. Then, it suffices to observe that the minimum of the  function $x\mapsto \frac{2}{\eps  \sin\theta_0} cx^2 - x$ is larger than $-C\eps$. Therefore, we obtain
\begin{align}\label{for later.end}
  \frac{2}{\eps  \sin\theta_0}|A | & \ge  \mathcal{H}^1(\partial B_{\rm heal}) -  C \eps   \big(\#\mathbf{M}^{\rm nh}_0(B ) + \#  \mathbf{M}^{\rm nh}_1(B)\big)\,, 
\end{align}
\GGG and Lemma~\ref{healing1} gives the corresponding control of the function on $\Omega \setminus B_{\rm heal}$. \EEE  
\EEE
\GGG By \eqref{iddd} and \eqref{1408240756}, in order to control $\#\mathbf{M}^{\rm nh}_0(B ) + \#  \mathbf{M}^{\rm nh}_1(B)$, it would be enough to control $l \EEE \# \mathcal{D}_l$ for every $l\in \N$. 
At this stage, 
one could use Lemma~\ref{lemma:afirststep} (which applies to $B$ in view of Modification 1)  to control $ l \EEE \# \mathcal{D}_l$ for $l\geq 3$ in terms of $\# \mathcal{D}_1$. Concerning the  components in $\mathcal{D}_1 $ and $\mathcal{D}_2 $, we will treat them \EEE differently depending on whether \EEE they are \emph{small} or \EEE  \emph{large}, according to the following definition: \MMM  for a generic set $\compo$ and \EEE  for $l=0,1,2$ we let  \EEE
\begin{equation}\label{defDsmalllarge}
 \mathcal{D}_l^{\rm small}(\compo):=\Big\{ \compo_j\in \mathcal{D}_l(\compo) \EEE \colon  |  {\rm sat}(\compo_j) \EEE |\leq \frac{\varepsilon^2}{\eta^2}\EEE \Big\},\quad   \mathcal{D}_l^{\rm large}(\compo):= \mathcal{D}_l(\compo)\setminus  \mathcal{D}_l^{\rm small}(\compo),
\end{equation}
\MMM where ${\rm sat}(\cdot)$ is defined at the beginning of Section \ref{se: subset prep}. \EEE Now the idea is that,  to bound $\#\mathcal{D}_1 $ and $\#\mathcal{D}_2 $,  it suffices to apply Lemma~\ref{heal entire com} to treat the small components in $\mathcal{D}_1^{\rm small}$, $\mathcal{D}_2^{\rm small}$, since the cardinality of the remaining ones is less than $C \eta/\varepsilon$. 
 In fact, for $H$ with   $\mathcal{H}^1(\partial H)\leq C$  (as it holds for $B$, see \eqref{def of bbb} and \eqref{global bound on b}) \EEE it follows that 
\begin{align}\label{not-many-large-compoenents}
 \#\mathcal{D}_{l}^{\rm large}(H) \leq  C  \frac{\eta}{\eps}, \quad \quad   l=0,1,2. \EEE 
\end{align}
 To see this, for each $\compo_j \in \mathcal{D}_{l}^{\rm large}(H)$,  we apply the isoperimetric inequality to obtain  $\mathcal{H}^1(\partial \compo_j) \ge  \mathcal{H}^1(\partial \,  {\rm sat} (\compo_j)) \ge c |{\rm sat} (\compo_j)|^{1/2} \ge  c\varepsilon/\eta$ for some universal $c>0$.  Then,  $ \#   \mathcal{D}_{l}^{\rm large}(H)  \le C\mathcal{H}^1(\partial H)\, \eta/\eps \le C\eta/\eps \EEE$.  
 \EEE

This program, however, cannot be pursued for $B$ because $\# \mathcal{D}_1 (B) \EEE$ explicitly appears on the right-hand side of \eqref{afirststep}, so simply healing the small components would not be enough.  Unfortunately, in general, there is no direct bound available for $\# \mathcal{D}_1(B)$ and we need to perform another preliminary modification to control \emph{a priori} $\# \mathcal{D}_1$ on a suitable subset of $B$. Roughly speaking, \EEE we need to get rid of $\mathcal{D}_1^{\rm small}$ \emph{before} applying Lemma~\ref{lemma:afirststep}. Therefore, we will never actually use $B_{\rm heal}$ in the proof but only the object \eqref{AAmod} below which is obtained after the   following   preliminary modification.

\vspace{1em}
\noindent \textbf{Modification 2: Dealing with unions of components in $\mathcal{D}_1(B)$ in terms of separating vertices.} A vertex $v$ in  $\bigcup_{k \ge 2} \mathcal{V}_{2k}(B)$ is called a \emph{separating vertex} if removing $v$ and the associated edges  from $(\mathcal{V}(B), \mathcal{E}(B))$ would increase the number of connected components of the graph $(\mathcal{V}(B), \mathcal{E}(B))$. Equivalently, in the topology of subsets  of $\R^2$, this corresponds to considering the connected components of  $\overline{B} \setminus \lbrace v \rbrace$ whose number would increase compared to the number of connected components of $\overline{B}$.  We denote the  set of \emph{separating vertices} in $B$ by  $\mathcal{V}_{\rm sep}(B) \subset \mathcal{V}(B)$,  see Figure~\ref{fig: bridges}.  \EEE

 For $v\in \mathcal{V}_{\rm sep}(B)$ let $\mathcal{G}_{\rm sep}(v)$ be the connected components of $\overline{B} \setminus \lbrace v \rbrace$.  We define 
$$\mathcal{G}_{\rm sep} \defas
\bigcup_{v \in \mathcal{V}_{\rm sep}(B)}  \mathcal{G}_{\rm sep}(v) \cup  
\mathcal{C}(\ol{B}) \,.
$$
Here, we explicitly add  also $\mathcal{C}(\ol{B})$, i.e.,  all connected components of $\ol{B}$. (Note that this is redundant whenever there are at least two vertices in  $\mathcal{V}_{\rm sep}(B)$ in different components of $\mathcal{C}(\ol{B})$.)   


\begin{figure}
  \begin{tikzpicture}[scale=0.3]

    \draw[fill=blue!15!white, scale=2, decoration={zigzag}] decorate{\myellipse};
      
%
\draw[color=magenta, thick,fill=blue!15!white, scale=1,rotate=90, shift={(5.58 cm,5 cm)}, decoration={zigzag, segment length=5mm}] decorate{\myellipse};
      \draw[fill=green!15!white, scale=1,rotate=90, shift={(-5.58 cm,5 cm)}, decoration={zigzag, segment length=4mm}] decorate{ \myellipse};
      \draw[fill=green!15!white, scale=0.8,rotate=90, shift={(-6.18 cm,-5 cm)}, decoration={zigzag, segment length=4mm}] decorate{ \myellipse};
      \filldraw[red] (-5,1.58) circle (3pt) ;
      \filldraw[red] (4,-1.72) circle (3pt);
      \filldraw[red] (-5,-1.58) circle (3pt);
      
      \draw[fill=green!15!white, scale=0.5, shift={(-4.08 cm,14 cm)}, decoration={zigzag, segment length=3mm}] decorate{ \myellipse};
      \draw[fill=green!15!white, scale=0.5,rotate=0,  shift={(-4.08 cm,8.42 cm)}, decoration={zigzag, segment length=4mm}] decorate{ \myellipse};

      \filldraw[red] (4,-1.72) circle (2pt);
      \filldraw[red] (-5,-1.58) circle (2pt);

%
 \draw[fill=green!15!white, scale=0.3, decoration={zigzag, segment length=7mm}, rotate=180] decorate{ (10.06,10)+\myellipse};
     \draw[fill=green!15!white, scale=0.3, decoration={zigzag, segment length=7mm}, rotate=180] decorate{(10.02,27)+\myellipse};
      \draw[fill=green!15!white, scale=0.3, decoration={zigzag, segment length=7mm},rotate=180] decorate{(9.29,19)+\myellipse};
  
      \draw[fill=green!15!white, scale=0.3, decoration={zigzag}, rotate=0] decorate{(-23.23,10)+\myellipse};
    \draw[fill=green!15!white, scale=0.3, decoration={zigzag}, rotate=0] decorate{(-23.29,27)+\myellipse};
     \draw[fill=green!15!white, scale=0.3, decoration={zigzag}, rotate=0] decorate{(-24.03,19)+\myellipse};
      \filldraw[red] (-5.75,2.98) circle (3pt) ;
      \filldraw[red] (-5.78,8.08) circle (3pt) ;
      \filldraw[red] (-5.99,5.7) circle (3pt) ;
      \filldraw[red] (-4.05,4.2) circle (3pt) ;
      \filldraw[red] (-4.05,7.) circle (3pt) ;
       \filldraw[red] (-4.22,-8.1) circle (3pt) ;
       \filldraw[red] (-4.22,-3.) circle (3pt) ;
       \filldraw[red] (-3.99,-5.7) circle (3pt) ;
  \end{tikzpicture}
  \caption{Schematic illustration of different components of $\ol{B}$. \MMM Note that $\ol{B}$ is connected and $\# \mathcal{C}(B) = 12$. \EEE The red dots depict the separating vertices. \MMM In this example the green parts correspond to  the $7$ elements in \EEE $\mathcal{G}_{\rm sep}^{\rm small}$. \JJJ The component with magenta border depicts a possible component in $\mathcal{D}_1^{\rm small}(B_{\rm sep})$. \EEE  } 
  \label{fig: bridges}
\end{figure}
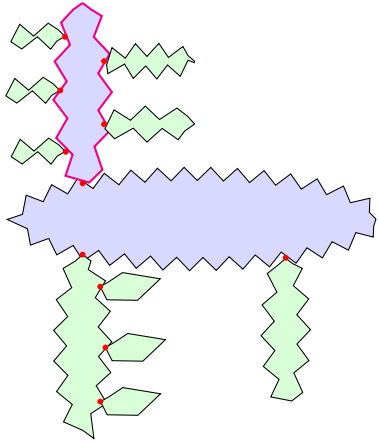

Note that each element in $\mathcal{G}_{\rm sep}$ consists  of unions of components in $\mathcal{C}(B) = (B_j)_j$ (up to \EEE a set of negligible measure). We observe that two elements $G_1 , G_2 \in  \mathcal{G}_{\rm sep}$ satisfy $|G_1 \cap G_2| = 0$ or, up to relabeling, $G_1 \subset G_2$. We \EEE define 
\begin{align}\label{eq: sepsamll}
\mathcal{G}^{\rm small}_{\rm sep}  & \VIT := \EEE  \big\{ G \in \mathcal{G}_{\rm sep} \colon   | {\rm sat}(G) \EEE | \le \eps^2/\eta^2,   \, \text{there is no }  \tilde{G} \in \mathcal{G}_{\rm sep}   \text{ with }   | {\rm sat}(\tilde{G}) \EEE| \le \eps^2/\eta^2  \EEE \text{ and }  G \subset \tilde G  \big\}. 
\end{align} 
For each $G\in \mathcal{G}^{\rm small}_{\rm sep}$, it holds $\mathcal{H}^1(\partial G)\leq \eps/ \eta^3$. In fact, $G = {\rm sat}(G)$  by the definition of $B$,  see \eqref{def of bbb}. \EEE  As $c\eps^2\leq |T|$  for  some $c>0$,   we obtain $\# \mathbf{T}_{G}\leq \frac{1}{c\eta^2}$. Thus,  by the discrete H\"older inequality \EEE
\begin{equation}\label{manus-blackboard}
  \mathcal{H}^1(\partial G)\leq \sum_{T\in \mathbf{T}_G} \mathcal{H}^1(\partial T) \leq C \sum_{T\in \mathbf{T}_G} \sqrt{|T|} \leq C \big(\# \mathbf{T}_G\big)^{1/2} \Big(\sum_{T\in \mathbf{T}_G} |T|\Big)^{1/2} \leq C \frac{\eps}{\eta^2}  \le \frac{\eps}{\eta^3}\,, \EEE  
\end{equation} 
where we used that
$\mathcal{H}^1(\partial T)^2\leq C |T|$ for some $C>0$ only depending on $\theta_{0}$,  and the last step holds for $\eta \le \eta_0$ with $\eta_0$ small enough. \EEE Note that on each element in $\mathcal{G}^{\rm small}_{\rm sep}$ the assumptions of Lemma \ref{heal entire com} are satisfied:  if \EEE $G \in \mathcal{G}^{\rm small}_{\rm sep}\cap \mathcal{G}_{\rm sep}(v)$, in the notation of Lemma~\ref{heal entire com}, $G 	=  {\rm sat}(G) \EEE $ corresponds to $\MMM\overline{Z}\EEE$ and $Y$ is the union of the triangles containing $v$ and included in $(\ol B \setminus G) \cup \{v\}$ (therefore $\overline{Y} \cap \overline{G}=\{v\}$) (see Figure~\ref{fig: bridges}).  The condition ${\rm dist}(B,\partial\Omega) \ge \omega(\eps)$ ensures that $N_G \setminus Y \subset \Omega \setminus B$. \EEE  If $G\in \mathcal{G}^{\rm small}_{\rm sep}\cap \mathcal{C}(\ol{B})$, the assumptions of Lemma \ref{heal entire com} are  fulfilled with $Y=\emptyset$. \EEE We define 
\begin{equation}\label{defBbri}
B_{\rm sep} \VIT := \EEE B     \setminus   \bigcup_{G \in\mathcal{G}^{\rm small}_{\rm sep}} G,  
\end{equation}
and modify \MMM $u$ \EEE on each $G \in \mathcal{G}^{\rm small}_{\rm sep}$ as in Lemma \ref{heal entire com}. 
This  leads  to a function $ u_{\rm sep} \in H^1(\Omega;\R^2)$ with
\begin{align}\label{vvvheal}
 \Vert e( u_{\rm sep} \EEE)    \Vert_{L^2(\Omega \setminus B_{\rm sep})} \le  \frac{C}{\eta^\alpha} \EEE \Vert e(  u  ) \Vert_{L^2(\Omega \setminus B)}.
 \end{align}
\GGG Above, we used \EEE that the neighborhoods $N_G$ in Lemma \ref{heal entire com} related to different components $G \in \mathcal{G}^{\rm small}_{\rm sep}$ overlap only a  bounded number of times depending on $\theta_0$: \EEE in fact, a triangle belongs to $N_G$ if it has nonempty intersection with some triangle in $G$, and for every $G\neq G' \in \mathcal{G}^{\rm small}_{\rm sep}$ it holds $G\cap G'=\emptyset$. Thus, since any $T\in \mathbf{T}$ has nonempty intersection with at most $c$ ($c$ depending only on $\theta_0$) different triangles in $\mathbf{T}$, then any $T\in \mathbf{T}$ belongs to at most $c$ different neighborhoods $N_G$. \EEE

Moreover, we have $\partial B_{\rm sep} \subset \partial B$ and  by using \eqref{energy3} we get  
\begin{align}\label{energy4}
\frac{2}{\eps  \sin\theta_0}|A | & \ge \mathcal{H}^1(\partial B_{\rm sep}) +  \sum_{j=0,1} \sum_{\GGG T \EEE  \in \mathbf{M}_j(B_{\rm sep})  }   \Big( \frac{2}{\eps  \sin\theta_0}|\GGG T \EEE | -   \mathcal{H}^1( \partial B_{\rm sep} \cap \GGG T \EEE ) \Big).
\end{align}
This simply follows from the fact that we remove entire components from $B$: indeed,   it holds that $\mathbf{M}_j(B_{\rm sep})=\mathbf{M}_j(B) \cap \mathbf{T}_{\rm sep}$ for $\mathbf{T}_{\rm sep}:=\{T\in \mathbf{T}\colon T\subset \ol{B_{\rm sep}}\}$. Moreover, for $j=0,1,2$, we have that $\mathcal{H}^1(\partial B \cap T)= \mathcal{H}^1( \partial B_{\rm sep} \cap T)$ for every $T\in \mathbf{M}_j(B_{\rm sep})$ and  $\mathcal{H}^1( \partial B \cap T)=\mathcal{H}^1((\partial B \setminus \partial B_{\rm sep}) \cap T)$ for every $T\in \mathbf{M}_j(B)\setminus \mathbf{M}_j(B_{\rm sep})$. Therefore, $\mathcal{H}^1( \partial B)-\mathcal{H}^1( \partial B_{\rm sep})=\mathcal{H}^1( \partial B \setminus \partial B_{\rm sep})$ and hence 
\begin{equation*}
\begin{split}
& \mathcal{H}^1(\partial B) +  \sum_{j=0,1} \sum_{ T   \in \mathbf{M}_j(B)  }   \Big( \frac{2}{\eps  \sin\theta_0}| T  | -   \mathcal{H}^1( \partial B \cap  T  ) \Big) \\& \hspace{2em}-  \Big( \mathcal{H}^1(\partial B_{\rm sep}) +  \sum_{j=0,1}  \sum_{ T   \in \mathbf{M}_j(B_{\rm sep})  }   \Big( \frac{2}{\eps  \sin\theta_0}| T | -   \mathcal{H}^1( \partial B_{\rm sep} \cap  T ) \Big)  \Big)
\\& \hspace{1em}
= \mathcal{H}^1( \partial B \setminus \partial B_{\rm sep}) + \sum_{j=0,1}\sum_{ T  \in \mathbf{M}_j(B)\setminus \mathbf{M}_j(B_{\rm sep}) }   \Big( \frac{2}{\eps  \sin\theta_0}| T  | -   \mathcal{H}^1( (\partial B \setminus\partial B_{\rm sep}) \cap  T  ) \Big) 
\\& \hspace{1em}
= \sum_{ T  \in \mathbf{M}_2(B)\setminus \mathbf{M}_2(B_{\rm sep}) } \mathcal{H}^1( (\partial B \setminus\partial B_{\rm sep}) \cap  T  ) + \sum_{j=0,1}\sum_{ T  \in \mathbf{M}_j(B)\setminus \mathbf{M}_j(B_{\rm sep}) }   \frac{2}{\eps  \sin\theta_0}| T  | \geq 0.
\end{split}
\end{equation*} \EEE
 This along with \eqref{energy3} shows \eqref{energy4}. \EEE By Lemma~\ref{lemma:afirststep} applied to $\compo=B_{\rm sep}$ we get \EEE
\begin{align}\label{lllge3}
  \sum_{l \ge 3} l \# \mathcal{D}_l(B_{\rm sep}) \le    C   \# \mathcal{D}_1(B_{\rm sep}) +  C \frac{\eta}{\eps}. 
  \end{align}
   Here, by using \eqref{global bound on b},  the constant $C$ only depends on $C_0$. \EEE The \GGG crucial point is that, differently from $\mathcal{D}_1(B)$, \EEE now $\mathcal{D}_1(B_{\rm sep})$ satisfies the additional fundamental property  
\begin{align}\label{eq: YYYYYY}
 \# \mathcal{D}_1(B_{\rm sep}) \le  C\eta/\eps.
 \end{align}
Indeed, \GGG recalling definition \eqref{defDsmalllarge}, 
 by \eqref{global bound on b} and \eqref{not-many-large-compoenents} we have
\begin{align}\label{eq: also2}
\# \mathcal{D}^{\rm large}_l(B_{\rm sep}) \le C\eta/\eps \quad  \text{for } l =0,1,2.\EEE
\end{align} 
Now, consider a component $B^i_{\rm sep} \in \mathcal{D}^{\rm small}_1(B_{\rm sep})$ and the corresponding $v \in \mathcal{V}(B^i_{\rm sep})$ with $v \in \partial B^i_{\rm sep}$   and  $n(v) \ge 4$, see \eqref{Dll}.  Note that $v$ is a separating vertex of the set $B$ considered above,  i.e., $v\in \mathcal{V}_{\rm sep}(B)$. \EEE
 Then, the only reason why this component has not been removed from $B$ in the construction of $B_{\rm sep}$  \MMM  is the fact that there is some $G_i \in \mathcal{G}_{\rm sep}(v)$ with $G_i \supset B^i_{\rm sep}$ and $| {\rm sat}(G_i)  | > \eps^2/\eta^2$. \EEE \JJJ For an example of such $B_{\rm sep}^i$, see the magenta-bordered component in Figure \ref{fig: bridges}. \EEE As the sets $(G_i)_i$ are pairwise disjoint for different $B^i_{\rm sep} \in \mathcal{D}^{\rm small}_1(B_{\rm sep})$,  by  \eqref{global bound on b} and repeating the argument below   \eqref{not-many-large-compoenents}  we can compute 
 \[
 \#  \mathcal{D}^{\rm small}_1(B_{\rm sep}) \frac{\eps}{\eta} \leq   \sum_{B_{\rm sep}^{i}\in \mathcal{D}^{\rm small}_1(B_{\rm sep})} |{\rm sat}  (G_i) \EEE|^{1/2} \le C \sum_{B_{\rm sep}^{i}\in \mathcal{D}^{\rm small}_1(B_{\rm sep})} \mathcal{H}^1( \partial  G_i \EEE )  \leq C\mathcal{H}^1(\partial B) \MMM \le C \EEE  \, ,
 \] 
 and thus, \JJJ together with \eqref{eq: also2} we get \EEE \eqref{eq: YYYYYY}. \EEE  Summarizing, in view of \eqref{lllge3}--\eqref{eq: also2}, we obtain \EEE  
\begin{align}\label{counti}
 \# \mathcal{D}^{\rm large}_0(B_{\rm sep}) + \EEE   \# \mathcal{D}_1(B_{\rm sep})  +   \# \mathcal{D}^{\rm large}_2(B_{\rm sep}) +  \sum_{l \ge 3}  l\# \mathcal{D}_l(B_{\rm sep}) \le    C \frac{\eta}{\eps}.
  \end{align}

\noindent\textbf{Modification 3: Healing.}
Eventually, we define 
\GGG \begin{equation*}
\widehat{B}_{\rm sep}:=B_{\rm sep} \setminus  \bigcup_{B^j_{\rm sep}\in \mathcal{D}^{\rm small}_2( B_{\rm sep})  }   B^j_{\rm sep},
\end{equation*} \EEE
i.e., we remove the small components  $\mathcal{D}^{\rm small}_2( B_{\rm sep})$, and,   recalling \eqref{healoftri}, 
\begin{equation}\label{AAmod}
A_{\rm mod}:=\Big(\widehat{B}_{\rm sep}\Big)_{\rm heal}.
\end{equation}
 Note that by construction $A_{\rm mod}$ cannot have `holes' smaller than $\eps^2/\eta^2$. \EEE Hence, we can use \GGG first  Lemma~\ref{heal entire com} 
for   components in \EEE $\mathcal{D}^{\rm small}_2(B_{\rm sep})$ and then Lemma \ref{healing1} for  triangles in \EEE$\mathbf{M}^{\rm heal}_j(\widehat{B}_{\rm sep})$, $j=0,1$, \EEE to find a function $u_{\rm mod} \in H^1(\Omega;\R^2)$ with
 $$\Vert e(u_{\rm mod })  \Vert_{L^2(\Omega \setminus A_{\rm mod})} \le  \frac{C}{\eta^\alpha}   \Vert e( u_{\rm sep } )  \Vert_{L^2(\Omega \setminus B_{\rm sep})}.$$
  (Here, for Lemma \ref{heal entire com}, we again use that neighborhoods only overlap a finite number of times, see \JJJ the argument \EEE below \eqref{vvvheal}.)
Using that $A\subset B$ and then $\|e(u)\|_{L^2(\Omega\sm B)}\leq \|e(u)\|_{L^2(\Omega\sm A)}$, \EEE with \eqref{vvvheal} we get
\begin{align}\label{eq. finaaal}
\Vert e(u_{\rm mod })  \Vert_{L^2(\Omega \setminus A_{\rm mod})}  \le  \frac{C}{\eta^{2\alpha}} \EEE\|e(u)\|_{L^2(\Omega\sm A)}.
\end{align}
 Since we assumed \eqref{energybound1}, this gives the first part of \eqref{themainthing}.

Let us now confirm the second part:  in view of \eqref{counti},   the main property of $\widehat{B}_{\rm sep}$ is that
\begin{align}\label{counti2}
 \sum_{l \ge 1}  l\# \mathcal{D}_l(\widehat{B}_{\rm sep}) \le    C \frac{\eta}{\eps}.  
  \end{align}
  \GGG Therefore, taking $\compo=\widehat{B}_{\rm sep}$ in \eqref{1308241138}, \EEE along with \eqref{energy4} and \eqref{defMnotheal}, by arguing as in \eqref{for later.end}, \EEE we deduce
  \begin{align}\label{energy2}
\frac{2}{\eps  \sin\theta_0}|A| & \ge  \mathcal{H}^1(\partial A_{\rm mod}) -  C \eps   \big(\#\mathbf{M}^{\rm nh}_0(\widehat{B}_{\rm sep}) + \#  \mathbf{M}^{\rm nh}_1(\widehat{B}_{\rm sep})\big).
\end{align}
 \GGG Taking $\widehat{B}_{\rm sep}$ as $\compo$ in \eqref{iddd} and \eqref{1408240756},  \EEE
we then get
\begin{align}\label{finalstep}
\frac{2}{\eps  \sin\theta_0}|A | & \ge \mathcal{H}^1(\partial A_{\rm mod}) -  C \eps   \sum_{l \ge 1} l \# \mathcal{D}_l(  \widehat{B}_{\rm sep}\EEE).
\end{align}
This along with \eqref{counti2} shows the second part of \eqref{themainthing}. \EEE Next,  \eqref{themainthing0} follows from the fact that $\lbrace u \neq u_{\rm mod} \rbrace \subset  (B \setminus A) \cup A = B$, $A_{\rm mod} \subset B$,  and \eqref{global bound on b}. Finally,  to validate \eqref{extra-statement}, \EEE in view of \eqref{counti2}, we are left to estimate the number of components $\mathcal{D}_{0}(\widehat{B}_{\rm sep}) $.  First, $ \# \mathcal{D}_{0}^{\rm large}(\widehat{B}_{\rm sep})$ is already controlled by \eqref{counti}. \MMM Each $A\in \mathcal{D}_{0}^{\rm small}(\widehat{B}_{\rm sep})$ is either  some element of $\mathcal{G}_{\rm sep}^{\rm small}$, see the definition in \eqref{eq: sepsamll}, or a component \JJJ in \MMM $\bigcup_{l \ge 1}\mathcal{D}_l(B_{\rm sep})\setminus \mathcal{D}^{\rm small}_2(B_{\rm sep})$. In the first case,   such small isolated components were already removed in Modification ~2 (see \eqref{defBbri}), i.e., \JJJ it holds \MMM $\mathcal{D}_{0}^{\rm small}(\widehat{B}_{\rm sep}) \subset \bigcup_{l \ge 1}\mathcal{D}_l(B_{\rm sep}) \setminus \mathcal{D}^{\rm small}_2(B_{\rm sep})$. Then, the desired control follows from \eqref{counti}. \EEE 

\noindent \textbf{General case: Components close to $\partial \Omega$\VIT .\EEE} Recall   that,   so far,    we assumed \MMM that \EEE   ${\rm dist}(A,\partial\Omega) \ge \omega(\eps)$ as this allowed us to apply  Lemmas \ref{healing1}--\ref{heal entire com} throughout the proof.  In particular, we have healed the components $\mathcal{G}^{\rm small}_{\rm sep}$ in Modification  2 and the components  $\mathcal{D}_2^{\rm small}(B_{\rm sep})$ in Modification~3. If such components have distance from $\partial \Omega$ smaller than $\omega(\eps)$, the extension in  Lemma \ref{heal entire com} cannot be performed. Yet, we observe that all such components have diameter smaller than $\eps/\eta^3$, see \eqref{manus-blackboard} for $\mathcal{G}^{\rm small}_{\rm sep}$ (the computation for $\mathcal{D}_2^{\rm small}(B_{\rm sep})$ is exactly the same). Thus, such components are  contained in $\lbrace x \in \Omega \colon {\rm dist}(x,\partial \Omega) \le \omega(\eps) + \eps/\eta^3 \rbrace$, i.e., have empty intersection with $\Omega_{\eps,\eta}$. In a similar fashion, triangles in Lemma \ref{healing1} cannot be healed if their distance from $\partial \Omega$ is smaller than $\omega(\eps)$ which means they are contained in  $\lbrace x \in \Omega \colon {\rm dist}(x,\partial \Omega) \le 2\omega(\eps) \rbrace \subset \Omega \setminus \Omega_{\eps,\eta}$.  Summarizing,  for sets $A \subset \Omega$ which do not satisfy  ${\rm dist}(A,\partial\Omega) \ge \omega(\eps)$ we get that  \eqref{eq. finaaal} holds with $\Omega_{\eps,\eta} \setminus A_{\rm mod}$ in place of $\Omega \setminus A_{\rm mod}$ which shows   \eqref{themainthing}.   
\end{proof}


\begin{remark}\label{remark-on-holes}\MMM 
    The construction implies that  $\partial T\cap \partial A_{\rm mod}\,\tilde{=}\,\emptyset$ for all $T \subset A_{\rm mod}$ with $T \notin \mathbf{T_A}$. In fact, for the set $B$ defined in Modification 1 (see \eqref{def of bbb}) it clearly holds  $\partial T\cap \partial B\,\tilde{=}\,\emptyset$ for all $T \in \mathbf{T}$ with $T \subset B\setminus A$. Then, we recall that in Modifications 2 and 3 we only remove entire saturated components or heal away triangles. \EEE
\end{remark} 

\noindent
We now proceed with the proof of \EEE  Corollary \ref{monotonicity-corollary}. 
\begin{proof}[Proof of Corollary \ref{monotonicity-corollary}]
Let  $\mathbf{T}_{A^1} \subset \mathbf{T}_{A^2}$. \EEE First, we apply Theorem \ref{generaltheorem} to $A^{1}$ and $A^2$ and obtain $A_{\rm mod}^1,A_{\rm mod}^2$ as well as $u_{\rm mod}^{1}$ and $u_{\rm mod}^{2}$. We  note that in general the monotonicity is not preserved in all the different modification steps carried out in the proof of Theorem \ref{generaltheorem}. In fact, as will will explain below, this is the case for Modifications 1--2 of the construction above,  but in Modification  3 one might remove components $D\in \mathcal{D}_{2}^{\rm small}( B^2_{\rm sep}\EEE)$ with $A_{\rm mod}^1 \cap D\neq \emptyset$. However, as we will point out, such components could have been `healed away' already before in the construction of $A_{\rm mod}^1$.

  Before we start, we observe that it suffices to show    $A_{\rm mod}^1\subset A_{\rm mod}^2$. Indeed,  due to $\mathbf{T}_{A^1} \subset \mathbf{T}_{A^2}$ and $A_{\rm mod}^1\subset A_{\rm mod}^2$, $T \in \mathbf{T}^{\rm mod}_{A^1}$ implies $T  \in \mathbf{T}^{\rm mod}_{A^2}$, and thus  $\mathbf{T}^{\rm mod}_{A^1} \subset \mathbf{T}^{\rm mod}_{A^2}$ directly follows. \EEE

\noindent \textbf{Modification 1}: For each connected component $A_c\in \mathcal{A}_{\rm small}^1$ of $\R^2\sm \overline{A^1}$ \MMM with $ | A_c|\leq \eps^2/\eta^2$, \EEE  we have that either $A_c\subset A^2$ or $A_c\sm \overline{A^2}$ is a connected component of $\R^2\sm \overline{A^2}$ with $ | A_c\sm \overline{A^2}|\leq \eps^2/\eta^2$. Thus, $A_c\sm \overline{A^2}\in \mathcal{A}^2_{\rm small}$. By the definition of $B^1$ and $B^2$ \EEE in Modification 1, \JJJ see \eqref{def of bbb}\EEE, we hence obtain $B^1\subset B^2$.  

\noindent \textbf{Modification 2}:
Next, we need to show that $B^1_{\rm sep}\subset B^2_{\rm sep}$. Recalling the definition in \eqref{defBbri} it is enough to show that \begin{equation}\label{crucial-implication}
  \bigcup_{G\in \mathcal{G}^{{\rm small},2}_{\rm sep}} G \cap \overline{B^1} \subset \bigcup_{G\in \mathcal{G}^{{\rm small},1}_{\rm sep}} G 
\end{equation} 
 For notational convenience, \EEE we formally denote the connected components  \MMM $\mathcal{C}(\ol{B^i})$ of $\ol{B^i}= \ol{B^i}\sm \emptyset$ with $\mathcal{G}^i_{\rm sep}(\emptyset)$. \EEE Then, by a slight abuse of notation, we add $\{\emptyset\}$ as a placeholder for a vertex to the set of separating vertices, i.e., $\mathcal{\tilde{V}}_{\rm sep}(B^{i})= \mathcal{V}_{\rm sep}(B^{i}) \cup \{\emptyset\}$ and write $G\in   \mathcal{G}^{i}_{\rm sep} \EEE= \bigcup_{v\in \mathcal{\tilde{V}}_{\rm sep}(B^{i})} \mathcal{G}^i_{\rm sep}(v)$,  and accordingly $ \mathcal{G}^{{\rm small},i}_{\rm sep}$ as in \eqref{eq: sepsamll}. \EEE

Let $x\in G\cap \overline{B^1} $ for some $G\in \mathcal{G}^{{\rm small},2}_{\rm sep}$. By definition we know that $G = {\rm sat}(G) \in \mathcal{G}^2_{\rm sep} (v)$ for some $v\in \mathcal{\tilde{V}}_{\rm sep}(B^2)$, i.e., $G$ is a connected component of $\overline{B^2}\sm \{v\}$ with $|G|\leq \eps^2/\eta^2$. Note that, as $B^1\subset B^2$, the set $G\cap \overline{B^1}$ possibly consists of different connected components of $\overline{B^1}\sm \{v\}$, i.e., $G\cap \overline{B^1}= \bigcup_{j}   G_j$ for suitable   components $(G_j)_j$. 

If $v\in \mathcal{\tilde{V}}_{\rm sep}(B^1)$, i.e., $v$ is also a separating vertex of $B^1$, we can conclude by definition that $G_j\in \mathcal{G}^1_{\rm sep}$ for all $j$.  As $| {\rm sat}({G}_j)|\leq |{\rm sat}(G\cap B^1)| \leq |G| \leq \eps^2/\eta^2$, we get ${G}_j \in \mathcal{G}^{{\rm small},1}_{\rm sep}$ for all $j$.  Since $G\cap \overline{B^1} =\bigcup_{j} {G}_j$, we have $x\in  {G}_j $ for some $j$.

  If instead \EEE $v\notin \mathcal{\tilde{V}}_{\rm sep}(B^1)$, the connected components of $\overline{B^1}$ intersected with $\R^2 \setminus \lbrace v\rbrace$ are exactly the components of $\overline{B^1}\sm \{v\}$, i.e., $G_j$  \MMM as above \EEE are components of $\overline{B^1}$ intersected with $\R^2 \setminus \lbrace v\rbrace$. 
 Hence, by setting accordingly $\tilde{G}_j=G_j$ or $\tilde{G}_j=G_j\cup \{v\}$,   for all $j$ we find a set $\tilde{G}_j\in \mathcal{G}^1_{\rm sep}(\emptyset)\subset \mathcal{G}_{\rm sep}^{1}$.   
Since $|G\cup \{v\}|= |G|$, we have in particular that $| {\rm sat}(\tilde{G}_j) \EEE | \leq |G| \leq \eps^2/\eta^2$ and thus $\tilde{G}_j \in \mathcal{G}^{{\rm small},1}_{\rm sep}$ for all $j$.  As  $G\cap \overline{B^1}  = \bigcup_{j} \tilde{G}_j \setminus \lbrace v \rbrace \EEE $, we have $x\in  G_j \EEE $ for some $j$, which concludes the proof of \eqref{crucial-implication}.  
\EEE

\noindent \textbf{Modification 3}: Note that $\widehat{B}_{\rm sep}^1\subset \widehat{B}^2_{\rm sep}$ is  in principle \EEE not true as for a component $ D \EEE \in \mathcal{D}^{\rm small}_2( B^2_{\rm sep})$ we might have $A_{\rm mod}^1 \cap  D \EEE \neq \emptyset$. However, because $D\in \mathcal{D}^{\rm small}_2(B^2_{\rm sep})$ we know that  $\overline{D\cap B_{\rm sep}^1}$ touches $\overline{B^1_{\rm sep}}\sm (D\cap B_{\rm sep}^1)$ \EEE at most at two points.   Since $D = {\rm sat}(D)$, we further get  $|{\rm sat}(D\cap B_{\rm sep}^1)|\leq |D| \EEE \leq \frac{\eps^2}{\eta^2}$. In particular, ${\rm sat}(D\cap B_{\rm sep}^1) = D\cap B_{\rm sep}^1$ due to the construction of $B^1$, and  \EEE the set $D\cap B_{\rm sep}^1$ (or, respectively, all its connected components) fulfill the assumptions of Lemma \ref{heal entire com}. This means we could have healed $D\cap B_{\rm sep}^1$ already in the construction of $\widehat{B}^1_{\rm sep}$, which then ensures $\widehat{B}_{\rm sep}^1\subset \widehat{B}^2_{\rm sep}$. Since the healing of triangles as in \eqref{AAmod} also  preserves \EEE the monotonicity, we have $A^1_{\rm mod}\subset A^2_{\rm mod}$.
\end{proof}

\EEE

\subsection{Proof of healing lemmas}\label{sec: proflem} It \EEE now remains to prove Lemma \ref{healing1} and Lemma \ref{heal entire com}. 

\begin{proof}[Proof of Lemma \ref{healing1}]
  For each $T\in \mathbf{M}^{\rm heal}_0(\compo) \cup \mathbf{M}^{\rm heal}_1(\compo)$, let $N^*_T$ be the  union of triangles in $\Omega\setminus  {\compo}$ having nonempty intersection with $T$. Then, let $N_T$ be the connected component of $N^*_T$ containing the vertex in the definition \eqref{defMheal}. Moreover, let $T^1$ and $T^2$ be two adjacent triangles to $T$ (i.e., sharing an edge with $T$)  that are contained in $N_{T}$. (Notice that the choice is unique for $\mathbf{M}^{\rm heal}_1(\compo)$, and that there are up to three different choices of pairs $\lbrace T^1,T^2\rbrace $ for $\mathbf{M}^{\rm heal}_0(\compo)$.)  \EEE
  For every $u$ affine on any triangle $T$, we denote by  $e(u)_T$ and \EEE $(\nabla u)_T$ the constant matrices $e(u)$ and \EEE $\nabla u$ on $T$, respectively. 
  
 \MMM By Korn's inequality we  get \EEE a function $z(x)=u(x)-Ax$,  where  $A \in \R^{2 \times 2}_{\rm skew}$, \EEE   such that 
  \begin{equation}\label{1008241052}
  \|\nabla z\|^2_{L^2(N_T)} \leq \GGG K_{N_T} \EEE \|e(u)\|^2_{L^2(N_T)}\,.
  \end{equation}  
  We notice that the Korn constant $K_{N_T}$ corresponding to $N_T$ depends only on the parameter $\theta_0$ associated to the family of admissible triangulations.
  Defining \EEE  
  \[
  A^{j}:=(\nabla u)_{T^j} - e(u)_{T^j}
  \] 
  for $j=1,2$, \EEE we deduce 
  \begin{equation}\label{eq: AAj}
  |A-A^{j}|^2\,|T^j|= \|A-A^{j}\|^2_{L^2(T^j)}\leq  2 \EEE  \|\nabla z\|^2_{L^2(T^j)}+ 2 \EEE  \|e(u)\|^2_{L^2(T^j)} \leq 2 \EEE  (K_{N_T} +1) \|e(u)\|^2_{L^2(N_T)},
  \end{equation}
    by  the identity \EEE  $\|e(u)\|^2_{L^2(T^j)}= |T^j| |e(u)_{T^j}|^2$,   the \EEE triangle inequality, and \eqref{1008241052}.
  Denoting by $l_j$ the unit vectors parallel to the edge in common between $T$ and $T^j$ for $j=1,2$, we get
  \begin{equation*}
  \begin{split}
  \|e(u)\|^2_{L^2(T)}& =|T||e(u)_T|^2  =|T||e(z)_T|^2\leq |T| |(\nabla z)_T|^2 \leq  \bar{C} \EEE  |T| \Big( |(\nabla z)_T \cdot l_1|^2 + |(\nabla z)_T \cdot l_2|^2 \Big) 
  \\
  &=  \bar{C} \EEE |T| \Big( |(\nabla z)_{T^1} \cdot l_1|^2 + |(\nabla z)_{T^2} \cdot l_2|^2 \Big) 
  \end{split}
  \end{equation*}
  for $ \bar{C}>0 \EEE$ depending only on $\theta_0$, where we used that by the continuity of $z$ it holds $(\nabla z)_T \cdot l_j= (\nabla z)_{T^j} \cdot l_j$ for $j=1,2$.   Since  $(\nabla z)_{T^j}=(\nabla u)_{T^j}-A= e(u)_{T^j}+A^j-A$, we hence obtain by \eqref{eq: AAj} \EEE
  \begin{equation*}
    \begin{aligned}
       \|e(u)\|^2_{L^2(T)} & \leq   \bar{C} \EEE |T| \Big( |(\nabla z)_{T^1} \cdot l_1|^2 + |(\nabla z)_{T^2} \cdot l_2|^2 \Big)\leq   \bar{C} \EEE |T| \sum_{j=1,2}\Big( |e(u)_{T^j}|^2 + |A-A^j|^2 \Big)
  \\&  \leq  \bar{C} \EEE  \sum_{j=1,2}\frac{|T|}{|T^j|} \Big( \|e(u)\|^2_{L^2(T_j)} +  2(K_{N_T}+1)\|e(u)\|^2_{L^2(N_T)}\Big) \leq  \bar{C}' \EEE \|e(u)\|^2_{L^2(N_T)} ,
    \end{aligned}
  \end{equation*}
  where $ \bar{C}' >0 \EEE$ just depends on $\theta_0$, recalling that both $K_{N_T}$ and the volume ratio between adjacent triangles   depend only on $\theta_0$. 
  We conclude by summing over \EEE $T\in \mathbf{M}^{\rm heal}_0(\compo) \cup \mathbf{M}^{\rm heal}_1(\compo)$, observing that each triangle in $\MMM \Omega\setminus \compo \EEE $ could belong at most to a bounded number (depending on $\theta_0$) of different $N_T$.
  \end{proof} 

  \EEE
 
\begin{proof}[Proof of Lemma \ref{heal entire com}]
 Along the proof we denote by $N$, $\mathbf{T}_N$ the sets $N_Z$, $\mathbf{T}_{N_Z}$ and \MMM by $C$  a universal  positive   constant, \EEE possibly varying from line to line,     depending only  on the parameter of the triangulation $\theta_0$.  Moreover, $C_\eta$ denotes a generic constant of the form $C \eta^{-\alpha}$ \JJJ for some $\alpha\in \N$. \EEE   
 
  We first show that \EEE
 \begin{equation}\label{1108241555}
\# \mathbf{T}_{N} \leq  \frac{C}{\eta^2}, \quad \quad  |T| \le C\eps^2/\eta^2 \ \text{for all $T \in \mathbf{T}_{N}$}\,. \EEE
 \end{equation} 
  Indeed, \EEE recall that each edge in $\mathbf{T}$ has \MMM at least length \EEE  $\eps$,  so that we obtain $\eps \#\mathcal{V}(Z)\leq \, \mathcal{H}^1(\partial Z)$.  Using  $| Z| \le \eps^2/\eta^2$, by repeating the calculation in \eqref{manus-blackboard}, we get $\mathcal{H}^1(\partial Z) \le C\eps/\eta^2$, and thus $  \#\mathcal{V}(Z)\leq  \eps^{-1} \EEE \mathcal{H}^1(\partial Z) \le C/\eta^2$.  Since each vertex in $\mathcal{V}(Z)$ is contained in only a bounded number of triangles in $\mathbf{T}_N$ (depending on $\theta_{0}$), we obtain the estimate $\# \mathbf{T}_{N} \le  C/\eta^2$. Each $T \in \mathbf{T}$, $T \subset Z$, satisfies $|T| \le \eps^2/\eta^2$.  Thus, as the area of adjacent triangles is comparable by a constant depending on $\theta_0$, we conclude the second part of \eqref{1108241555}. \EEE

Since $Z$ has no holes, we 
 can suppose (up to enlarging $Y$) that  
 $N\setminus Y$ consists of two \MMM connected \EEE components $N_1$ and $N_2$ whose closures contain the two touching points $p,\,q\in   \overline{Y}\cap \overline{Z} \EEE$ (see Figure \ref{fig: two-domains-pic}). In fact, if $N\sm Y$ had further components, their closure would only intersect with $Z$ at one of the touching points and not share an edge with a triangle in $\overline{Z}$. 
We now claim that 
 there are $\JJJ A_1,A_2 \EEE \in \R^{2 \times 2}_{\rm skew}$  such that 
\begin{equation}\label{from-korn}
  \|\nabla u- A_j\|_{L^{2}(T)}\leq   C_\eta \EEE  \|e(u)\|_{L^2(N_j)}\quad \quad \text{for \MMM each triangle \EEE } T\subset \ol N_j,\,\; j=1,2.
 \end{equation}
 To see this, \EEE for every $T \subset \ol N_j$, let 
 \begin{equation}\label{1108241258}
 (\nabla u)_T=e(u)_T + A_T,
 \end{equation}
 for $(\nabla u)_T$, $e(u)_T$, and $A_T$ suitable matrices representing the constant values of $\nabla u$, $e(u)$, and $\nabla u-e(u)$ on $T$.
 Given two adjacent triangles $T_1,T_2 \in \EEE \mathbf{T}_N$ in $N_j$, i.e., sharing an edge, let us consider the circle $C_{1,2}$ with the maximal radius among those \JJJ circles \EEE centered on a point of the common edge and included in $T_1\cup T_2$.  The radius of any $C_{1,2}$ is larger than $ \varepsilon / \MMM C $, where $C$ \EEE  depends only on $\theta_0$. Applying Korn's inequality on $C_{1,2}$, we find $A_{1,2} \in \R^{2 \times 2}_{\rm skew}$ and $K>0$ (the Korn constant for a circle) such that \EEE
 \begin{equation*}
 \|\nabla u - A_{1,2}\|_{L^2(C_{1,2})}\leq K \|e(u)\|_{L^2(C_{1,2})},
 \end{equation*}
 and then
 \begin{equation*}
 \|A_{T_j} - A_{1,2}\|_{L^2(C_{1,2}\cap T_j)} \leq (K+1) \|e(u)\|_{L^2(C_{1,2})}.
 \end{equation*}
Since $A_{T_j}$ are   constant matrices,  noticing that $\|A_{T_j} - A_{1,2}\|^2_{L^2(C_{1,2}\cap T_j)}= \frac{|C_{1,2}|}{2}|A_{T_j} - A_{1,2}|^2$,   we deduce by the triangle inequality that  
 \begin{equation}\label{1108241254}
 |A_{T_1} - A_{T_2}|^2\leq    \frac{8}{|C_{1,2}|} (K+1)^2 \|e(u)\|^2_{L^2(C_{1,2})}.
 \end{equation}
 Being $N_j$ connected, given any $\tilde{T}_1$, $\tilde{T}_2 \subset \ol N_j$, there are triangles $\tilde{\mathbf{T}}  := (\hat{T}_j)_{j=1}^n \subset \mathbf{T}$    included in $\overline{N_j}$  with $\hat{T}_1 = \tilde{T}_1$, $\hat{T}_n = \tilde{T}_2$, such that  $\hat{T}_j$, $\hat{T}_{j+1}$ are adjacent for all $j=1, \ldots, n-1$. \EEE  We   apply  the estimate \eqref{1108241254}   for  \EEE pairs of adjacent triangles  a finite number of times (less than $\# \mathbf{T}_N$). Then, using the \EEE triangle inequality  and also noting  that \MMM the \EEE sets $C_{1,2}$ are such that $|C_{1,2}|\geq \varepsilon^2/C$ and that circles corresponding to different  pairs \EEE overlap at most twice,
 \EEE we find that 
\begin{equation}\label{newbyjoscha}
 |A_{\tilde{T}_1}-A_{\tilde{T}_2}|\leq   \frac{C}{\eps}  \sqrt{\# \mathbf{T}_N}   \|e(u)\|_{L^2(N_j)} \leq  \frac{C_\eta}{ \eps}   \|e(u)\|_{L^2(N_j)},
 \end{equation}
  where in the last step we used \eqref{1108241555}. \EEE We now confirm \eqref{from-korn} by fixing $A_j$ as one of the $A_T$, $T\subset \overline{N_j}$. Notice that for every $T \subset \ol N_j$,  recalling the notation \eqref{1108241258}, we  indeed  have  by \eqref{1108241555}  and \EEE   \eqref{newbyjoscha} \EEE
 \begin{equation*}
 \|\nabla u - A_j\|_{L^2(T)}\leq \|\nabla u - A_T\|_{L^2(T)} + \|A_j - A_T\|_{L^2(T)}= \|e(u)\|_{L^2(T)} + \|A_j - A_T\|_{L^2(T)} \leq   C_\eta  \EEE  \|e(u)\|_{L^2(N_j)}.
 \end{equation*}
This concludes the proof of \eqref{from-korn}. 

By the fact that  $u$ is  affine \EEE on each $T \in \mathbf{T}_N$, $|T| \ge \eps^2/C$, and by \eqref{from-korn} it also follows that 
 \begin{equation}\label{infty-1}
   \JJJ |(\nabla u)_{T}-A_j| = \EEE \|\nabla u- A_j\|_{L^{\infty}( T )}\leq  \frac{C_\eta}{\eps}  \|e(u)\|_{L^2(N_j)}\quad \quad \text{for}\; j=1,2.    
    \end{equation}
The fundamental point in the proof is now that \EEE 
\begin{equation}\label{1108241658}
|A_2-A_1|\leq  \frac{C_\eta}{ \eps} \EEE \|e(u)\|_{L^2(N\setminus Y)}.
\end{equation}    
In fact, recall \EEE  that $p$ and $q$ are connected by a path consisting of \EEE less than $\# \mathbf{T}_N$  edges of triangles in $\mathbf{T}_N$. Then, by \eqref{infty-1} and the Fundamental Theorem of Calculus applied on any edge \EEE of the path connecting $p$ and $q$ to the scalar-valued functions $(u-A_j \cdot)_i$, $j=1,2$, $i=1,2$ \MMM (here, \EEE $(u-A_j \cdot)_i$ are the two components of $x\mapsto u(x)-A_j x$ for $j=1,2$), we get 
\begin{equation}\label{1108241703}
\big|\big(u(q)-u(p)-A_j(q-p)\big)_i \big|\leq  C_\eta \EEE \|e(u)\|_{L^2(N_j)},
\end{equation}
where we used that the path of edges has a length of order   $\sim \frac{\eps}{\eta^2}$,  because $\mathcal{H}^1(\partial Z)\leq   C \eps/\eta^2$,  see \MMM below \EEE \eqref{1108241555}. \EEE 
Subtracting the two \JJJ terms  in \eqref{1108241703} \EEE for $j=1,2$ and for fixed $i$, we get
\begin{equation*}
|\big((A_2-A_1)(q-p)\big)_i|\leq  C_\eta \EEE \|e(u)\|_{L^2(N\setminus Y)},
\end{equation*}
which confirms \eqref{1108241658} since $A_j$ are  skew symmetric    \EEE and $|q-p| \in [\varepsilon,  C \eta^{-2} \eps]$.
    
Combining \eqref{infty-1} and \eqref{1108241658} we deduce that
 \begin{equation}\label{1108241717}
    \|\nabla u- A_1\|_{L^{\infty}(N\setminus Y)}\leq  \frac{C_\eta}{ \eps} \EEE \|e(u)\|_{L^2(N\setminus Y)}.    
    \end{equation}
 By McShane's theorem we find a Lipschitz extension $\widetilde{w}$ of $u-A_1 \cdot$ from $N\setminus  \overline{Y}\EEE$ to $ {\rm int}(\overline{Z\cup(N\setminus Y)}) \EEE$ whose components have the same Lipschitz constant as  $u-A_1 \cdot$. In particular,
 \begin{equation}\label{1108241721}
 \|\nabla \widetilde{w}\|_{L^\infty((N\setminus Y) \cup Z)}\leq  \frac{C_\eta}{\eps} \EEE\|e(u)\|_{L^2(N\setminus Y)}, \qquad \widetilde{w}= u-A_1 \cdot \text{ on }N\setminus  \ol{Y}. \EEE
 \end{equation}
 We set
 \begin{equation*}
 u_{\rm heal}:=\widetilde{w} + A_1 \cdot.
 \end{equation*}
The second condition in \eqref{1108241721} immediately gives that $u_{\rm heal}= u$ on $N\setminus  \ol{Y} \EEE $.  Moreover, \EEE
\begin{equation*}
\begin{split}
\|e(u_{\rm heal})\|_{L^2((N\setminus Y) \cup Z)}&=\|e(\widetilde{w})\|_{L^2((N\setminus Y) \cup Z)} \leq \|\nabla \widetilde{w}\|_{L^2((N\setminus Y) \cup Z)} \\& \leq |(N\setminus Y) \cup Z|^{1/2} \|\nabla \widetilde{w}\|_{L^\infty((N\setminus Y) \cup Z)} \leq  C_\eta  \EEE \|e(u)\|_{L^2(N\setminus Y)},
\end{split}
\end{equation*}
where in the last inequality we used \eqref{1108241721} and the fact that 
  $|Z|\leq  \varepsilon^2/\eta^2$ by 
assumption as well as  that $|N|$ is controlled by  $C\varepsilon^2/\eta^4$ \EEE due to \eqref{1108241555}, respectively. \EEE This concludes the proof.
\end{proof}

%
%
%
%

\section{Approximation of quasi-static crack growth}\label{sec: evo}

This section is devoted to the formulation of our main result. We present a convergence result for an evolutionary problem with respect to the adaptive finite-element approximation introduced in Section~\ref{sec: gamma}.  More precisely, we set up a time-incremental minimization scheme and prove the convergence to a continuum quasi-static crack growth in the spirit of {\sc Francfort and Larsen} \cite{Francfort-Larsen:2003}.  In particular, \JJJ we \EEE recover the existence result of a fracture evolution in linearized elasticity  \cite{FriedrichSolombrino}. The main issue compared to the $\Gamma$-convergence result in Section~\ref{subsec: Gamma-conv} consists in dealing with the    irreversibility of the fracture process.

\subsection{Quasi-static adaptive finite element model}\label{sec: QS FEM}

We  define an arbitrary sequence $(\eps_n)_n \subset (0,\infty)$ with $\eps_n \to 0$ as $n \to \infty$.   Instead of considering general densities $f$ with properties \eqref{f-assump}, we consider  for simplicity  only the  special case  $f(t)=  t\wedge \kappa$  and $\C=\mathrm{Id}_{2{\times}2 \times 2 \times 2}$.  The case of general $\C$ can be treated in the same way, adjusting the notation accordingly. Moreover, we assume that 
the function $\omega \colon \R^+ \to \R^+$ is given by
\begin{align}\label{eq: 10.6}
\omega(\eps_n)  = 10^6\eps_n.
\end{align}
The constant $10^6$ is chosen for definiteness only. This assumption could be removed at the expense of additional estimates which we omit for simplicity. \EEE
 
  In order to introduce boundary conditions on a part $\partial_D \Omega \subset \partial \Omega$ of the boundary, we impose boundary conditions in a \emph{neighborhood} of the boundary. More precisely, we suppose that there exists another Lipschitz set ${\Omega'} \supset \Omega$ with  $\partial_D \Omega = \partial \Omega \cap {\Omega'}$ such that also  $\Omega' \setminus \overline{\Omega}$ is Lipschitz. For a given boundary datum $g\in W^{2,\infty}( \Omega' \EEE; \R^2)$ and  a triangulation $ \mathbf{T}_n \EEE\in \mathcal{T}_{\eps_n}(\Omega')$, we define 
\begin{align}\label{triangul}
\text{$g_{{ \mathbf{T}_n \EEE}}$ as the piecewise affine \JJJ interpolation \EEE of $g$ on $ \mathbf{T}_n \EEE$.}
\end{align}
 Recalling \eqref{energy-static}, we then consider the energy  
\begin{align}\label{energy-static-new}
E_{n}(u)\defas  \int_{\Omega}|e(u)|^2\wedge \frac{\kappa}{\eps_n} \,    {\rm d}x 
\end{align}
if $u\in \mathcal{A}_{\varepsilon_n}( \Omega' \EEE )$ and if for  the (possibly non uniquely chosen) triangulation $ \mathbf{T}_n \EEE(u)\in \mathcal{T}_{\eps_n}(\Omega')$ (see \eqref{Veps}) it holds   $u = g_{ \mathbf{T}_n \EEE(u)}$ on each triangle $T \in \mathbf{T}_n \EEE(u)$ such that \MMM $T \cap  \overline{\Omega} = \emptyset$. \EEE Otherwise, we set $E_{n}(u) = + \infty$.  We emphasize that the energy is still defined as an integral over $\Omega$  although the functions $u$ are defined on the larger set $\Omega'$.

Now we introduce a time discrete evolution which is driven by time-dependent boundary conditions $g\in W^{1,1}(0,T;W^{2,\infty}( \Omega'; \R^2 \EEE))$. \VIT Given \EEE a sequence $(\delta_n)_n \subset (0,\infty)$ with $\delta_n \to 0$ and for each $\delta_n$, we consider the subdivision $0=t^0_n <\dots<t_n^{ T /\delta_n}=T$  of the interval $[0,T]$ with step size $\delta_n$. (Without restriction, we assume that $T /\delta_n \in \N$.) Correspondingly, let $(g(t^k_n))_k$ be the sequence of boundary data at different time steps $k\in \{0,\dots, T/\delta_n \}$.

We assume for the moment that a displacement history $(u^j_n)_{j<k}$ at time steps  $(t^j_n)_{j<k}$ is given, and introduce admissible competitors and the energy for the next time step, taking into account the  irreversibility of the process.

 Consider $u\in \mathcal{A}_{\eps_n}(\Omega')$ and the corresponding \EEE  triangulation  $  \mathbf{T}_n (u) $.   Recall the definition of  $\MMM \mathbf{T}_{\eps_n}^{\rm big}  \EEE  =   \{T \in  \mathbf{T}_n  (u)\colon \eps_n |e(u)_{T}|^2\geq \kappa  \}$ \MMM and the definition of $\Omega^{\rm big}_{\eps_n} (u) $ \EEE  in \eqref{omegacrack}. \EEE  In place of $\Omega^{\rm big}_{\eps_n} (u)$,  we   now define a  possibly \EEE  larger `crack set' $\Omega_{n}^{\rm crack}(u)$ by   considering also \EEE triangles that are very far away from a regular background mesh. More precisely, let $\mathbf{Z}_n$ be the triangulation that is based on the square grid of size $\eps'_n \defas 2\eps_n \cos(\theta_{0})$ with nodes contained in $\eps_n' \Z^2 \cap \Omega'$, where each square is then cut into two triangles along the diagonal (see  \cite[Figure~5.11]{ChamboDalMaso}). \EEE For a function $u\in \mathcal{A}_{\eps_n}(\Omega')$  with triangulation $\mathbf{T}_n (u)$, \EEE we define $\mathbf{Z}_n(u)\defas  \mathbf{T}_n \EEE(u)\cap \mathbf{Z}_n$ as the part of the triangulation that belongs to this regular background mesh and let $\dist(T, \mathbf{Z}_n(u))\defas \min\{\dist(T,\tilde{T})\colon \tilde{T}\in \mathbf{Z}_n(u)\}$. We then  define \EEE
\begin{equation}\label{def: new-crackset}
    \mathbf{T}^{\rm crack}_n(u)\defas \big\{T \in  \mathbf{T}_n \EEE(u)\colon \eps_n |e(u)_{T}|^2\geq \kappa  \quad \text{or} \quad \dist(T,\mathbf{Z}_n(u))\geq 10^6 \eps_n \big\}.
\end{equation}
The associated crack set is then defined as the union of all such triangles  in $\Omega'$, this means
\begin{equation}\label{def: new-cracksetXXXXX}
    \Omega_{n}^{\rm crack}(u)\defas \mathrm{int}\Big(\bigcup_{T\in   \mathbf{T}^{\rm crack}_n \EEE(u)} T\Big) \cap \Omega' \,.\EEE
\end{equation} 
Note that, additionally to the condition on the gradient, we also  regard \EEE triangles as `cracked' if they are  far away from a fixed background mesh. This means in particular that, if $\mathbf{Z}_n(u)=\emptyset$, we would have $\Omega_n^{\rm crack}(u)=\Omega'$.  This condition is inspired by the construction of recovery sequences in \cite[Appendix]{ChamboDalMaso} where all triangles are close to a background mesh, i.e., in that situation  the additional condition is not active. In our evolutionary setting, we expect the same, and thus the condition is merely of technical nature. Let us also emphasize that the constant $10^6$ is chosen for definiteness only and could be chosen arbitrarily large, but \EEE   fixed.   Both requirements in \eqref{def: new-crackset} will  turn out to be crucial  for our proof of the stability of the static equilibrium property, see Theorem \ref{stability} below. \EEE  

\MMM Given a displacement history $(u^j_n)_{j<k}$, we \EEE  define 
 \begin{equation}\label{Tcrack}\mathbf{T}_{n,k-1}^{\rm crack}\defas \bigcup_{j<k}  \mathbf{T}^{\rm crack}_n \EEE(u^{j}_n) \quad \quad \Omega_{n,k-1}^{\rm crack}\defas \bigcup_{j<k}  \Omega^{\rm crack}_n \EEE(u^{j}_n)\,. \end{equation}
 For a given displacement $u\in \mathcal{A}_{\eps_n}(\Omega')$ we also  set \EEE
  \[\mathbf{T}_{n,k-1}^{\rm crack}(u)\defas \mathbf{T}_{n,k-1}^{\rm crack}\cup \mathbf{T}^{\rm crack}_{n}(u) \quad \quad \Omega_{n,k-1}^{\rm crack}(u)\defas \Omega_{n,k-1}^{\rm crack} \cup  \Omega_{n}^{\rm crack} \EEE(u)\,.\]
\MMM Similar to  the splitting in  \eqref{first-lb}, we \EEE   define the corresponding history-dependent energy by
 \begin{align}\label{energy-split}
    \mathcal{E}_{n}(u, (u^j_n)_{j<k}):=
      \int_{\Omega\setminus  \Omega_{ n ,k-1}^{\rm crack}(u) \EEE } |e(u)|^2 \, +\,\kappa \,\frac{| \Omega_{n,k-1}^{\rm crack}(u) \EEE  |}{ \eps_n \EEE} =: \EEE
         \mathcal{E}^{\rm elast}_{n}(u,(u^j_n)_{j<k})+ \mathcal{E}^{\rm crack}_{n}(u,(u^j_n)_{j<k}).
  \end{align}
   Note that the set $\Omega_{n,k-1}^{\rm crack}$  and thus the energy take the `cracked triangles' of all previous time steps into account. \EEE   In general, the triangulations at each time could be different and without additional requirements it is not guaranteed that the union $\Omega_{n,k-1}^{\rm crack}$ is consistent with an admissible triangulation. In particular, $\mathbf{T}_{n, k-1}^{\rm crack}$ is not necessarily a triangulation \GGG partitioning \EEE $\Omega_{n,k-1}^{\rm crack}$. For this reason, we introduce a further restriction as we set up the time-incremental \MMM minimization \EEE scheme. More precisely, recalling \eqref{Veps}, we set $\hat{\mathcal{A}}_n^0(\Omega'):= \mathcal{A}_{\varepsilon_n}(\Omega')$ and for $k \ge 1$ we introduce the set $\hat{\mathcal{A}}^{k}_{n}(\Omega')\subset \mathcal{A}_{\eps_n}(\Omega')$ that depends on the displacement history $(u^{j}_n)_{j<k}$ and consists of all functions $u\colon\Omega'\to \R^2$ such that there exists a triangulation $ \mathbf{T}_n \EEE\in \mathcal{T}_{\eps_n}(\Omega')$ with $u$ being piecewise affine on $ \mathbf{T}_n \EEE$ and  such that $ \mathbf{T}_n \EEE$ fulfills  
  \begin{align}\label{eq: admis tria}
\mathbf{T}_{n,k-1}^{\rm crack}\subset  \mathbf{T}_n. \EEE
\end{align}
We then define 
\begin{align}\label{akeps}
\mathcal{A}^{k}_{n}\defas \big\{v\in   \hat{\mathcal{A}}^{k}_{n}(\Omega') \EEE \; \text{and} \;   v=g(t^k_n)_{ \mathbf{T}_n \EEE(v)} \;\text{on all $T \in \mathbf{T}_n \EEE(v)$ with }  \MMM T \cap  \overline{\Omega} =\emptyset \EEE
\big\}.
\end{align}
Inductively,  provided that $u^j_n \in \mathcal{A}^{j}_{n}$ for $0 \le j \le k-1$, \EEE we  see that $ \hat{\mathcal{A}}^{k}_{n}(\Omega') \neq \emptyset$ and then also $  {\mathcal{A}}^{k}_{n} \neq \emptyset$ for all $k \ge 1$ since the triangulation $ \mathbf{T}_n \EEE(u^{k-1}_n)$ satisfies $ \mathbf{T}_{n,k-1}^{\rm crack}\subset \mathbf{T}_n \EEE(u^{k-1}_n)\EEE$.

 We suppose that the \emph{initial value} $u_n^0 \in \mathcal{A}^{0}_n$ is a minimum configuration in the sense that 
  \begin{equation}\label{minimizing-scheme0}
  u_n^{0}\in  {\rm  argmin} \big\{ E_n(v) \colon v \in \mathcal{A}^{0}_n  \big\} \,,
  \end{equation}
  with $E_n$ given as in \eqref{energy-static-new}. We inductively define an evolution as follows: given $(u_n^j)_{0 \le j \le k-1}$, we let 
  \begin{equation}\label{minimizing-scheme}
    u_{n}^{k}\in \argmin \big\{\mathcal{E}_{n}(v,(u_n^{j})_{j< k})\, \colon v\in \mathcal{A}^{k}_{n}\big\},
  \end{equation}  
  i.e., the minimization problem involves the previous time steps, according to the definition in \eqref{energy-split}. The existence of minimizers in \eqref{minimizing-scheme} immediately follows from the direct method since  $\mathcal{A}^{k}_{n} \neq \emptyset$, the problem is finite dimensional for a fixed triangulation, \JJJ $\mathcal{E}_{n}(\cdot, (u^j_n)_{j<k})$ is continuous, \EEE  and the set of admissible \JJJ interpolations \EEE $\mathcal{A}^{k}_{n}$ is compact.

\subsection{Quasi-static fracture evolution}

We consider the  \emph{Griffith energy} \EEE 
  \begin{align}\label{eq: lim-en}
  \mathcal{E}(u,K):=   \int_{\Omega}|e(u)|^2\,{\rm d}x+ \kappa \,\GGG \sin(\theta_0) \EEE \, \mathcal{H}^{1}(K)\,,
  \end{align}
  for each $u \in GSBD^2 (\Omega') \EEE $ and each rectifiable set $K\subset  \Omega \cup  \partial_D \Omega \EEE$ with $\mathcal{H}^1(K) < + \infty$, where $e(u)$ denotes the approximate symmetric gradient and \GGG $J_u$ \EEE is the jump set of $u$, which is subject to the constraint $\GGG J_u \EEE \, \tilde{\subset} \,  K$. (Here and in the following, $\, \tilde{\subset} \, $ stands for inclusions up to $\mathcal{H}^1$-negligible sets.) We highlight that, although the elastic energy is defined on $\Omega$, the functions are defined on the larger set $\Omega'$ and the crack sets $K$ may intersect the Dirichlet boundary $\partial_D \Omega$. \EEE  
  
By $AD(g,H)$ we denote all functions $v\in GSBD^2(\Omega') \EEE$ such that   
\begin{equation}\label{def:ADgH}
  v=g \text{ on } \Omega' \setminus \overline{\Omega}, \quad J_v \ \tilde{\subset} \ H.
  \end{equation}

 \begin{definition}\label{main def}
  We define an \emph{irreversible quasi-static crack evolution} with respect to the boundary condition $g \MMM \in W^{1,1}(0,T;W^{2,\infty}( \Omega'; \R^2)) \EEE $  as any \EEE mapping  $t\to (u(t),\Gamma(t))$ with $u(t) \in AD(g(t),\Gamma(t))$ for all $t \in [0,T]$ such that the following  four \EEE conditions hold:
  
  \begin{itemize}
   \item[(a)] \emph{Initial condition}: $u(0)$ minimizes $\mathcal{E}(u,J_u)$ given in \eqref{eq: lim-en} among all $v \in GSBD^2(\Omega')$ with
$v = g(0)$ on $\Omega' \setminus \overline{\Omega}$. 
  \item[(b)] \emph{Irreversibility}:   $\Gamma(t_1) \, \tilde{\subset} \,  \Gamma(t_2)$   for all $0\leq t_1\leq t_2\leq T$.
  \item[(c)] \emph{Global stability}:  \BBB For every $t \in (0,T]$,   for  every  $H$ with $\Gamma(t) \, \tilde{\subset} \, H $, and  for \EEE every $v\in AD(g(t),H)$    it holds that 
      \begin{equation}\label{finalstability}
      \GGG \mathcal{E}(u(t),\Gamma(t)) \EEE\leq \mathcal{E}(v,H)\,.
      \end{equation}
      \item[(d)]   \emph{Energy balance}:    The function $t\mapsto \mathcal{E}(u(t),\Gamma(t))$ is absolutely continuous and it holds that
      \begin{equation}\label{energybalance}
          \frac{\rm d}{{\rm d}t} \mathcal{E}(u(t),\Gamma(t)) =\int_{\Omega}  e(u  (t) \EEE ) : \nabla \partial_{t}g(t)\, {\rm d}x \BBB \quad \text{for a.e.\ $t \in [0,T]$}\,, \EEE 
      \end{equation} 
      where by $\partial_t$ we denote the time derivative of $g$.
  \end{itemize}
\end{definition}

   In \cite{FriedrichSolombrino} \EEE (see \cite{Francfort-Larsen:2003} for the scalar case), the existence of an  {irreversible quasi-static crack evolution} with respect to the boundary displacement $g$ has been shown. In this present work, our goal consists in approximating such an evolution with the time-discretized evolution defined in Section \ref{sec: QS FEM}.

\subsection{Main result: Approximation of quasi-static crack growth}

To formulate our main result, we need to introduce some further notation. First, we associate a `crack surface' to the finite element evolution $(u_n^k)_k$.  To this end,
  we choose $\eta_n\to 0$ depending on $\eps_n$ such that 
  \begin{equation}\label{3112241116}
   (C_{\eta_n})^7 \eps_n \to 0\quad \text{as }n\to \infty, 
   \end{equation}
  where $C_{\eta_n}\geq 1$ denotes the constant in Theorem~\ref{generaltheorem}, and  
 we apply Theorem \ref{generaltheorem} to \MMM $\eta_n$, \EEE $u_n^k$, and $  A   =\Omega^{\rm crack}_{n,k}$.  (Here, we consider $\Omega'$ as the ambient space in place of $\Omega$, in particular we replace $\Omega$ with $\Omega'$ in   \eqref{EEE}, see also \eqref{def: new-cracksetXXXXX}.) 
 We obtain a set $\Omega^{\rm mod}_{n,k} $ and a function $u^{\rm mod}_{n,k} $  \EEE satisfying
\begin{equation}\label{eq: combined}
\begin{split}
&|\Omega^{\rm mod}_{n,k}| \le C_{\eta_n} \eps_n,   \qquad   | \lbrace u_n^{k} \neq  u^{\rm mod}_{n,k} \rbrace  \cap \MMM \Omega'_{\eps_n, \eta_n}\EEE  | \le C_{\eta_n} \eps_n, \quad \\
 \Vert   e(u^{\rm mod}_{n,k}) &  \Vert_{L^2(\MMM \Omega'_{\eps_n, \eta_n} \EEE  \setminus \Omega^{\rm mod}_{n,k})} \le C_{\eta_n},  \quad  \mathcal{H}^1(\partial \Omega^{\rm mod}_{n,k}) \EEE \le  \frac{2}{\eps_n  \sin\theta_0}|   \Omega^{\rm crack}_{n,k} \EEE  |  +  C  \eta_n, 
\end{split}
\end{equation}
 where   the constants $C$ and $C_{\eta_n}$ also depend   on $\max_{0 \le k \le T/\delta_n} \EEE \mathcal{E}_{n}(u_n^k,(u_n^{j})_{j< k})$. \EEE

 We define the evolution $u_n\colon [0,T]\times  \Omega' \EEE \to \R$, piecewise affine in space and piecewise constant in time, by
  \begin{equation}\label{definitionofun}
  u_n(t):= u^k_n \chi_{\Omega' \setminus  \Omega^{\rm crack}_{n,k}  } \; \text{for}\; t\in[t^{k}_n,t^{k+1}_n)\,.
  \end{equation}
The crack set $K_n$ is defined by
 \begin{equation}\label{definition-crackset}
    K_n(t)\defas \partial{\Omega}^{\rm mod}_{n,k}\quad  \text{for $t\in [t^{k}_n,t^{k+1}_n)$}\,.
    \end{equation}  
   For the crack sets, we use the notion of \emph{$\sigma$-convergence} recalled in Definition \ref{def:sigmaconv} below, which is a suitable notion of convergence for crack sets.  In particular, below we will obtain the existence of $K(t)  \tsubset \EEE \overline{\Omega} \cap \Omega'$  for \EEE $t \in [0,T]$  such \EEE 
   that $K_n(t)$  $\sigma$-converges to $K(t)$ for $t \in [0,T]$. \EEE

 As a final preparation, we identify sets on which convergence of displacement fields can be guaranteed. For a crack set $\Gamma(t)  \subset  \overline{\Omega} \cap \Omega' \EEE $ with $\mathcal{H}^1(\Gamma(t) \EEE)<\infty$, by $B(t) \subset \Omega$ we denote the largest set of finite perimeter (with respect to set inclusion) which satisfies $  \partial^* B(t)  \cap \Omega' \EEE \, \tilde{\subset} \,   \Gamma(t) \EEE  $. \EEE This set represents the `broken off pieces', and by $G(t) := \Omega' \setminus B(t)$ instead we denote the `good set', {which}   in particular satisfies $\Omega' \setminus \overline{\Omega} \subset G(t)$. Note that  convergence of the displacements can only be expected on $G(t)$, see  \cite[Subsection 2.4]{steinke} for a \MMM thorough \EEE discussion. Moreover, we note that for an evolution $t\mapsto (u(t),\Gamma(t))$ in the sense of Definition~\ref{main def} it holds that 
\begin{equation}\label{b(t)}
 e(u(t) ) = 0 \quad \text{ on }  B(t) \quad \text{ for all $t \in \VVV [0,T]$}. \EEE   
\end{equation}
In fact, this follows by   applying \eqref{finalstability} with test set $H= \JJJ \Gamma(t) \EEE$ and test function $v = u(t) \chi_{G(t)}$. \EEE

  The main result of this paper is the following convergence theorem. 
\begin{theorem}[Approximation of quasi-static \EEE crack growth]\label{maintheorem}
     \GGG There exists \VIT a \EEE quasi-static crack evolution \EEE $t\to (u(t), K(t) \EEE)$ with respect to the boundary condition $g$ such that, up to a subsequence, we have \VIT that \EEE
  \begin{align}\label{eq: sigma coniiiii}
   K_n(t) \ \sigma\text{-converges to }   K(t) \text{ for all }t \in [0,T]
  \end{align}
  as $n \to \infty$, \VVV $K(0)=J_{u(0)}$, \EEE  and \VIT that, \EEE   for each $t \in [0,T]$, 
      \begin{align}\label{eq: l1convi}
      \text{ $\GGG u_{n}(t) \to u(t)\EEE$ as $k \to \infty$ in measure on $G(t)$,} \quad e(u_n(t))\to e(u(t)) \text{ in }L^2(\Omega; \R^{2 \times 2}_{\rm sym}),
      \end{align}
        where $G(t)$ is the  set corresponding to $K(t)$  defined before \eqref{b(t)}. \EEE
      Moreover, for all $t \in [0,T]$ we have \VIT that \EEE
 \begin{align}\label{energ convi}
  \mathcal{E}_n(u_n(t), (u_{n}^{j})_{j <  k(t)}) \to  \GGG \mathcal{E}(u(t), K(t) \EEE ) \EEE \quad \text{as $n \to \infty$},    
  \end{align}
      where for each $n$ the \BBB ($n$-dependent)  \EEE index  $k(t)  \in \N \EEE $ is chosen such that $t \in [t^{k(t)}_n,t^{k(t)+1}_n)$. 
\end{theorem}
 
  \begin{remark}\label{man: remark}
  {\normalfont
  We proceed with \MMM two \EEE comments on the result:
  
\begin{itemize}
\item[(i)] The energy convergence \eqref{energ convi} can be improved to separate energy convergence in the sense that, for all $t \in [0,T]$, 
$$\mathcal{E}^{\rm elast}_{n}\big( \MMM u_n(t), \EEE (u^j_n)_{j<k(t)}\big) \to  \int_{\Omega}|e(u(t))|^2\,{\rm d}x  $$
(recall \eqref {energy-split}) and  
\begin{align}\label{convvvi}\lim_{n \to \infty} \mathcal{E}^{\rm crack}_{n}\big( \MMM u_n(t), \EEE (u^j_n)_{j<k(t)}\big) = \frac{ \kappa \sin(\theta_0)  }{2} \EEE \lim_{n \to \infty} \mathcal{H}^1(K_n(t))    =   \kappa \sin(\theta_0) \EEE \, \mathcal{H}^{1}(K(t))\,.
\end{align}
\item[(ii)] The identity \eqref{convvvi} is the reason why \EEE crack sets along the sequence are defined in terms of $\partial{\Omega}^{\rm mod}_{n,k}$ and not in terms of $\partial \Omega^{\rm crack}_{n,k} $. In fact, \EEE for the latter the identity \eqref{convvvi}  does not hold in general. 
\end{itemize}

  }
  \end{remark}
  \section{Preparations}\label{sec: preps}
  In this section, we collect some tools and remarks that we will need to prove Theorem \ref{maintheorem} in the next section.  Before providing \EEE the necessary compactness and semicontinuity statements, we  recall \EEE a suitable notion of convergence for  crack sets and prove, respectively recall, the necessary compactness and irreversiblity results.

\subsection{Convergence of sets}\label{subsec:sigmaconv}

 Let $\Omega \subset \R^d$ for $d \ge 2$ and denote by $e_1$ the first unit vector. We let \EEE 
\begin{equation}\label{defPC}
{\rm PC}( \Omega \EEE ):=\{v \in L^1( \Omega; \EEE \R^d) \colon v= e_1 \chi_T \text{ with } T\subset \MMM \Omega \EEE \text{ of finite perimeter}\}
\end{equation}
 be the collection of piecewise constant functions taking values in $\lbrace 0, e_1  \rbrace$. \MMM In this subsection, \EEE by $\mathcal{A}(\Omega)$ we denote the family of open subsets of $\Omega$. \EEE 
We recall the notion of $\sigma$-convergence from \cite{GiacPonsi}.
\begin{definition}{\cite[Definition~5.1]{GiacPonsi}}\label{def:sigmaconv}
A sequence of rectifiable sets $(K_n)_n$ in  $\Omega$ \EEE $\sigma$-converges in $\Omega$ to $K$  if \EEE the functionals $\mathcal{H}_n^-\colon {\rm PC}(\Omega) {\times} \mathcal{A}(\Omega) \to [0,+\infty)$ defined by
\begin{equation}\label{1708240952}
\mathcal{H}_n^-(u, A):=\mathcal{H}^{d-1}\big( (J_u \sm K_n)\cap A)
\end{equation}
are such that, for every $A \in \mathcal{A}(\Omega)$, $\mathcal{H}_n^-(\cdot, A)$ $\Gamma$-converges with respect to the strong topology of $L^1(\Omega)$  to $\mathcal{H}^-(\cdot, A)$, where \EEE  $\mathcal{H}^-\colon {\rm PC}(\Omega) {\times} \mathcal{A}(\Omega) \to [0,+\infty)$  is \EEE given by
\begin{equation}\label{1708240957}
\mathcal{H}^-(u, A):=\int_{J_u \cap A} h^-(x,   \nu_u) \EEE \,  \mathrm{d}\mathcal{H}^{d-1}(x),
\end{equation}
and $K$ is the maximal (with respect to  $\tilde{\subset}$) rectifiable set in $\Omega$ such that
\begin{equation}\label{maxK}
h^-(x, \nu_K(x))=0 \quad\text{for } \mathcal{H}^{d-1}\text{-a.e.\ }x \in K.
\end{equation}
\end{definition}

%

\begin{remark}\label{rem: sigma1}
 (i) More precisely, \VIT if $\mathcal{H}^{d-1}(K_n)\leq C$ \MMM for all $n \in \N$,  by   \cite[Proposition 3.3]{GiacPonsi} \EEE  the density $h^-$ in  \eqref{1708240957} is characterized by
\begin{equation}\label{0812242041}
h^-(x, \nu)=\limsup_{\varrho\to 0^+}\liminf_{n\to +\infty} \frac{\mathbf{m}^{\mathrm{PC}}_{\mathcal{H}_n^-}(\ol u_{x,\nu}, Q_\varrho^\nu(x))}{\varrho^{d-1}},  
\end{equation}
for $x \in \Omega$ and $\nu \in \VIT \mathbb{S}^{d-1} \EEE = \lbrace y\in \R^d \colon |y| = 1\rbrace$, where $Q_\varrho^\nu(x)$ denotes a suitable cube with sidelength $\varrho$ and two sides orthogonal to $\nu$,  $\ol u_{x,\nu}:= e_1 \chi_{Q_\varrho^{\nu,-}(x)}$ (with $Q_\varrho^{\nu,-}(x) = \lbrace y\in Q_\varrho^{\nu}(x) \colon (y- x)\cdot \nu < 0\rbrace$),   and
\begin{equation}\label{0812242009}
\mathbf{m}^{\mathrm{PC}}_{\mathcal{H}}(\ol v, A):=\inf_{v \in \mathrm{PC}(A)}\{\mathcal{H}(v, A)\colon v=\ol v \text{ in a neighborhood of } \partial A\} \ \text{ for }\ol v \in \mathrm{PC}(A), \, A \in \mathcal{A}(\Omega).
\end{equation}
In particular, using $\ol u_{x,\nu}$ as competitor, we get  $h^- \le 1$. \EEE

(ii)  Let $(K_n)_{n}$ be a sequence of rectifiable sets in $\Omega$ such that $K_n$ $\sigma$-converges to $K$ in $\Omega$ and let $(\tilde{K}_n)_{n}$ be another sequence of rectifiable sets such that $\tilde{K}_n$ $\sigma$-converges to $\tilde{K}$ in $\Omega$, and $\tilde{K}_n \VVV \tsubset \EEE K_n$ \MMM for all $n \in \N$. \EEE Then, we have $\tilde{K}\VVV \tsubset \EEE K$. 

(iii) If $F$ is a closed set such that $K_n \tsubset F$, then the $\sigma$-limit $K$ \JJJ of $(K_n)_n$\EEE,  satisfies $K \JJJ \, \tilde{\subset}\, \EEE F$: in fact, $\mathcal{H}_n^-(u,\Omega\sm F)=\hn(J_u \cap (\Omega\sm F))$ for every $n \in \N$,  and thus \EEE $   \mathcal{H}^-(u, \Omega\sm F) =  \EEE \hn(J_u \cap (\Omega\sm F))$,  i.e., \EEE  $h^-$ cannot be 0 on a subset of $\Omega\sm F$ of positive $\hn$-measure.

\end{remark}

In the following, we  need  the  compactness and lower semicontinuity properties of $\sigma$-convergence,  see  \cite[Propositions 5.3, proof of Theorem 8.1]{GiacPonsi}.  \EEE

\begin{proposition}[Compactness and lower semicontinuity\EEE]\label{comp-sigma}
  Let $(K_n)_{n}$ be a sequence of rectifiable sets in $\Omega$  with $\hn(K_n)\leq C$. Then there exists a subsequence $(n_k)_{k}$ and a rectifiable set $K$  in $\Omega$ such that $K_{n_k}$  $\sigma$-converges in $\Omega$ to $K$. Moreover,
\begin{equation*}
  \hn(K)\leq \liminf_{n\to \infty} \hn(K_n)\,. 
\end{equation*}
\end{proposition}

\begin{theorem}[Variant of Helly's theorem for $\sigma$-convergence]\label{helly-sigma}
    Let $t\mapsto \JJJ K \EEE_n(t)$ be a sequence of increasing set functions defined on an interval $I\subset \R$ with values contained in $\Omega$, i.e., $\JJJ K \EEE(s)\subset \JJJ K \EEE(t)\subset \Omega$ for every $s,t\in I$ with $s<t$. Assume that $\hn(\JJJ K \EEE_n(t))$ is bounded uniformly with respect to $n$ and $t$. Then, there exist a subsequence $(\JJJ K \EEE_{n_k})_k$ and an increasing set function $t\mapsto \JJJ K \EEE(t)$ on $I$ such that for every $t\in I$ we have
    \begin{enumerate}[label=(\alph*)]
        \item $\JJJ K \EEE_{n_k}(t)\to \JJJ K \EEE(t)$ in the sense of $\sigma$-convergence,
        \item  $\mathcal{H}^{d-1}(\JJJ K \EEE(t))   \leq \liminf_{k\to \infty} \mathcal{H}^{d-1}(\JJJ K \EEE_{n_k}(t))$. \EEE
\end{enumerate}
\end{theorem}

 For definition and properties of $SBV^{p}(\Omega)$, $p \in (1,\infty)$,  we refer the reader to \cite{Ambrosio-Fusco-Pallara:2000}. In the following, we say $v_n\rightharpoonup v$ weakly in $SBV^{p}(\Omega)$ if $v_n \to v$ in $L^1(\Omega)$ and $\sup_n (\Vert \nabla v_n \Vert_{L^p(\Omega)} + \mathcal{H}^{d-1}(J_{v_n})) < + \infty$. \EEE 
We will also make use of the following property of $\sigma$-convergence   (see \cite[Proposition 5.8]{GiacPonsi}). 
\begin{proposition}\label{sigma-property}
Let $(K_n)_{n}$ be a sequence of rectifiable sets in $\Omega$ such that $K_n$ $\sigma$-converges to $K$ in $\Omega$. Let $(v_n)_{n}$ be a sequence $SBV^{p}(\Omega)$ with $v_n\rightharpoonup v$ weakly in $SBV^{p}(\Omega)$ and $\hn(J_{v_n}\setminus K_n)\to 0$. Then $J_{v} \tsubset \EEE K$.      
\end{proposition} 
\EEE 

\GGG

 We close this section with a lower semicontinuity result for the boundaries of void sets.

\begin{lemma}[Lower semicontinuity for void sets\EEE]\label{le:semicontfactortwo}
Let $K_n$ be a sequence of rectifiable sets of the form $K_n = \partial V_n$ for closed sets $ V_n \EEE \subset \R^d$, $n \in \N$,  with 
finite perimeter   and \EEE $|V_n| \to 0$. Suppose that $K_n $ $\sigma$-converges to $K$. Then
$$\liminf_{n \to \infty} \hn(K_n) \ge 2 \hn(K). $$
\end{lemma}

\begin{proof}

Given $\mu>0$, by a covering argument there exist an open and smooth set $U\subset \Omega$ and a function $v\in {\rm PC}(\Omega)$ such that 
\begin{equation}\label{sigma-covering'}
  \hn(K\setminus U )\leq \mu, \quad \quad \hn((K\, \triangle \, J_{v})\cap U )\leq \mu \,. 
\end{equation} 
 Let $(v_n)_{n}\subset  \EEE {\rm PC}(U)$ \EEE be a recovery sequence  in $L^1(U; \VIT \R^d \EEE)$ \EEE for the restriction of $v$ to $U$ with respect to the functionals   \eqref{1708240952} and \eqref{1708240957}. \EEE Since $h^{-} \leq 1$,  see Remark \ref{rem: sigma1}(i), \EEE  we have
\begin{equation}\label{almost-recovery-seq'}
  \limsup_{n\to +\infty} \hn((J_{v_n}\setminus K_n)\cap U)\leq  \int_{(J_{v}\setminus K)\cap U} h^{-}(x,\MMM \nu_v \EEE )\, \dhn(x)\leq \hn((J_{v}\setminus K)\cap U)\leq \mu \,. 
\end{equation} 
We notice that the function 
\[
\tilde{v}_n\defas \MMM  \chi_{U\setminus {V}_n} \EEE v_n
\] 
is such that
\begin{equation}\label{0812240755}
\tilde{v}_n\to v \quad \text{in }L^1(U; \VIT \R^d \EEE),
\end{equation}
since $|V_n|\to 0$.  \JJJ Notice that the function $\tilde{v}_n$ might also jump outside of $\partial V_n$ \MMM  because $J_{v_n}\cap (U\sm {V_n})\neq \emptyset$ is in general possible. \EEE 
\EEE Therefore, we apply \EEE
 \cite[Proposition~9]{Braides-Chambolle-Solci} to the open set $\tilde{\Omega}_n:= U\sm V_n $ (in place of $\Omega$ therein) to obtain  a sequence of approximating pairs $ ( (z_h, F_h) )_{h }$, with $z_h \in SBV^2(\tilde\Omega_n; \MMM \R^d \EEE )$    and $F_h$ of class $C^\infty$, \EEE  such that 
\begin{equation*}
J_{z_h}\tsubset \partial F_h,\quad z_h=0 \text{ in }F_h, \quad z_h\to \tilde{v}_n=v_n \ \MMM \text{ in } \EEE L^1(\tilde\Omega_n;\R^d), \quad F_h \to \emptyset, 
\end{equation*} 
and 
\begin{equation}\label{void: approxi'}
\limsup_{h\to \infty} \hn(\partial F_h)=  2 \hn(J_{\tilde{v}_n}\cap \tilde{\Omega}_n) \MMM \le \EEE 2\hn((J_{v_n}\setminus K_n) \cap U).
\end{equation} 
For $h_n \in \N$ such that 
\begin{equation}\label{0812241143}
\|z_{h_n}- v_n\|_{L^1(\tilde{\Omega}_n)}+ |F_{h_n}|+ \MMM  \hn(\partial F_{h_n}) -   2\hn((J_{v_n}\setminus K_n) \cap U)   \EEE  \leq \frac{1}{n},
\end{equation}
 setting
 \begin{equation*}
 w_n:= z_{h_n} \chi_{\tilde{\Omega}_n} \in SBV^2(U;\R^d), \qquad G_n:= F_{h_n}\cup   (   V_n  \cap U ),  
 \end{equation*}
 it holds that
 \begin{equation*}
 w_n= 0  \text{ in } G_n, \quad J_{w_n}\tsubset \partial G_n, \quad  w_n \to v \quad\text{ in }L^1(U; \VIT \R^d \EEE).
 \end{equation*}
 Therefore, \cite[Proposition 3]{Braides-Chambolle-Solci} gives  
 \begin{equation}\label{0812241138}
 2 \hn(J_v \cap U)\leq \liminf_{n\to +\infty} \hn(\partial G_n).
 \end{equation}
 By  \EEE \eqref{almost-recovery-seq'} and \eqref{0812241143} it holds  that
 \begin{equation*}
 \limsup_{n\to +\infty} \hn(\partial F_{h_n}) \leq \limsup_{n\to +\infty} 2 \hn((J_{v_n}\setminus K_n)\cap U) \EEE \leq  2\mu.  \EEE
\end{equation*}  
  Since $\hn(\partial G_n) \MMM \le \EEE \hn(K_n\cap U)+ \hn(\partial F_{h_n})$, 
  we get
\begin{equation}\label{0812241159}
\liminf_{n\to +\infty} \hn(K_n\cap U)\geq \liminf_{n\to +\infty} \hn(\partial G_n)- 2 \EEE \mu.
\end{equation}
By \eqref{sigma-covering'}, \eqref{0812241138},  and \EEE \eqref{0812241159} we deduce
\begin{equation*}
\liminf_{n\to +\infty} \hn(K_n)\geq \liminf_{n\to +\infty} \hn(K_n\cap U)\geq  2 \hn(J_v \cap U) -  2 \EEE \mu \geq 2 \hn(K) -  6 \EEE \mu,
\end{equation*}
and we conclude by the arbitrariness  of $\mu >0$. \EEE  
\end{proof}

\subsection{Compactness}\label{subsec:comp}



\JJJ From now on we restrict ourselves to the case $d=2$. \EEE For technical reasons, we need to  consider the space of $GSBD^p$ functions (with $p \in (1,\infty)$) which may also attain  a limit  value $\infty$. Following \cite[Subsection 4.4]{Friedrich-Crismale}, we define  $\bar{\R}^2 := \R^2 \cup \lbrace \infty \rbrace$, and for $U \subset \R^2$ open, we let 
\begin{align}\label{eq: compact extension}
GSBD^p_\infty(U) := \Big\{ &u \in L^0(U;\bar{\R}^2)\colon \,   A^\infty_u  := \lbrace u = \infty \rbrace \text{ satisfies } \mathcal{H}^{1}(\partial^* A^\infty_u)< +\infty, \notag \\
&  \ \ \ \ \ \ \ \ \ \  \ \ \ \ \ \tilde{u}_t := u \chi_{U \setminus A^\infty_u} + t \chi_{A^\infty_u} \in GSBD^p(U) \ \text{ for all $t \in  \R^2 \EEE $} \Big\}. 
\end{align}
Symbolically, we  also write
$$u = u \chi_{U \setminus A^\infty_u} + \infty \chi_{A^\infty_u},$$
and for any $u \in GSBD^p_\infty(U)$, we   set  \EEE
\begin{align}\label{eq: general jump}
e(u) = 0 \text{ in } A^\infty_u, \quad \quad \quad 
J_u = J_{u \chi_{\Omega \setminus A^\infty_u}} \cup (\partial^*A^\infty_u \cap U).
\end{align}
We also define the subspaces  $PR(U) \subset GSBD^p(U)$ and  $PR_\infty(U) \subset GSBD^p_\infty(U)$ as the functions in $u \in GSBD^p(U)$ and  $u \in GSBD^p_\infty(U)$,  respectively, \EEE with $e(u) \equiv 0$. Functions $a \in PR_\infty(U)$ are \emph{piecewise rigid} in the sense that they can be represented as    
$$a=\sum_{j \in \N} a^j \chi_{P^j} + \infty \chi_{A_a^\infty},  $$
where $a^j $ are affine mappings with $e(a^j) = 0$ and $(P^j)_{j \in \N}$ is a Caccioppoli partition of $U \setminus A_a^\infty$.  If $a \in PR(U)$, then $A_a^\infty = \emptyset$.

There exists a metric $\bar{d}$ on $GSBD^p_\infty(U)$, see \cite[Equation (3.13)]{Friedrich-Crismale}, which induces the following convergence: $\bar{d}(u_n,u) \to 0$ if and only if  
$$  u_n \to u  \quad \text{ in measure on }  U \setminus A_u^\infty, \quad \quad  |u_n| \to \infty \quad \text{ on $A_u^\infty$}.$$ 
 In the following, we say that a sequence $(u_n)_n \subset GSBD^p_\infty(U)$ \emph{converges weakly} to $u \in GSBD^p_\infty(U)$ if 
\begin{align}\label{eq: weak gsbd convergence}
 \sup\nolimits_{n\in \N} \big( \Vert e(u_n) \Vert_{L^p(U)} +  \mathcal{H}^{1} \EEE (J_{u_n})\big) < + \infty  
 \end{align}
and 
$$   \text{$\bar{d}(u_n,u) \to 0$, \quad \quad \quad $e(u_n)\rightharpoonup e(u)$ weakly in $L^p(U\setminus A_u^\infty; \R_{\rm sym}^{2 \times 2} \EEE)$}.$$

\begin{proposition}[Compactness]\label{prop: compactness}
Let $(u_n)_n \subset GSBD^p(U)$   satisfy \EEE \eqref{eq: weak gsbd convergence}. Then there exists $u \in GSBD^p_\infty(U)$ such that, up to a subsequence, $u_n$ converges weakly to $u$ in $GSBD^p_\infty(U)$. If additionally $u_n \in PR_\infty(U)$ for all $n \in \N$, we get $u \in PR_\infty(U)$. 
\end{proposition}

\begin{proof}
The result follows from \cite[Theorem 3.5 and Lemma 3.6]{Friedrich-Crismale} which itself is a consequence of \cite[Theorem 1.1]{CC-JEMS}. The closedness of $PR_\infty$ follows by repeating the argument in  \cite[Lemma 3.3]{FM}.
\end{proof}

For the next result, we again consider Lipschitz sets $\Omega \subset \Omega'  \subset \R^2 \EEE$ such that $\Omega' \setminus \overline{\Omega}$ is Lipschitz, as considered in Section \ref{sec: evo}.

\begin{proposition}\label{prop:sigmasigmap}
Let $(K_n)_n$ be a sequence of rectifiable sets in $\Omega' \cap \overline{\Omega}$, \JJJ with $\mathcal{H}^1(K_n)\leq C$ \MMM which \EEE  $\sigma$-converges to $K \subset \Omega'$. \EEE
Let $(v_n)_n$ be a sequence converging weakly in $GSBD^p_\infty(\Omega')$ to $v \in GSBD^p_\infty(\Omega')$  such that \EEE  $\mathcal{H}^1(J_{v_n} \setminus K_n) \to 0$ as $n \to \infty$  and $v_n|_{\Omega' \setminus \overline{\Omega}}$ is a bounded sequence in $W^{1,p}(\Omega' \setminus \overline{\Omega};\R^2)$.  \EEE Then the following hold:
\begin{itemize}
\item[(i)]  $J_{v} \, \tilde{\subset} \, K \VVV \tsubset \Omega' \cap \overline{\Omega}$.  
\item[(ii)]  It holds that \EEE $|A^\infty_v \cap G_K|=0$, \EEE where $G_K$ denotes the smallest set of finite perimeter with \VVV $| \Omega' \setminus (\overline{\Omega} \cup G_K)|=0$ \EEE and  $\partial^* G_K \cap \Omega' \, \tilde{\subset} \, K$.
\end{itemize}
\end{proposition}
 Later we will apply this proposition \VIT to \EEE sequences satisfying Dirichlet boundary conditions which guarantees the boundedness assumption on $\Omega' \setminus \overline{\Omega}$. \MMM The set $G(t)$ introduced before \eqref{b(t)} will play the role of $G_K$. \EEE The result \EEE ensures the compatibility of the notion of $\sigma$-convergence with the weak notion of $GSBD^p_{\infty}$-convergence for the displacement fields,  and can be seen as the $GSBD$-analog of \cite[Proposition~5.8]{GiacPonsi}. \EEE More precisely, it enables us to show that   limiting evolutions $t\mapsto u(t)$ fulfill $J_{u(t)} \, \tilde{\subset} \,  K(t)$ and are thus admissible with respect to $t\mapsto K(t)$. Before we come to the proof, we show the following preliminary result. \EEE    
\begin{lemma}\label{le:0812241813}
 Let   $(K_n)_n$ be a sequence of rectifiable sets in $\Omega' \cap \overline{\Omega}$, \JJJ with $\mathcal{H}^1(K_n)\leq C$ \MMM which \EEE $\sigma$-converges to  $K \subset \Omega'$. \EEE Consider the functionals 
\begin{equation}\label{eq: e''}
\mathcal{E}'_n(u):=\int_{\Omega'} |e(u)|^p \,   {\rm d}x + \mathcal{H}^1(J_u \sm K_n)
\end{equation}
if $u \in GSBD^p(\Omega')$  and \EEE $+\infty$ otherwise in $L^0(\Omega'; \R^2)$.  Then $\mathcal{E}'_n(u)$   $\Gamma$-converges, with respect to the convergence in measure in $\Omega'$, to the functional
\begin{equation}\label{nocheinefleichung}
\mathcal{E}'(u):=\int_{\Omega'} |e(u)|^p \,  {\rm d}x + \int_{J_u} h^-(x,  \nu_u) \EEE \, \dhu(x)
\end{equation}
if $u \in GSBD^p(\Omega')$   and \EEE $+\infty$ otherwise in $L^0(\Omega'; \R^2)$, where   $h^-$ is the density given in Remark \ref{rem: sigma1}(i) for $d=2$. \EEE 
\end{lemma}

\begin{proof}[Proof of Lemma~\ref{le:0812241813}]
 Our strategy is to apply the abstract $\Gamma$-convergence result in \cite{FriPerSol20} to the sequence $\mathcal{E}'_n$. As a pointwise lower bound on the density of the surface integral is needed, cf.\ \cite[Assumption~(2.12)]{FriPerSol20}, we consider a suitable perturbation. 
 Precisely,   \EEE given $\varepsilon>0$, let us define the functionals $\mathcal{E}'_{n,\varepsilon}(u):= \mathcal{E}'_{n}(u)+\varepsilon \hu(\MMM J_u \cap \EEE K_n)$,  i.e., \EEE
\begin{equation*}
\mathcal{E}'_{n,\varepsilon}(u)=\int_{\Omega'} |e(u)|^p \,  {\rm d}x  + \mathcal{H}^1(J_u \sm K_n)+\varepsilon \hu(\MMM J_u \cap \EEE K_n)
\end{equation*}
if $u \in GSBD^p(\Omega')$  and \EEE $+\infty$ otherwise in $L^0(\Omega'; \R^2)$. The characterization of the $\Gamma$-limit of $\mathcal E'_{n,\varepsilon}$ with respect to the convergence in measure follows from \cite[Theorem 2.4, \EEE Remark~3.15]{FriPerSol20}: 
the $\Gamma$-limit of $\mathcal E'_{n,\varepsilon}$ is
\begin{equation}\label{0812241718}
\mathcal{E}_\varepsilon'(u):=\int_{\Omega'} |e(u)|^p \,  {\rm d}x + \int_{J_u} h_\varepsilon^-(x,  \nu_u \EEE ) \, \dhu(x),
\end{equation}
if $u \in GSBD^p(\Omega')$   and \EEE $+\infty$ otherwise in $L^0(\Omega'; \R^2)$,
 with \EEE
\begin{equation}\label{0812242057}
h_\varepsilon^-(x, \nu)=\limsup_{\varrho\to 0^+}\liminf_{n\to +\infty} \frac{\mathbf{m}^{\mathrm{PC}}_{\mathcal{H}_{n,\varepsilon}^-}(\ol u_{x,\nu}, Q_\varrho^\nu(x))}{\varrho}
\end{equation}
 for $x \in \Omega'$ and $\nu \in \VIT \mathbb{S}^1$, \EEE where $\mathcal{H}_{n,\varepsilon}^-(u,A):= \mathcal{H}_n^-(u, A) \EEE + \varepsilon \hu(\MMM J_u \cap \EEE K_n \cap A)$. Here,  we recall the definition of $\mathcal{H}_n^-$ and $\mathbf{m}^{\mathrm{PC}}_{\mathcal{H}_{n,\varepsilon}^-}$ in \eqref{1708240952}
 and  \EEE \eqref{0812242009}, respectively.

We notice that \cite[Theorem~2.4]{FriPerSol20} proves an integral representation result for the $\Gamma$-limit of $\mathcal{E}'_{n,\varepsilon}$,  and \EEE \cite[Remark~3.15]{FriPerSol20}  shows \EEE that the surface density is exactly $h_\varepsilon^-$.  Strictly speaking, \EEE the result has been explicitly detailed for $p=2$, but it holds with minor changes in the proof also for any $p>1$, cf.\  \cite[Remark~5.3]{FriPerSol20}. 

As $\mathcal{H}_n^-\leq \mathcal{H}_{n,\varepsilon}^-$ and $\mathcal{H}_{n,\varepsilon}^-$  is \EEE increasing in $\varepsilon$, it is immediate that 
\[
h^-\leq \lim_{\varepsilon\to 0}h_\varepsilon^-,
\]
 for $h^-$ defined in \eqref{0812242041}.  We observe that \EEE $H(x):=\limsup_{\varrho\to 0^+}  \liminf_{n\to +\infty}  \EEE \varrho^{-1}\hu(K_n \cap Q_\varrho^\nu(x)) $ is finite up to a set of negligible $\hu$-measure.  Indeed, \EEE if $H(x)=+\infty$ on $B$ with $\hu(B)>0$, then the  weak$^*$ limit of $\mu_n:=\hu\mres_{K_n}$,  denoted by $\mu$,  \EEE   would satisfy  $\mu(B)=+\infty$  by  \cite[Theorem 2.56]{Ambrosio-Fusco-Pallara:2000}, \EEE which contradicts the  weak$^*$ lower semicontinuity of the total variation  $\mu(\Omega') \le \liminf_{n \to \infty}\mu_n(\Omega')<+\infty$.  Recalling \EEE \eqref{0812242041} and \eqref{0812242057}  this yields \EEE
 \begin{equation*}
h_\varepsilon^-(x,\nu)\leq h^-(x,\nu) + \varepsilon H(x) \quad  \text{for all $x \in \Omega'$ and $\nu \in \VIT \mathbb{S}^1$}, \EEE
 \end{equation*}
 and then
 \begin{equation*}
 h^-\geq \lim_{\varepsilon\to 0} h_\varepsilon^-.
 \end{equation*}
 Therefore, for any $A \in \mathcal{A}(\Omega')$, the functionals $\mathcal{E}'_\varepsilon(\cdot, A)$ pointwise (and monotonically) converge to $\mathcal{E}'(\cdot, A)$, $\mathcal{E}'_{n,\varepsilon}-\mathcal{E}'_n \in (0, M\varepsilon)$ for $M:=\sup_n \hu(K_n)$, and then the representation result holds true also for  the functional $\mathcal{E}'$ given in \eqref{nocheinefleichung}. \EEE  
\end{proof}
\begin{proof}[Proof of Proposition~\ref{prop:sigmasigmap}]
(i) We first remark that $K \tsubset \Omega'\cap \ol \Omega$ by  Remark~\ref{rem: sigma1}(iii). By assumption we get  $v|_{\Omega'\setminus \overline{\Omega}} \in W^{1,p}(\Omega'\setminus \overline{\Omega};\R^2)$ and $v_n \to v \in L^p(\Omega'\setminus \overline{\Omega};\R^2)$, up to a subsequence.   \EEE Let us apply the compactness result \cite[Theorem~3.8]{FriPerSol20} to $v_n$ (which, similarly to the integral representation result employed in the proof of Lemma~\ref{le:0812241813}, has been proven only for $p=2$ but holds for every $p>1$, with minor changes in the proof). Recalling the precise form of the modifications,  see \cite[Theorem 6.1]{FriedrichSolombrino} and also  \cite[Theorem~3.1]{steinke}, \EEE   there are functions $y_n$ with $y_n= v_n$ on $\Omega'\setminus   \overline{\Omega} \EEE $, $r_n$ with $|\{r_n\neq 0\}|\le \frac{1}{n}$, $a_n=\sum_{j} a_n^j \chi_{P_n^j} \in {\rm PR}(\Omega')$, and  $u \in GSBD^p(\Omega')$  such that 
\begin{equation}\label{1708241612}
y_n=v_n + r_n + a_n,
\end{equation}
and
\begin{equation}\label{1708241613}
\mathcal{H}^1\big(( J_{a_n} \cup J_{r_n})\sm J_{v_n}\big) \le \frac{1}{n}, \qquad \mathcal{E}'_n(y_n) \le \mathcal{E}'_n(v_n)+\frac{1}{n}, \qquad y_n \to u \text{  in measure on $\Omega'$},
\end{equation}
 where $\mathcal{E}'_n$ is defined as in \eqref{eq: e''}. \EEE Since $a_n \in {\rm PR}(\Omega')$ and \eqref{1708241613} holds, there exists $a= \sum_j a_j \chi_{P_j} \in {\rm PR}_\infty(\Omega')$ such that $\bar{\rm d}(a_n,a)\to 0$ by Proposition \ref{prop: compactness}.    \EEE We notice that $a=0$ on $\Omega'\setminus \ol\Omega$ by \eqref{1708241612} and since $y_n=v_n $  on $\Omega'\setminus \ol\Omega$ \EEE as well as  $|\{r_n\neq 0\}| \to 0$. Next, we observe that
\begin{equation}\label{1708241817}
A^\infty_v = A^\infty_a.
\end{equation}
 Indeed, \EEE otherwise $|y_n|\to +\infty$ a.e.\ on $A^\infty_v\triangle A^\infty_a$ since $A^\infty_v = \lbrace |v_n|\to +\infty\rbrace $, $A^\infty_a = \lbrace |a_n|\to +\infty\rbrace $, and $|\{r_n\neq 0\}| \to 0$.  This  contradicts the third property in  \EEE   \eqref{1708241613}.
Since $a_n$, $a$ are piecewise rigid functions, denoting by $B_n:=\bigcup\{P^j_n \colon P^j_n \cap  A^\infty_v \EEE \neq \emptyset\}$, it holds that  
\begin{equation}\label{2108240713}
|B_n\triangle A^\infty_v|\to 0, \qquad  \partial^*B_n \cap \Omega'  \,  \tilde{\subset} \, \EEE J_{a_n}, \EEE
\end{equation}
 where for the second property we assumed without restriction that the affine mappings $(a_n^j)_j$ are pairwise distinct, cf.\ \cite[(3.1)]{FM}. \EEE
Then, for every $b\in \R^2$,
\[
{y}^b_n:=y_n+b\chi_{B_n}-a_n\chi_{\Omega' \EEE \sm B_n}
\] 
 is \EEE such that,  by \eqref{1708241612}--\eqref{2108240713} and the fact that $v_n $ converges weakly in $GSBD^{p}_\infty(\Omega')$ to $v$, \EEE
\begin{equation}\label{2108240729}
{y}^b_n \to {u}^b:=v\chi_{\Omega'\sm A^\infty_v}+(u+b)\chi_{A^\infty_v} \quad \text{a.e.\ in } \Omega',
\end{equation} 
and 
\begin{equation}\label{2108240732}
\partial^*A^\infty_v  \cap \Omega' \,  \tilde{\subset} \,  \EEE J_{{u}^b} \text{ for a.e.\ }b \in \R^2.
\end{equation} 
Let us fix  some $b$ satisfying \eqref{2108240732}.   Recalling \eqref{eq: general jump} and \EEE combining  \eqref{2108240729}--\eqref{2108240732} we get 
\begin{align}\label{eventually} 
J_v \tsubset J_{{u}^b}.
\end{align}
 \VVV 
We now apply the result \cite[Lemma~7.1]{FriPerSol20} to (the surface parts of) $\mathcal{E}'_n$  and \EEE $\mathcal{E}$ as in Lemma~\ref{le:0812241813}, and to the converging sequence ${y}^b_n\to u^b$,  see \EEE \eqref{2108240729}: this gives that \EEE
\begin{equation*}
\int_{J_{u^b}} h^-(x,  \nu_{u^b}) \EEE \,\mathrm{d}\mathcal{H}^1(x) \leq \liminf_{n\to \infty} \mathcal{H}^1(J_{y^b_n} \sm  K_n)  .
\end{equation*}
Using the definition of $y^b_n$,  by \EEE the first property \EEE in \eqref{1708241613},  \eqref{2108240713}, \EEE and the assumption that $\hu(J_{v_n}\sm K_n)\to 0$, \EEE  we get 
$$ \int_{J_{{u}^b}} h^-(x,  \nu_{u^b}) \EEE \,\mathrm{d}\mathcal{H}^1(x) = 0.$$
By \VVV Definition~\ref{def:sigmaconv} \EEE we find $J_{{u}^b} \tsubset K$ and \VVV by \eqref{eventually} \EEE we conclude the proof of (i). 

(ii) Since $v_n \to v$ in $ L^p \EEE (\Omega'\sm \ol \Omega; \R^2)$, up to a subsequence we have $v_n \to v$ a.e.\ in $\Omega'\sm \ol \Omega$ and $|v_n|\to +\infty$ a.e.\ in $A_v^\infty$, by definition of weak convergence in $GSBD^p_\infty(\Omega')$.  This shows \EEE $|(\Omega'\sm \ol \Omega) \cap A_v^\infty|=0$.  Moreover, the fact that $J_{{u}^b} \tsubset K$ and \eqref{2108240732} give $\partial^*A^\infty_v  \cap \Omega' \,  \tilde{\subset} \,  K$. 

Let $G_K$ be the smallest set with (a)  $|\Omega'\sm (\ol \Omega \cup G_k)|=0$ and  (b) $\partial^* G_k \cap \Omega' \, \tilde{\subset} \, K$. As  \MMM $|(\Omega'\sm \ol \Omega) \cap A_v^\infty|=0$, \EEE also $G_K \setminus A_v^\infty$ is a set satisfying (a) and (b). Therefore, the minimality of $G_K$ implies $|A^\infty_v \cap G_K|=0$.
\end{proof}

\section{Proof of  the \EEE main result}\label{sec: main}

This section is devoted to the proof of Theorem  \ref{maintheorem}.  We start by establishing  a \EEE uniform energy bound on $u_n(t)$, defined in \eqref{definitionofun}, which will   enable us to pass to the limit by the compactness and semicontinuity results   that were established in the previous section.   In order to conclude that \eqref{finalstability} holds, we need a result on the stability of unilateral minimizers (Theorem \ref{stability}) whose proof is deferred to Section \ref{sec: stability}. With this result at hand, we are finally able to conclude the proof of Theorem \ref{maintheorem}. 

\EEE

\subsection{Energy bound}

In this section, our goal is to show that the energy of the evolution is bounded uniformly in time. For this, we need to prove an energy estimate on the time-discrete level, which will be also crucial to establish the energy balance. \EEE 

We recall the notation for the time discretization $\{0=t^{0}_n< t^{1}_n<\dots< t^{T /\delta_n}_n =T\}$ of the interval $[0,T]$ with step size $\delta_n$, and introduce the following shorthand notation.   Given an arbitrary $v\in  \MMM \mathcal{A}^{k}_n \EEE $, we write 
\begin{equation}\label{1908241552}
\mathcal{E}^{0}_n(v)=E_n(v)\, \quad \text{and} \quad \mathcal{E}^{k}_n(v) \defas \mathcal{E}_n(v, (u^j_{n})_{j<k})\quad \text{for} \, { k\ge 1, \EEE }
\end{equation}
where $E_n$ and $\mathcal{E}_n$ are defined in  \eqref{energy-static-new} and \eqref{energy-split}, respectively, and $(u^j_{n})_{j<k}$ denote the displacements that have been found  at previous time steps $(t^j_n)_{j<k}$.
We start by proving a bound on the elastic part of the energy.  
\begin{proposition}[Elastic energy]\label{prop-control-on-ela}
  Let $t \mapsto u_n(t)$ be the discrete evolution defined in \eqref{definitionofun}.   \MMM There exists \EEE a constant  $C>0$ depending on $g$ such that   $\mathcal{E}_n^{0}(u^0_n) \le C$ and 
  \begin{align*}
   \int_{ \Omega'} |e({u}_n(t))|^2  \, {\rm d} x \le C \quad \text{ for all $t \in [0,T]$ and for all $n \in \N$.}
   \end{align*}
      
\end{proposition}
\EEE
\begin{proof}
For the time step $t_{n}^{0}=0$, we consider   the background mesh $\mathbf{Z}_{n}\in \mathcal{T}_n(\Omega')$ as introduced before \eqref{def: new-crackset} and \EEE  the test function $g_n^{0} :=g(0)_{\mathbf{Z}_n}$, see \eqref{triangul}.   
  We then have  $g_n^{0}\in \mathcal{A}_n^{0}$, see \eqref{akeps}. Since $g\in W^{1,1}(0,T;W^{2,\infty}(\Omega;  \R^2  ))$, we can deduce that $e(g_n^{0})_{T}$ is uniformly bounded on each triangle $T$, hence we obtain $\mathcal{E}_n^{0}(g_n^{0})  \le C$ for some $C>0$   independently of $n$. By \eqref{minimizing-scheme0} this implies that $\mathcal{E}_n^{0}(u_n^{0})\le C$ uniformly in $n$.

In a similar fashion, we can test with a function $g_n^{k}\in \mathcal{A}^{k}_n$, namely $g_n^{k} = g(t_n^{k})_{\mathbf{T}_n}$, where $\mathbf{T}_n$ is \JJJ given by  $\mathbf{T}_{n}(u_n^{k-1})$. \EEE As before, the regularity of $g$ implies that \EEE $e(g_n^{k})_{T}$ is uniformly bounded, and thus $\eps_n\,|e(g_n^{k})_{T}|^2<  \kappa $ for all $T\in \mathbf{T}_n$ for $n$ large enough.  In particular, recalling  \eqref{def: new-crackset}--\eqref{def: new-cracksetXXXXX} this means   $ \Omega_{n,k-1}^{\rm crack}(g_n^{k}) 
\subset \EEE \Omega_{n,k-1}^{\rm crack}(u_n^{k})$. 
\EEE As $\mathcal{E}_n^{k}(u_n^{k})\leq \mathcal{E}_n^{k}(g_n^{k})$  by \eqref{minimizing-scheme}, using \eqref{energy-split} we can deduce for all $t\in  [ \EEE t_n^k,t_n^{k+1})$  
  \begin{align*}
    \int_{\Omega}|e({u}_n(t))|^2\, {\rm d}x \EEE & = \int_{\Omega \setminus \Omega_{n,k-1}^{\rm crack}(u_n^{k})} |e(u_n^{k})|^2\, {\rm d}x \EEE\leq \int_{\Omega \setminus \Omega_{n,k-1}^{\rm crack}(g_n^{k})} |e(g_n^{k})|^2\, {\rm d}x \EEE  +     \kappa \,\frac{|\Omega_{n,k-1}^{\rm crack}(g_n^k)|}{\eps_n} - \kappa \,\frac{|\Omega_{n,k-1}^{\rm crack}(u_n^k)|}{\eps_n}  \\
    & \le \int_{ \Omega \setminus \Omega_{n,k-1}^{\rm crack}(g_n^{k}) \EEE } |e(g_n^{k})|^2\,  {\rm d}x \EEE   \le C \,.
  \end{align*} 
Eventually, we have $\| e(u_n(t)) \|_{L^2(\Omega' \setminus \overline{\Omega})}   \le \EEE C $ by the definition  of   $\mathcal{A}^{k}_{n}$ \EEE in \eqref{akeps}   and the regularity of $g$. 
In fact, if \MMM   $T \cap \overline{\Omega} = \emptyset$, \EEE then $u_n(t)=g(t_n^k)_{\mathbf{T}_n(u_n(t))}$ in $T$ by the definition  of  $\mathcal{A}^{k}_{n}$. Instead, for  $\MMM \mathbf{T}^{\rm bdy}_n(u_n(t)) \EEE :=\{T \in \mathbf{T}_n(u_n(t)) \colon T \cap (\Omega'\sm \ol \Omega) \neq \emptyset,\, T \cap \Omega\neq \emptyset\}$, it holds $\# \mathbf{T}^{\rm bdy}_n(u_n(t))\leq \frac{\tilde{C}}{\eps_n}$  with $\tilde{C}$ depending on the Lipschitz constant of $\JJJ \partial \Omega \EEE$. \MMM Therefore,  in view of \eqref{def: new-crackset}, \EEE $\|e(u_n(t))\|^2_{L^2(\Omega^{\rm bdy}_n(u_n(t)))}\leq \MMM \tilde{C}  \EEE $, \JJJ where \begin{equation*}\Omega^{\rm bdy}_n(u_n(t))\defas 
 \mathrm{int}\Big(
  \bigcup_{T\in \mathbf{T}^{\rm bdy}_n(u_n(t))} T \Big)\,. \end{equation*}\EEE  This concludes the proof.  
\end{proof}
\noindent
As an immediate consequence, we obtain the following corollary. 
\begin{corollary}\label{cor: full-ela-control}
    Let $t \mapsto u_n(t)$ be the discrete evolution defined in \eqref{definitionofun}.
    Then, there exists a constant $C>0$ depending on $g$, but independent of $k$ and $n$, such that 
\begin{align}\label{vio}
\int_{0}^{t_{n}^{k}} \int_{\Omega} e ({u}_n(\tau)) :  e(\partial_{t}g(\tau))\,{\rm d} x\,{\rm d}\tau \le   C \quad \MMM \text{for all $k=0,\ldots,  T/\delta_n$}. \EEE
\end{align} 
\end{corollary}
\begin{proof}
    By H\"older's inequality we find 
$$ \int_{0}^{t_{n}^{k}} \int_{\Omega} e (u_n(\tau)) :  e(\partial_{t}g(\tau))\,{\rm d} x\,{\rm d}\tau  \leq    \, \Vert e(\partial_t  g)\Vert_{L^1(0,T; L^2(\Omega))}\; \|e({u}_n)\|_{L^{\infty}([0,t_n^k); L^2(\Omega))} \,.$$
 Using \EEE Proposition~\ref{prop-control-on-ela}
and   $g\in W^{1,1}(0,T; W^{2,\infty}(\Omega; \R^2 ))$ we deduce \eqref{vio}. 
\end{proof}
We continue with the following discrete energy estimate that will be fundamental to establish a uniform bound on the energy and to prove the energy-balance law.   
\begin{lemma}[Discrete energy estimate]\label{e-balance-lemma}
  Let $t \mapsto u_n(t)$ be the discrete evolution defined in \eqref{definitionofun}. Let $k=0,\ldots,  T/\delta_n  $. Then, there exists  $(\beta_n)_n$   independent \EEE of $k$ with $\beta_n \to 0$  \MMM as $n \to \infty$ \EEE such that 
    \begin{equation}\label{eq: for en bal}
{        \mathcal{E}_n^k( u_n^k ) - \mathcal{E}_{n}^{0}(u^0_n) \leq 2  \int_{0}^{t_{n}^{k}} \int_{\Omega}  e({u}_n(\tau)) :  e(\partial_{t} {g}(\tau)  )\,{\rm d} x\,{\rm d}\tau  +  \beta_n  .   }
    \end{equation}
\end{lemma}

\begin{proof}
The argumentation follows a well-known strategy, see e.g.\   \cite[Section~3.2]{Francfort-Larsen:2003}, \EEE \cite[Lemma 6.1]{dMasoFranToad}, or   \cite[Lemma 4.3]{FriedrichSeutter} \EEE for  an \EEE  application in a discrete setting.   We notice that \MMM we have \EEE a quadratic bulk energy as in \cite{Francfort-Larsen:2003}, which simplifies the computations. \EEE
We start by introducing a notation: for the boundary function $g\in W^{1,1}( 0, T; W^{2,\infty}(\Omega;\R^2))\EEE$ and a time step $0\leq l\leq T/\delta_n$,    we define $\tilde{g}_n^{l}(t):= g(t)_{\mathbf{T}_n(u_n^{l})}$ \EEE as the piecewise affine interpolation of $g(t)$ with respect to $\mathbf{T}_n(u_n^{l})$.

Given the time $t_n^l$, we define the test function $\xi_n^{l}\defas u_n^{l-1}+\tilde{g}_n^{l-1}(t_n^{l})-\tilde{g}_n^{l-1}(t_n^{l-1})$. Note that $\xi_n^{l}$ is piecewise affine with respect to $\mathbf{T}_n(u_n^{l-1})$ and by definition we have $\xi_n^{l}\in \mathcal{A}^{l}_n$. In view of \eqref{minimizing-scheme}, we obtain $\mathcal{E}_n^{l}(u_n^{l})\leq \mathcal{E}_n^{l}(\xi_n^{l})$. Our goal is   to prove that there exists a bounded sequence $(\vartheta_n)_n$ in $L^\infty([0,T] \times \Omega )$ with $ \Vert \vartheta_n(\tau) \Vert_{ L^2(\Omega) } \to 0 $  uniformly in $\tau$ such that, for each $l$, we have 
\begin{equation}\label{wanttoprove}
\mathcal{E}_n^{l}(\xi_n^{l})-\mathcal{E}_n^{l}(u_n^{l-1})\leq 2 \,\int^{t^{l}_{n}}_{t^{l-1}_{n}}\int_{\Omega} (e(  {u}_n(\tau) \EEE )+\vartheta_n(\tau)) : \partial_{t} e(g(\tau)) \,  {\rm d}x \EEE\,  {\rm d}\tau \EEE +\,  C \sigma^l_n     {\eps_n} \EEE    \,,
\end{equation}
where $  \sigma^l_n \EEE :=\int_{t^{l-1}_n}^{t^l_n}\Vert  \partial_t g \Vert_{W^{2,\infty}(\Omega)} \, {\rm d}s$. Once this is shown, \EEE  we can deduce \eqref{eq: for en bal} as follows. Note first that $\mathcal{E}^{l-1}_{n}(u_{n}^{l-1}) = \mathcal{E}^l_{n}(u_{n}^{l-1})$ for each step $l$. Since $\mathcal{E}_n^{l}(u_n^{l})\leq \mathcal{E}_n^{l}(\xi_n^{l})$, we can sum up over all time steps $  1 \EEE \le l \le k$ to obtain a telescopic sum on the left-hand side of \eqref{wanttoprove}  which leads to \EEE
  \begin{equation}\label{eq: a number}
      \begin{aligned}
          \mathcal{E}^k_{n}(u_n^k)- \mathcal{E}_{n}^{0}(u_n^0)
   \leq  2\,\int_{0}^{t_{n}^{k}} \int_{\Omega} (e(u_n(\tau))+\vartheta_n(\tau)) : \partial_{t} e(g(\tau)) \,  {\rm d}x \EEE\,  {\rm d}\tau \EEE  +   C    {\eps_n} \EEE ,
      \end{aligned}
  \end{equation}
 where we used $\sum_l \sigma^l_n \le C$. Since $\|\vartheta_n(\tau)\|_{L^2  (\Omega) \EEE}\to 0$ uniformly in $\tau\in [0,T]$, we get by Hölder's inequality \[\int_{0}^{t_{n}^{k}} \int_{\Omega} |\vartheta_n(\tau) : \partial_{t} e(g(\tau))|\,  {\rm d}x \EEE\,  {\rm d}\tau \EEE \leq \int_{0}^{t_{n}^{k}}  \|\vartheta_n(\tau)\|_{L^2(\Omega)}  \|\partial_{t}e(g(\tau))\|_{L^2(\Omega)}  \EEE \,  {\rm d}\tau \EEE\to 0\,. \]
In view of \eqref{eq: a number}, setting 
\[\beta_n\defas C   {\eps_n} \EEE + \MMM \int_{0}^T \EEE \int_{\Omega} |\vartheta_n(\tau) : \partial_{t} e(g(\tau))|\,  {\rm d}x \,  {\rm d}\tau,  \] 
 we obtain \EEE  \eqref{eq: for en bal}.

It now remains to prove \eqref{wanttoprove}. By definition we have $\Omega^{\rm crack}_{n,l-1}(u_n^{l-1})=\Omega^{\rm crack}_{n,l-1}\subset \Omega^{\rm crack}_{n,l-1}(\xi_n^{l})$ and hence  \GGG (recall \eqref{energy-split}) \EEE
  \begin{equation}\label{energy-difference}
    \mathcal{E}_n^{l}(\xi_n^{l})-\mathcal{E}_n^{l}(u_n^{l-1})= \int_{\Omega\setminus \Omega^{\rm crack}_{n,l-1}(\xi_n^{l})}  |e(\xi_n^{l})|^2\,  {\rm d}x \EEE - \int_{\Omega\setminus\Omega^{\rm crack}_{n,l-1}} |e(u_n^{l-1})|^2\,  {\rm d}x \EEE + \frac{\GGG \kappa\, \EEE |\Omega^{\rm crack}_{n,l-1}(\xi_n^{l})\setminus \Omega^{\rm crack}_{n,l-1}|}{\eps_n}\,.
  \end{equation} 
  We can split the  first \EEE term into
   \begin{align}\label{energy: elasticpart}
 \int_{\Omega\setminus\Omega^{\rm crack}_{n,l-1}(\xi_n^{l})}     |e(\xi_n^{l})|^2 \EEE\,  {\rm d}x \EEE  =      \int_{\Omega\setminus\Omega^{\rm crack}_{n,l-1}}   |e(\xi_n^{l})|^2 \EEE \,  {\rm d}x  - \int_{\Omega^{\rm crack}_{n,l-1}(\xi_n^{l})\setminus\Omega^{\rm crack}_{n,l-1}} \,   |e(\xi_n^{l})|^2 \EEE \,  {\rm d}x \,.
    \end{align} 
    Note that $\Omega^{\rm crack}_{n,l-1}(\xi_n^{l})\setminus\Omega^{\rm crack}_{n,l-1}$ consists of all triangles $T\in \mathbf{T}^{\rm crack}_{n,l-1}(\xi_n^{l})\setminus \mathbf{T}^{\rm crack}_{n,l-1}$, i.e., it holds that $ \eps_n   |e(u_n^{j})_{T}|^2 < \kappa $ and $ \dist(T,\mathbf{Z}_n(u_n^{j}))\MMM < \EEE 10^6 \eps'_n $ for all $ 0 \le \EEE j\leq l-1$ \MMM such that $T \in \mathbf{T}_n(u_n^j)$. \VIT Since \JOS $\mathbf{T}_n(\xi_n^{l})= \mathbf{T}_n(u_n^{l-1})$, we hence have \EEE $\dist(T,\mathbf{Z}_n(\xi_n^{l}))= \dist(T,\mathbf{Z}_n(u_n^{l-1})) <  10^6 \eps'_n $ for all $T\in \mathbf{T}^{\rm crack}_{n,l-1}(\xi_n^{l})\setminus \mathbf{T}^{\rm crack}_{n,l-1}$.  
    Therefore we conclude by \eqref{def: new-crackset} that $\eps_n |e(\xi_n^{l})_{T}|^2\ge \kappa $ \JJJ for all $T\in \mathbf{T}^{\rm crack}_{n,l-1}(\xi_n^{l})\setminus \mathbf{T}^{\rm crack}_{n,l-1}$.  
    This leads to
   \[    \int_{\Omega^{\rm crack}_{n,l-1}(\xi_n^{l})\setminus\Omega^{\rm crack}_{n,l-1}} \, |e(\xi_n^{l})|^2  \,  {\rm d}x \EEE \geq \kappa \frac{|\Omega^{\rm crack}_{n,l-1}(\xi_n^{l})\setminus\Omega^{\rm crack}_{n,l-1}|}{\eps_n} \,.\] \EEE
Combining this with \eqref{energy-difference} and \eqref{energy: elasticpart} we obtain 
  \begin{equation}\label{preliminary-estimate}
  \begin{split}
    \mathcal{E}_n^{l}(\xi_n^{l})-\mathcal{E}_n^{l}(u_n^{l-1}) \leq  &  \int_{\Omega\setminus \Omega^{\rm crack}_{n,l-1} }   (|e(\xi_n^{l})|^{2}-|e(u_n^{l-1})|^2 )  \, {\rm d}x \,. \EEE  
    \end{split}
  \end{equation}
   Next, recalling the definition $\xi_n^{l}\defas u_n^{l-1}+\tilde{g}_n^{l-1}(t_n^{l})-\tilde{g}_n^{l-1}(t_n^{l-1})$, we apply the mean value theorem   to the function $h\colon [0,1]\to \R; \, s\mapsto |e(u_n^{l-1})+\, s\,  (e(\tilde{g}_n^{l-1}(t_n^l)-\tilde{g}_n^{l-1}(t_n^{l-1})))|^2$ to obtain some $\rho_n^{l-1}\in [0,1]$ with 
   \begin{equation}\label{meanvalue1}
   \begin{aligned}
    &  \int_{\Omega\setminus \Omega^{\rm crack}_{n,l-1} } \EEE   (|e(\xi_n^{l})|^{2}-|e(u_n^{l-1})|^2 )   \, {\rm d}x \EEE \\  = 2 &  \int_{\Omega\setminus \Omega^{\rm crack}_{n,l-1} } \EEE   \bigl(e(u_n^{l-1})+\, \rho_n^{l-1}\,  e(\tilde{g}_n^{l-1}(t_n^l)-\tilde{g}_n^{l-1}(t_n^{l-1}))\bigr):\big(e(\tilde{g}_n^{l-1}(t_n^l)-\tilde{g}_n^{l-1}(t_n^{l-1}))\big)\,  {\rm d}x \EEE\,.  
   \end{aligned}
  \end{equation}
We can now define the function $ \vartheta_n \EEE  \colon \EEE [0,T]\to L^2(\Omega; \R_{\rm sym}^{2 \times 2} \EEE)$   by setting, for $s\in [t_n^{l-1},t_n^{l} )$, \EEE  
  \begin{equation}\label{meanvalue2}\vartheta_n(s)\defas \rho_n^{l-1}\,\bigl(  e(\tilde{g}_n^{l-1}(t_n^l))-e(\tilde{g}_n^{l-1}(t_n^{l-1}))\bigr) =  \rho_n^{l-1}\,  \int_{t_n^{l-1}}^{t_n^{l}}\partial_{t} e(\tilde{g}_n^{l-1}(\tau))\, {\rm d}\tau 
  \quad \text{   in $\Omega\setminus  \Omega^{\rm crack}_{n,l-1} \EEE $} \,,\end{equation} 
 and $0$ outside of it.    Since $g\in W^{1,1}(0,T; W^{2,\infty}(\Omega; \R^2))$, the function $\tau\mapsto  \partial_{t} e(\tilde{g}^{l-1}_{n}(\tau))$ \EEE belongs to $L^1(0,T; \MMM L^{\infty}\EEE(\Omega;  \R_{\rm sym}^{2 \times 2} \EEE ))$ and we can use the absolute continuity of the integral and $|t_{n}^{l}-t_{n}^{l-1}| =\delta_n \to 0$ to conclude that $\|\vartheta_n(s)\|_{L^2(\Omega)}\to 0$ uniformly in $s$.
Moreover, we write  
\begin{align}\label{meanvalue3} e(\tilde{g}_n^{l-1}(t_n^l))-e(\tilde{g}_n^{l-1}(t_n^{l-1})) = \int_{t_n^{l-1}}^{t_n^{l}}\partial_{t} e(g(\tau))\, {\rm d}\tau +  \zeta_n^{l-1} 
 \,, \end{align}
  with
  $ \zeta_n^{l-1}\defas \int_{t_n^{l-1}}^{t_n^{l}}\partial_{t} e(\tilde{g}_n^{l-1}(\tau)-g(\tau)) \EEE \, {\rm d}\tau\, $. 
  \JJJ By the regularity of $g$ and the definition of $\tilde{g}_n^{l}(t)$ we obtain  
  \begin{equation}\label{rest-estimate}
    \|\zeta_n^{l-1}\|_{L^{\infty}(\Omega)}  \leq  C \sigma^l_n\,  \omega( \eps_n) \leq  C \sigma^l_n\,   \eps_n , 
  \end{equation} 
  where \MMM  $\sigma^l_n$ is given in \eqref{wanttoprove}, and \EEE the last step follows from \eqref{eq: 10.6}. \EEE   Noting  that $\vartheta_n$ is piecewise constant in time, we deduce from \eqref{meanvalue1}--\eqref{meanvalue3} and by Fubini's theorem
  \begin{equation}\label{with-rest-term}
    \begin{aligned}
   \int_{\Omega\setminus \Omega^{\rm crack}_{n,l-1} } \EEE   (|e(\xi_n^{l})|^{2}-|e(u_n^{l-1})|^2 )  \,  {\rm d}x  = 2 &\int_{t_n^{l-1}}^{t_n^{l}}  \int_{\Omega\setminus \Omega^{\rm crack}_{n,l-1} } \EEE  (e(u_n^{l-1})+\vartheta_n(\tau)) :  e(\partial_{t} g(\tau)) \, {\rm d}x \,  {\rm d}\tau\\
    + & 2 \EEE  \int_{\Omega\setminus \Omega^{\rm crack}_{n,l-1} } \EEE   (e(u_n^{l-1})+\vartheta_n(\tau)) : \VVV \zeta_n^{l-1} \EEE \, {\rm d}x
    \,.
    \end{aligned}
  \end{equation}
By Proposition \ref{prop-control-on-ela}, the definition in \eqref{definitionofun}, \EEE and the fact that $\|\vartheta_n(s)\|_{L^{2}(\Omega)}\to 0$, we can estimate the last term by \eqref{rest-estimate} and obtain 
\begin{equation}\label{1001251948}
 \int_{\Omega\setminus \Omega^{\rm crack}_{n,l-1} } \EEE   (e(u_n^{l-1})+\vartheta_n(\tau)) : \zeta_n^{l-1} \EEE \, {\rm d}x \leq C \sigma^l_n \,  \eps_n \,. \EEE
\end{equation}\EEE 
 By \EEE \eqref{definitionofun},  \eqref{preliminary-estimate},   \EEE  \eqref{with-rest-term}, and \eqref{1001251948} we obtain \[\mathcal{E}_n^{l}(\xi_n^{l})-\mathcal{E}_n^{l}(u_n^{l-1}) \leq 2  \int_{t_n^{l-1}}^{t_n^{l}}\int_{\Omega}(e({u}_n(\tau))+\vartheta_n(\MMM \tau \EEE )): 
 e(\partial_{t} g(\tau))\, \, {\rm d}x \EEE \,  {\rm d}\tau  + \VVV C \sigma^l_n    \eps_n \EEE \,.\]
  This \EEE yields \eqref{wanttoprove} and concludes the proof. 
\end{proof}

As a direct consequence, we obtain the following bound on the energy. 

\begin{corollary}[Energy bound]\label{cor: energy bound}
 Let $t \mapsto u_n(t)$ be the discrete evolution defined in \eqref{definitionofun}. Then, there exists a constant $C>0$ only depending on $g$ such that 
\[\mathcal{E}^k_n(u^k_n) \le C \quad \text{for all $k = 0, \ldots, T/\delta_n  $ and $n \in \N$.} \]
\end{corollary}

\begin{proof}
   The proof follows by combining $\mathcal{E}_n^{0}(u^0_n) \le C$ (see Proposition \ref{prop-control-on-ela}), \eqref{vio}, and Lemma \ref{e-balance-lemma}.
\end{proof}

\subsection{Compactness and lower semicontinuity}
Based on the energy bound in Corollary \ref{cor: energy bound}, we can pass to the limit in the crack sets and displacements by compactness arguments. We start with the crack sets.

\begin{proposition}[Convergence of crack sets\EEE]\label{limit-crack-evo}
  \GGG Let $t\mapsto (u_n(t), K_n(t))$ be the evolution defined in \eqref{definitionofun}  and \eqref{definition-crackset}. There exists an increasing set function $t\mapsto K(t)$ in $\Omega' \cap \overline{\Omega}$ \EEE 
  and a subsequence (not relabeled) such that, for every $t\in [0,T]$, the set $K_n(t)$ $\sigma$-converges to $K(t)$. \EEE
\end{proposition}

\begin{proof}
By \eqref{eq: combined}, \eqref{definition-crackset}, and the energy bound  in  Corollary \ref{cor: energy bound} \EEE we have that $\mathcal{H}^1(K_n(t))\leq C$ for all $t \in [0,T]$. We therefore can apply Proposition \ref{comp-sigma} for each $t\in [0,T]$ to obtain the limiting set function $K(t)  \tsubset \Omega'$ such that $K_n(t)$ $\sigma$-converges to $K(t)$ for a $t$-dependent subsequence. Since $K_n(t) \tsubset \Omega' \cap \overline{\Omega}$, \EEE the definition of $\sigma$-convergence directly implies $K(t) \tsubset \Omega' \cap \overline{\Omega}$, see  Remark~\ref{rem: sigma1}(iii).  \EEE Our goal is to apply Helly's theorem in the version of Theorem \ref{helly-sigma} to show that the subsequence can be chosen  independently \EEE of $t$ and that $t\mapsto K(t)$ is increasing. This is however impeded by the fact that  \EEE  $K_n(t)$ is not an increasing set function. As a remedy, we define  
 \[ \tilde{K}_n(t)  \defas \bigcup_{ T  \in  \mathbf{T}_{n,k}^{\rm mod}} \partial  T  \,\quad \text{for}\; t\in[t_n^{k},t_n^{k+1}) \,,\] 
 where $\mathbf{T}_{n,k}^{\rm mod} := \lbrace T \in \mathbf{T}_{n,k}^{\rm crack} \colon  T \subset \Omega_{n,k}^{\rm mod}\rbrace$. \EEE As   $\Omega_{n,k-1}^{\rm crack} \subset \Omega_{n,k}^{\rm crack}$, see \eqref{Tcrack}, we get   $ \mathbf{T}_{n,k-1}^{\rm mod} \subset   \mathbf{T}_{n,k}^{\rm mod}$ \EEE by Corollary \ref{monotonicity-corollary}, and thus the set function $t\mapsto  \tilde{K}_n(t) $ in fact fulfills $  \tilde{K}_n(s)   \tsubset \EEE \tilde{K}_n(t)  $  for $s \le t$. Moreover, $\mathcal{H}^1(  \tilde{K}_n(t) ) \le C$ for all $t \in [0,T]$, again  by the energy bound in  Corollary \ref{cor: energy bound}. (Use   $|T| \ge c\eps \mathcal{H}^1(\partial T)$ for all triangles $T$ for $c$ depending on $\theta_0$  to get \EEE $\# \mathbf{T}_{n,k}^{\rm mod}\leq  C/ \eps_n \EEE $,  as well as \eqref{eq: 10.6}.) \EEE  Hence, we can apply Theorem~\ref{helly-sigma} to obtain an increasing set function $t\mapsto  \tilde{K}(t) \EEE$ such that $ \tilde{K}_n(t) \EEE$ $\sigma$-converges to $ \tilde{K}(t) \EEE$ for every $t\in [0,T]$, up to extracting a subsequence \MMM (not relabeled). \EEE We now want to show that $ \tilde{K}(t) \EEE\,\tilde{=}\, K(t)$ for each $t\in [0,T]$. 
 \JJJ Note that each triangle $T\in \{T\notin \mathbf{T}_{n,k}^{\rm crack}\colon T \subset \Omega_{n,k}^{\rm mod}\}$ fulfills $\partial T\cap K_n\,\tilde{=}\,\emptyset$, by Remark \ref{remark-on-holes}. Therefore, we have $K_n(t) \tsubset  \tilde{K}_n(t) $ because
\begin{equation*}
  K_n(t)=\partial \Omega_{n,k}^{\rm mod} \sm \bigcup_{\substack{T\subset \Omega_{n,k}^{\rm mod}\\T\notin \mathbf{T}_{n,k}^{\rm crack}} }\partial T \subset  \bigcup_{\substack{T\subset \Omega_{n,k}^{\rm mod}\\ T\in \mathbf{T}_{n,k}^{\rm crack}}} \partial T = \tilde{K}_n(t)\,.
\end{equation*}

 \EEE By  Remark \ref{rem: sigma1}(ii),  this implies $ K(t) \tsubset \EEE  \tilde{K}(t) \EEE$. To show the reverse inclusion, we \EEE fix a time $t\in [0,T]$.   We denote by $\mathcal{H}_{\tilde{K}(t)}^{-}$   the $\Gamma$-limit of $\mathcal{H}^-_{\tilde{K}_n(t)}$ given by $\mathcal{H}^-_{\tilde{K}_n(t)}(u, A)=\mathcal{H}^1((J_u \sm \tilde{K}_n(t)) \cap A)$,  and \EEE let $h^-_{\tilde{K}(t)}$    be the corresponding density function from \eqref{1708240957}.  In the analogous way, we denote $\mathcal{H}_{ {K}(t)}^{-}$ and  $h^-_{ {K}(t)}$. \EEE 

 We suppose by contradiction that $\mathcal{H}^1(\tilde{K}(t) \sm  K(t)) >0$.  As $K(t)$ is the maximal set   on which $h^-_{K(t)}$ is $\mathcal{H}^1$-a.e.\ equal to $0$ (cf.\ Definition~\ref{def:sigmaconv}), we find  $\mu >0$ such that
\begin{align}\label{formycontra}
\int_{\tilde{K}(t)} h^-_{K(t)}(x, \nu_{\tilde{K}(t)}(x)) \, \mathrm{d} \mathcal{H}^1(x) > 3\mu. 
\end{align}
\EEE  We use a covering argument to find an  open set  $U  \EEE \subset \Omega' \EEE $ and a function $v\in {\rm PC}( U \EEE )$ such that 
  \begin{equation}\label{sigma-covering1}
    \mathcal{H}^1( \tilde{K}(t) \EEE\setminus U )\leq \mu \quad \quad \mathcal{H}^1\big(( \tilde{K}(t) \EEE\, \triangle \, J_{v})\cap U \big)\leq \mu \,.
  \end{equation}  
 Let $(v_n)_{n}$ be a recovery sequence for $v$ with respect to the functional $\mathcal{H}_{ \tilde{K}(t) }^{-}(\cdot, U)$, i.e.,  particularly 
 $${\limsup_{n \to \infty } \mathcal{H}^1((J_{v_n} \setminus   \tilde{K}_n(t))  \cap U ) \le \MMM \int_{J_v \cap U} h^-_{ \tilde{K}(t)}(x,   \nu_v(x)) \,  \mathrm{d}\mathcal{H}^{1}(x) \le \EEE \mu}$$
   by \eqref{sigma-covering1} and the fact that $h^-_{\tilde{K}(t)} \le 1$, see Remark~\ref{rem: sigma1}(i). \EEE 
We define a modified sequence by 
  \begin{equation}
    \tilde{v}_n \MMM (x) \EEE \defas \begin{cases}
      v_n(x) & x\in \Omega' \setminus \Omega_{n,k}^{\rm mod},\\
      0 & x\in \Omega_{n,k}^{\rm mod}\,.
    \end{cases}
  \end{equation}  
Then, recalling \eqref{definition-crackset}   we have 
$J_{\tilde{v}_n}\sm K_n(t)\tsubset J_{v_n}\sm \tilde{K}_n(t)$ and \EEE thus $\limsup_{n \to \infty } \mathcal{H}^1((J_{ \tilde{v}_n } \setminus K_n(t)\EEE) \cap U \EEE) \le \mu$. In view of \eqref{3112241116}--\eqref{eq: combined}, we obtain $|\{\tilde{v}_n\neq v_n\}|  \le \EEE |\Omega_{n,k}^{\rm mod}| \leq  C_{\eta_n} \EEE \eps_n \to 0 $.  As ${v}_n\to v$ in $ L^1(U; \VIT \R^2 \EEE)$, we thus \EEE  get  $\tilde{v}_n\to v$ in $ L^1(U; \VIT \R^2 \EEE) \EEE $. Using the $\Gamma$-liminf inequality for the $\sigma$-convergence of $K_n(t)$ to $K(t)$ we get 
\begin{equation}\label{0802252022}
 \mathcal{H}_{K(t)}^{-}(v,U)=\int_{J_v \cap U} h^-_{K(t)}(x,  \nu_v \EEE (x)) \, \mathrm{d} \mathcal{H}^1(x) \EEE \leq \limsup_{n \to \infty } \mathcal{H}^1((J_{ \tilde{v}_n \EEE} \setminus K_n(t)\EEE)  \cap U \EEE) \le \mu.   
\end{equation}
 Now, \eqref{sigma-covering1} and the fact that $h^-_{K(t)}\le 1$ show
$$\int_{\tilde{K}(t)} h^-_{K(t)}(x, \nu_{\tilde{K}(t)}(x)) \, \mathrm{d} \mathcal{H}^1(x) \le \mathcal{H}_{K(t)}^{-}(v,U) + 2\mu \le 3\mu, $$
which contradicts \eqref{formycontra} and concludes the proof. \EEE
\end{proof}



In the sequel, it will be  convenient to express some of the quantities considered so far in terms of the time $t$ in place of the iteration step.  As before,  let $\{0=t^{0}_n< t^{1}_n<\dots<t^{T/\delta_n}_n=T\}$ be the  time discretization of the interval $[0,T]$, and let $(u^j_{n})_{j<k}$ be a  corresponding  displacement history. We define the set of functions $\mathcal{A}_n(t)\defas \mathcal{A}_n^{k} $ for $t \in [t^k_n,t_n^{k+1})$, see \eqref{akeps}. Recalling \eqref{energy-split} and \eqref{1908241552}, we define
\begin{align}\label{energy along evo}
   \mathcal{E}_n(v_n;t)  := \mathcal{E}_n^{k}(v_n) =     \mathcal{E}_n(v_n ;(u_n^j)_{j<k})\quad \text{for}\; t \in [t^k_n,t_n^{k+1})
\end{align}
for each $v_n \in \mathcal{A}_n(t)$. We use similar notation for the parts of the energy introduced in \eqref{energy-split}. Recall also the Griffith energy defined in \eqref{eq: lim-en}.

\begin{proposition}\label{liminf-ineq'}
 \GGG Let $t\mapsto (u_n(t), K_n(t))$ be the evolution defined in \eqref{definitionofun}  and \eqref{definition-crackset}. \EEE Let  $K(t)$ be the $\sigma$-limit of $K_n(t)$ given by Proposition \ref{limit-crack-evo}, and let $G(t)$ be the  set  corresponding to $K(t)$ \EEE defined before \eqref{b(t)}. \EEE   Then,  there exists a function $ u(t) \EEE \in AD(g(t),K(t))$  with $u(t) = 0$ on $\Omega' \setminus G(t)$, \EEE \EEE    and a subsequence $(n_l)_l$ depending on $t$, such that $u_{n_l}(t)\to u(t)$ in measure on \EEE  $G(t) \EEE $ and $e(u_{n_l}(t)) \rightharpoonup e(u(t))$ weakly in $L^2( G(t); \EEE \R^{2 \times 2}_{\rm sym} )$.
    Moreover,  \EEE
 \begin{equation}\label{enlobo}
  {   
  \liminf_{n\to \infty} \mathcal{E}^{\rm crack}_{n}(u_{n}(t);t)\geq \kappa \sin \theta_0 \hu(K(t)) \,. }
  \end{equation}  
\end{proposition}

\begin{proof}
We fix $t \in [0,T]$ and for each $n \in \N$ we choose $k_n$ such that $t \in [t^{k_n}_n, t^{k_n+1}_n)$. Recall the definition of $\Omega_{n,k_n}^{\rm crack} $  and $\Omega_{n,k_n}^{\rm mod} $ in \eqref{Tcrack} and \eqref{eq: combined}, respectively.  By $u^{\rm mod}_n$ we denote the function given by Theorem \ref{generaltheorem} applied on $u_n(t)$, which by \eqref{themainthing}  and the energy bounds in  Proposition~\ref{prop-control-on-ela}, and Corollary \ref{cor: energy bound} satisfies \begin{align}\label{themainthingXXX}
    \|e(u^{\rm mod}_n) \|_{L^2(   \Omega'_{\eps_n,\eta_n}  \EEE \setminus \Omega_{n,k_n}^{\rm mod})} \le C_{\eta_n}. 
   \end{align}
    Choose $\Omega_* \subset \subset \Omega'$ and suppose that $n$ is large enough such that $\Omega_* \subset \Omega'_{\eps_n,\eta_n} $. \EEE 
We define $v_{n}\in GSBD^{2} ( \Omega_* \EEE )$ by $v_{n}\defas (1-\chi_{\Omega_{n,k_n}^{\rm mod} }) u^{\rm mod}_n$. Then, by \eqref{3112241116}, \eqref{eq: combined}, and  Corollary \ref{cor: energy bound} we get 
\begin{align}\label{Dn}
   |D_n | \le   C_{\eta_n}  \eps_n + C \eps_n  \le   C_{\eta_n}  \eps_n  \to 0, \quad \text{ where } D_n:=  \big( \lbrace u_n(t) \neq  v_n \rbrace \cap \Omega_* \big) \EEE \cup \Omega_{n,k_n}^{\rm crack}  .  
\end{align}

\EEE

\noindent
For $p:=3/2$ we have   by H\"older's inequality, \eqref{themainthingXXX}--\eqref{Dn},   and the energy bound in  Proposition~\ref{prop-control-on-ela} that
       \begin{align*}
       \|e({v}_n) \|_{L^p( \Omega_* \EEE )} & =  \|e(v_n) \|_{L^p(D_n)} + \|e(u_n^{k_n}) \|_{L^p(\Omega' \setminus \Omega_{n,k_n}^{\rm crack})} \\
       & \le   C  |D_n|^{1/6} \|e(u^{\rm mod}_n) \|_{L^2(D_n \setminus \Omega_{n,k_n}^{\rm mod})}  +    C\|e(u_n^{k_n}) \|_{L^2(\Omega' \setminus \Omega_{n,k_n}^{\rm crack})} 
        \le C_{\eta_n}^{7/6} \eps_n^{1/6} + C \le C,       
       \end{align*} 
       where in the last step we used \eqref{3112241116}.   By \eqref{eq: combined} and the energy bound we get   $\mathcal{H}^1(J_{{v}_n}   )   \leq \mathcal{H}^1(K_n(t) ) \leq C  $. \EEE   Thus, by Proposition \ref{prop: compactness}   we find $u(t) \in  GSBD^p_\infty \EEE (\Omega_*)$  such that  (up to a  subsequence, not relabeled)
\begin{equation}\label{1001252053}
       v_n \to u(t)\text{ in measure on }\Omega_* \setminus A_{u(t)}^\infty.
\end{equation}        By Proposition \ref{prop:sigmasigmap}(i) and the $\sigma$-convergence of $K_n(t)$ to $K(t)$, we get $J_{u(t)} \cap \Omega_* \, \tilde{\subset} \,  K(t)$. Then, by a diagonal argument  (letting $\Omega_* \nearrow \Omega'$), we find $u(t) \in   GSBD^p_\infty \EEE (\Omega')$  such that $v_n \to u(t)$ in measure on $\Omega' \setminus A_{u(t)}^\infty$ with $J_{u(t)} \, \tilde{\subset} \,  K(t)$.

  Note that   $ g(t_n^{k_n})_{\mathbf{T}(u_n^{k_n})} \EEE \to g(t)$ in $W^{1,2}(\Omega';\R^2)$ by the regularity of $g$   (recall \eqref{triangul}). \EEE As $u_n(t) \in \mathcal{A}_{n} \MMM (t) \EEE $ (see \eqref{akeps}), it is elementary to check that $u(t) = g(t)$ on $\Omega' \setminus \overline{\Omega}$. 
  In particular, we find $A_{u(t)}^\infty \subset \Omega $. Then, we can apply Proposition \ref{prop:sigmasigmap}(ii) which along with the definition of $G(t)$ before \eqref{b(t)} implies $|A_{u(t)}^\infty \cap G(t)| = \emptyset$. Thus, in view of \eqref{1001252053}, we conclude $v_n \to u(t)$ in measure on $G(t)$.       
Next, by \eqref{Dn} we get $u_n(t) \to u(t)$ in measure on $G(t)$.  Moreover, \EEE by Proposition~\ref{prop-control-on-ela},  $|\Omega_{n,k_n}^{\rm crack}| \to 0$, and \VIT by \EEE weak compactness we get $e(u_n(t)) \rightharpoonup e(u(t))$ weakly in $L^2( G(t); \EEE \R^{2 \times 2}_{\rm sym})$. 
  Up to replacing $u(t)$ by $0$ in $\Omega' \setminus G(t)$ (not relabeled), we thus obtain $u(t) \in GSBD^2   (\Omega')$, and  $J_{u(t)} \, \tilde{\subset} \,  K(t)$ as well as $u(t) = g(t)$ on $\Omega' \setminus \overline{\Omega}$ still hold. \EEE   Summarizing, we have shown that $  u(t)  \in AD(g(t),K(t))$, see \eqref{def:ADgH}.

It remains to show \eqref{enlobo}.   
We recall from \eqref{eq: combined} and \eqref{definition-crackset} that  
$$   \mathcal{H}^1( K_n(t)  ) \le  \frac{2}{\eps_n  \sin\theta_0}|   \Omega^{\rm crack}_{n,k}    |  +  C\eta_n \le  \frac{2}{\kappa\sin\theta_0}\mathcal{E}^{\rm crack}_{n}(u_{n}(t);t) +C\eta_n .$$         
 For $n \to \infty$ we have \EEE  $\eta_n \to 0$ and  $|\Omega^{\rm mod}_{n,k}| \to 0$ (see \eqref{3112241116}--\eqref{eq: combined}). Then,  in view of $K_n(t) = \partial \Omega^{\rm mod}_{n,k}$, \EEE by means of Lemma~\ref{le:semicontfactortwo} we derive 
$$ 2  \mathcal{H}^1( K(t)  ) \le  \frac{2}{\kappa\sin\theta_0} \liminf_{n\to \infty}\mathcal{E}^{\rm crack}_{n}(u_{n}(t);t)\,.  $$
This concludes the proof of \eqref{enlobo}. 
\end{proof}
\subsection{Proof of Theorem \ref{maintheorem}}

In this section we give the proof of the main result.  We \EEE   need a final preparation.  Recall the Griffith energy $\mathcal{E}$ defined in \eqref{eq: lim-en}. \EEE 

\begin{theorem}[Stability]\label{stability}
  Let $t \mapsto (u_n(t), K_n(t))$ \EEE be the evolution defined by \eqref{definitionofun}--\eqref{definition-crackset} and let $K(t)$ be the $\sigma$-limit of $K_n(t)$.  
  For any $\psi\in \GGG GSBD^2(\Omega')$ \EEE with $\psi = g(t)$ on $\Omega' \setminus \overline{\Omega}$ 
  there exists a sequence $(\psi_{n})_{n}$ of piecewise affine displacements with $\psi_n \in   \MMM \mathcal{A}_{n}(t) \EEE $  
  converging to  $\psi$ in measure on $\Omega'$ \EEE such that
\begin{equation}\label{stab1}
  \limsup_{n\to \infty} \, \Big( 
 \   \mathcal{E}^{ {\rm crack}}_{n} \EEE (\psi_n;t) - \mathcal{E}^{ {\rm crack}}_{n}(u_n(t);t)  \Big)
     \leq \kappa \sin(\theta_0) \,  \mathcal{H}^1(J_\psi\sm K(t))  ,
  \end{equation}
and
 \begin{equation}\label{stab2}
      \lim_{n\to \infty}\,  \mathcal{E}^{ {\rm elast}}_{n}(\psi_n;t)   \leq   \int_{\Omega}  |e (\psi)|^2 \, {\rm d}x.
  \end{equation}
\end{theorem}

\begin{corollary}[Recovery sequence]\label{cor: stability}
  For each $\psi\in GSBD^2(\Omega') $ with $ \psi = g(0)\EEE$ on $\Omega' \setminus \overline{\Omega}$ there exists a sequence $(\psi_{n})_{n}$ with $\psi_n \in \mathcal{A}^{0}_n  $ such that $\psi_n$ converges to  $\psi$  in  measure on $\Omega'$  and
     \begin{equation*}
     \limsup_{n\to \infty} \mathcal{E}^{0}_n(\psi_n) \le \mathcal{E}(\psi,J_{\psi}).
     \end{equation*}
 \end{corollary}
 
%

This result is essential to pass from the minimality condition \eqref{minimizing-scheme} in the finite element  model to the global stability in \eqref{finalstability}. For this reason, estimates of this kind are often referred to as \emph{stability of unilateral minimizers}, see e.g.\ \cite{GiacPonsi}. The proof of Theorem \ref{stability}  will be deferred to the next section; \VIT it could be directly adapted (indeed, simplified: it is enough to drop the dependence on $t$ and formally consider $K_n=\emptyset$) to prove Corollary~\ref{cor: stability}, which in turn confirms the upper bound in \EEE
the $\Gamma$-convergence result \cite[Theorem~3.1]{BonBab}.   The strategy follows the one of the \emph{jump transfer lemma} introduced \JJJ by {\sc Francfort and Larsen} in \EEE \cite{Francfort-Larsen:2003}, see the  works \cite{dMasoFranToad, Lazzaroni,   FriedrichSeutter,   FriedrichSolombrino,  Giacomini:2005b} for several variants.  Our situation is slightly different and more delicate because we have to construct an admissible mesh associated to the recovery sequence. For this, we will exploit the ideas by {\sc  Dal  Maso and Chambolle} \cite{ChamboDalMaso} \EEE for the explicit construction of a minimizing adaptive mesh. 

We are now in \BBB the \EEE position to prove the main result of this paper. 

\begin{proof}[Proof of Theorem \ref{maintheorem}]
    We split the proof into  five \EEE   parts. First, we prove the existence of a limiting evolution and validate the irreversibility of the crack sets. Subsequently, in a second step, we use Theorem~\ref{stability}  to prove the stability   property \EEE \eqref{finalstability}.  Afterwards, we show    the convergence of displacement fields without the necessity of passing to $t$-dependent subsequences, see Step 3. \EEE   Next, we want to confirm that the mapping  $t\to (u(t),K(t))$ 
    actually fulfills the energy balance \eqref{energybalance} which is content of Step 4. In the last step, we then prove the convergence of energies and thus the   strong convergence of   the linear strains.

    \emph{Step 1: Limiting evolution.} Let  $K(t)$ be the $\sigma$-limit of $K_n(t)$ for $t \in [0,T]$ given by Proposition~\ref{limit-crack-evo}.  Note that \EEE $t \mapsto K(t)$  is an increasing set function. For each fixed time $t\in [0,T]$, by virtue of  Proposition~\ref{liminf-ineq'}, \EEE there exists $u(t) \in  AD(g(t),K(t)) \EEE$  with $u(t) = 0$ on $\Omega' \setminus G(t)$ \EEE    and a subsequence \GGG $n_l$ \EEE depending on $t$ 
      such that $u_{n_l}(t)$ converges to $u(t)$   in measure on $G(t)$ and $e(u_{n_l}(t)) \rightharpoonup e(u(t))$ weakly  in $L^2( G(t); \EEE  \R^{2 \times 2}_{\rm sym} )$. \EEE   



\emph{Step 2: Stability  \eqref{finalstability}.} 
Fix $t \in [0,T]$. Let $H$ with  $K(t) \, \tilde{\subset} \, H $ and $\psi\in AD(g(t),H)$. \EEE  We employ Theorem \ref{stability} to obtain a sequence of piecewise affine \EEE displacements $\psi_n \in \mathcal{A}_{n} \MMM (t) \EEE $ \GGG approximating $\psi$ and \EEE satisfying \eqref{stab1}--\eqref{stab2}. 
By the minimality property of the solution $u_n(t)$, see \eqref{minimizing-scheme0}--\eqref{minimizing-scheme},  \eqref{definitionofun}, and the shorthand notation in \eqref{energy along evo} we have  
  \[\mathcal
  {E}_{n}(u_n(t); t)\leq \mathcal{E}_{n}(\psi_{n} ;t)\,.\]
  We now split the energy on both sides like in \eqref{energy-split}: subtracting the crack energy on the left-hand side gives us   
  \begin{equation}\label{Passtolimit}
  \begin{aligned}    
 \mathcal{E}_n^{{\rm{elast}}}(u_{n}(t);t)\leq  & \;\mathcal{E}_n^{{\rm{elast}}}(\psi_{n}; t) + \mathcal{E}_n^{{\rm{crack}}}(\psi_{n};t)-\mathcal{E}_n^{{\rm{crack}}}(u_{n}(t);t). 
  \end{aligned}
\end{equation}
Passing to the limit in \eqref{Passtolimit}, by employing \eqref{stab1} 
and \eqref{stab2}, we find 
\begin{equation}\label{2108241743}
\limsup_{n\to \infty} \mathcal{E}_n^{{ \rm{elast}}}(u_{n}(t);t)   \le    \int_{\Omega } |e(\psi)|^2 \, {\rm d}x +   \kappa \sin(\theta_0)   \,  \mathcal{H}^1(J_\psi\sm K(t)) .  
\end{equation} 
 By Proposition \ref{prop-control-on-ela}, Proposition \ref{liminf-ineq'},  and  a compactness argument, we get  $e(u_{n_l}(t))\rightharpoonup W(t)$ in $L^2(\Omega';  \R^{2 \times 2}_{\rm sym})$ for some $W(t) \in L^2(\Omega'; \R^{2 \times 2}_{\rm sym})$ with $W(t) = e(u(t))$ on $G(t)$. As  $u(t) = 0$ on $\Omega' \setminus G(t)$, \EEE we
\EEE derive
\begin{equation}\label{2108241743-new'XXX}
  \int_{\Omega}|e(u(t))|^2\, {\rm d}x   +   \int_{\Omega \setminus G(t)}|W(t)|^2\, {\rm d}x   \EEE   \le     \int_{\Omega } |e(\psi)|^2 \, {\rm d}x +   \kappa \sin(\theta_0)  \,  \mathcal{H}^1(J_\psi\sm K(t)) .  \EEE
\end{equation} 
Since $\psi\in AD(\psi, H)$, we have $J_{\psi}  \, \tilde{\subset} \, \EEE H$,  and thus \EEE  
\begin{equation}\label{2108241743-new}
  \int_{\Omega}|e(u(t))|^2\, {\rm d}x   \le     \int_{\Omega } |e(\psi)|^2 \, {\rm d}x +   \kappa \sin(\theta_0) \,   \mathcal{H}^1(H\sm K(t))  . \EEE
\end{equation} 
Because $K(t) \, \tilde{\subset} \,  H$ we conclude that \eqref{finalstability} holds.  By particularly choosing $\psi = u(t)$ in \eqref{2108241743-new'XXX} \MMM and recalling $J_{u(t)} \, \tilde{\subset} \, K(t)$, \EEE we get $W(t) \equiv 0$ on $\Omega \setminus G(t)$, and thus  $e(u_{n_l}(t)) \rightharpoonup e(u(t))$ weakly  in $L^2(\Omega';  \R^{2 \times 2}_{\rm sym} )$. \EEE 

\emph{Step 3: Uniqueness of limiting displacements.}  We argue that the obtained limit $u(t)$ is uniquely determined on $G(t)$ and satisfies  $e(u(t)) = 0 $ on $\Omega' \setminus G(t)$ for all $t \in [0,T]$. This along with Uyrsohn's principle  shows that the subsequence $(n_l)_l$ in Step 1 can be chosen independently of $t$. This shows \eqref{eq: l1convi} except for strong convergence,  which we defer to Step 5. 
 
 Let us suppose that $\hat{u}(t) \in AD(g(t),K(t))$ denotes another limit of the sequence $u_n(t)$ found in Step~1.   As in Step~2, one can show that \eqref{2108241743-new} holds for $\hat{u}(t)$ in place of $u(t)$. This shows that both $u(t)$ and $\hat{u}(t)$ are minimizers of the strictly convex minimization problem $v \mapsto   \int_{\Omega}|e(v)|^2\, {\rm d}x $    for $v \in AD(g(t),K(t))$. Consequently, $e(u(t)) = e(\hat{u}(t))$ on $\Omega'$. Moreover,    $e(u(t)) = e(\hat{u}(t)) = 0$ on $\Omega' \setminus G(t)$, see \eqref{b(t)}. By a piecewise rigidity argument (see \cite{Chambolle-Giacomini-Ponsiglione:2007}), taking the definition of $G(t)$  and the fact $u(t)=\hat{u}(t)=g(t)$ on $\Omega'\sm \overline{\Omega}\subset G(t)$ \EEE  into account and using the fact that adding infinitesimal rigid motions is only admissible on the connected components of $\Omega' \setminus K(t)$ inside $B(t)$, we then see that $u(t) = \hat{u}(t)$ on $G(t)$.   We refer to \cite[Proof of Theorem 2.2]{steinke} for details. 
 
Moreover, arguing as in Step~2 and above in Step~3 (cf.\ also \cite[Lemma~3.8]{Francfort-Larsen:2003} and \cite[Theorem~7.5]{FriedrichSolombrino}), we can ensure that, for each $t \in [0,T]$ outside of an at most countable subset and every $t_n \MMM \nearrow \EEE t$, we have that $u(t_n)\to u(t)$ in measure on $G(t)$ and $e(u(t_n))\rightharpoonup e(u(t))$ in $L^2(\Omega'; \R_{\rm sym}^{2 \times 2})$.

\emph{Step 4: Energy balance: 
} From Step 1 and Step 3 we recall that \EEE   $e({u}_{n}  \EEE (\tau))\rightharpoonup e(u(\tau))$ in $L^2( \Omega'; \R_{\rm sym}^{2 \times 2} \EEE)$ for all $\tau \in [0,T]$.  As $e(\partial_{t}g(\MMM \tau, \cdot \EEE ))\in L^2(  \Omega', \R_{\rm sym}^{2 \times 2} \EEE)$  is  uniformly  bounded in time, using Proposition~\ref{prop-control-on-ela} \EEE  we can apply the reverse Fatou's lemma  to  deduce that \EEE 
     \begin{align}\label{limit-testing}
      & \limsup_{n\to \infty} \int_{0}^{t} \int_{\Omega}  e( {u}_n  (\tau)) : ( \partial_{t}g(\tau)) \,  {\rm d}x  \, {\rm d}\tau \le  \int_{0}^{t} \int_{\Omega}  e(u(\tau)) :( \partial_{t}g(\tau))\,  {\rm d}x \,{\rm d}\tau  \,.
     \end{align} By Lemma \ref{e-balance-lemma} we find $(\beta_n)_n$ with $\beta_n \to 0$ such that for any $t\in [t_n^{k}, t_n^{k+1})$ we have 
  \begin{align*}
\mathcal{E}_n(u_n(t);t)  - \mathcal{E}_{n}^{0}(u^0_n) &\leq 2  \int_{0}^{t_{n}^{k}} \int_{\Omega}  e({u}_n(\tau)) :  e(\partial_{t} {g}(\tau) )\,{\rm d} x\,{\rm d}\tau  +  \beta_n. 
\end{align*}
By passing to another sequence $(\beta_n)_n$, still satisfying $\beta_n \to 0 $, \EEE \EEE  for every $t\in [0,T]$ we have   
    \begin{align}\label{energy-diff}   
    \mathcal{E}_n(u_n(t);t) &\leq 2 \,\int_{0}^{t} \int_{\Omega}  e(  {u}_n  (\tau)) :  e(\partial_{t}g(\tau))\,{\rm d} x\,{\rm d}\tau  +  C\beta_n  +  \mathcal{E}_{n}^{0}(u^0_n)\,,
  \end{align}      
  where we have used $e(  {u}_n )\in  L^\infty(0,T; L^2(\Omega, \R_{\rm sym}^{2 \times 2} \EEE))$ (see Proposition \ref{prop-control-on-ela}) \EEE and $g\in W^{1,1}(0,T; W^{2,\infty}(\Omega;\R^2))$ to  estimate the integral from $t_n^k$ to $t$ in terms of  $\beta_n$. 
Recalling   \eqref{minimizing-scheme0}, \eqref{1908241552}, and Corollary \ref{cor: stability} \EEE we have 
  \[
 \mathcal{E}(u(0), \MMM K(0) \EEE ) \le \EEE \limsup_{n\to \infty} \mathcal{E}^0_{n}(u_n^0) =   \limsup_{n\to \infty}\min_{\mathcal{A}^{0}_n} \mathcal{E}^0_{n} \le   \mathcal{E}(\psi, J_{\psi})  \EEE
  \] 
   for all  $\psi\in GSBD^2( \Omega') \EEE $ with $ \psi = g(0)\EEE$ on $\Omega' \setminus \overline{\Omega}$,  where the first inequality follows from Proposition~\ref{liminf-ineq'} and  weak lower semicontinuity of norms. \MMM As $J_{u(0)} \, \tilde{\subset} \,   K(0)$, \EEE  this shows that condition (a) in Definition \ref{main def} holds. By choosing $\psi = u(0)$  and using \MMM again \EEE $J_{u(0)} \, \tilde{\subset} \,   K(0)$, we get \MMM $J_{u(0)} =K(0)$ and \EEE $\limsup_{n\to \infty} \mathcal{E}^0_{n}(u_n^0) \le \mathcal{E}(u(0), K(0))$. \EEE  Combining this with Proposition~\ref{liminf-ineq'} and \eqref{limit-testing}--\eqref{energy-diff} we get
  \begin{align}\label{eq: for en ba}
          & \mathcal{E}(u(t), K(t) )  \leq \liminf_{n\to \infty} \mathcal{E}_{n}(u_n(t);t) \GGG \leq \limsup_{n\to \infty} \mathcal{E}_{n}(u_n(t);t) \EEE \leq \limsup_{n\to \infty} \mathcal{E}_{n}^{0}(u_{n}^0) \\& \hspace{1em}+  \limsup_{n\to \infty} 2\int_{0}^{t} \int_{\Omega} e( {u}_n \EEE (\tau)) : e(\partial_{t}g(\tau))\,{\rm d} x\,{\rm d}\tau \notag  \MMM \le \EEE  \mathcal{E}(u(0), K(0)) + 2 \int_0^t  \int_{\Omega} e(u(\tau)) : e(\partial_{t} g(\tau))\, {\rm d}x \, {\rm d}\tau.
  \end{align}
Thus, it remains to prove the reverse inequality
  \begin{equation}\label{eq: reverse}
      \mathcal{E}(u(t),  K(t)   )\geq \,\mathcal{E}(u(0),K(0))+2\int_0^t  \int_{\Omega} e(u(\tau)) : e(\partial_{t} g(\tau))\, {\rm d}x \, {\rm d}\tau.
  \end{equation}
Fixed $t \in [0,T]$, let $(s_k^i)_i$ be a partition of $[0,t]$ with 
\[
\lim_{k\to +\infty} \max_{1\leq i\leq k} (s_k^i-s_k^{i-1})=0.
\]
  For any $i$ and $k$, we take \EEE  $ H= K(s_k^i)$ and $ v= (u(s_k^i)-g(s_k^{i})+g(s_k^{i-1}))  \in AD(g(s_k^{i-1}),  K(s_k^i) )$ as an admissible competitor at time $s_k^{i-1}$.  This yields \EEE
$${\mathcal{E}(u(s_k^{i-1}),   K(s_k^{i-1})\EEE )\leq \mathcal{E}\big(u(s_k^i) +  \big(g(s_k^{i-1}) - g(s_k^{i})\big)  \EEE , K(s_k^i) \EEE \big),}$$
that is
\begin{equation*}
\begin{split}
\mathcal{E}(u(s_k^i), K(s_k^i))\geq \mathcal{E}(u(s_k^{i-1}),   K(s_k^{i-1})) + 2 \int_{\Omega} e(u(s_k^i))\colon e(g(s_k^i)-g(s_k^{i-1})) \dx 
- \int_\Omega |e(g(s_k^i)-g(s_k^{i-1}))|^2 \dx.
\end{split}
\end{equation*}
Summing over $i$, and observing that $g(s_k^i)-g(s_k^{i-1})=\int_{s_k^{i-1}}^{s_k^{i}} \partial_t g(\tau) \,\mathrm{d}\tau$ \MMM are \EEE functions on $L^2(\Omega'; \R^2)$ (where the integral is in the Bochner sense), we get
\begin{equation*}
\mathcal{E}(u(t), K(t))\geq \mathcal{E}(u(0),   K(0)) + 2 \int_0^t \int_{\Omega} e(\ol u(\tau))\colon e(\partial_t g(\tau))\dx\, \mathrm{d}\tau + o_{k\to +\infty}(1),
\end{equation*}
with $\ol u(\tau):= u(s_k^i)$ for $\tau \in (s_k^{i-1}, s_k^i]$. \MMM Now, we use \EEE $g\in W^{1,1}(0,T; W^{1,2}(\Omega'; \R^2))$ 
together with the uniform bound from Proposition~\ref{prop-control-on-ela} and the weak convergence in $L^2(\Omega; \R^{2 \times 2}_{\rm sym})$, \MMM as $k \to \infty$, \EEE of $e(\ol u(\tau))$ to $e(u(\tau))$ for all $\tau \in [0,t]$,  except for at most countable many (see end of Step~3). \EEE  
\EEE 
\EEE This shows \eqref{eq: reverse} and 
  concludes the proof of the energy balance \eqref{energybalance}.

 \emph{Step 5: Energy convergence  and strong convergence of displacements}.  
 \GGG Gathering \eqref{eq: for en ba} with \eqref{eq: reverse}, and recalling \eqref{energy along evo} we deduce \eqref{energ convi}. Moreover, since the two parts of the energy are separately lower semicontinuous,  we obtain the statement in Remark \ref{man: remark}(i). \EEE In particular, $\|e(  {u}_{n}   (t))\|_{L^2(\Omega)} \to \|e(u(t))\|_{L^2(\Omega)}$, which shows $e(u_{n}(t))\to e(u(t))$ strongly   in $L^2(\Omega'; \R^{2 \times 2}_{\rm sym})$ and concludes the proof of \eqref{eq: l1convi}. \EEE
 \end{proof}

\EEE

\section{Unilateral stability: Proof of Theorem \ref{stability}}\label{sec: stability}
  This section is entirely devoted to  the proof of the stability result in  Theorem \ref{stability}.

 \subsection{Preparations}\label{sec:stabprep}

We start with some preparations for the proof. \EEE

\subsection*{Density argument} 
We first observe that it suffices to prove the statement for functions $\psi$ with more regularity, employing a suitable density argument. 
Let $\mathcal{W}(\Omega' ) \subset GSBD^2(\Omega')$ be the collection of functions $v$ such that $J_v$ is closed and included in a finite union of closed and  connected  pieces of $C^1$-curves and $v$ lies in $W^{2,\infty}(  \Omega' \EEE  \setminus J_v;\R^2)$. By the density result \cite[Theorem~3.2]{steinke} we can choose a sequence $(v_n)_n\subset \mathcal{W}( \Omega' \EEE )$ with $v_n=g(t)$ on  $\Omega'\setminus \overline{\Omega}$  such that \begin{equation}
\begin{cases}
  \text{$v_n\to v$ in measure on  $\Omega' $}\,,\\  
 \| e(v_n) - e(v) \|_{L^2(\Omega')}\to 0 \EEE\,,\\
   \mathcal{H}^1(J_{v_n}\triangle J_v) \to 0.
\end{cases}
\end{equation}\EEE
Bearing in mind this density result,  by a diagonal argument \EEE it suffices to construct a sequence as in Theorem~\ref{stability} for a function $\psi \in \mathcal{W}(\Omega ')$ with  $\psi = g(t)$ on  $\Omega'  \setminus \overline{\Omega}$. Furthermore, for simplicity we only treat the case that $J_{\psi} \cap \partial_D \Omega =\emptyset$ (no jump along the boundary), for the general case follows by minor adaptations of the construction at the boundary (see \cite{Francfort-Larsen:2003}) which would merely overburden notation in the sequel.  

\MMM Choose $k(t)  \in \N \EEE $  such that $t \in [t^{k(t)}_n,t^{k(t)+1}_n)$ and \EEE define  $\mathcal{A}^{k(t)}_{n}(g(t))$ as in \eqref{akeps} with $g(t)$ in place of $g(t^k_n)$. \EEE We fix $\theta>0$ and observe that it suffices to construct a sequence $(\psi_n)_n \in \mathcal{A}^{k(t)}_n(g(t))  $ 
of displacements such that  
  \begin{equation}\label{stab0neu}
    \limsup_{n\to \infty} |\{|\psi_n-\psi|\geq \delta\}| \le C\theta\quad \text{for all}\; \delta>0\,,
    \end{equation}
    \begin{equation}\label{stab1neu}
    \limsup_{n\to \infty} \,  
   \mathcal{E}^{ {\rm crack}}_{n}(\psi_n;t) -\mathcal{E}^{ {\rm crack}}_{n}(u_n(t);t)  
    \leq \kappa \sin(\theta_0)\,\mathcal{H}^1(J_\psi\sm K(t)) + C\theta\,,
    \end{equation}
   \begin{equation}\label{stab2neu}
        \limsup_{n\to \infty}\,  \mathcal{E}^{ {\rm elast}}_{n}(\psi_n;t)   \GGG \leq \EEE   \int_{\Omega} |e (\psi)|^2 \, {\rm d}x + C\theta,\EEE
  \end{equation}
   where $C>0$ is a universal constant.   Strictly speaking, we need to construct a sequence $\psi_n$ that attains the boundary values $g(t^{k(t)}_n)$,  i.e., lies in $\MMM \mathcal{A}_n(t) \EEE = \mathcal{A}^{k(t)}_n = \mathcal{A}^{k(t)}_n(g(t^{k(t)}_n))$. \EEE Therefore, we eventually need to replace the sequence  $(\psi_n)_n $ by $\psi_n -g(t)_{\mathbf{T}(\psi_n)} + g(t^{k(t)}_n)_{\mathbf{T}(\psi_n)}$. \EEE Due to the regularity of $g$, this still leads to  \eqref{stab0neu}--\eqref{stab2neu}.    Then, the statement follows again by a diagonal argument, sending $\theta \to 0$.


\GGG 
\EEE

  \subsection*{Besicovitch covering}
We follow the procedure from \cite{FriedrichSeutter}, which in turn stems from 
\cite{Francfort-Larsen:2003}, i.e., we introduce a fine cover of 
 $J_{\psi}  $ \EEE with closed squares satisfying certain additional properties. By $\nu_\psi$ we denote \MMM a \EEE measure-theoretic \MMM unit \EEE normal at $J_{\psi}$.  We split our considerations into (1) $K(t) \cap J_\psi$ and (2) $J_\psi \setminus K(t)$. \EEE

 (1)  By a covering argument, \EEE for given $\theta>0$ there exists an open set $U \subset \Omega'$ \EEE and  a function $v\in {\rm PC}( U) \EEE $ such that 
\begin{equation}\label{sigma-covering}
 \mathcal{H}^1(K(t) \setminus U) \le \theta, \EEE \quad \quad \mathcal{H}^1((K(t)\, \triangle \, J_{v})\cap U \EEE )\leq \theta \,. 
\end{equation} 
  By  the $\sigma$-convergence of $K_n(t)$ to $K(t)$, \EEE  there exists a sequence  $(v_n)_{n}\in  PC(U) \EEE $ with $v_n\to v$ in $L^1(U)$ and  
\begin{equation}\label{almost-recovery-seq}
  \limsup_{n \to \infty} \mathcal{H}^1(J_{v_n}\setminus K_n(t))\leq   \mathcal{H}^1(J_{v}\setminus K(t) )\leq \theta \,,
\end{equation} 
 see  \eqref{almost-recovery-seq'} for a similar argument. \EEE We can choose a suitable subset $G_j \subset J_{v}  $ as done preceding  \cite[(2.2)]{Francfort-Larsen:2003} with $\mathcal{H}^1( \GGG J_v \EEE \setminus G_j )\leq \theta$ such that
\begin{align}\label{eq gj}
  \mathcal{H}^1( K(t) \setminus G_j  ) \le  \mathcal{H}^1( (K(t) \setminus G_j) \cap U  ) + \theta \leq  \EEE  \mathcal{H}^1( \GGG J_v \EEE \setminus G_j ) + 2\theta \le  3\theta\,. \EEE
\end{align}  
For each $x\in G_j \cap \GGG J_\psi \EEE $ we consider closed squares $Q_r(x)$ with sidelength $2r$ and \BBB two sides \EEE orthogonal to $\nu_\psi(x)$ which are contained in $ U \EEE $ and satisfy \cite[(2.3), (2.5)]{Francfort-Larsen:2003}. 
 
 (2) \EEE For closed squares $Q_r(x) \subset \Omega' \EEE$ with a center $x\in \GGG J_\psi \EEE\setminus K(t)$, still oriented in direction $\nu_\psi(x)$, we can assume \JJJ that \EEE
\VVV for \EEE $\mathcal{H}^1$-a.e.\ $x\in \GGG J_\psi \EEE \setminus K(t)$ and for $r$ sufficiently small it holds
    \begin{equation}\label{almostnoK}
\mathcal{H}^{1}\big(K(t) \cap Q_{r}(x)\big)\leq \theta r\,.
    \end{equation}
Here, we use \EEE the fact that $K(t)$ has $\mathcal{H}^{1}$-density $0$ 
\VIT in a subset of $J_\psi \setminus K(t)$ of full $\mathcal{H}^1$ measure. \EEE

 As $\GGG J_\psi \EEE$ is contained in a finite 
union of  closed $C^1$-curves, for a.e.\ $x \in J_\psi$, possibly passing to smaller $r$, the above squares  in cases (1) and (2) \EEE can be chosen such that they also satisfy 
\begin{align}
  \label{normal-angle-estimate0}
2r &\le \mathcal{H}^1\big(\GGG J_\psi \EEE \cap Q_r(x) \big) \le 4r, \\
\label{normal-angle-estimate2}
 J_\psi \cap  Q_{r   } \EEE (x)  &\subset \lbrace y \colon   |(y-x) \cdot \nu_\psi(x)| \le \theta r \rbrace.        
\end{align}
With this, we obtain a fine cover of  $\Gamma:= (G_j \cap  \GGG J_\psi \EEE) \cup (\GGG J_\psi \EEE \setminus K(t))$ to which we can apply the Besicovitch covering theorem  with respect to the Radon measure $\mathcal{L}^{2}+\mathcal{H}^{1}|_{\Gamma}$. For $\theta>0$ fixed as above, we hereby find a finite and disjoint subcollection $\mathcal{B}:= (Q_{r_i}(x_i))_i$, or shortly denoted by $(Q_i)_i$, such that $(Q_i)_i$ satisfy the properties \eqref{almostnoK}--\eqref{normal-angle-estimate2} (\eqref{almostnoK} only for $x_i \notin K(t)$) \EEE as well as 
    \begin{equation}\label{besicovitch-props}
            \mathcal{L}^2\big(\bigcup\nolimits_{\mathcal{B}}  Q_{i} \EEE \big)  \le \theta, \EEE    
            \quad \quad \quad 
            \mathcal{H}^{1}\big( \Gamma \setminus \bigcup\nolimits_{\mathcal{B}}Q_i\big)  \le \theta\,.
        \end{equation} 
Here and in the following, we use $\bigcup_{\mathcal{B}} Q_i$ as a shorthand for $\bigcup_{Q_{r_i}(x_i) \in \mathcal{B}} Q_{r_i}(x_i)$. Using  \eqref{eq gj},   \eqref{besicovitch-props}, and \EEE the definition of $\Gamma$  we get 
      \begin{equation}\label{besicovitch-props2}   
\mathcal{H}^{1}\big( \GGG J_\psi \EEE\setminus \bigcup\nolimits_{\mathcal{B}}Q_i\big) \le 4  \theta.
              \end{equation} 
 Without further notice,  we will frequently use that the squares are pairwise disjoint. 

By $\mathcal{B}_{\rm good}\subset \mathcal{B}$ we denote the collection of closed  squares $Q_i = Q_{r_i}(x_i)$ with $x_i\in \GGG J_\psi \EEE\setminus K(t)$, and similarly we let $\mathcal{B}_{\rm bad}\subset \mathcal{B}$ be the collection of all squares $Q_i =Q_{r_i}(x_i)$ with $x_i\in \GGG J_\psi \EEE\cap \BBB G_j \EEE $. Clearly, we have $\mathcal{B}=\mathcal{B}_{\rm good}\cup \mathcal{B}_{\rm bad}$.  We also define the sets 
$$B_{\rm good} =  \bigcup\nolimits_{\mathcal{B}_{\rm good}} Q_i, \quad \quad \quad B_{\rm bad} =  \bigcup\nolimits_{\mathcal{B}_{\rm bad}} Q_i. $$
 Note that by construction we have $B_{\rm bad} \subset U$. \EEE In order to construct a sequence $\psi_n\colon  \Omega' \EEE\to \R^2$ of piecewise affine functions   satisfying \eqref{stab0neu}--\eqref{stab2neu}, we need to specify the triangulation $\mathbf{T}_{n}(\psi_n)$ that is associated to $\psi_n$. For this, we split   $\Omega'$  into \EEE three different \MMM subsets \EEE $B_{\rm good}$, $B_{\rm bad}$ and $B_{\rm rest}=  \Omega'  \setminus (B_{\rm good}\cup B_{\rm bad})$.   Inside of $B_{\rm good}$ and $B_{\rm bad}$ we will need two kinds of  a \emph{transfer of jump sets} which we discuss next. \EEE    

\subsection*{Jump transfer}

 We again consider the good and bad squares separately.

\subsubsection*{Bad squares} \EEE

  The sequence    $(v_n)_{n}\subset PC( U )$  defined above, satisfying \EEE $\limsup_{n \to \infty} \mathcal{H}^1(J_{v_n}\setminus K_n(t))\leq \theta$ and $v_n \to v$   in $L^1(U; \VIT \R^2 \EEE)$,  can be used to transfer  \EEE the jump set of $ \psi \EEE$  inside $B_{\rm bad}$.
 \VVV First, we recall the main points from \cite{Francfort-Larsen:2003}, for any $Q_i = Q_{r_i}(x_i)\in\mathcal{B}_{\rm bad}$: \EEE there are  two lines $H^{n,+}_i$ and $H^{n,-}_i$ with normal $\nu_\phi(x_i)$ which lie above and below the middle line $H_i$  containing the point $x_i$, also with normal $\nu_\phi(x_i)$, such that 
      \begin{equation}\label{Ri-estimates2}
R^{n}_i \supset Q_i \cap   \lbrace y \colon   |(y-x_i) \cdot \nu_\phi(x_i)| \le  2   \theta r_i \rbrace,  
    \end{equation}
    \begin{equation}\label{Ri-estimates}
      \mathcal{H}^{1}\big(\bigcup\nolimits_{\mathcal{B}_{\rm bad}}  L^n_i \big)\leq C\theta
    \end{equation}
for a universal constant $C>0$, where $R^{n}_i$ denotes  the rectangular subset that lies between $H^{n,+}_i$ and $H^{n,-}_i$, and $L^{n}_i:=\partial R^{n}_i \setminus (H^{n,+}_i \cup H^{n,-}_i)$ denotes its lateral boundaries,  cf.\ \cite[(2.10)--(2.11)]{Francfort-Larsen:2003}  and Figure~\ref{Gamma-constr}. \EEE  Moreover, the construction provides a set of   finite perimeter $P^n_i\subset Q_i = Q_{r_i}(x_i)\in\mathcal{B}_{\rm bad}$ \EEE such that 
\begin{equation}\label{def: conticurve}
  \Gamma^n_i := \partial^* P^n_i \cap Q_i \subset R_i^n
 \end{equation}  
  satisfies \EEE
 \begin{equation}\label{1301251232}
 \hu\big((\Gamma^n_i \cup L_i^n)\sm J_{v_n}\big)\leq C \theta r_i.
 \end{equation}
  Note that, without restriction,   $P^n_i$ can be chosen such that both $P^n_i$ and $Q_i \setminus P^n_i$ are connected sets   and hence  $ \Gamma^n_i$ is a curve,  see \EEE \cite[below (5.17)]{FriedrichSeutter}  for the precise argument. \EEE We refer to \cite[Section 2]{Francfort-Larsen:2003} or  \cite[Section 5]{FriedrichSeutter} for more details and also refer   to   Figure \ref{Gamma-constr} for an illustration of the construction. We mention that the construction is simplified compared to \cite{FriedrichSeutter, Francfort-Larsen:2003} as the sequence $(v_n)_n$ lies in $PC(U)$ and thus $\nabla v_n \equiv 0$.

 In \cite{Francfort-Larsen:2003}, the approximating functions $\phi_n$ are defined  by $\phi_n = \psi$ on $\Omega'\sm B_{\rm bad}$ and, \MMM in  $B_{\rm bad}$, as \EEE $\phi_n=\psi$ outside of $\bigcup_{\mathcal{B}_{\rm bad }} R^{n}_i $ and inside $\bigcup_{\mathcal{B}_{\rm bad }} R^{n}_i $ by \MMM reflection  \EEE ensuring that \EEE 
\begin{align}\label{franfort again}
J_{\phi_n} \cap Q_i \, \tilde{ \subset} \,  \Gamma^n_i \BBB \cup L_i^n \EEE  \quad \quad \text{for all $Q_i \in \mathcal{B}_{\rm bad}$}.
\end{align}
  Here, we will use a variant of this construction, see \cite[Section 5]{FriedrichSeutter}, which is based on a cut-off construction and  further  ensures     that,    for a constant $C > 0$, 
\begin{align}\label{eq: Lipschitz1before}
 & \Vert  \nabla \phi_n \Vert_{L^\infty(\Omega' )} \le   C\Vert    \nabla \psi  \Vert_{L^\infty(\Omega' )}, \\ \EEE
& \nabla \phi_n \in SBV( \Omega' ;\VVV \R^{2{\times}2}\EEE) \  \text{ with }   \Vert  \nabla (\nabla \phi_n) \Vert_{L^\infty(\Omega' )} \le   C \theta^{-1} (\min\nolimits_i r_i)^{-1}\Vert    \nabla \psi \Vert_{\VVV L^\infty(\Omega' ) \EEE} + \EEE  C\Vert    \VVV \nabla(\nabla \psi) \EEE \Vert_{\VVV L^\infty(\Omega' ) \EEE}, \notag
\end{align}
\EEE and that, for a sufficiently small constant $c_\theta>0$ depending on $\theta$, 
\begin{align}\label{eq: Lipschitz2alt}
\JJJ {\rm (i)} \ \bigcup_{\mathcal{B}_{\rm bad}}  \Gamma_i^{n} \cup L_i^n \, \tilde{\subset} \,  J_{\phi_n} \,,\EEE \quad \quad {\rm (ii)}\; \ \mathcal{H}^1\big( \big\{ x \in J_{\phi_n} \cap B_{\rm bad} \colon|[\phi_n](x)| \le \BBB c_\theta \EEE  \big\} \big)    \le C\theta.
\end{align} 
 We refer the reader to  \cite[(5.16), (5.21), (5.22)]{FriedrichSeutter}. \EEE 

  For later purposes, we stress \EEE that we can assume  without restriction that \EEE each connected component of $(\Gamma_i^{n}\setminus K_n(t)) \cap Q_i$ connects two different connected components of $\Omega_{n,k}^{\rm mod}$. \MMM (Here and in the sequel, we write $k$ in place of $k(t)$.) \EEE In fact, if the start and endpoint of a component lie in the same connected component of $\Omega_{n,k}^{\rm mod}$, in the construction of $\phi_n$ the curve  $\Gamma^n_{i}$ \EEE could be replaced by a curve that lies \VIT inside \EEE $K_n(t)$, i.e., on the boundary of $\Omega_{n,k}^{\rm mod}$, and such that \VVV \eqref{1301251232} \EEE would still hold. \MMM For the same reason, it is not restrictive to assume that each connected components of $\Omega_{n,k}^{\rm mod}$ meets at most two connected components of $(\Gamma_i^{n}\setminus K_n(t)) \cap Q_i$ as otherwise the curve could be replaced accordingly without affecting \eqref{1301251232}. \EEE \JOS Furthermore, we can assume that for each component $\compo$ of $\Omega_{n,k}^{\rm mod}$, the set $\ol \compo$ is not self-intersecting (cf., e.g., Figure~\ref{fig:self-intersection}), since otherwise we could modify $K_n(t)$ (and thus $J_{v_n}$) such that no self-intersections appear but \eqref{1301251232} still holds. \EEE

For future notational purposes, we denote the two connected sets $P^n_i$ and $Q_i \setminus P^n_i$ and the restriction of the function on these sets by  
\begin{equation}\label{1301252256}
P_i^{n,+}:= P^n_i, \quad  P_i^{n,-}:=  Q_i \sm P^n_i, \quad \quad  \quad  \MMM \phi_{i}^{n, \pm} \EEE :=\phi_n|_{P_i^{n,\pm}}.
\end{equation}

  \begin{figure}
    \begin{minipage}{.3\textwidth}
    \begin{tikzpicture}[scale=0.5]
    \draw[color=black] (0,0) -- (8,0) --  (8,8) -- node[anchor=north]{\footnotesize$Q_i$}(0,8) -- (0,0);
      \filldraw[fill=blue!15!white, draw=black]{} (0,6)-- (8,6) --(8,2)-- (0,2)--cycle;
        \draw[color=blue!55!white] (0,2) -- node[anchor=north] {$H_i^{n,-}$} (8,2);
      \draw[color=blue!55!white] (0,6) --node[anchor=south] {$H_i^{n,+}$} (8,6);
      \draw[color=white] (0,4)--
      (8,4);
      \draw[color=orange](0,2) -- node[anchor=east] {$L_i$}(0,6);
      \draw[color=orange](8,2) -- (8,6);
      \node[anchor=south, color=black] at (3.7,4) {$K_n(t)$};

  
      \filldraw[fill=blue!50!white, draw=black] (1.5,4.1) -- ++(-120:0.4) -- ++(0:0.4)-- cycle;
      \filldraw[fill=blue!50!white, draw=black] (1.5,4.1)++(-120:0.4) -- ++(0:0.4) -- ++(-120:0.4)-- cycle;
      \filldraw[fill=blue!50!white, draw=black] (1.5,4.1) -- ++(0:0.4) -- ++(-120:0.4)-- cycle; 
      \filldraw[fill=blue!50!white, draw=black] (1.5,4.1)++(0:0.4) -- ++(-120:0.4) -- ++(0:0.4)-- cycle;
      \filldraw[fill=blue!50!white, draw=black] (1.5,4.1)++(-60:0.4)  -- ++(-120:0.4)-- ++(0:0.4)-- cycle;
      \filldraw[fill=blue!50!white, draw=black] (1.5,4.1)++(0:0.4)++(-120:0.4) -- ++(0:0.4) -- ++(-120:0.4)-- cycle;
      \filldraw[fill=blue!50!white, draw=black] (1.5,4.1)++(0:0.4) -- ++(0:0.4) -- ++(-120:0.4)-- cycle;
      \filldraw[fill=blue!50!white, draw=black] (1.5,4.1)++(0:0.4)++(-60:0.4)  -- ++(-120:0.4)-- ++(0:0.4)-- cycle;
      \filldraw[fill=blue!50!white, draw=black] (1.5,4.1)++(0:0.8) -- ++(-120:0.4) -- ++(0:0.4)-- cycle;
      \filldraw[fill=blue!50!white, draw=black] (1.5,4.1)++(0:0.8)++(-120:0.4) -- ++(0:0.4) -- ++(-120:0.4)-- cycle;

      \filldraw[fill=blue!50!white, draw=black] (3.5,3.8) -- ++(-100:0.4) -- ++(20:0.4)-- cycle;
      \filldraw[fill=blue!50!white, draw=black] (3.5,3.8) -- ++(20:0.4) -- ++(-100:0.4)-- cycle; 
      \filldraw[fill=blue!50!white, draw=black] (3.5,3.8)++(20:0.4) -- ++(-100:0.4) -- ++(20:0.4)-- cycle;
      \filldraw[fill=blue!50!white, draw=black] (3.5,3.8)++(20:0.4) -- ++(20:0.4) -- ++(-100:0.4)-- cycle;
      \filldraw[fill=blue!50!white, draw=black] (3.5,3.8)++(20:0.8) -- ++(-100:0.4) -- ++(20:0.4)-- cycle;

      \filldraw[fill=blue!50!white, draw=black] (3.5,3.8)++(20:0.8)  ++(-100:0.4) ++(20:0.4) -- ++(20:0.4) -- ++(-100:0.4)-- cycle; 
      \filldraw[fill=blue!50!white, draw=black] (3.5,3.8)++(20:0.8)  ++(-100:0.4) ++(20:0.4)++(20:0.4) -- ++(-100:0.4) -- ++(20:0.4)-- cycle;
      \filldraw[fill=blue!50!white, draw=black] (3.5,3.8)++(20:0.8)  ++(-100:0.4) ++(20:0.4)++(20:0.4) -- ++(20:0.4) -- ++(-100:0.4)-- cycle;
      \filldraw[fill=blue!50!white, draw=black] (3.5,3.8)++(20:0.8)  ++(-100:0.4) ++(20:0.4)++(20:0.8) -- ++(-100:0.4) -- ++(20:0.4)-- cycle;

      \filldraw[fill=blue!50!white, draw=black] (6.5,3.5) -- ++(-120:0.4) -- ++(0:0.4)-- cycle;
      \filldraw[fill=blue!50!white, draw=black] (6.5,3.5)++(-120:0.4) -- ++(0:0.4) -- ++(-120:0.4)-- cycle;
      \filldraw[fill=blue!50!white, draw=black] (6.5,3.5) -- ++(0:0.4) -- ++(-120:0.4)-- cycle; 
      \filldraw[fill=blue!50!white, draw=black] (6.5,3.5)++(0:0.4) -- ++(-120:0.4) -- ++(0:0.4)-- cycle;
      \filldraw[fill=blue!50!white, draw=black] (6.5,3.5)++(-60:0.4)  -- ++(-120:0.4)-- ++(0:0.4)-- cycle;
      \filldraw[fill=blue!50!white, draw=black] (6.5,3.5)++(0:0.4)++(-120:0.4) -- ++(0:0.4) -- ++(-120:0.4)-- cycle;
      \filldraw[fill=blue!50!white, draw=black] (6.5,3.5)++(0:0.4) -- ++(0:0.4) -- ++(-120:0.4)-- cycle;
      \filldraw[fill=blue!50!white, draw=black] (6.5,3.5)++(0:0.4)++(-60:0.4)  -- ++(-120:0.4)-- ++(0:0.4)-- cycle;
      \filldraw[fill=blue!50!white, draw=black] (6.5,3.5)++(0:0.8) -- ++(-120:0.4) -- ++(0:0.4)-- cycle;
      \filldraw[fill=blue!50!white, draw=black] (6.5,3.5)++(0:0.8)++(-120:0.4) -- ++(0:0.4) -- ++(-120:0.4)-- cycle;

  
  \draw[color=black,thick] (1.5,4.1)-- ++(0:0.8)--++(-60:0.4);
  \draw[color=black,thick] (1.5,4.1)--++(-120:0.4) -- ++(-60:0.4) -- ++(0:0.8)-- ++(60:0.4);
  \draw[color=black,thick] (3.5,3.8) --++ (20:0.8)--++(-40:0.8)--++(20:0.8);
  \draw[color=black,thick] (3.5,3.8) --++ (-100:0.4)--++(20:2.0)--++(-40:0.4);
  \draw[color=black,thick] (6.5,3.5)--++(-120:0.4)--++(-60:0.4)--++(0:0.8)--++(60:0.4)--++(-120:0.4);
  \draw[color=black,thick] (6.5,3.5)--++(0:0.8)--++(-60:0.4);
  
  
  \end{tikzpicture}
\end{minipage}
\begin{minipage}{.3\textwidth}
    \begin{tikzpicture}[scale=0.5]
    \draw[color=black] (0,0) -- (8,0) --  (8,8) -- node[anchor=north]{\footnotesize $Q_i$}(0,8) -- (0,0);
      \filldraw[fill=blue!15!white, draw=black]{} (0,6)-- (8,6) --(8,2)-- (0,2)--cycle;
            \draw[color=blue!55!white] (0,2) -- node[anchor=north] {$H_i^{n,-}$} (8,2);
      \draw[color=blue!55!white] (0,6) --node[anchor=south] {$H_i^{n,+}$} (8,6);
      \draw[color=white] (0,4)--
      (8,4);
      \draw[color=orange](0,2) -- node[anchor=east] {$L_i$}(0,6);
      \draw[color=orange](8,2) --  (8,6);

  
      \filldraw[fill=blue!50!white, draw=black] (1.5,4.1) -- ++(-120:0.4) -- ++(0:0.4)-- cycle;
      \filldraw[fill=blue!50!white, draw=black] (1.5,4.1)++(-120:0.4) -- ++(0:0.4) -- ++(-120:0.4)-- cycle;
      \filldraw[fill=blue!50!white, draw=black] (1.5,4.1) -- ++(0:0.4) -- ++(-120:0.4)-- cycle; 
      \filldraw[fill=blue!50!white, draw=black] (1.5,4.1)++(0:0.4) -- ++(-120:0.4) -- ++(0:0.4)-- cycle;
      \filldraw[fill=blue!50!white, draw=black] (1.5,4.1)++(-60:0.4)  -- ++(-120:0.4)-- ++(0:0.4)-- cycle;
      \filldraw[fill=blue!50!white, draw=black] (1.5,4.1)++(0:0.4)++(-120:0.4) -- ++(0:0.4) -- ++(-120:0.4)-- cycle;
      \filldraw[fill=blue!50!white, draw=black] (1.5,4.1)++(0:0.4) -- ++(0:0.4) -- ++(-120:0.4)-- cycle;
      \filldraw[fill=blue!50!white, draw=black] (1.5,4.1)++(0:0.4)++(-60:0.4)  -- ++(-120:0.4)-- ++(0:0.4)-- cycle;
      \filldraw[fill=blue!50!white, draw=black] (1.5,4.1)++(0:0.8) -- ++(-120:0.4) -- ++(0:0.4)-- cycle;
      \filldraw[fill=blue!50!white, draw=black] (1.5,4.1)++(0:0.8)++(-120:0.4) -- ++(0:0.4) -- ++(-120:0.4)-- cycle;

      \filldraw[fill=blue!50!white, draw=black] (3.5,3.8) -- ++(-100:0.4) -- ++(20:0.4)-- cycle;
      \filldraw[fill=blue!50!white, draw=black] (3.5,3.8) -- ++(20:0.4) -- ++(-100:0.4)-- cycle; 
      \filldraw[fill=blue!50!white, draw=black] (3.5,3.8)++(20:0.4) -- ++(-100:0.4) -- ++(20:0.4)-- cycle;
      \filldraw[fill=blue!50!white, draw=black] (3.5,3.8)++(20:0.4) -- ++(20:0.4) -- ++(-100:0.4)-- cycle;
      \filldraw[fill=blue!50!white, draw=black] (3.5,3.8)++(20:0.8) -- ++(-100:0.4) -- ++(20:0.4)-- cycle;

      \filldraw[fill=blue!50!white, draw=black] (3.5,3.8)++(20:0.8)  ++(-100:0.4) ++(20:0.4) -- ++(20:0.4) -- ++(-100:0.4)-- cycle; 
      \filldraw[fill=blue!50!white, draw=black] (3.5,3.8)++(20:0.8)  ++(-100:0.4) ++(20:0.4)++(20:0.4) -- ++(-100:0.4) -- ++(20:0.4)-- cycle;
      \filldraw[fill=blue!50!white, draw=black] (3.5,3.8)++(20:0.8)  ++(-100:0.4) ++(20:0.4)++(20:0.4) -- ++(20:0.4) -- ++(-100:0.4)-- cycle;
      \filldraw[fill=blue!50!white, draw=black] (3.5,3.8)++(20:0.8)  ++(-100:0.4) ++(20:0.4)++(20:0.8) -- ++(-100:0.4) -- ++(20:0.4)-- cycle;

      \filldraw[fill=blue!50!white, draw=black] (6.5,3.5) -- ++(-120:0.4) -- ++(0:0.4)-- cycle;
      \filldraw[fill=blue!50!white, draw=black] (6.5,3.5)++(-120:0.4) -- ++(0:0.4) -- ++(-120:0.4)-- cycle;
      \filldraw[fill=blue!50!white, draw=black] (6.5,3.5) -- ++(0:0.4) -- ++(-120:0.4)-- cycle; 
      \filldraw[fill=blue!50!white, draw=black] (6.5,3.5)++(0:0.4) -- ++(-120:0.4) -- ++(0:0.4)-- cycle;
      \filldraw[fill=blue!50!white, draw=black] (6.5,3.5)++(-60:0.4)  -- ++(-120:0.4)-- ++(0:0.4)-- cycle;
      \filldraw[fill=blue!50!white, draw=black] (6.5,3.5)++(0:0.4)++(-120:0.4) -- ++(0:0.4) -- ++(-120:0.4)-- cycle;
      \filldraw[fill=blue!50!white, draw=black] (6.5,3.5)++(0:0.4) -- ++(0:0.4) -- ++(-120:0.4)-- cycle;
      \filldraw[fill=blue!50!white, draw=black] (6.5,3.5)++(0:0.4)++(-60:0.4)  -- ++(-120:0.4)-- ++(0:0.4)-- cycle;
      \filldraw[fill=blue!50!white, draw=black] (6.5,3.5)++(0:0.8) -- ++(-120:0.4) -- ++(0:0.4)-- cycle;
      \filldraw[fill=blue!50!white, draw=black] (6.5,3.5)++(0:0.8)++(-120:0.4) -- ++(0:0.4) -- ++(-120:0.4)-- cycle;

  
  \draw[color=black,thick] (1.5,4.1)-- ++(0:0.8)--++(-60:0.4);
  \draw[color=black,thick] (1.5,4.1)--++(-120:0.4) -- ++(-60:0.4) -- ++(0:0.8)-- ++(60:0.4);
  \draw[color=black,thick] (3.5,3.8) --++ (20:0.8)--++(-40:0.8)--++(20:0.8);
  \draw[color=black,thick] (3.5,3.8) --++ (-100:0.4)--++(20:2.0)--++(-40:0.4);
  \draw[color=black,thick] (6.5,3.5)--++(-120:0.4)--++(-60:0.4)--++(0:0.8)--++(60:0.4)--++(-120:0.4);
  \draw[color=black,thick] (6.5,3.5)--++(0:0.8)--++(-60:0.4);

  
  \draw[color=red,thick] (0.,4.) -- (1.5,4.1)+(-120:0.4);
  \draw[color=black,thick] (1.5,4.1)-- ++(0:0.8)--++(-60:0.4);
  \draw[color=red,thick] (1.5,4.1)++(-120:0.4) ++(60:0.4) ++(0:0.8)++(-60:0.4) -- (3.5,3.8);
  \draw[color=black,thick] (3.5,3.8) --++ (20:0.8)--++(-40:0.8)--++(20:0.8);
  \draw[color=red,thick] (3.5,3.8) ++ (20:0.8)++(-40:0.8)++(20:0.8) --
  (6.5,3.5)++(-120:0.4);
  \draw[color=black,thick] (6.5,3.5)--++(-120:0.4)--++(-60:0.4)--++(0:0.8)--++(60:0.4);
  \draw[color=red,thick] (6.5,3.5)++(-120:0.4)++(-60:0.4)++(0:0.8)++(60:0.4)--(8,3.15);
  \node[color=red!65!white, anchor=south] at (3.7,4) {$J_{v_n}$};

  \end{tikzpicture}
\end{minipage}
  \begin{minipage}{.3\textwidth}
  \begin{tikzpicture}[scale=0.5]
        \filldraw[fill=yellow!15!white, draw=white]{} (0,8)-- (8,8) --(8,0)-- (0,0)--cycle;
          \filldraw[fill=green!10!white, draw=white]{} 
        (0.,4.) -- (1.5,4.1) --++ (0:0.8)--++(-60:0.4)-- (3.5,3.8) --++ (20:0.8)--++(-40:0.8)--++(20:0.8) --(6.5,3.5) --++(-120:0.4)--++(-60:0.4)--++(0:0.8)--++(60:0.4)--(8,3.15)--(8,8)-- (0,8)--cycle;
      \draw[color=black] (0,0) -- (8,0) --  (8,8) -- node[anchor=north]{\footnotesize$Q_i$}(0,8) -- (0,0);
        \draw[color=orange](0,2) -- node[anchor=east] {$L_i$}(0,6);
        \draw[color=orange](8,2) -- node[anchor=west] {$L_i$} (8,6);

    \draw[color=red,thick] (0.,4.) -- (1.5,4.1)+(-120:0.4);
    \draw[color=red,thick] (1.5,4.1)-- ++(0:0.8)--++(-60:0.4);
    \draw[color=red,thick] (1.5,4.1)++(-120:0.4) ++(60:0.4) ++(0:0.8)++(-60:0.4) -- (3.5,3.8);
    \draw[color=red,thick] (3.5,3.8) --++ (20:0.8)--++(-40:0.8)--++(20:0.8);
    \draw[color=red,thick] (3.5,3.8) ++ (20:0.8)++(-40:0.8)++(20:0.8) --(6.5,3.5)++(-120:0.4);
    \draw[color=red,thick] (6.5,3.5)--++(-120:0.4)--++(-60:0.4)--++(0:0.8)--++(60:0.4)--(8,3.15);

     \node[color=black!65!white, anchor=south] at (3.5,5.5) {$ P_i^{n,+}$};
      \node[color=black!65!white, anchor=south] at (3.5,2) {$ P_i^{n,-}$};
     \node[color=red!65!white, anchor=south] at (3.5,4) {$ \Gamma_i^{n}$};
  \end{tikzpicture}
\end{minipage}
  \caption{Construction of the continuous curve $ \Gamma^n_i \subset R_{i}^{n}$ from $J_{v_n}$, \VIT for $Q_i \in\mathcal{B}_{\rm bad}$. \EEE \MMM In general, the curve $\Gamma_i^{n}$  is a subset  of $J_{v_n}$ that is almost contained in $K_n$, which in turn is the boundary of the \VIT `modified crack set' $\Omega_{n,k}^{\rm mod}$.
  \EEE Note that the figure is only a schematic illustration: in fact, the \MMM actual \EEE thickness of $L_i^n$ is much smaller. 
   } \label{Gamma-constr} 
\end{figure}
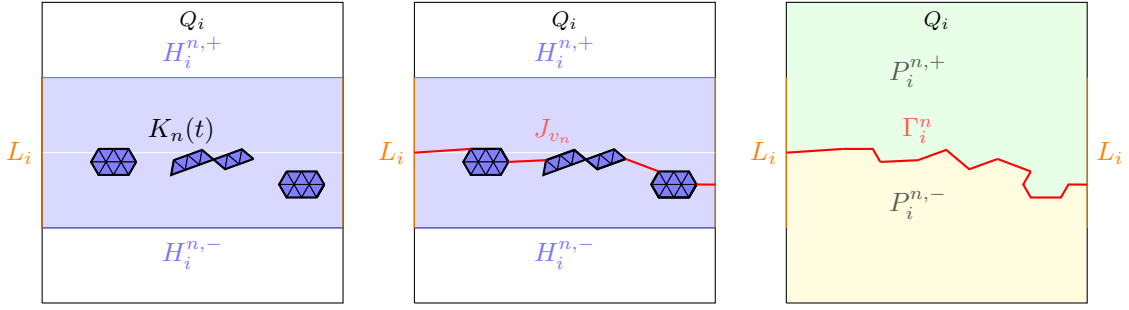

\subsubsection*{Good squares}

 In the good squares we employ a similar construction \EEE to guarantee that the jump sets of the approximating functions  are \EEE flat in 
 suitable \EEE squares of $\mathcal{B}_{\rm good}$ \EEE 
and that \EEE also $\nabla \phi_n$   lies in $W^{1,\infty}(\Omega' \setminus J_{\phi_n};\R^2)$. Such additional properties will be crucial in our following constructions.

First, we define  the collection of squares 
\begin{align}\label{bigsmall}
\begin{split}
\mathcal{B}_{\rm good}^{\rm large}:=\{Q_i \in \mathcal{B}_{\rm good} & \colon | \Omega_{n,k}^{\rm crack} \EEE \cap Q_i|  > \theta^{-1} r_i \eps_n\}, \quad \mathcal{B}_{\rm good}^{\rm small}:=\mathcal{B}_{\rm good} \sm \mathcal{B}_{\rm good}^{\rm large}, 
\end{split}
\end{align}
and the sets  
$$B_{\rm good}^{\rm large}:=\bigcup_{\mathcal{B}_{\rm good}^{\rm large}}Q_i, \quad \quad  B_{\rm good}^{\rm small}:=\bigcup_{\mathcal{B}_{\rm good}^{\rm small}}Q_i\,.
$$
\VVV (We notice that the sets defined above depend on $n$ and $k$, but we neglect this dependence in the notation.)
We state a technical lemma, whose \EEE proof is deferred to Section \ref{sec: lemma prof}. 

\begin{lemma}\label{stripe-prop}
  Let $\theta>0$,  let \EEE  $E\subset \Omega'$ be induced by a union of triangles $\mathbf{T}_E  \in \EEE \mathcal{T}_{\eps_n}  (\Omega') \EEE$  (cf.\ \eqref{EEE}), and $Q_i\in \mathcal{B}_{\rm good}$ with $| \VVV E \EEE \cap Q_i|\leq  \theta^{-1} r_i \eps_n \EEE $. 
  Then, there exists $\xi\in (-\theta r_i,\theta r_i) \EEE $ and a constant  $\hat{C}>0$ \EEE such that the set 
  \begin{equation*}\label{def: stripe}
    A^{n}_i(\xi; \VVV E \EEE) :=    \big\{ y \in  Q_i \colon    (y-x_i) \cdot \nu_\psi(x_i) \EEE  \in (\xi- 10^{8} \eps_n,\xi +  10^{8} \eps_n) \big\}
  \end{equation*} 
   fulfills 
  \begin{align}\label{property-lemma7.1}
    \# \big\{ T\in \VVV \mathbf{T}_E \EEE \colon T\cap A^{n}_i(\xi; \VVV E) \EEE \neq \emptyset \big\}\leq  \frac{\hat{C}}{\theta^2}\,. \EEE
  \end{align}
\end{lemma}
 For $Q_i \in \mathcal{B}_{\rm good}^{\rm small}$, we define
$$R^n_i = \lbrace y \in Q_i \colon   |(y-x) \cdot \nu_\psi(x)| \le \theta r_i \rbrace, \quad \quad  H^{n,\pm}_i = \lbrace y   \in Q_i \colon   (y-x) \cdot \nu_\psi(x) = \pm \theta r_i\rbrace. $$
(The sets are actually independent of $n$. Yet, $n$ is added to have the same notation as for bad squares.) \EEE  For any $Q_i=Q_{r_i}(x_i)\in \mathcal{B}^{\rm small}_{\rm good}$, let $A^{n}_i:=A^{n}_i( \MMM \xi_i^n; \EEE \Omega_{n,k}^{\rm crack})$ from Lemma~\ref{stripe-prop}, and for \EEE
\begin{equation}\label{1401251329}
  \VVV \Gamma_i^n := \lbrace y\in  Q_i \EEE \colon (y-x_i) \cdot \nu_\psi(x_i)=   \MMM  \xi_i^n \EEE \EEE \rbrace, \quad        P_i^{n,\pm}:=  \lbrace y\in Q_i\colon  \pm (y-x_i) \cdot \nu_\psi(x_i) > \xi_i^n  \rbrace \,, \EEE
\end{equation}
 we apply the reflection lemma \cite[Lemma~3.4]{steinke} to the (restriction of the) function $\psi \in W^{2,\infty}(\Omega' \setminus J_{\psi};\R^2)$ to  $Q_i \setminus R_i^n$, which has Sobolev regularity by \eqref{normal-angle-estimate2}.    \EEE The line $\{x_2=0\}$ in the lemma corresponds to each of the two lines $H^{n,+}_i$ and $H^{n,-}_i$.  By this, we  \EEE get $\phi_{i}^{n-}$, $\phi_{i}^{n,+}$ of class $W^{2,\infty}$ and set 
\begin{align}\label{phiningood}
\phi_n:=\phi_{i}^{n,-} \chi_{P_i^{n,-}}+ \phi_{i}^{n,+} \chi_{P_i^{n,+}},
\end{align}
  see also Figure~\ref{fig: precrack-constr} for an illustration.  
\JJJ By the above construction we get that \EEE
\begin{equation}\label{14012513295}
J_{\phi_n}\cap Q_i \EEE \subset \Gamma_i^n\MMM \cup L_i^n, \EEE  \quad  \MMM  \Vert  \nabla \phi_n \Vert_{L^\infty(\Omega' )} \le   C\Vert    \nabla \psi  \Vert_{L^\infty(\Omega' )}, \EEE \quad 
\|\phi_n\|_{W^{2,\infty}(Q_i\sm \Gamma_i^n)}\leq   C_{\psi,\theta} \EEE \|\psi\|_{W^{2,\infty}(Q_i\sm J_{\psi})}\,,
\end{equation} 
\MMM where $L_i^n$ denotes the lateral boundary  also for $Q_i \in \mathcal{B}^{\rm small}_{\rm good}$, i.e., $L^{n}_i:=\partial R^{n}_i \setminus (H^{n,+}_i \cup H^{n,-}_i)$, and \EEE
where for shortness  we denote by $C_{\psi,\theta}$ the constant \MMM on the right-hand side of \EEE  \eqref{eq: Lipschitz1before}. Following again the strategy of \cite[Section 5]{FriedrichSeutter} as done for \MMM squares \EEE in $\mathcal{B}_{\rm bad}$, we get that for a small constant $c_\theta$ only depending on $\theta$ it holds \EEE
\begin{align}\label{eq: Lipschitz2good}
    \JJJ {\rm (i)} \ \bigcup_{\mathcal{B}_{\rm good}^{\rm small}}  \Gamma_i^{n} \cup L_i^n  \, \tilde{\subset} \,  J_{\phi_n} \,, \EEE\quad \quad {\rm (ii)}\
 \mathcal{H}^1\big( \big\{ x \in J_{\phi_n} \cap B^{\rm small}_{\rm good} \colon|[\phi_n](x)| \le \BBB c_\theta \EEE  \big\} \big)    \le C\theta.
\end{align}


 \subsubsection*{Definition of $\phi_n$} \EEE
  Eventually, we define the function $\phi_n$ on $B_{\rm bad}$ and $B^{\rm small}_{\rm good}$ as discussed above, and let  \EEE
\begin{equation}\label{1501251236}
\phi_n=\psi \text{ on } \Omega'\sm (B_{\rm bad} \cup B^{\rm small}_{\rm good}).
\end{equation}
 Let us collect some of the main properties that we will use in the following.  \MMM First, \JJJ by the above construction it also holds for all $Q_i\in \mathcal{B}_{\rm bad}\cup\mathcal{B}_{\rm good}^{\rm small}$ that 
\begin{equation}\label{outside-R-i}
  \phi_n=\psi \quad \text{in $Q_i\sm R^{n}_i$} \,.   
\end{equation} \EEE
In view of  \eqref{eq: Lipschitz1before} and \eqref{14012513295}, there is  $C_\psi>0$ depending on $\psi$ and $C_{\psi,\theta}$ additionally depending on $\theta$ such that
\begin{align}\label{eq: Lipschitz1}
\Vert  \nabla \phi_n \Vert_{L^\infty(\Omega' \setminus J_{\phi_n} )} \le C_\psi, \quad \quad \quad   \Vert   \nabla^2 \phi_n \Vert_{L^\infty(\Omega' \setminus J_{\phi_n} )} \le   C_{\psi,\theta}.
\end{align}
 By \eqref{eq: Lipschitz2alt} and \eqref{eq: Lipschitz2good}, \EEE for a sufficiently small constant $c_\theta>0$ depending on $\theta$,  it holds that 
\begin{align}\label{eq: Lipschitz2}
\JJJ {\rm (i)} \bigcup_{\mathcal{B}_{\rm bad} \cup \mathcal{B}^{\rm small}_{\rm good}} \Gamma_i^n \cup L_i^n  \, \tilde{\subset} \,  J_{\phi_n}\,, \EEE \quad {\rm (ii)} \
\mathcal{H}^1\big( \big\{ x \in J_{\phi_n} \cap  (B_{\rm bad} \cup B^{\rm small}_{\rm good}) \EEE \colon|[\phi_n](x)| \le \BBB c_\theta \EEE  \big\} \big)    \le C\theta.
\end{align} 
 Moreover, by construction and recalling \eqref{franfort again},  \eqref{14012513295} we find  \EEE
   \begin{equation}\label{1401251338}
 J_{\phi_n} \cap {\rm int}(Q_i) \, \tilde{ \subset} \,  \Gamma^n_i, \EEE \quad J_{\phi_n} \cap \partial Q_i \subset L_i^n \quad \text{for any } Q_i \in \mathcal{B}_{\rm bad}  \cup \mathcal{B}^{\rm small}_{\rm good}
\end{equation}
 such that by \eqref{normal-angle-estimate0} it holds \EEE 
\begin{align}\label{crack  for later}
 \mathcal{H}^1( \Gamma_i^n \cup    L^{n}_i    ) \le   (1+2\theta) \EEE \mathcal{H}^1(J_{\psi} \cap Q_i) \quad \text{for all $Q_i\in \mathcal{B}^{\rm small}_{\rm good}$ }.
\end{align}

\EEE

\subsection{Definition of $\psi_n$ and $\mathbf{T}_n(\psi_n)$, and proof of \eqref{stab0neu}}
We  now  proceed with the definition of the sequence of displacements $(\psi_n)_n$ and the corresponding triangulations $\mathbf{T}_{n}(\psi_n)$. For $k= k(t)$ such that $t\in [t_n^k,t_n^{k+1})$, we consider   $u_n(t)=u_n^k\in \mathcal{A}_n^k$ and the corresponding triangulation $\mathbf{T}_n(u_n^k)$. In the following, we will use the identity  $\mathbf{T}_{n,k-1}^{\rm crack}(u_n^k) = \mathbf{T}_{n,k}^{\rm crack}$.  Recalling the partition of $\Omega'$ into $B_{\rm good}$, $B_{\rm bad}$, and $B_{\rm rest}$, we let  
$$\mathbf{T}^{\rm bad}_{n,k} := \big\{ T \in \mathbf{T}_n(u_n^k)\colon  T \cap B_{\rm bad} \neq \emptyset \big\}, \quad \quad  \mathbf{T}^{\rm rest}_{n,k} := \big\{ T \in \mathbf{T}_n(u_n^k)\colon  T \cap (B_{\rm good} \cup B_{\rm bad}) = \emptyset \big\},$$
and define the set $B^n_{\rm good} := \Omega' \setminus  \bigcup_{T \in \mathbf{T}^{\rm bad}_{n,k} \cup  \mathbf{T}^{\rm rest}_{n,k}} T$.  Note that \EEE $ B^n_{\rm good} \cap B_{\rm bad}=\emptyset$ \JJJ and for $n$ small enough we have $B_{\rm good}\subset B_{\rm good}^n$\EEE. 
We will define the triangulation $\mathbf{T}_{n}(\psi_n)$ as 
\begin{align}\label{eq: splitti}
\mathbf{T}_{n}(\psi_n) = \mathbf{T}^{\rm bad}_{n,k} \cup \mathbf{T}^{\rm rest}_{n,k} \cup \mathbf{T}^{\rm good}_{n,k}, 
\end{align}
where $\mathbf{T}^{\rm good}_{n,k}$ denotes a suitable triangulation of $ B^{n}_{\rm good}$, which will be specified below. On every $T \in \mathbf{T}^{\rm  rest}_{n,k}$ we choose $\psi_n$ as the affine  interpolation of $\psi$.

The main steps consist now in (a) defining $\psi_n$ on the triangles of $\mathbf{T}^{\rm bad }_{n,k}$  and in (b) defining  $\mathbf{T}^{\rm good}_{n,k}$ and the interpolation $\psi_n$ on $\mathbf{T}^{\rm good}_{n,k}$. \EEE  


\subsubsection*{(a) Bad squares}
We   define the sequence of displacements $(\psi_n)_n$ on triangles in $\mathbf{T}^{\rm bad}_{n,k}$. We fix  $Q_i\in \mathcal{B}_{\rm bad}$, and define  the set of vertices 
\begin{equation}
\mathcal{V}_{i,n}\defas \big\{x\in \mathcal{V}(T) \colon\; T \in \mathbf{T}^{\rm bad}_{n,k}, \  T \cap Q_i \neq \emptyset   \big\}\,,
\end{equation} 
where we denote the three vertices of the triangle $T$ by $\mathcal{V}(T)$. \EEE We  specify the value $\psi_n(x)$ on each $x\in \mathcal{V}_{i,n}$ in order to compute the piecewise affine interpolation on each triangle  formed by the vertices $\mathcal{V}_{i,n}$.  Recall that inside the closed set $Q_i$ the jump $J_{\phi_n}$ is included in the curve $\VVV \Gamma^n_i \EEE$ and the two lines $L^n_i$,  see \eqref{franfort again}. For any vertex $x\in \mathcal{V}_{i,n}$ with $ x\notin \Gamma_i^{n} \cup L^n_i \EEE$,   
we simply set $\psi_n(x)\defas \phi_n(x)$.  Instead, if $x\in \Gamma_i^{n} \cup L^n_i\EEE$, we proceed as follows. Since $\phi_n\in W^{1,\infty}(\Omega\setminus J_{\phi_n};  \R^2)$, we know that \VVV $\phi^{n, +}_{i}(x)$ and $\phi^{n,-}_{i}(x)$ \EEE are well defined \VVV as boundary traces
\EEE for each $x\in \Gamma_i^{n}$, \VVV and as boundary traces also on $P_i^{n,\pm}\cap L^n_i$, respectively,  cf.\ \eqref{1301252256}. \EEE     For $x \in L^n_i \cap P_i^{n,\pm}$, we set $\psi_n(x)\defas \phi_{i}^{n,\pm}(x)$. \EEE  For 
$x\in \Gamma_i^{n}\setminus K_n(t)$, we choose one of the values, e.g., $\psi_n(x)\defas \VVV \phi^{n,+}_{i}(x)$. For $x\in \Gamma_i^{n}\cap K_n(t)$,  there exists at least one triangle $T$ with $x\in T$ and $|T \cap  \Omega_{n, k}^{\rm mod} \EEE |=0 $ because $K_n(t) = \partial \Omega_{n,k}^{\rm mod}$. We collect all such triangles in the set 
\begin{equation}\label{below}
 \mathbf{T}_{\rm out}^x  \defas \big\{T\in \mathbf{T}^x \, \colon|T \cap \Omega_{n,k}^{\rm mod}|=0  \big\},\quad \text{where}\quad \mathbf{T}^x\defas \big\{T\in \mathbf{T}^{\rm bad}_{n,k}\colon \EEE  x\in T\big\} \,.
\end{equation} 
Moreover, we let $N^x_{\rm out} \EEE \defas \bigcup_{T\in \mathbf{T}^x_{\rm out}} T$ be the neighborhood of $x$ consisting of these surrounding triangles. 
If $N^x_{\rm out} \, \tilde{\subset} \,  P_i^{n,+}$, we set $\psi_n(x)\defas \VVV \phi^{n,+}_{i}(x)$ (\JJJ see the point $x_1$ in Figure \ref{N-x-out-figure}\EEE). \EEE If $N^x_{\rm out} \, \tilde{\subset} \,  P_i^{n,-}$, we set $\psi_n(x)\defas \VVV \phi^{n-}_{i}(x)$ (\JJJ see the point $x_4$ in Figure \ref{N-x-out-figure}). \EEE  Otherwise, we again choose an arbitrary value, e.g.,  $\psi_n(x):=\VVV \phi_{i}^{n,+}(x)$ (\JJJ see e.g. the point $x_3$ in Figure \ref{N-x-out-figure}\EEE). For later purposes, we also define
\begin{align}\label{NNN}
N_{\rm out}(T) = \bigcup_{x \in \mathcal{V}(T)} N_{\rm out}^x.
\end{align}
Now, we can define $\psi_n$ on the union of the triangles    in $\mathbf{T}^{\rm bad}_{n,k}$ (and thus in particular on $B_{\rm bad}$)     
by taking the piecewise affine interpolation on each triangle $T\in \mathbf{T}^{\rm bad}_{n,k}$. 

\begin{figure}
  \begin{minipage}{0.492\textwidth}
    \begin{tikzpicture}[scale=0.88]
      \filldraw[fill=blue!15!white, draw=white]{} (0,8)-- (8,8) --(8,0)-- (0,0)--cycle;
  \filldraw[fill=green!10!white, draw=white]{} 
      (0.,4.) -- (1.5,4.1) --++ (0:0.8)--++(-60:0.4)-- (3.5,3.8) --++ (20:0.8)--++(-40:0.8)--++(20:0.8) --(6.5,3.5) --++(-120:0.4)--++(-60:0.4)--++(0:0.8)--++(60:0.4)--(8,3.15)--(8,8)-- (0,8)--cycle;
    \draw[color=black] (0,0) -- (8,0) --  (8,8) -- node[anchor=north]{$Q_i$}(0,8) -- (0,0);
      \draw[color=orange](0,2) -- node[anchor=east] {$L_i$}(0,6);
      \draw[color=orange](8,2) -- (8,6);

  
      \filldraw[fill=blue!50!white, draw=black] (1.5,4.1) -- ++(-120:0.4) -- ++(0:0.4)-- cycle;
      \filldraw[fill=blue!50!white, draw=black] (1.5,4.1)++(-120:0.4) -- ++(0:0.4) -- ++(-120:0.4)-- cycle;
      \filldraw[fill=blue!50!white, draw=black] (1.5,4.1) -- ++(0:0.4) -- ++(-120:0.4)-- cycle; 
      \filldraw[fill=blue!50!white, draw=black] (1.5,4.1)++(0:0.4) -- ++(-120:0.4) -- ++(0:0.4)-- cycle;
      \filldraw[fill=blue!50!white, draw=black] (1.5,4.1)++(-60:0.4)  -- ++(-120:0.4)-- ++(0:0.4)-- cycle;
      \filldraw[fill=blue!50!white, draw=black] (1.5,4.1)++(0:0.4)++(-120:0.4) -- ++(0:0.4) -- ++(-120:0.4)-- cycle;
      \filldraw[fill=blue!50!white, draw=black] (1.5,4.1)++(0:0.4) -- ++(0:0.4) -- ++(-120:0.4)-- cycle;
      \filldraw[fill=blue!50!white, draw=black] (1.5,4.1)++(0:0.4)++(-60:0.4)  -- ++(-120:0.4)-- ++(0:0.4)-- cycle;
      \filldraw[fill=blue!50!white, draw=black] (1.5,4.1)++(0:0.8) -- ++(-120:0.4) -- ++(0:0.4)-- cycle;
      \filldraw[fill=blue!50!white, draw=black] (1.5,4.1)++(0:0.8)++(-120:0.4) -- ++(0:0.4) -- ++(-120:0.4)-- cycle;

      \filldraw[fill=blue!50!white, draw=black] (3.5,3.8) -- ++(-100:0.4) -- ++(20:0.4)-- cycle;
      \filldraw[fill=blue!50!white, draw=black] (3.5,3.8) -- ++(20:0.4) -- ++(-100:0.4)-- cycle; 
      \filldraw[fill=blue!50!white, draw=black] (3.5,3.8)++(20:0.4) -- ++(-100:0.4) -- ++(20:0.4)-- cycle;
      \filldraw[fill=blue!50!white, draw=black] (3.5,3.8)++(20:0.4) -- ++(20:0.4) -- ++(-100:0.4)-- cycle;
      \filldraw[fill=blue!50!white, draw=black] (3.5,3.8)++(20:0.8) -- ++(-100:0.4) -- ++(20:0.4)-- cycle;

      \filldraw[fill=blue!50!white, draw=black] (3.5,3.8)++(20:0.8)  ++(-100:0.4) ++(20:0.4) -- ++(20:0.4) -- ++(-100:0.4)-- cycle; 
      \filldraw[fill=blue!50!white, draw=black] (3.5,3.8)++(20:0.8)  ++(-100:0.4) ++(20:0.4)++(20:0.4) -- ++(-100:0.4) -- ++(20:0.4)-- cycle;
      \filldraw[fill=blue!50!white, draw=black] (3.5,3.8)++(20:0.8)  ++(-100:0.4) ++(20:0.4)++(20:0.4) -- ++(20:0.4) -- ++(-100:0.4)-- cycle;
      \filldraw[fill=blue!50!white, draw=black] (3.5,3.8)++(20:0.8)  ++(-100:0.4) ++(20:0.4)++(20:0.8) -- ++(-100:0.4) -- ++(20:0.4)-- cycle;

      \filldraw[fill=blue!50!white, draw=black] (6.5,3.5) -- ++(-120:0.4) -- ++(0:0.4)-- cycle;
      \filldraw[fill=blue!50!white, draw=black] (6.5,3.5)++(-120:0.4) -- ++(0:0.4) -- ++(-120:0.4)-- cycle;
      \filldraw[fill=blue!50!white, draw=black] (6.5,3.5) -- ++(0:0.4) -- ++(-120:0.4)-- cycle; 
      \filldraw[fill=blue!50!white, draw=black] (6.5,3.5)++(0:0.4) -- ++(-120:0.4) -- ++(0:0.4)-- cycle;
      \filldraw[fill=blue!50!white, draw=black] (6.5,3.5)++(-60:0.4)  -- ++(-120:0.4)-- ++(0:0.4)-- cycle;
      \filldraw[fill=blue!50!white, draw=black] (6.5,3.5)++(0:0.4)++(-120:0.4) -- ++(0:0.4) -- ++(-120:0.4)-- cycle;
      \filldraw[fill=blue!50!white, draw=black] (6.5,3.5)++(0:0.4) -- ++(0:0.4) -- ++(-120:0.4)-- cycle;
      \filldraw[fill=blue!50!white, draw=black] (6.5,3.5)++(0:0.4)++(-60:0.4)  -- ++(-120:0.4)-- ++(0:0.4)-- cycle;
      \filldraw[fill=blue!50!white, draw=black] (6.5,3.5)++(0:0.8) -- ++(-120:0.4) -- ++(0:0.4)-- cycle;
      \filldraw[fill=blue!50!white, draw=black] (6.5,3.5)++(0:0.8)++(-120:0.4) -- ++(0:0.4) -- ++(-120:0.4)-- cycle;

  
  \draw[color=red,thick] (0.,4.) -- (1.5,4.1)+(-120:0.4);
  \draw[color=red,thick] (1.5,4.1)-- ++(0:0.8)--++(-60:0.4);
  \draw[color=red,thick] (1.5,4.1)++(-120:0.4) ++(60:0.4) ++(0:0.8)++(-60:0.4) -- (3.5,3.8);
  \draw[color=red,thick] (3.5,3.8) --++ (20:0.8)--++(-40:0.8)--++(20:0.8);
  \draw[color=red,thick] (3.5,3.8) ++ (20:0.8)++(-40:0.8)++(20:0.8) --
  (6.5,3.5)++(-120:0.4);
  \draw[color=red,thick] (6.5,3.5)--++(-120:0.4)--++(-60:0.4)--++(0:0.8)--++(60:0.4)--(8,3.15);
  
      \filldraw[black] (1.5,4.1)++(0:0.4) circle (1pt) node[anchor=south]{\footnotesize $x_1$};
      \filldraw[black] (3.5,3.8) circle (1pt) node[anchor=south]{\footnotesize $x_2$};
      \filldraw[black] (3.5,3.8)++(20:0.8) ++(-40:0.4) circle (1pt) node[anchor=south]{\footnotesize $x_3$};
       \filldraw[black] (6.5,3.5) ++(-120:0.8)++(0:0.4) circle (1pt) node[anchor=north]{\footnotesize $x_4$};

   \node[color=black!65!white, anchor=south] at (3.5,6) {$ P_i^{n,+}$};
    \node[color=black!65!white, anchor=south] at (3.5,2) {$ P_i^{n,-}$};
    \node[color=red, anchor=south] at (6,4) {$ \Gamma_i^{n}$};
  \end{tikzpicture}
  \end{minipage}
  \begin{minipage}{0.5\textwidth}
    \begin{tikzpicture}[scale=0.88]
      \filldraw[fill=blue!15!white, draw=white]{} (0,8)-- (8,8) --(8,0)-- (0,0)--cycle;
  \filldraw[fill=green!10!white, draw=white]{} 
      (0.,4.) -- (1.5,4.1) --++ (0:0.8)--++(-60:0.4)-- (3.5,3.8) --++ (20:0.8)--++(-40:0.8)--++(20:0.8) --(6.5,3.5) --++(-120:0.4)--++(-60:0.4)--++(0:0.8)--++(60:0.4)--(8,3.15)--(8,8)-- (0,8)--cycle;
    \draw[color=black] (0,0) -- (8,0) --  (8,8) -- node[anchor=north]{$Q_i$}(0,8) -- (0,0);
      \draw[color=orange](0,2) -- node[anchor=east] {$L_i$}(0,6);
      \draw[color=orange](8,2) -- node[anchor=west] {$L_i$} (8,6);

  
      \filldraw[fill=blue!50!white, draw=black] (1.5,4.1) -- ++(-120:0.4) -- ++(0:0.4)-- cycle;
      \filldraw[fill=blue!50!white, draw=black] (1.5,4.1)++(-120:0.4) -- ++(0:0.4) -- ++(-120:0.4)-- cycle;
      \filldraw[fill=blue!50!white, draw=black] (1.5,4.1) -- ++(0:0.4) -- ++(-120:0.4)-- cycle; 
      \filldraw[fill=blue!50!white, draw=black] (1.5,4.1)++(0:0.4) -- ++(-120:0.4) -- ++(0:0.4)-- cycle;
      \filldraw[fill=blue!50!white, draw=black] (1.5,4.1)++(-60:0.4)  -- ++(-120:0.4)-- ++(0:0.4)-- cycle;
      \filldraw[fill=blue!50!white, draw=black] (1.5,4.1)++(0:0.4)++(-120:0.4) -- ++(0:0.4) -- ++(-120:0.4)-- cycle;
      \filldraw[fill=blue!50!white, draw=black] (1.5,4.1)++(0:0.4) -- ++(0:0.4) -- ++(-120:0.4)-- cycle;
      \filldraw[fill=blue!50!white, draw=black] (1.5,4.1)++(0:0.4)++(-60:0.4)  -- ++(-120:0.4)-- ++(0:0.4)-- cycle;
      \filldraw[fill=blue!50!white, draw=black] (1.5,4.1)++(0:0.8) -- ++(-120:0.4) -- ++(0:0.4)-- cycle;
      \filldraw[fill=blue!50!white, draw=black] (1.5,4.1)++(0:0.8)++(-120:0.4) -- ++(0:0.4) -- ++(-120:0.4)-- cycle;

      \filldraw[fill=blue!50!white, draw=black] (3.5,3.8) -- ++(-100:0.4) -- ++(20:0.4)-- cycle;
      \filldraw[fill=blue!50!white, draw=black] (3.5,3.8) -- ++(20:0.4) -- ++(-100:0.4)-- cycle; 
      \filldraw[fill=blue!50!white, draw=black] (3.5,3.8)++(20:0.4) -- ++(-100:0.4) -- ++(20:0.4)-- cycle;
      \filldraw[fill=blue!50!white, draw=black] (3.5,3.8)++(20:0.4) -- ++(20:0.4) -- ++(-100:0.4)-- cycle;
      \filldraw[fill=blue!50!white, draw=black] (3.5,3.8)++(20:0.8) -- ++(-100:0.4) -- ++(20:0.4)-- cycle;

      \filldraw[fill=blue!50!white, draw=black] (3.5,3.8)++(20:0.8)  ++(-100:0.4) ++(20:0.4) -- ++(20:0.4) -- ++(-100:0.4)-- cycle; 
      \filldraw[fill=blue!50!white, draw=black] (3.5,3.8)++(20:0.8)  ++(-100:0.4) ++(20:0.4)++(20:0.4) -- ++(-100:0.4) -- ++(20:0.4)-- cycle;
      \filldraw[fill=blue!50!white, draw=black] (3.5,3.8)++(20:0.8)  ++(-100:0.4) ++(20:0.4)++(20:0.4) -- ++(20:0.4) -- ++(-100:0.4)-- cycle;
      \filldraw[fill=blue!50!white, draw=black] (3.5,3.8)++(20:0.8)  ++(-100:0.4) ++(20:0.4)++(20:0.8) -- ++(-100:0.4) -- ++(20:0.4)-- cycle;

      \filldraw[fill=blue!50!white, draw=black] (6.5,3.5) -- ++(-120:0.4) -- ++(0:0.4)-- cycle;
      \filldraw[fill=blue!50!white, draw=black] (6.5,3.5)++(-120:0.4) -- ++(0:0.4) -- ++(-120:0.4)-- cycle;
      \filldraw[fill=blue!50!white, draw=black] (6.5,3.5) -- ++(0:0.4) -- ++(-120:0.4)-- cycle; 
      \filldraw[fill=blue!50!white, draw=black] (6.5,3.5)++(0:0.4) -- ++(-120:0.4) -- ++(0:0.4)-- cycle;
      \filldraw[fill=blue!50!white, draw=black] (6.5,3.5)++(-60:0.4)  -- ++(-120:0.4)-- ++(0:0.4)-- cycle;
      \filldraw[fill=blue!50!white, draw=black] (6.5,3.5)++(0:0.4)++(-120:0.4) -- ++(0:0.4) -- ++(-120:0.4)-- cycle;
      \filldraw[fill=blue!50!white, draw=black] (6.5,3.5)++(0:0.4) -- ++(0:0.4) -- ++(-120:0.4)-- cycle;
      \filldraw[fill=blue!50!white, draw=black] (6.5,3.5)++(0:0.4)++(-60:0.4)  -- ++(-120:0.4)-- ++(0:0.4)-- cycle;
      \filldraw[fill=blue!50!white, draw=black] (6.5,3.5)++(0:0.8) -- ++(-120:0.4) -- ++(0:0.4)-- cycle;
      \filldraw[fill=blue!50!white, draw=black] (6.5,3.5)++(0:0.8)++(-120:0.4) -- ++(0:0.4) -- ++(-120:0.4)-- cycle;
  
  
  \filldraw[fill=magenta!20!white, draw=magenta] (1.5,4.1) -- ++(60:0.4) -- ++(-60:0.4)-- cycle;
     \filldraw[fill=magenta!20!white, draw=magenta] (1.5,4.1)++(60:0.4) -- ++(0:0.4) -- ++(-120:0.4)-- cycle;
      \filldraw[fill=magenta!20!white, draw=magenta] (1.5,4.1)++(0:0.4) -- ++(0:0.4) -- ++(120:0.4)-- cycle; 
      \node[anchor=north, color=magenta] at (1.5,5.1) {$N^{x_1}_{\rm out}$};

      \filldraw[fill=magenta!20!white, draw=magenta] (3.5,3.8) -- ++(-100:0.4) -- ++(140:0.4)-- cycle;
      \filldraw[fill=magenta!20!white, draw=magenta] (3.5,3.8) -- ++(-160:0.4) -- ++(80:0.4)-- cycle;
      \filldraw[fill=magenta!20!white, draw=magenta] (3.5,3.8) -- ++(140:0.4) -- ++(20:0.4)-- cycle;
      \filldraw[fill=magenta!20!white, draw=magenta] (3.5,3.8) -- ++(20:0.4) -- ++(140:0.4)-- cycle;
      
    \node[anchor=north, color=magenta] at (3.5,4.8) {$N^{x_2}_{\rm out}$};


      \filldraw[fill=magenta!20!white, draw=magenta] (3.5,3.8)++(20:0.8) -- ++(20:0.4) -- ++(-100:0.4)-- cycle;
     \filldraw[fill=magenta!20!white, draw=magenta] (3.5,3.8)++(20:1.2) -- ++(-100:0.4) -- ++(20:0.4)-- cycle;
      \filldraw[fill=magenta!20!white, draw=magenta] (3.5,3.8)++(20:0.8)  ++(-100:0.4) --++(-40:0.4) -- ++(80:0.4) -- cycle; 
      \filldraw[fill=magenta!20!white, draw=magenta] (3.5,3.8)++(20:0.8)  ++(-40:0.4) --++(-40:0.4) -- ++(-160:0.4) -- cycle; 
      \node[anchor=north, color=magenta] at (4.5,3.5) { $N^{x_3}_{\rm out}$};
  
      \filldraw[fill=magenta!20!white, draw=magenta] (6.5,3.5) ++(-120:0.4) -- ++(-120:0.4) -- ++(0:0.4)-- cycle;
      \filldraw[fill=magenta!20!white, draw=magenta] (6.5,3.5) ++(-120:0.8) -- ++(-60:0.4) -- ++(60:0.4)-- cycle;
      \filldraw[fill=magenta!20!white, draw=magenta] (6.5,3.5) ++(-120:0.8)++(0:0.4) -- ++(-120:0.4) -- ++(0:0.4)-- cycle;
      \filldraw[fill=magenta!20!white, draw=magenta] (6.5,3.5) ++(-120:0.8)++(0:0.4) -- ++(-60:0.4) -- ++(60:0.4)-- cycle;
  
      \node[anchor=north, color=magenta] at (6.5,2.5) { $N^{x_4}_{\rm out}$};
  
  
  \draw[color=red,thick] (0.,4.) -- (1.5,4.1)+(-120:0.4);
  \draw[color=red,thick] (1.5,4.1)-- ++(0:0.8)--++(-60:0.4);
  \draw[color=red,thick] (1.5,4.1)++(-120:0.4) ++(60:0.4) ++(0:0.8)++(-60:0.4) -- (3.5,3.8);
  \draw[color=red,thick] (3.5,3.8) --++ (20:0.8)--++(-40:0.8)--++(20:0.8);
  \draw[color=red,thick] (3.5,3.8) ++ (20:0.8)++(-40:0.8)++(20:0.8) --
  (6.5,3.5)++(-120:0.4);
  \draw[color=red,thick] (6.5,3.5)--++(-120:0.4)--++(-60:0.4)--++(0:0.8)--++(60:0.4)--(8,3.15);
  
      \filldraw[black] (1.5,4.1)++(0:0.4) circle (1pt) node[anchor=south]{\footnotesize $x_1$};
      \filldraw[black] (3.5,3.8) circle (1pt) node[anchor=south]{\footnotesize $x_2$};
      \filldraw[black] (3.5,3.8)++(20:0.8) ++(-40:0.4) circle (1pt) node[anchor=south]{\footnotesize $x_3$};
       \filldraw[black] (6.5,3.5) ++(-120:0.8)++(0:0.4) circle (1pt) node[anchor=north]{\footnotesize $x_4$};

   \node[color=black!65!white, anchor=south] at (3.5,6) {$ P_i^{n,+}$};
    \node[color=black!65!white, anchor=south] at (3.5,2) {$ P_i^{n,-}$};
   \node[color=red, anchor=south] at (6,4) {$ \Gamma_i^{n}$};
  \end{tikzpicture}
  \end{minipage}
  \caption{\JJJ Construction of $\psi_n$ depending on the position of $N^x_{\rm out}$ relative to $\Gamma_i^{n}$, \VIT for $Q_i \in\mathcal{B}_{\rm bad}$. \JJJ In this case, we have e.g. $\psi_n(x_1)=\phi_{i}^{n,+}(x_1)$ and $\psi_n(x_4)=\phi_{i}^{n,-}(x_4)$. For \JOS $x_2$ and $x_3$ \EEE we choose one of the two values. \EEE
  } \label{N-x-out-figure} 
  \end{figure}
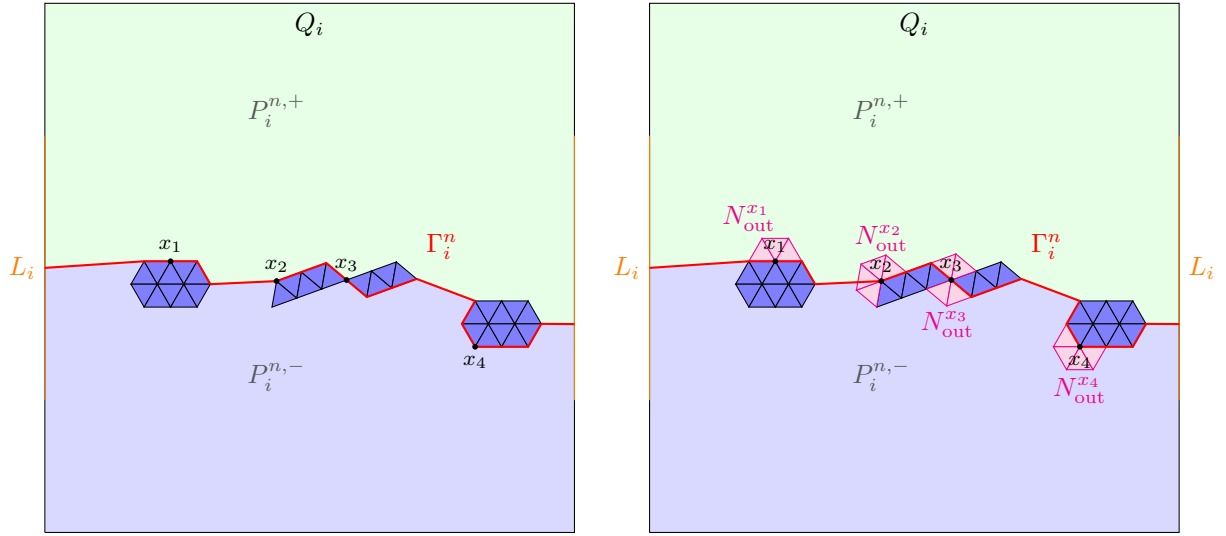

\subsubsection*{(b) Good squares}

We come to the definition of $\mathbf{T}^{\rm good}_{n,k}$ and the corresponding interpolation $\psi_n$ on $B_{\rm good}^n$.  \EEE 
As the argument is local, without restriction we suppose that   $\mathcal{B} _{\rm good}$   consists of a single square $Q_i$.   In the remaining part of this section, we will therefore drop the index $i$ for simplicity,  and write $Q$,  $r$,   as well as $A^n$ (see before \eqref{1401251329}).  \EEE

For $Q\in \mathcal{B}_{\rm good}^{\rm large}$, we set $\mathbf{T}^{\rm good}_{n,k} = \lbrace T \in  \mathbf{T}_n(u_n^k) \colon \, T \cap Q \neq \emptyset\rbrace $, and \EEE  we define $\psi_n$ as the piecewise affine interpolation of $\psi$ on each $T\in \mathbf{T}^{\rm good}_{n,k}$.
 
Let us now come to the construction in the second case $Q\in \mathcal{B}_{\rm good}^{\rm small}$, i.e., $|\Omega_{n,k}^{\rm crack}\cap Q|\leq  \theta^{-1} r \eps_n$.  
As the underlying triangulation in $Q\setminus \VVV A^{n}$  we will use $\mathbf{T}_n(u_n^k)$, whereas \VIT inside \EEE \VVV $A^{n}$ \EEE
 we  use the explicit construction of an optimal triangulation from \cite{ChamboDalMaso}. However, in order to ensure  that  the triangulation $\mathbf{T}_n(\psi_n)$ is admissible, it \EEE has to fulfill the condition $ \mathbf{T}^{\rm crack}_{n,k-1}\subset \mathbf{T}_n(\psi_n)$,  see \eqref{eq: admis tria}. \EEE Therefore, we need to treat the parts of \VVV $A^{n}$ \EEE that are contained in the \VIT `pre-crack' \EEE $ \Omega^{\rm crack}_{n,k-1}$ differently, which leads to a more involved construction, see Figure \ref{fig: precrack-constr}.

We start by considering neighborhoods of $\JJJ \Omega^{\rm crack}_{n,k}$ in $Q$, by adding a buffer zone of order \MMM $10^7\eps_n'$ with \EEE $\eps'_{n}= 2 \eps_n \cos(\theta_0)$ around each \VIT `broken \EEE triangle'. 
  More precisely, we define the \VIT `enlarged \EEE crack sets' $N^{\rm crack}_{n,k}$ and the corresponding collection of triangles as  
\begin{equation}\label{enlarged-crackinA}
  \mathbf{T}^{\rm neigh}_{n,k}:=\big\{T\in \mathbf{T}_n(u_n^{k})\colon \dist\big(T,  ( \Omega_{n,k}^{\rm crack} \cap Q) \EEE \big)\leq    10^{7}  \eps'_n\big\}\quad \text{and}\quad N^{\rm crack}_{n,k} \defas \bigcup_{T\in \mathbf{T}^{\rm neigh}_{n,k}}T  \,.
\end{equation} 
By the definition of $\Omega_{n,k}^{\rm crack}$, we know that outside of $N^{\rm crack}_{n,k}$ the triangles in $\mathbf{T}_n(u_n^{k})$ are part of a regular   `background triangulation' $\mathbf{Z}_n$ \EEE with size $\eps'_{n}$. By the construction \JJJ of $\Gamma_i^{n}$ in $B_{\rm good}^{\rm small}$ \EEE and due to Lemma~\ref{stripe-prop}, the set $(J_{\phi_n}\setminus N^{\rm crack}_{n,k})\cap A^n$ \EEE consists of at most $M:=\hat{C}  \theta^{-2} \EEE +1$ \EEE straight segments $(S^n_k)_{k=1}^{M}$. \EEE 
According to \cite[Proof of \MMM Proposition \EEE 4.1]{ChamboDalMaso} (see in particular \cite[(4.9) and Figure 4.10\EEE]{ChamboDalMaso}) there exists a triangulation $\mathbf{\tilde{T}}_{n}\in \mathcal{T}_{\eps_n}(\R^2 )$ such that, setting $\mathbf{\tilde{T}'}_{n}\defas \{T\in \mathbf{\tilde{T}}_{n}\colon T\cap \bigcup_{k=1}^{M}S^n_k\neq \emptyset\}$, we have   
\begin{enumerate}
  \item\label{mesh-constr-1} $J_{\phi_n}\cap  \mathcal{V}(T) =\emptyset$ for all $T \in  \mathbf{\tilde{T}}_{n}$. 
  \item\label{mesh-constr-2} All triangles $T\in \mathbf{\tilde{T}}_{n}$ with $\dist(T, \MMM \bigcup_{k=1}^{M} \EEE S^n_k)\geq 5 \eps'_n $ are part of the \VIT `background' \EEE triangulation \JJJ $\mathbf{Z}_n$. \EEE 
  \item\label{mesh-constr-3} It holds \begin{equation}\label{mostcrucialproperty}
 \limsup_{n \to \infty} \EEE   \sum_{T\in \mathbf{\tilde{T}'}_{n}}\frac{ |T|  }{\eps_n} \le \limsup_{n \to \infty} \EEE  \,  \sin(\theta_0)\mathcal{H}^1\Big(\bigcup\nolimits_{k=1}^M S^n_k\Big)  \le   (1+2\theta)\sin(\theta_{0})\mathcal{H}^1(J_{\psi}\cap  Q \EEE ) \EEE \,,  
    \end{equation}  
\end{enumerate}
 where the \MMM second \EEE inequality in \eqref{mostcrucialproperty} follows from the construction of the segments $(S^n_k)_k$ and \eqref{crack  for later}. \EEE  

We now 
\VVV define \EEE $\mathbf{T}^{\rm good}_{n,k}$ \VVV as the set of triangles of the following type,  see also Figure \ref{fig: precrack-constr} for an illustration: \EEE
\begin{itemize}
\item[(i)] $T\in \mathbf{T}_n(u_n^k)$  with $T \cap Q \neq \emptyset$ and  $T\cap A^{n} = \emptyset$; 
\item[(ii)] $T\in \mathbf{T}_{n,k}^{\rm neigh}$ with $T \cap A^n\neq \emptyset$;
\item[(iii)] $T \in \mathbf{\tilde{T}}_{n}$ with  $T \cap A^n\neq \emptyset$ and \EEE $T \notin \mathbf{T}_{n,k}^{\rm neigh}$.
\end{itemize}  
Recalling  \eqref{enlarged-crackinA},  we notice that this \EEE leads to a triangulation of $B^{n}_{\rm good}$. In fact, by construction all triangles $T\in \mathbf{\tilde{T}}_{n} $ with $\dist(T, \bigcup_{k=1}^{N}S^n_k)\geq 5 \eps'_n $ are part of $\mathbf{Z}_n$ (see point \eqref{mesh-constr-2} above)
and by the definition in \eqref{def: new-crackset} all triangles of $\mathbf{T}_n(u_n^k)$  inside $\lbrace x \in N^{\rm crack}_{n,k}\colon {\rm dist}(x,   \Omega_{n,k}^{\rm crack}) \ge 10^{6} \eps_n'\rbrace $  are part of $\mathbf{Z}_n$.

Eventually, on the triangles in $\mathbf{T}^{\rm good}_{n,k}$,   we can define $\psi_n$ as the  interpolation of $\phi_n$ on $\mathbf{T}^{\rm good}_{n,k}$. Note that also for $T\in  \mathbf{T}^{\rm good}_{n,k} \cap \EEE \mathbf{\tilde{T}}_n$ the interpolation is well defined  because by construction none of the vertices of $T$ belongs to $\bigcup_{k} S^n_k$ (see point \eqref{mesh-constr-1} above). 
 
\subsubsection*{(c) Conclusion} \EEE
 We have now concluded the definition of the function $\psi_n$ and \VVV of \EEE the triangulation $\mathbf{T}_n(\psi_n)$. As by construction $\mathbf{T}_n(\psi_n)$ fulfills $\mathbf{T}_{n,k-1}^{\rm crack}\subset \mathbf{T}_n(\psi_n)$ \JJJ (actually, it even holds $\mathbf{T}_{n,k}^{\rm crack}\subset \mathbf{T}_n(\psi_n)$\EEE), we have $\psi_n\in  \hat{\mathcal{A}}_n^k(\Omega') $,  see before \eqref{eq: admis tria}. \EEE 
For later purposes, we note that 
\begin{align}\label{extraremark}
\eps_n|e(\psi_n)_{T}|^2\geq \kappa \quad \quad \text{for all $T\in \mathbf{T}_{n,k-1}^{\rm crack}(\psi_n)\sm \mathbf{T}_{n,k}^{\rm crack}$}.
\end{align}
Indeed, \EEE  since $T\notin \mathbf{T}_{n,k}^{\rm crack}$,  by \eqref{def: new-crackset} \EEE we know that $ \eps_n   |e(u_n^{j})_{T}|^2 < \kappa $ and $ \dist(T,\mathbf{Z}_n(u_n^{j})) < \EEE 10^6  \eps_n  \EEE $ for all $j\leq k$  \MMM such that $T \in \mathbf{T}_n(u_n^j)$.  
As $T\in \mathbf{T}_{n,k-1}^{\rm crack}(\psi_n)$, we can conclude that either $\eps_n   |e( \psi_n \EEE )_{T}|^2 \geq  \kappa$ or $\dist(T, \mathbf{Z}_n(\psi_n))\geq 10^6  \eps_n \EEE $. \JOS If $T \notin \mathbf{T}_n(u_n^k)$ we have, by the construction of $\mathbf{T}_n(\psi_n)$, that $T\in \tilde{\mathbf{T}}_n$ and thus by \eqref{mesh-constr-2} it holds $\dist(T, \mathbf{Z}_n(\psi_n)) <  10^6  \eps_n $. 
However, if $T\in \mathbf{T}_n(u_n^k)$, we have by the above argument that $\dist(T,\mathbf{Z}_n(u_n^{k})) <  10^6  \eps_n  $, which leads to $\dist(T, \mathbf{Z}_n(\psi_n)) <  10^6  \eps_n  $. \EEE 
In both cases, \EEE this yields  $\eps_n   |e( \psi_n \EEE )_{T}|^2 \geq  \kappa$,  i.e., \eqref{extraremark} holds. 

 Since we assumed that \EEE $\partial_{D} \Omega\cap J_{\psi}=\emptyset$, we can also suppose that for $n$  large enough we have  $(B_{\rm bad}\cup B^{n}_{\rm good})\cap \partial_{D}\Omega =\emptyset $. Therefore, we obtain  $\psi_n=g(t)_{\mathbf{T}_n(\psi_n)}$ on all $T \in \mathbf{T}_n(\psi_n)$ with  \MMM $T \cap  \overline{\Omega} = \emptyset$, \EEE which yields $\psi_n \in \mathcal{A}_n^{k}(g(t))$, see the definition before \eqref{stab0neu}.  Moreover,  by the regularity of $\psi\in \mathcal{W}(\Omega')$ we have  that $\psi_n\to \psi$ in measure on $\Omega'\sm (B_{\rm good}\cup B_{\rm bad})$. \MMM Therefore, \EEE by the first estimate in \eqref{besicovitch-props} it  follows that, for all $\delta>0$, 
$$ \limsup_{n\to \infty} |\{ |\psi_{n}  - \psi|>\delta\}| \leq |(B_{\rm good} \cup B_{\rm bad})|  \le   \theta.  $$
We thus have validated that \eqref{stab0neu} holds for $\psi_n$.  
 \EEE

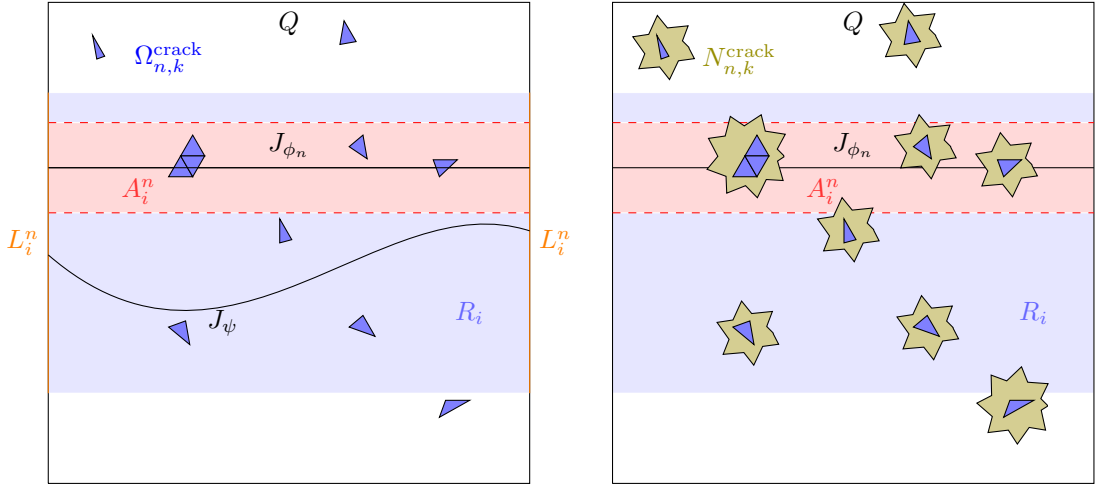
\begin{figure}

  \begin{minipage}{0.5\textwidth}
    \begin{tikzpicture}[scale=0.8]
      \filldraw[fill=blue!10!white, draw=white]{} (0,6.5)-- (8,6.5) --(8,1.5)-- (0,1.5)--cycle;
    \filldraw[fill=red!15!white, draw=white]{} (0,6)-- (8,6) --(8,4.5)-- (0,4.5)--cycle;
    \draw[color=black] (0,0) -- (8,0) --  (8,8) -- node[anchor=north]{$Q$}(0,8) -- (0,0);
    \draw[color=red, dashed] (0,6) -- (8,6);
      \draw[color=red, dashed] (0,4.5) -- (8,4.5);
      \draw[color=black] (0,5.25)--
      (8,5.25);
      \node[anchor=south, color=red!80!white] at (1.5,4.5) {$A_i^n$}; 
      \draw[color=orange](0,1.5) -- node[anchor=east] {$L_i^n$}(0,6.5);
      \draw[color=orange](8,1.5) -- node[anchor=west] {$L_i^n$} (8,6.5);
  
   \draw[color=black] (0,3.8) .. controls (3,1.2) and (5.5,5) .. (8,4.2);
        \node[anchor=south, color=black] at (2.9,2.3) {$J_{\psi}$};
  
  

      \filldraw[fill=blue!50!white, draw=black] (5,5.6) -- ++(40:0.3) -- ++(-80:0.4)-- cycle;
  
      \filldraw[fill=blue!50!white, draw=black] (3.5,3.8)++(30:0.4) -- ++(90:0.4) -- ++(-60:0.4)-- cycle;
  
      \filldraw[fill=blue!50!white, draw=black] (6.3,5.1)++(0:0.2) -- ++(110:0.3) -- ++(0:0.4)-- cycle;
  
      \filldraw[fill=blue!50!white, draw=black] (2,5.1) -- ++(60:0.4) -- ++(-60:0.4)-- cycle;
      \filldraw[fill=blue!50!white, draw=black] (2,5.1)++(60:0.4) -- ++(60:0.4) -- ++(-60:0.4)-- cycle;
      \filldraw[fill=blue!50!white, draw=black] (2,5.1)++(0:0.4) -- ++(120:0.4) -- ++(0:0.4)-- cycle;
  
  
      \filldraw[fill= blue!50!white, draw=black] (2,2.6) -- ++(20:0.3) -- ++(-80:0.4)-- cycle;
      \filldraw[fill=blue!50!white, draw=black] (4.5,7.1)++(30:0.4) -- ++(80:0.4) -- ++(-60:0.4)-- cycle;
      \filldraw[fill=blue!50!white, draw=black] (6.3,1.1)++(0:0.2) -- ++(70:0.3) -- ++(0:0.4)-- cycle;
      \filldraw[fill=blue!50!white, draw=black] (5,2.6) -- ++(40:0.3) -- ++(-60:0.4)-- cycle;
      \filldraw[fill=blue!50!white, draw=black] (0.5,6.8)++(40:0.4) -- ++(100:0.4) -- ++(-60:0.4)-- cycle;
          \draw[color=black] (0,5.25)--node[anchor=south] {$J_{\phi_n}$}(8,5.25);
      \node[color=blue!100!white, anchor=north] at (2.,7.5) {$ \Omega_{n,k}^{\rm crack}$};
      \node[anchor=south, color=blue!60!white] at (7.,2.5) {$R_i$}; 
    \end{tikzpicture}
  \end{minipage}
\begin{minipage}{0.4\textwidth}
	\begin{tikzpicture}[scale=0.8]
        \filldraw[fill=blue!10!white, draw=white]{} (0,6.5)-- (8,6.5) --(8,1.5)-- (0,1.5)--cycle;
        \node[anchor=south, color=blue!60!white] at (7.,2.5) {$R_i$}; 
     \filldraw[fill=red!15!white, draw=white]{} (0,6)-- (8,6) --(8,4.5)-- (0,4.5)--cycle;
     \draw[color=black] (0,0) -- (8,0) --  (8,8) -- node[anchor=north]{$Q$}(0,8) -- (0,0);
     \draw[color=red, dashed] (0,6) -- (8,6);
     \draw[color=red, dashed] (0,4.5) -- (8,4.5);
     \draw[color=black] (0,5.25)--node[anchor=south] {$J_{\phi_n}$}(8,5.25);
     \node[anchor=south, color=red!80!white] at (3.5,4.5) {$A_i^n$}; 

 
  \draw[fill=olive!40!white, thin, decoration={zigzag, }] decorate{ (6.55,5.3) circle (12pt)};
   \draw[fill=olive!40!white, thin, decoration={zigzag, }] decorate{(2.25,5.45) circle (17pt)};
   \draw[fill=olive!40!white, thin, decoration={zigzag, }] decorate{(3.92,4.22) circle (12pt)};  
   \draw[fill=olive!40!white, thin, decoration={zigzag, }] decorate{ (5.2,5.6) circle (12pt)};
    \draw[fill=olive!40!white, thin, decoration={zigzag, }] decorate{(2.25,2.5) circle (12pt)};
     \draw[fill=olive!40!white, thin, decoration={zigzag, }] decorate{ (0.85,7.25) circle (12pt)};
     \draw[fill=olive!40!white, thin, decoration={zigzag, }] decorate{ (5.24,2.59) circle (12pt)};  
    \draw[fill=olive!40!white, thin, decoration={zigzag, }] decorate{ (4.95,7.45) circle (12pt)};
  \draw[fill=olive!40!white, thin, decoration={zigzag, }] decorate{ (6.7,1.3) circle (15pt)};

     \filldraw[fill=blue!50!white, draw=black] (5,5.6) -- ++(40:0.3) -- ++(-80:0.4)-- cycle;
 
     \filldraw[fill=blue!50!white, draw=black] (3.5,3.8)++(30:0.4) -- ++(90:0.4) -- ++(-60:0.4)-- cycle;
 
     \filldraw[fill=blue!50!white, draw=black] (6.3,5.1)++(0:0.2) -- ++(110:0.3) -- ++(0:0.4)-- cycle;
 
     \filldraw[fill=blue!50!white, draw=black] (2,5.1) -- ++(60:0.4) -- ++(-60:0.4)-- cycle;
     \filldraw[fill=blue!50!white, draw=black] (2,5.1)++(60:0.4) -- ++(60:0.4) -- ++(-60:0.4)-- cycle;
     \filldraw[fill=blue!50!white, draw=black] (2,5.1)++(0:0.4) -- ++(120:0.4) -- ++(0:0.4)-- cycle;
 
 
     \filldraw[fill=blue!50!white, draw=black] (2,2.6) -- ++(20:0.3) -- ++(-80:0.4)-- cycle;
     \filldraw[fill=blue!50!white, draw=black] (4.5,7.1)++(30:0.4) -- ++(80:0.4) -- ++(-60:0.4)-- cycle;
     \filldraw[fill=blue!50!white, draw=black] (6.3,1.1)++(0:0.2) -- ++(70:0.3) -- ++(0:0.4)-- cycle;
     \filldraw[fill=blue!50!white, draw=black] (5,2.6) -- ++(40:0.3) -- ++(-60:0.4)-- cycle;
     \filldraw[fill=blue!50!white, draw=black] (0.5,6.8)++(40:0.4) -- ++(100:0.4) -- ++(-60:0.4)-- cycle;
     \node[color=olive, anchor=south] at (2.1,6.6) {$ N_{n,k}^{\rm crack}$};

     \end{tikzpicture}
\end{minipage}
\hspace*{0.6cm}
\caption{Construction of $\mathbf{T}_n(\psi_n)$ in $Q\in \MMM \mathcal{B}^{\rm small}_{\rm good}$. \EEE  First, we \VVV transfer the jump from \EEE $J_{\psi}$ to \VVV the flat set \VIT $J_{\phi_n}$, \VVV whose vertical coordinate is determined by Lemma~\ref{stripe-prop}. In the \VIT golden \EEE
neighborhoods $N_{n,k}^{\rm crack}$ of connected components of $\Omega_{n,k}^{\rm crack}$ (which are actually 
 of \EEE size much larger than the triangles), we  keep \EEE the triangulation $\mathbf{T}_n(u_n^k)$,  and in the red region we use the triangulation $\mathbf{\tilde{T}}_{n}$ given by \cite{ChamboDalMaso}. \EEE  
}
\label{fig: precrack-constr}
\end{figure}

\subsection{Proof of \eqref{stab1neu}: Stability \MMM estimate \EEE of the crack part}

We now proceed with the proof of  \eqref{stab1neu}. Again let $t\in [0,T]$ be given and, for each $n \in \N$, choose $k=k(t)$ such that $t \in [t^k_n,t_n^{k+1})$. Let $(\psi_n)_n$ be the sequence defined above with the associated triangulation $\mathbf{T}_n(\psi_n)$. Note that by definition  we have
\begin{align}\label{split-stability}
  \mathcal{E}_n^{\rm crack}(\psi_n; t)-\mathcal{E}_n^{\rm crack}(u_n(t); t)  \leq  
  \sum_{T\in \VVV \mathbf{T}^{\rm crack}_{n,k-1}(\psi_n) \EEE \setminus \mathbf{T}^{\rm crack}_{n,k} } \JJJ \kappa \EEE \frac{|T|}{ \eps_n} \,.
\end{align}
We split \VVV $\mathbf{T}^{\rm crack}_{n,k-1}(\psi_n)$ \EEE into three different parts, namely $ \mathbf{T}^{\rm cr, g}_{n,k}(\psi_n)$, $ \mathbf{T}^{\rm cr, b}_{n,k}(\psi_n)$, and  $ \mathbf{T}^{\rm cr, r}_{n,k}(\psi_n)$, which corresponds to the intersection of \VVV $\mathbf{T}^{\rm crack}_{n,k-1}(\psi_n)$ \EEE with $\mathbf{T}^{\rm good}_{n,k}$, $\mathbf{T}^{\rm bad}_{n,k}$, and  $\mathbf{T}^{\rm rest}_{n,k}$, respectively, see \eqref{eq: splitti}.

%
%
%
%

\subsection*{Bad squares}

We start with the bad squares and aim at estimating 
\begin{equation}\label{toprove-bad}
  \limsup_{n\to \infty}\sum_{T\in \mathbf{T}^{\rm cr, b}_{n,k}(\psi_n) \setminus  \mathbf{T}^{\rm crack}_{n,k}  } \frac{|T|}{ \eps_n \EEE }\leq C\theta 
\end{equation} 
for a universal constant $C>0$.  This will rely on the following two lemmas whose proofs will be deferred to Section \ref{sec: lemma prof} below. For their formulation, we recall the definition of the neighborhood $N_{\rm out}(T)$ in \eqref{NNN} and \JJJ of the sets $P_i^{n,\pm}$ in \eqref{1301252256}, see also Figure \ref{N-x-out-figure}. We then \EEE introduce the set of triangles, where $N_{\rm out}(T)$ is contained in one of the sets $P_i^{n,+}$ or $P_i^{n,-}$, namely 
\begin{align}\label{sidedef}
\mathbf{T}^{\rm side}_{n} = \Big\{ T \in \mathbf{T}^{\rm bad}_{n,k} \colon \ \   N_{\rm out}(T) \subset \overline{P_i^{n,+}} \text{ or }  N_{\rm out}(T) \subset \overline{P_i^{n,-}} \text{ for some }Q_i \in \mathcal{B}_{\rm bad}  \Big\}.
\end{align}
We also recall the definition of $ L^n_i$ and $\Gamma_i^n$   in \EEE \eqref{Ri-estimates} and  \eqref{def: conticurve}, respectively.

\begin{lemma}\label{lemma7.1.5}
Consider $T \in \mathbf{T}^{\rm bad}_{n,k}$ satisfying one of the conditions


\begin{itemize}
\item[(a)]   $T \cap (L^n_i \cup \Gamma^i_n) = \emptyset$,
\item[(b)]   $T \cap L^n_i  = \emptyset$ and $T \in \mathbf{T}^{\rm side}_{n} \MMM \sm \mathbf{T}_{n,k}^{\rm crack}$. 
\end{itemize}
Then, it holds $| (  \nabla  \psi_n)_T|\leq C$  for some $C$ only depending on $\psi$, \EEE and in particular, for $n$ large enough, we have by   \eqref{extraremark} that $ T \notin  \mathbf{T}^{\rm crack}_{n,k-1}(\psi_n) \EEE \setminus \mathbf{T}^{\rm crack}_{n,k}$.   
\end{lemma}

\begin{lemma}\label{lemma7.1}
For $n$ large enough,  the curves \EEE  $ \Gamma^n_i \EEE$  satisfy the properties    
    \begin{align}\label{rare-situation}
 \sum_{Q_i \in \mathcal{B}_{\rm bad}}      \# \big\{T\in \mathbf{T}_{n}(\psi_n)\colon \ \emptyset \neq T\cap  \Gamma^n_i \EEE \subset K_n(t), \ T \notin \mathbf{T}^{\rm side}_{n}\JJJ \cup \mathbf{T}_{n,k}^{\rm crack} \EEE \big\}\leq \frac{C \,\theta}{\eps_n}\,.  
    \end{align}
\begin{equation}\label{firsttoproveX}
 \sum_{Q_i \in \mathcal{B}_{\rm bad}} \EEE\#\big\{T\in \mathbf{T}_{n}(\psi_n)  \colon \, ( \Gamma^n_i \EEE\setminus K_n(t))\cap T\neq \emptyset \text{ or } L^n_i \EEE \cap T \neq \emptyset \big\}\leq  \frac{C\,\theta}{\eps_n}\,.\end{equation}
\end{lemma}
 
With these auxiliary results we can prove  \eqref{toprove-bad}. 
    In view of \eqref{rare-situation}, \eqref{firsttoproveX}, and the fact that $|T| \le C  \eps^2_n \EEE$ for all $T \in \mathbf{T}^{\rm cr, b}_{n,k}(\psi_n) $  (recall \eqref{eq: 10.6}), \EEE   it is enough to confirm that each $T \in \mathbf{T}^{\rm cr, b}_{n,k}(\psi_n) \setminus  \mathbf{T}_{n,k}^{\rm crack} $ lies in the collection of triangles estimated in \eqref{rare-situation} or \eqref{firsttoproveX}.  Indeed, all triangles not lying in these collections \JOS either satisfy (a) or (b) of Lemma \ref{lemma7.1.5} or fulfill $T\in \mathbf{T}_{n,k}^{\rm crack}$. In both cases, it holds $T\notin \mathbf{T}^{\rm cr, b}_{n,k}(\psi_n) \setminus  \mathbf{T}_{n,k}^{\rm crack} $ for $n$ large enough, and thus we obtain a contradiction. \EEE

\subsection*{Good squares}
For the good squares, our goal is to prove  
\begin{equation}\label{enoughtoprove}
  \limsup_{n\to \infty}\sum_{T\in \mathbf{T}^{\rm cr, g}_{n,k}(\psi_n) \setminus  \mathbf{T}^{\rm crack}_{n,k}} \frac{|T|}{\eps_n}\leq \sin(\theta_{0})\mathcal{H}^1\big((J_{\psi}\setminus K(t))\cap B_{\rm good}\big)+C\theta\,.
\end{equation}
 We \EEE define the `cracked triangles' in $Q_i\JJJ \in \mathcal{B}_{\rm good}$ according to $\psi_n$ by   
\begin{equation}\label{split-of-jumps}
    \mathbf{T}'_n(\psi_n, Q_i)= \big\{T\in\mathbf{T}_n(\psi_n) \colon T\cap (  J_{\phi_n} \cap Q_i)\EEE \neq \emptyset \big\}\,. 
\end{equation}
 We state an important property whose proof is again deferred to Section \ref{sec: lemma prof}. \EEE

\begin{lemma}\label{lemma7.4}
It holds that 
\begin{equation}\label{crucial-inclusion}
 \mathbf{T}^{\rm cr, g}_{n,k}(\psi_n)\setminus \mathbf{T}^{\rm crack}_{n,k}    \subset  \bigcup_{Q_i\in \mathcal{B}_{\rm good}}\mathbf{T}'_n(\psi_n,Q_i)\,.
\end{equation}
\end{lemma}
\noindent
 Recalling the distinction in  \eqref{bigsmall}, we first address squares $Q_i\in \mathcal{B}_{\rm good}^{\rm large}$. In fact, we want to prove that 

\begin{equation}\label{enoughtoprove-large}
  \sum_{Q_i \in \mathcal{B}_{\rm good}^{\rm large}}  \sum_{ T\in \mathbf{T}'_n(\psi_n,Q_i) } \frac{|T|}{\eps_n} \le C \theta.
\end{equation}
By definition we find that for each $Q_i\in \mathcal{B}_{\rm good}^{\rm large} $ we have $r_i<\theta \, \eps_n^{-1}| \Omega_{n,k}^{\rm crack} \EEE \cap Q_i| $ such that  the energy bound in Corollary \ref{cor: energy bound}  implies \EEE 
$\sum_{Q_i\in  \mathcal{B}_{\rm good}^{\rm large}} r_i \le C\theta\,$.  By  \eqref{normal-angle-estimate0} we find
\begin{align}\label{auch das noch}
\mathcal{H}^1( J_\psi \cap \MMM Q_i) \le Cr_i \quad  \text{and} \quad  \EEE \mathcal{H}^1( J_\psi \cap B_{\rm good}^{\rm large}) \le C\theta.
\end{align} \EEE
Then, by construction \MMM and the regularity of $J_\psi$ \EEE we have $\# \mathbf{T}'_n(\psi_n,Q_i) \le Cr_i/\eps_n$ for all  $Q_i\in  \mathcal{B}_{\rm good}^{\rm large}$ (see \eqref{1501251236}), \EEE i.e.,  by \eqref{eq: 10.6} \EEE we get  
$$  \sum_{Q_i \in \mathcal{B}_{\rm good}^{\rm large}}  \sum_{ T\in \mathbf{T}'_n(\psi_n,Q_i) } \frac{|T|}{\eps_n} \le  \sum_{Q_i \in \mathcal{B}_{\rm good}^{\rm large}} C \eps_n \# \mathbf{T}'_n(\psi_n,Q_i) \le \sum_{Q_i\in  \mathcal{B}_{\rm good}^{\rm large}}  C r_i \le C\theta\,, $$ and hence \eqref{enoughtoprove-large}.  
We now proceed with $\mathcal{B}_{\rm good}^{\rm small}$.   Recall that $J_{\phi_n}\cap Q_i\subset \Gamma_i^n \cup   L^{n}_i$, see \eqref{1401251329}--\eqref{14012513295}. Recalling also \EEE the construction of $\mathbf{T}_n(\psi_n)$ in $Q_i$, \VVV in particular  the \EEE definition of $\mathbf{\tilde{T}'}_{n,i}$ for each $Q_i$, we have 
\begin{equation}\label{schonauchnoch}
    \mathbf{T}'_n(\psi_n, Q_i)\subset  \mathbf{\tilde{T}'}_{n,i}\cup \big\{T\in  \VVV \mathbf{T}_{n,k,i}^{\rm neigh}\EEE   \colon T\cap  \Gamma^n_i \EEE  \cap A_i^n \EEE \neq \emptyset \big\}  \cup \{T\in \mathbf{T}_n(\psi_n)\colon T\cap \VVV L_i^n \EEE \neq \emptyset\}\,.
\end{equation}
 Here, 
\JJJ $  \mathbf{T}_{n,k,i}^{\rm neigh} $ and \EEE
$\mathbf{\tilde{T}'}_{n,i}$ \VVV are defined as in  \eqref{enlarged-crackinA} and \JJJ above \EEE \eqref{mostcrucialproperty}, where we also include the subscript $i$ in the notation since the objects are related to $Q_i$. \EEE  
\VVV We estimate the energetic contribution for each of the three families of triangles in \eqref{schonauchnoch}:
for the first family, \EEE
we can make use of \eqref{mostcrucialproperty}  and \EEE we obtain, for any $Q_i\in \mathcal{B}^{\rm small}_{\rm good}$,
\begin{equation}\label{comb3}
  \limsup_{n \to \infty}  \sum_{T\in \tilde{\mathbf{T}}'_{n,i}} \frac{|T|}{\eps_n} \le     (1+2\theta) \EEE \sin(\theta_{0})\mathcal{H}^1( J_{\psi} \EEE \cap Q_i) \,.
  \end{equation}
For the second family, by \EEE definition of $  \mathbf{T}^{\rm neigh}_{n,k,i}$  in \eqref{enlarged-crackinA} and Lemma \ref{stripe-prop}, we get
  \begin{equation}\label{1501251805}
  \# \big\{T\in  \VVV \mathbf{T}_{n,k,i}^{\rm neigh}\EEE   \colon T\cap  \Gamma^n_i \EEE  \cap A_i^n \EEE \neq \emptyset \big\} \leq  C\EEE \#\big\{T\in  \mathbf{T}_{n,k}^{\rm crack} \EEE \colon T\cap \VVV A^{n}_i\EEE \neq \emptyset \big\}\leq C\frac{\hat{C}}{ \theta^2\EEE},
  \end{equation}
   where the $C$ depends on the constant $10^7$. \EEE  For the third family, \EEE since $\mathcal{H}^1(\VVV L_i^n \EEE)\leq C r_i \EEE \theta$,  we deduce \EEE
\begin{equation}\label{lateral-good-small}
  \#  \{T\in \mathbf{T}_n(\psi_n)\colon T\cap \VVV L_i^n \EEE\neq \emptyset\} \leq C\frac{\theta  r_i \EEE }{\eps_n}\,.
\end{equation}
Therefore, \VVV in view of the fact that $|T|\leq C \eps_n^2$  by \eqref{eq: 10.6}, \EEE  collecting \eqref{comb3}--\eqref{lateral-good-small}  and using \eqref{normal-angle-estimate0} \EEE we get that for each $Q_i\in \mathcal{B}_{\rm good}^{\rm small}$ \JJJ and for $n$ large enough \EEE it holds  
\begin{equation*}\label{good-small-collection}
   \sum_{ T\in \mathbf{T}'_n(\psi_n,Q_i)}  \frac{|T|}{\eps_n} \le \sin(\theta_{0})\mathcal{H}^1( J_{\psi} \EEE \cap Q_i)  + \VVV C\frac{\hat{C}}{ \theta^2 \EEE} \eps_n \EEE + C\theta r_i .
\end{equation*}
From Lemma \ref{lemma7.4} and the fact that \EEE $\sum_i r_i \le C$
we can conclude that 
\begin{equation}
  \limsup_{n\to \infty}\sum_{T\in \mathbf{T}^{\rm cr, g}_{n,k}(\psi_n) \setminus  \mathbf{T}^{\rm crack}_{n,k}} \frac{|T|}{\eps_n}\leq \limsup_{n\to \infty} \sum_{Q_i \in \mathcal{B}_{\rm good}}\sum_{T\in \mathbf{T}'_n(\psi_n,Q_i) } \frac{|T|}{\eps_n}\leq \sum_{Q_i \in \mathcal{B}_{\rm good}} \sin(\theta_{0})\mathcal{H}^1(J_{\psi}\cap Q_i) + C\theta.
\end{equation}
 In view of \eqref{almostnoK} and using again that \EEE $\sum_i r_i \le C$, this implies \eqref{enoughtoprove}.

\subsection*{The remaining part}
Finally, we want to show that  
\begin{equation}\label{toprove-remaining}
    \limsup_{n\to \infty}\sum_{T\in \mathbf{T}^{\rm cr, r}_{n,k}(\psi_n) \setminus  \mathbf{T}^{\rm crack}_{n,k} } \frac{|T|}{\eps_n}\leq C\theta\,.
  \end{equation}
Recall that for all $T\in \mathbf{T}^{\rm rest}_{n,k}$ we have $\psi_n=\psi$ on the vertices of $T$,  see below \eqref{eq: splitti}. \EEE In particular, for all triangles 
$T\in\mathbf{T}^{\rm rest}_{n,k}$ with    $T \cap J_{\psi} =  \emptyset$ \EEE we have $ |e(\psi_n)_{T}| = \EEE |e(\psi)_{T}|\leq |\nabla \psi_{T}| \leq C$ and hence,     $T\notin \mathbf{T}^{\rm cr, r}_{n,k}(\psi_n)\setminus  \mathbf{T}^{\rm crack}_{n,k}$   for $n$ large enough \EEE by  \eqref{extraremark}. Therefore,  it \EEE suffices to consider all $T\in \mathbf{T}^{\rm rest}_{n,k}$ with    $T\cap J_{\psi}\neq \emptyset$.  Due to \EEE  \eqref{besicovitch-props2} and the regularity of $J_{\psi}$, we obtain  
\begin{equation}\label{anothersmalljump}
    \# \{T\in \mathbf{T}^{\rm rest}_{n,k} \colon T\cap J_{\psi}\neq\emptyset\} \leq \frac{C}{\eps_n} \mathcal{H}^1(J_{\psi}\cap B_{\rm rest}) \leq C\frac{\theta}{\eps_n}\,. 
\end{equation}  
Altogether, we obtain $\# (\mathbf{T}^{\rm cr, r}_{n,k}(\psi_n) \setminus  \mathbf{T}^{\rm crack}_{n,k} ) \le  C\frac{\theta}{\eps_n}$.  This along with \eqref{eq: 10.6} yields \EEE \eqref{toprove-remaining}.
   
\subsection*{Conclusion} 
By putting together \eqref{toprove-bad}, \eqref{enoughtoprove}, and \eqref{toprove-remaining} we conclude 
\begin{equation}
   \limsup_{n\to \infty} \sum_{T\in  \mathbf{T}^{\rm crack}_{n,k-1}(\psi_n) \EEE \setminus \mathbf{T}^{\rm crack}_{n,k} } \frac{|T|}{\eps_n} \leq \sin(\theta_{0}) \mathcal{H}^1(J_{\psi}\setminus K(t)) + C\theta \,.
\end{equation}
By \eqref{split-stability} this finally validates \eqref{stab1neu}.

\subsection{Proof of \eqref{stab2neu}: Stability estimate for elastic part}

Finally, we come to the proof of \eqref{stab2neu}. Recall the definition of approximating functions $\phi_n$ from the paragraph `Jump transfer'  in Section~\ref{sec:stabprep}. \EEE   We define   the collection of triangles which do not intersect the jump by \EEE
\begin{equation*}
  \mathbf{D}_n(\phi_n)\defas \{T\in \mathbf{T}_n(\psi_n)\colon \VVV T \cap J_{\phi_n}=\emptyset \EEE \}, \quad \quad D_n(\phi_n)\defas \bigcup_{T\in \mathbf{D}_n(\phi_n)} T \,.
\end{equation*} 
 Next, we introduce the set of triangles $T$ \VIT outside \EEE $\mathbf{T}_{n,k-1}^{\rm crack}(\psi_n)$ where   $|e(\psi_n)_{T}|^2$ is potentially not uniformly controlled. To this end, we let \EEE  
\begin{equation*}\label{1501251117'}
    G_{n,k}(\psi_n):=  \bigcup_ {T \in \mathcal{G}_{n,k}(\psi_n)} T , \EEE
    \end{equation*}
where
\begin{equation*}\label{1501251119}
    \begin{split}
    \mathcal{G}_{n,k}(\psi_n):= \big\{ T \in \mathbf{T}_n(\psi_n) \sm \mathbf{T}_{n,k-1}^{\rm crack}(\psi_n) \colon T \cap  (\Gamma_i^n \cup L_i^n) \neq \emptyset \text{ for }Q_i \in \mathcal{B}_{\rm bad} \cup   \mathcal{B}^{\rm small}_{\rm good} \EEE \text{ or }  T \notin \mathbf{D}_n(\phi_n)  \big\}.
    \end{split}
\end{equation*}
We  split the elastic energy $ \mathcal{E}_n^{\rm elast} (\psi_n,t)$ into the two parts 
\begin{equation}\label{split-elastic}
    \int_{\Omega\sm \Omega_{n,k-1}^{\rm crack}(\psi_n) }|e(\psi_n)|^2\, {\rm d}x = \int_{\Omega\sm (G_{n,k}(\psi_n)\cup \Omega_{n,k-1}^{\rm crack}(\psi_n))} |e(\psi_n)|^2\, {\rm d}x +\int_{G_{n,k}(\psi_n)} |e(\psi_n)|^2\, {\rm d}x\,. 
  \end{equation}
In order to estimate the second term, we  need \EEE the following lemma. \EEE
\begin{lemma}\label{lemma: negli-jump}
  For $\eps_n$ small enough, it holds that
  \VVV 
\begin{equation}\label{1501251053}
\sum_{Q_i \in \mathcal{B}_{\rm bad} \cup \mathcal{B}_{\rm good}^{\rm small}}  \#\, \big( \{T \in \mathbf{T}_n(\psi_n) \colon T \cap  \Gamma_i^n \neq \emptyset\} \sm  \mathbf{T}_{n,k-1}^{\rm crack}(\psi_n) \big) \leq C\theta\eps_n^{-1}.
  \end{equation}
\end{lemma}

We defer the proof to Section \ref{sec: lemma prof}.  In view of   \eqref{firsttoproveX} and \eqref{lateral-good-small}, we have \EEE  
\begin{equation}\label{1501251131}
\sum_{Q_i \in \mathcal{B}_{\rm bad} \cup \mathcal{B}_{\rm good}^{\rm small}}\#\, \{T \in \mathbf{T}_n(\psi_n) \colon T \cap  L_i^n \neq \emptyset\}  \leq C\theta\eps_n^{-1}.
\end{equation}
 Let $\mathbf{T}^{\rm good, large}_{n,k}   = \lbrace T \in  \mathbf{T}_n(u_n^k) \colon \, T \cap B_{\rm good}^{\rm large} \neq \emptyset\rbrace$.  From \eqref{1501251236}, \eqref{auch das noch}, and the argument in \eqref{anothersmalljump} we deduce \EEE
 \begin{equation}\label{final-control-rest}
    \#\big(   (\mathbf{T}^{\rm  rest}_{n,k}  \cup   \mathbf{T}^{\rm good, large}_{n,k}) \EEE \sm \mathbf{D}_n(\phi_n)\big) \leq C\theta \eps_n^{-1}\,.
  \end{equation}  
  By Lemma~\ref{lemma: negli-jump}, \eqref{1501251131}, \eqref{final-control-rest}, and  \eqref{1401251338}, \EEE recalling $|T|\leq C \eps_n^2$ for any $T \in \mathbf{T}_n(\psi_n)$  by \eqref{eq: 10.6} \EEE and using  $|e(\psi_n)_{T}|^2\leq  \kappa \EEE /\eps_n $ for all $T \notin \mathbf{T}_{n,k-1}^{\rm crack}(\psi_n) $, we obtain 
\begin{equation}\label{1501251117}
 | G_{n,k}(\psi_n) | \leq C\theta \eps_n\EEE \quad \text{and}\quad \int_{G_{n,k}(\psi_n)} |e(\psi_n)|^2 \dx \leq C\theta \,.
\end{equation}
Therefore, we are left to estimate the elastic energy contribution on $\Omega\sm (G_{n,k}(\psi_n)\cup \Omega_{n,k-1}^{\rm crack}(\psi_n))$.  We first show that for each $T \in \mathbf{T}_n(\psi_n) \setminus ( \mathbf{T}_{n,k-1}^{\rm crack}(\psi_n) \cup  \mathcal{G}_{n,k}(\psi_n) )$, the function $\phi_n|_T$ has $W^{2,\infty}$-regularity. Indeed,  if $T$ intersects $Q_i \in \mathcal{B}_{\rm bad} \cup \mathcal{B}_{\rm good}^{\rm small}$ with $T \cap  (\Gamma_i^n \cup L_i^n) = \emptyset$  and $T\subset Q_i$,  it follows that either $T \subset  {P_i^{n,-}}$ or $T \subset {P_i^{n,+}}$ (see \eqref{1301252256} and \eqref{1401251329}),  and thus $\phi_n$ coincides with $\phi_{i}^{n,-}$ or $\phi_{i}^{n,+}$   in $T$,   respectively   (see \eqref{1301252256} and  \eqref{phiningood}). If $T \cap  (\Gamma_i^n \cup L_i^n) = \emptyset$ and $T\sm Q_i\neq \emptyset$, then $\phi_n$ coincides with $\psi$ in $T$ \JJJ by \eqref{outside-R-i}.  If   $T \in (\mathbf{T}^{\rm  rest}_{n,k}   \cup   \mathbf{T}^{\rm good, large}_{n,k})  \cap \mathbf{D}_n(\phi_n)$,  \EEE   we have that $\phi_n= \psi$ on $T$, see \eqref{1501251236}. \EEE 

 In all cases, we get $\Vert \nabla \phi_n \Vert_{W^{1,\infty}(T)} \le C_{\psi,\theta}  $ by \eqref{eq: Lipschitz1}. Since on all such triangles  $\psi_n$ is the piecewise affine interpolation of $\phi_n$,  using \eqref{eq: 10.6}  we derive \EEE
\EEE
$$
 \VVV \| e (\psi_n)-e(\phi_n)\|_{L^{\infty}(T)}\leq \EEE \|\nabla \psi_n -\nabla \phi_n\|_{L^{\infty}(T)}   \leq \MMM C_{\psi,\theta} \EEE \,\eps_n \quad \ \    \text{for all  $T \in \mathbf{T}_n(\psi_n) \setminus ( \mathbf{T}_{n,k-1}^{\rm crack}(\psi_n) \cup  \mathcal{G}_{n,k}(\psi_n) )$}. \EEE
$$ 
By summing  over all triangles  this gives  
\begin{equation}\label{0902250955}
 \int_{\Omega\sm (G_{n,k}(\psi_n)\cup \Omega_{n,k-1}^{\rm crack}(\psi_n))}|e(\psi_n)-e(\phi_n)|^2 \dx\leq |\Omega|\, \MMM C^2_{\psi, \theta} \EEE \,\eps_n^2.
\end{equation} 
By the bound on   $\JJJ \nabla \phi_n\EEE$ in \eqref{eq: Lipschitz1}, the first property in  \eqref{besicovitch-props}, and  \eqref{1501251236} we further find
$$ \int_{\Omega}|e(\phi_n)|^2 \dx \le \int_{\Omega}|e(\psi)|^2 \dx + C\theta. $$
This and \eqref{0902250955} control the limsup of the first integral on the right-hand side of \eqref{split-elastic}. Therefore, in view of \eqref{1501251117}, the proof of \eqref{stab2neu} is concluded.  \EEE

\subsection{Proof of lemmas}\label{sec: lemma prof}
 Finally, we prove the lemmas used in the previous subsections. 

\begin{proof}[Proof of Lemma \ref{stripe-prop}]
  We consider $\xi_j\defas  ( -\theta  r_i   +  2    j\cdot  10^{8}  \eps_n)\EEE $ for $  j = 0, \ldots, \EEE   {N}:=\lceil\frac{ \theta  r_i \EEE }{  10^{8}   \eps_n}\rceil$ and observe  that for $n$ large enough    
   \begin{equation*}
    \bigcup_{j=0}^{ {N}} A^{n}_i(\xi_j) \supset R^n_i \quad \text{and}\quad |\VVV E \EEE\cap Q_i| \ge  \sum_{j=1}^{N-2} \EEE |\VVV E \EEE \cap A^{n}_i(\xi_j)  |,
  \end{equation*}  
  \VVV denoting, here and below, $A^{n}_i(\xi_j; E)$ by $A^{n}_i(\xi_j)$. \EEE
Therefore, we find $J  \in \lbrace 2,\ldots, N - 3\rbrace$ such that
$${\sum_{j=J-1}^{J+1} |  E  \cap A^{n}_i(\xi_j)  | \le \tilde{C}N^{-1} |  E  \cap Q_i| \le \tilde{C}\eps_n^2/ \theta^2 \EEE} $$
 for some universal $\tilde{C}>0$, \EEE where we used the assumption that $|\VVV E \EEE  \cap Q_i|\leq  \theta^{-1} r_i \eps_n \EEE$.  As $|T|\geq c \eps_n^2$  for $c$ only depending on $\theta_0$, we obtain 
$$\# \Big\{   T\in  \mathbf{T}_E  \EEE \colon T\subset\bigcup_{j=J-1}^{J+1} A^{n}_i(\xi_j)\ \Big\} \le \tilde{C}/ \theta^2. \EEE$$
\MMM In view of \eqref{eq: 10.6}, \EEE this shows the statement for $\xi = \xi_J$. 
\end{proof}
 
\begin{proof}[Proof of Lemma \ref{lemma7.1.5}]
We first consider the case that $T \cap (\VVV L^n_i \EEE\cup \VVV \Gamma^n_i \EEE) = \emptyset$. \JJJ As it holds $\Vert  \nabla \phi_n \Vert_{L^\infty(\Omega'\JJJ \setminus J_{\phi_n}\EEE)} \le C_\psi$ by \eqref{eq: Lipschitz1}, \EEE we find by \eqref{1401251338} that
$$  |\phi_n(x)-\phi_n(x')|\leq \|\nabla \phi_n\|_{\JJJ L^\infty (T)\EEE} |x-x'|\leq \MMM C_\psi \EEE |x-x'|
$$
for vertices $x,x' \in \mathcal{V}(T)$. In particular,   as $\psi_n(x)=\phi_n(x)$ and $\psi_n(x')=\phi_n(x')$, \EEE  this implies  
    \begin{equation}\label{bovi}
      |(\nabla \psi_n)_{T} (x-x')|= |\psi_n(x)-\psi_n(x')|\leq \MMM C_\psi \EEE  |x-x'|\,.
    \end{equation}  
    Since this holds true along all three edges of $T$, we deduce that also $| (  \nabla  \psi_n)_T|\leq C$ for some $C>0$,  as desired. \EEE
    
Now, we assume that $T \cap \VVV L^n_i \EEE = \emptyset$ but $T \cap \VVV \Gamma^n_i \EEE \neq \emptyset$ and $T \in \mathbf{T}^{\rm side}_{n} \MMM \sm \mathbf{T}_{n,k}^{\rm crack} \EEE $. \VVV By \EEE definition of $\mathbf{T}^{\rm side}_{n}$  in \eqref{sidedef}, \EEE we find that either \VVV $N_{\rm out}(T) \EEE \, \tilde{\subset} \, P^{n,+}_{i}$ or \VVV $N_{\rm out}(T) \EEE \, \tilde{\subset} \, P_i^{n,-}$ and
 $N_{\rm out}(T)   \, \tilde{\subset} \, \EEE P^{n,\pm}_{i}$ if and only if $T  \, \tilde{\subset} \, \EEE P^{n,\pm}_{i}$, i.e., $\psi_n$ is the piecewise affine interpolation of $\phi_{i}^{n,\pm}$. Arguing as above, we deduce that $| (  \nabla  \psi_n)_T|\leq C$.
\end{proof}

\begin{proof}[Proof of Lemma \ref{lemma7.1}]
We start with the proof of \EEE  \eqref{rare-situation}. Consider $T\in \mathbf{T}^{\rm bad}_{n} $ such that $\emptyset \neq T\cap \Gamma^n_i \EEE \subset K_n(t)$ and $T \notin \mathbf{T}^{\rm side}_{n} \JJJ \cup \mathbf{T}_{n,k}^{\rm crack}\EEE$.   As $\partial T \cap \partial \Omega_{n,k}^{\rm mod}  \neq  \emptyset$ and $T \notin  \mathbf{T}_{n,k}^{\rm crack}$, we first  notice that,  in view of Remark~\ref{remark-on-holes}, we have $|T \cap \Omega_{n,k}^{\rm mod}|=0$.   
 \VIT Then \JJJ
we \VIT may assume that \JJJ $T\subset N_{\rm out}^{x}$ for all $x\in \mathcal{V}(T)$, hence $N_{\rm out}(T)$ is a connected set. As $\emptyset \neq T\cap \Gamma^n_i \subset K_n(t)$, we  get $N_{\rm out}(T)\cap \partial \VIT \Omega_{n,k}^{\rm mod}\JJJ\neq \emptyset$. However, \VIT since \JJJ $T \notin \mathbf{T}^{\rm side}_{n}$, we also have $N_{\rm out}(T)\cap P^{n,+}_{i}\neq \emptyset$ and at the same time $N_{\rm out}(T)\cap P^{n,-}_{i}\neq \emptyset$. \MMM  This implies that $T$ is either close to the intersection of a component of $\Gamma_i^{n}\setminus K_n(t)$ and a component of $\Omega_{n,k}^{\rm mod}$ or it is close to the intersection of two different components in $\Omega_{n,k}^{\rm mod}$ (see, e.g., $x_2$ in  Figure~\ref{N-x-out-figure} or \VIT $x_3$ \MMM 
in this picture, \VIT respectively). \MMM 

Recall that each connected components of $\Omega_{n,k}^{\rm mod}$ meets at most two connected components of $\Gamma_i^{n}\setminus K_n(t) $, see the discussion before \eqref{1301252256}. For the same reason, we can assume that for each connected component $H$ of  $\Omega_{n,k}^{\rm mod}$, the set $H \cap \Gamma^n_i$ intersects at most two other components in $\Omega_{n,k}^{\rm mod}$. 
Summarizing, the number of triangles $T$ of the above form is controlled (up to a multiplicative constant) by the total number of connected components $\mathcal{C}(\Omega_{n,k}^{\rm mod})$. \EEE  By \eqref{extra-statement} and Corollary \ref{cor: energy bound}  we \MMM  thus  \EEE obtain 
\begin{equation}
  \# \big\{T\in \mathbf{T}^{\rm bad}_{n}\colon \ \emptyset \neq T\cap  \Gamma^n_i \EEE \subset K_n(t), \ T \notin \mathbf{T}^{\rm side}_{n}   \MMM \cup \mathbf{T}_{n,k}^{\rm crack}\EEE \big\}\leq C \MMM \frac{ \eta_n  }{\eps_n}. \EEE
\end{equation}
Since   $\eta_n<\theta$ for  $n$ large enough, we conclude the proof of \eqref{rare-situation}.


We now come to the proof of \eqref{firsttoproveX}.  The triangles $T$ with $L^n_i \EEE \cap T \neq \emptyset$ can be controlled \JJJ exactly \EEE as in \eqref{lateral-good-small}.  \EEE  By $(\Gamma^n_{i,l})_l$  we denote the connected components of $(\Gamma_i^{n}  \setminus K_n(t)) \cap Q_i$, and recall that we  assumed without restriction that $\Gamma^n_{i,l}$ connects two different connected components of $\Omega_{n,k}^{\rm mod}$, \VVV see the discussion \MMM before  \eqref{1301252256}. \EEE   Note that $K_n(t)$ is the boundary of a union of triangles \JJJ in \EEE $\mathbf{T}_{n}(u_n^{k})$, and thus  different connected components of $K_n(t)$ have distance at least  $c  \eps_n$  for some $c>0$ only depending on $\theta_0$. Therefore, each $\Gamma^n_{i,l}$ fulfills $\mathcal{H}^{1}(\Gamma^n_{i,l})\geq  c \EEE \eps_n$.  The key point consists in showing 
\begin{equation}\label{forcomponents}
    \#\{T\JJJ \in  \mathbf{T}_{n}(\psi_n) \EEE\colon T\cap \Gamma_{i,l}^n\neq \emptyset\}\leq \frac{C}{\eps_n} \mathcal{H}^1(\Gamma_{i,l}^n).
\end{equation}
Once this is shown,  by summing over all   components $(\Gamma^n_{i,l})_l$, \EEE we get
\begin{equation}\label{firsttoprove}
\sum_{Q_i \in \mathcal{B}_{\rm bad}}\#\{T\in  \mathbf{T}_{n}(\psi_n)  \colon \,   (\Gamma_i^{n}\setminus K_n(t))    \cap T\neq \emptyset\}\leq \sum_{Q_i \in \mathcal{B}_{\rm bad}} \frac{C}{\eps_n}\mathcal{H}^1  \big( ( \Gamma_i^{n} \setminus K_n(t)) \cap Q_i   \big)\,.\end{equation} 
By using \JJJ \eqref{almost-recovery-seq} 
and \eqref{1301251232} \EEE we have $\limsup_{n\to \infty}\mathcal{H}^1(( \Gamma^n_i \EEE\setminus K_n(t))\cap B_{\rm bad})\leq C\theta$, which    concludes the proof of \EEE  \eqref{firsttoproveX}.
 
To prove \eqref{forcomponents}, we consider the   $\omega(\varepsilon_n) $-neighborhood of $ \Gamma^n_{i,l}$, i.e., $
 N_{\varepsilon_n}( \Gamma^n_{i,l}):=\{x \in \R^2\colon {\rm dist}(x, \Gamma^n_{i,l})\leq \omega(\varepsilon_n)\}\,$.  
 As \EEE $\mathcal{H}^1 ( \Gamma^n_{i,l})\geq c\eps_n$, an elementary geometric argument yields the existence of a universal constant $C>0$ such that 
\[ \big|N_{\varepsilon_n}( \Gamma^n_{i,l})\big|\leq C \, \omega(\varepsilon_n) \mathcal{H}^{1}( \Gamma^n_{i,l})\leq C \varepsilon_n \mathcal{H}^{1}( \Gamma^n_{i,l})\,, \]
 where we used \eqref{eq: 10.6}. \EEE Since $|T| \geq c\varepsilon_n^2$, we hence obtain  
    \[\#\big\{T \JJJ \in  \mathbf{T}_{n}(\psi_n) \EEE\colon \, \Gamma^n_{i,l}\cap T\neq \emptyset  \big\}\le \# \big\{T\JJJ \in  \mathbf{T}_{n}(\psi_n) \EEE\colon \, T\subset N_{\varepsilon_n}(\Gamma^n_{i,l})\big\} \leq C\frac{\big|N_{\varepsilon_n}( \Gamma^n_{i,l})\big|}{  \eps_n^{2}\EEE}\leq \frac{C}{\varepsilon_n}\mathcal{H}^{1}( \Gamma^n_{i,l}) \,,\]
    \JJJ which validates \eqref{forcomponents}. \EEE 
This concludes the proof. 
\end{proof}

\begin{proof}[Proof of Lemma \ref{lemma7.4}]
  We argue by contradiction and assume that there exists $T\in  \mathbf{T}^{\rm cr, g}_{n,k}(\psi_n) \EEE \setminus \mathbf{T}^{\rm crack}_{n,k} $ with $T \notin  \mathbf{T}'_{n} \EEE (\psi_n,Q_i)$ for all $Q_i\in \mathcal{B}_{\rm good}$.  In view of   \eqref{extraremark}, we have \EEE  $|e(\psi_n)_{T}|^2\geq \kappa/\eps_n \EEE$ for $T\in  \mathbf{T}^{\rm cr, g}_{n,k}(\psi_n) \EEE \setminus \mathbf{T}^{\rm crack}_{n,k} $. Since $ T  \notin \mathbf{T}'_{n}(\psi_n,Q_i)$, \MMM i.e., \EEE $T\cap ( J_{\phi_n} \EEE \cap Q_i)=\emptyset$, we can deduce  $\| \phi_n\|_{W^{1,\infty}(T)}\leq  C_\psi \EEE $ \VVV by \eqref{eq: Lipschitz1}.  \EEE  Now, consider vertices $x,x'\in \mathcal{V}(T)$   and recall that by definition $\psi_n(x)=\phi_n(x)$ and $\psi_n(x')=\phi_n(x')$. \EEE  We obtain   
\begin{equation}
    |(\nabla \psi_n)_{T}  (x-x')|=|\psi_n(x)-\psi_n(x')|=|\phi_n(x)-\phi_n(x')|\leq \|\nabla \phi_n\|_{L^{\infty}(T)}  |x-x'| \leq \JJJ C_{\psi} \EEE \eps_n \,.  
  \end{equation}
  Since this holds along all three edges of $T$,  there exists a constant $C>0$, such that $|e(\psi_n)_{T}| \leq|(\nabla \psi_n)_{T}|\leq {C}$, which for large $n\in \N$  contradicts the fact that $|e(\psi_n)_{T}|^2\geq \kappa/\eps_n $. Hence, \eqref{crucial-inclusion} holds.  
  \end{proof}
  
\begin{proof}[Proof of Lemma \ref{lemma: negli-jump}]
 We follow \EEE the argumentation in \cite[Lemma 5.2]{FriedrichSeutter},  which relies on the fact that the measure \MMM of \EEE jump points with small jump height is small, cf.\ \eqref{eq: Lipschitz2}. \EEE  Fix $T \in \mathbf{T}_n(\psi_n)\sm  \mathbf{T}_{n,k-1}^{\rm crack}(\psi_n)$ such that $T \cap \Gamma_i^n \neq \emptyset$, for some $Q_i \in \mathcal{B}_{\rm bad} \cup \mathcal{B}^{\rm small}_{\rm good}$. \MMM By \eqref{eq: Lipschitz2}{\rm (i)} we then get  that also $J_{\phi_n}\cap T\neq \emptyset$. \EEE Let $\tilde{x}\in T \cap J_{\phi_n}$ and choose \VVV $x\in T\cap  \overline{P_{i}^{n,+}} \EEE$ and $x'\in  T\cap \overline{P_i^{n,-}}\EEE$ (recall   \eqref{1301252256}  and \eqref{phiningood}) \EEE  in such a way that $\tilde{x}$ lies on the segment  between $x$ and $x'$, \EEE and $|x-x'|\geq c\eps_n$.  (If \JJJ $T \subset  \overline{P_{i}^{n,+}}$\EEE, we choose $x' = \tilde{x}$, and if $\JJJ T\subset  \overline{P_{i}^{n,-}}\EEE $, we choose $x = \tilde{x}$. \MMM In the following, we assume that $\tilde{x} \neq x,x'$. In the other case, the argument is similar using the one-sided traces of $\phi_n$ at $\tilde{x}$.) \EEE From $T\notin \mathbf{T}_{n,k-1}^{\rm crack}(\psi_n)$, we have that $ |e(\psi_n)_{T}| \leq \sqrt{\kappa/\eps_n}$.  Since $\phi_n$ is Lipschitz on $Q_i\sm \VVV \Gamma^n_i \EEE$ (see \eqref{eq: Lipschitz1} \JJJ and \eqref{1401251338})\EEE,  \EEE we get $|\phi_n(x)-\psi_n(x)|, |\phi_n(x')-\psi_n(x')|\leq  C_{\psi}\eps_n + C\sqrt{\eps_n} \EEE$  by \eqref{eq: 10.6}, \EEE 
and therefore
  \begin{equation*}
    \Big\langle \phi_n(x)-\phi_n(x'),  \frac{x-x'}{|x-x'|}  \Big\rangle\leq \Big\langle \psi_n(x)-\psi_n(x'), \frac{x-x'}{|x-x'|} \Big\rangle +  C\sqrt{\eps_n} \EEE \leq \Big\langle e(\psi_n)_{T}\cdot (x-x'), \frac{x-x'}{|x-x'|}\Big\rangle+   C\sqrt{\eps_n}\,. \EEE 
  \end{equation*}  
By this fact and $|x-x'|\leq  \omega(\eps_n)  \leq C \eps_n$ (\JJJ see \eqref{eq: 10.6}\EEE), we hence can estimate
 \begin{equation}\label{jumpsize-esti1}
  \Big\langle \phi_n(x)-\phi_n(x'),  \frac{x-x'}{|x-x'|} \Big\rangle \leq |e(\psi_n)_{T}| \, \eps_n +  C\sqrt{\eps_n} \EEE \leq  C\sqrt{\eps_n}. \EEE
 \end{equation}
On the other hand,   by the Fundamental Theorem of Calculus we have that 
\begin{equation}\label{jumpsize-esti2}
  \Big\langle \phi_n(x)-\phi_n(x'), \frac{x-x'}{|x-x'|} \Big\rangle = \int_{0}^{1} \Big\langle \nabla \phi_n(x+ s(x'-x)) \cdot (x'-x), \frac{x'-x}{|x-x'|} \Big\rangle\, {\rm d}s + \Big\langle [\phi_n(\tilde{x})],\frac{x'-x}{|x-x'|}\Big\rangle .
\end{equation}
  By  \eqref{eq: Lipschitz1}  \EEE   and $|x-x'|\le C\eps_n$,    the first term on the right-hand side \EEE is of order $\eps_n$. Putting together \eqref{jumpsize-esti1} and \eqref{jumpsize-esti2} we thus obtain
\begin{equation*}
  \Big\langle [\phi_n(\tilde{x})],\frac{x'-x}{|x-x'|}\Big\rangle \leq C   \sqrt{ \,\eps_n} \EEE + C\eps_n   \leq C \sqrt{\eps_n}. \EEE
\end{equation*}
Letting $\nu_{1}\defas \frac{x'-x}{|x'-x|}$  this means $\langle [\phi_n(\tilde{x})], \nu_1\rangle \leq C\sqrt{\eps_n}$. Now we can repeat the procedure for different $\hat{x}$ and $\hat{x}'$, where \MMM the segment between $\hat{x}$ and $\hat{x}'$ \EEE intersects $J_{\phi_n}$ also exactly at $\tilde{x}$ and $\nu_{2}\defas \frac{\hat{x}-\hat{x}'}{|\hat{x}-\hat{x}'|}$ satisfies $\langle \nu_1,\nu_2 \rangle \ge c$ for a universal $c>0$. We then have $\langle [\phi_n(\tilde{x})], \nu_{i}\rangle \leq C\sqrt{\eps_n}$ for $i=1,2$, which leads to $|[\phi_n(\tilde{x})]|\leq C\sqrt{\eps_n}$. 

We consider
$N_{\varepsilon_n}(T):=\{x \in \R^2\colon {\rm dist}(x, T \EEE )\leq \varepsilon_n\}\,,$ where we use $\|\nabla \phi_n\|_{L^{\infty}(N_{\eps_n}(T)\sm J_{\phi_n})}\leq \MMM C_\psi \EEE $ to find
\begin{equation*}
  |[\phi_n(z)]|\leq C  \sqrt{\eps_n} \EEE + C \MMM C_\psi \EEE  \eps_n \quad \text{for all} \; z\in J_{\phi_n}\cap N_{\eps_n}(T)\,.
\end{equation*}  
(In fact, $\tilde{x}$ and each $z$ can be connected by two different curves on different components of  $N_{\varepsilon_n}(T) \setminus \overline{J_{\phi_n}}$ of length $\sim \eps_n$.) As $\mathcal{H}^1(J_{\phi_n} \cap N_{\varepsilon_n}(T))$ is at least $\eps_n$, we get for $\eps_n$ small enough depending on $\theta$ 
$${\mathcal{H}^1\big( \lbrace  z \EEE \in  J_{\phi_n} \cap N_{\varepsilon_n}(T)\colon   |[\phi_n](  z \EEE )|  \le c_\theta  \rbrace \big) = \mathcal{H}^1(J_{\phi_n} \cap N_{\varepsilon_n}(T))\ge \eps_n.} $$ 
Summing  over all $T$ \VVV of the form described above, \EEE  observing that each neighborhood $N_{\varepsilon_n}(T)$ intersects only a bounded number of neighborhoods $N_{\varepsilon_n}(T')$,  $T'\neq T$,  and using \eqref{eq: Lipschitz2} we conclude \EEE that \eqref{1501251053} holds.  
\end{proof}

 \section*{Acknowledgements} 
 This research was funded by the Deutsche Forschungsgemeinschaft (DFG, German Research Foundation) - 377472739/GRK 2423/2-2023.  
 
\VIT   VC is member of the Gruppo Nazionale per l'Analisi Matematica, la Probabilit\`a e le loro Applicazioni (GNAMPA) of the Istituto Nazionale di Alta Matematica (INdAM) and acknowledges the financial support of PRIN 2022J4FYNJ
 `Variational methods for stationary and evolution problems with singularities and interfaces', PNRR Italia Domani, funded by the European Union under NextGenerationEU, CUP B53D23009320006, and of Sapienza Università di Roma through a SEED PNR Project.

\EEE

\end{document}